\documentclass[11pt,leqno]{amsart}
   
\usepackage[top=2.5 cm,bottom=2 cm,left=2.5 cm,right=2 cm]{geometry}
\usepackage[utf8]{inputenc}
\usepackage[dvipsnames]{xcolor}
\usepackage{tikz}
\usepackage{pgfplots}
\usepackage{footmisc}
\pgfplotsset{compat=newest}
\usepgfplotslibrary{fillbetween}
\usetikzlibrary{patterns}
\usepackage{bbm}
\usepackage{esint}
\usepackage{graphicx}
\usepackage{hyperref}
\usepackage{color}
\usepackage{amsfonts}
\usepackage{psfrag}
\newcounter{stepnb}

\newtheorem{theorem}{Theorem}[section]
\newtheorem{lemma}[theorem]{Lemma}

\newtheorem{corol}[theorem]{Corollary}
\newtheorem{hyp}{Hypothesis}

\newtheorem{definition}{Definition}[section]
\newtheorem{remark}{Remark}[section]
\numberwithin{equation}{section}

\newcommand{\R}{\mathbb{R}}

\newcommand{\unpo}{\mathcal{O}(1)}

\newcommand{\ee}{\varepsilon}
\newcommand{\mf}{\mathbf}

\usepackage{pgfplots}

\newcommand{\be}{\begin{equation}}
\newcommand{\eq}{\end{equation}}

\newcommand{\loc}{\mathrm{loc}}

\begin{document}

\title[]{Existence of vanishing physical viscosity solutions  of characteristic initial-boundary value problems for systems of conservation laws}

\author[F.~Ancona]{Fabio Ancona}
\address{F.A. Dipartimento di Matematica ``Tullio Levi-Civita", Universit\`a degli Studi di Padova, Via Trieste 63,
35131 Padua, Italy}
\email{ancona@math.unipd.it}
\author[A.~Marson]{Andrea Marson}
\address{A.M. Dipartimento di Matematica ``Tullio Levi-Civita", Universit\`a degli Studi di Padova, Via Trieste 63,
35131 Padua, Italy}
\email{marson@math.unipd.it}
\author[L.~V.~Spinolo]{Laura V.~Spinolo}
\address{L.V.S. IMATI-CNR, via Ferrata 5, I-27100 Pavia, Italy.}
\email{spinolo@imati.cnr.it}
\maketitle
{
\rightskip .85 cm
\leftskip .85 cm
\parindent 0 pt
\begin{footnotesize}

{\sc Abstract.}
We consider initial boundary-value problems for nonlinear systems of conservation laws in one space variable. It is known that in general different viscous mechanisms yield different solutions in the zero-viscosity limit. Here we focus on the most technically demanding case, known as boundary characteristic case, which occurs when one of the characteristic velocities of the system vanishes. We work in small total variation regimes and assume that every characteristic field is either genuinely nonlinear or linearly degenerate. We 
establish existence of \emph{admissible solutions satisfying a boundary condition consistent with the vanishing viscosity approximation} given by a large class of physical (that is, mixed hyperbolic-parabolic) systems. In particular, our results apply to the zero-viscosity limit of the Navier-Stokes and viscous MHD equations, written in both Eulerian and Lagrangian coordinates.  Our analysis relies on a fine boundary layers analysis and is based on the introduction of a new wave front-tracking algorithm. From the technical viewpoint, the most innovative elements are i) a new class of interaction estimates for boundary layers and boundary characteristic wave fronts hitting the boundary, which yields the introduction of a new Glimm-type functional; ii) a detailed analysis of the behavior of the wave front-tracking algorithm close to the boundary, which in turn yields relevant information on the limit. 
\\

\medskip\noindent
{\sc Keywords:} systems of conservation laws, initial-boundary value problems, hyperbolic systems, wave front-tracking, boundary characteristic case, mixed hyperbolic-parabolic systems, boundary layers 

\medskip\noindent
{\sc MSC (2010):  35L65, 35B30}

\end{footnotesize}
}

\section{Introduction} 
We consider initial-boundary value problems for the nonlinear system of conservation laws 
\begin{equation} 
\label{e:claw}
       \mf g (\mf v)_t + \mf f(\mf v)_x = 0, \quad (t, x)\in \R_+ \times \R_+ ,
\end{equation}
where the unknown $\mf v$ attains values in $\R^N$ and $\mf g, \mf f: \R^N \to \R^N$ are smooth functions satisfying suitable assumptions that we discuss in the following. Note that~\eqref{e:claw} can be \emph{formally} obtained as the $\ee \to 0^+$ limit of the viscous approximation
\begin{equation} 
\label{e:vclaw}
       \mf g (\mf v^\ee)_t + \mf f(\mf v^\ee)_x = \ee \Big( \mf D(\mf v^\ee) \mf v^\ee_x \Big)_x, \quad \mf v \in \R^N, 
\end{equation}
where the function $\mf D$ attains values in the space of $N \times N$ matrices and is positive semi-definite. We postpone the discussion on the precise assumptions satisfied by $\mf D$, here we just  mention that we impose conditions, introduced in the fundamental paper by Kawashima and Shizuta~\cite{KawashimaShizuta1}, that are satisfied by several relevant physical systems like the Navier-Stokes equations and the viscous magneto-hydrodynamics (MHD) equations. In particular, we tackle the case of a singular (that is, non-invertible) matrix $\mf D$, which is the most relevant from the physical viewpoint but involves severe technical challenges. In particular, owing to the singularity of $\mf D$, the initial-boundary value problem for~\eqref{e:vclaw} is in general overdetermined if one imposes a \emph{full} boundary condition like $\mf v^\ee (t, 0) = \mf v_b (t)$. In the present work we regard the initial-boundary value problem for~\eqref{e:claw} as limit of~\eqref{e:vclaw} coupled with the initial and boundary conditions 
\be 
\label{e:datavclaw}
   \mf v^\ee(0, \cdot) = \mf v_0, \qquad \boldsymbol{\widetilde \beta} (\mf v^\ee (\cdot, 0), \mf v_b) = \mf 0_{N},
\eq
where $\mf v_0, \mf v_b: \R_+ \to \R$ are functions with finite total variation and the function $\boldsymbol{\widetilde \beta} $ is defined in \S\ref{ss:boucon}. 
Very loosely speaking, the basic idea underpinning the construction of $\boldsymbol{\widetilde \beta}$ is that we impose a full boundary condition on the \emph{parabolic} component of $\mf v^\ee$, whereas on the \emph{hyperbolic} component we impose boundary conditions along the characteristic fields entering the domain. We refer to \S\ref{ss:boucon} and Remark~\ref{r:beta} in \S\ref{s:hyp} for a more detailed discussion.

The analysis of the inviscid limit of~\eqref{e:vclaw} is particularly relevant for the analysis of initial-boundary value problems for~\eqref{e:claw} since it is known that (even in the most elementary linear case) the limit in general depends on $\mf D$, see~\cite{Gisclon}. This marks a fundamental difference with the case of the Cauchy problem, where in the conservative case the limit does not depend on the viscosity matrix, see~\cite{Bianchini}. Loosely speaking this difference is due to the fact that in the initial-boundary value problem the transient behavior in proximity of the domain boundary is described by the so-called \emph{boundary layers}, which are steady solutions of~\eqref{e:vclaw} and as such strongly affected by $\mf D$. Remarkably, the fact 
that solutions of~\eqref{e:claw} depend on the underlying viscous mechanism has also relevant consequences from the numerical viewpoint, see~\cite{MishraSpinolo}. Note that establishing the convergence of~\eqref{e:vclaw} to~\eqref{e:claw} is presently an open problem, even in the case of the Cauchy problem. In the case of initial-boundary value problems, partial results under restrictive assumptions are established for instance in~\cite{AnconaBianchini,ChenFrid,Gisclon,GrenierRousset,JosephLeFloch,Rousset,Spinolo}. See also~\cite{Serre1,Serre2} for a general introduction to hyperbolic initial-boundary value problems.

 To the best of our knowledge, the present paper provides the first global in time existence result for solutions of initial-boundary value problems for~\eqref{e:claw} satisfying a \emph{boundary condition that is consistent with the viscous mechanism~\eqref{e:vclaw}} in the boundary characteristic case, see Remark~\ref{r:bc} below. The boundary characteristic case occurs when one of the eigenvalues of the jacobian matrix of $\mf f$ vanishes and it is known to be particularly challenging since the waves of the characteristic family can have a fairly complex behavior and can alternatively enter the domain, leave it, or even be tangential to the boundary.  

We now provide the definition of the boundary condition we impose on~\eqref{e:claw} and point out that in the following we require that every eigenvector field of the jacobian matrix of the flux $\mf f$ of~\eqref{e:claw} is either genuinely nonlinear or linearly degenerate. 
\begin{definition}
\label{d:equiv} Given system~\eqref{e:vclaw} and $\mf {\bar v}$, $\mf{ v}_b \in \R^N$, we say that ``$\mf{\bar v} \sim_{\mf D} \mf{ v}_b$'' if there is 
$\mf{\underline v} \in \R^N$ such that the following conditions are both satisfied:
\begin{itemize}
\item[i)] $\mf f(\mf{\bar v}) = \mf f(\mf{\underline v})$ and the $0$-speed discontinuity between  $\mf{\bar v}$ (on the right) and $\mf{\underline v}$ (on the left) 
satisfies the Lax admissibility condition;
\item[ii)] there is a so-called ``boundary layer'' $\mf w: \R_+ \to \R^N$ such that
\be
\label{e:bl}
\left\{
\begin{array}{ll}
          \mf D(\mf w) \mf w' = \mf f(\mf w) - \mf f(\mf{\underline v}) \\
           \boldsymbol{\widetilde \beta} (\mf w(0), \mf v_b) =\mf 0, \quad \lim_{y \to + \infty} \mf w(y)=\mf{\underline v}. 
\end{array}
\right.
\eq
\end{itemize}
\end{definition}
Some remarks are here in order. First, the heuristic meaning of the above definition is that $\mf{\bar v} \sim_{\mf D} \mf{ v}_b$ if $\mf v_b$ is connected to $\mf{\bar v}$ by (either or both) a boundary layer and an admissible $0$-speed discontinuity. Second, we refer to~\cite{Dafermos,Lax} for the definition of Lax admissibility condition. Third, assume in condition i) that $\mf{\bar v}$ is close to $\mf{\underline v}$ and that the boundary is not characteristic, that is all the eigenvalues of the jacobian matrix of $\mf f$ are bounded away from $0$; then owing to the Local Invertibility Theorem the condition $\mf f(\mf{\bar v}) = \mf f(\mf{\underline v})$ yields $\mf{\bar v}= \mf{\underline v}$ and the Lax admissibility requirement is trivially satisfied. Fourth, 
contact discontinuities are always Lax admissible and hence the Lax admissibility requirement in condition i) is redundant if  the boundary is characteristic and the boundary characteristic field (that is, the field associate to the eigenvalue that vanishes) is linearly degenerate. Summing up, in the small total variation framework we consider in the following the Lax admissibility requirement in condition i) is only relevant in the boundary characteristic case when the boundary characteristic field is genuinely nonlinear. Finally, we remark in passing that there is a huge literature concerning the analysis of boundary layers for nonlinear systems of conservation laws. Here we only refer to the seminal papers~\cite{Rousset,SerreZumbrun,Xin} and to the references therein for a more extended discussion.

We can now pose the initial-boundary problem for~\eqref{e:claw} by assigning the initial and boundary conditions 
\be
\label{e:ibvp}
         \mf v(0, \cdot) = \mf v_0, \qquad 
         \mf v(\cdot, 0) \sim_{\mf D} \mf v_b          
\eq
Note that another famous way of assigning the boundary condition on~\eqref{e:claw} is discussed in~\cite{DuboisLeFloch} and used in~\cite{Amadori} to formulate the initial-boundary value problem and establish global-in-time existence results in the boundary characteristic case. Note however that the boundary condition given in~\eqref{e:ibvp} differs from the one in~\cite{DuboisLeFloch}, even in the simple case of a linear systems with an invertible viscosity matrix $\mf D$, provided $\mf D$ does not coincide with the identity (see also~\cite{Gisclon} and Remark~\ref{r:bc} below for a more detailed discussion about this point). This implies in particular that the in the boundary characteristic case the solution given by Theorem~\ref{t:main} below is in general different from the one in~\cite{Amadori}.  In all the cases where the convergence of~\eqref{e:claw} to~\eqref{e:vclaw} has been effectively established (see for instance~\cite{AnconaBianchini,Rousset}) the limit satisfies~\eqref{e:ibvp}.  

We can now state our main result. The precise formulation of the hypotheses requires some heavy notation and is therefore postponed to \S\ref{s:hyp}. We now provide instead an heuristic discussion: Hypothesis~\ref{h:normal} states that there is a change of variables $\mf v \leftrightarrow \mf u$ such that, up to the rescale $(t, x)  \leftrightarrow (t/\ee, x/\ee)$,~\eqref{e:vclaw} can be rewritten as
\be \label{e:symmetric}
    \mf E (\mf u) \mf u_t + \mf A (\mf u) \mf u_x = \mf B(\mf u) \mf u_{xx} + \mf G(\mf u, \mf u_x) \mf u_x, 
\eq
with~\eqref{e:symmetric} \emph{in the normal form} in the Kawashima-Shizuta sense~\cite{KawashimaShizuta1}. Hypothesis~\ref{h:ks} is the so-called Kawashima-Shizuta condition, which very loosely speaking is a coupling condition between the hyperbolic and the parabolic component of~\eqref{e:symmetric}. Hypothesis~\ref{h:sh} is a standard strict hyperbolicity assumption, whereas Hypothesis~\ref{h:ldgnl} states that every eigenvector field of~\eqref{e:symmetric} is either genuinely nonlinear or linearly degenerate. Finally, Hypothesis~\ref{h:eulerlag} collects the conditions introduced in~\cite{BianchiniSpinoloARMA,BianchiniSpinolo} to enusure that we can write the boundary layers equation for~\eqref{e:symmetric} in an explicit form (see the analysis in \S\ref{s:bl}). In the statement of the following theorem, $BV$ denotes the space of bounded total variation functions, see~\cite{AmbFuPal}. 
\begin{theorem}
\label{t:main} Assume Hypotheses $1, \cdots, 5$ in \S\ref{s:hyp} and fix $\mf v^\ast \in \R^N$; then there is a constant $\delta^\ast>0$ only depending on the  functions $\mf g$, $\mf f$, $\mf D$ in system~\eqref{e:vclaw}, on the change of variables $\mf v \leftrightarrow \mf u$ and on $\mf v^\ast$, such that the following holds. If $\mf v_0, \mf v_b \in BV(\R_+)$ satisfy
\be \label{e:hp}
     \mathrm{TotVar} \ \mf v_0 + \mathrm{TotVar} \ \mf v_b + |\mf v_0 (0^+) - \mf v_b (0^+)| \leq \delta^\ast, \quad |\mf v_0 (0^+) - \mf v^\ast| \leq \delta^\ast
\eq
there is a global-in-time Lax admissible distributional solution $\mf v \in BV_{\mathrm{loc}} (\R_+ \times \R_+)$ of~\eqref{e:claw} satisfying the initial and boundary\footnote{Note that, since $v \in BV_{\mathrm{loc}} (\R_+ \times \R_+)$, the trace $\mf v(\cdot, 0)$ is a well-defined, locally summable function, see~\cite{AmbFuPal}.} conditions~\eqref{e:ibvp}. Also, if system~\eqref{e:claw} admits a convex entropy then the solution we construct is  entropy admissible.   
\end{theorem}
In~\eqref{e:hp} we denote by $\mf v_0 (0^+)$ and $\mf v_b (0^+)$ the right limits at $x=0$ and $t=0$ of the functions $\mf v_0$ and $\mf v_b$, respectively; these limits exist because the functions have bounded total variation.  We now comment on Theorem~\ref{t:main}. 
\begin{itemize}
\item The small total variation assumption is fairly standard for the analysis of one-dimensional system of conservation laws, see~\cite{Bressan,Dafermos}. Note that the counter-examples in~\cite{BCZ2,BCZ,Jen}  imply that this assumption is necessary to obtain global-in-time existence results that apply to general systems of conservation laws. The examples in~\cite{BCZ2,BCZ,Jen} are technically speaking Cauchy problems, but can be easily adapted to the initial-boundary value problem owing to finite propagation speed. 
\item Theorem~\ref{t:main} is an \emph{existence} result. We are confident that one could establish stability and \emph{uniqueness} results by combining the analysis in the present paper with the Standard Riemann Semigroup approach \emph{\'a la} Bressan, see~\cite{Bressan}. However, since the proof of Theorem~\ref{t:main} is already long and technical this is left for future work. 
\item We refer to \S\ref{ss:roadmap} below for a detailed outline of the proof of Theorem~\ref{t:main}, here we point out that it relies on the introduction of a new wave front-tracking algorithm, which involves both wave fronts and boundary layers. As the limit of a wave front-tracking algorithm, the solution constructed in the proof of Theorem~\ref{t:main} enjoys several further regularity properties, see for instance~\cite[Theorem 10.4]{Bressan}. In particular, it satisfies the Lax admissibility condition along shock curves. 
\item The main novelty of Theorem~\ref{t:main} is that we deal with the boundary characteristic case, that is we take into account the possibility that one eigenvalue of the matrix $\mf E^{-1} \mf A$ in~\eqref{e:symmetric} attains the value $0$. In the non characteristic boundary case (i.e., when all the eigenvalues are bounded away from $0$), Goodman and Sabl\'e-Tougeron~\cite{Goodman,SableT} established global-in-time existence results by relying on a modified Glimm scheme~\cite{Glimm}, whereas Amadori~\cite{Amadori} used a wave front-tracking algorithm. The analysis in~\cite{Amadori} covers the boundary characteristic case, but as mentioned before the boundary condition in~\cite{Amadori} is different from~\eqref{e:ibvp}, see Remark~\ref{r:bc} below for a more detailed discussion. The fact that we use~\eqref{e:ibvp} accounts for severe technical challenges that we touch upon in Remark~\ref{r:bc}. As mentioned before, the treatment of the boundary characteristic case is much more complicated than the non-characteristic, for reasons that we now try to explain. At the heuristic level, in the boundary characteristic case there is a deep interplay between  the  behavior of the solution \emph{inside} the domain and the dynamics of boundary layers, which describe the transient behavior from~\eqref{e:vclaw} to~\eqref{e:claw} at the boundary. At a more technical level, one has to take into account the fact that boundary layers may have a component lying on a center manifold, which makes the analysis of their dynamics more delicate. To conclude, we point out that the analysis in~\cite{AnconaBianchini} gives among other things a proof of Theorem~\ref{t:main} in the case where the matrix $\mf D$ in~\eqref{e:vclaw} is the identity. Note however that in most of the physically relevant cases $\mf D$ is not even invertible. 
\item Fix $\mu \in \R$; by applying the change of variables $x  \leftrightarrow x - \mu t$ the analysis in the present paper immediately extends to the case where the boundary condition is posed at the boundary $x = \mu t$ rather than at $x=0$. We are confident that our analysis could also be extended to the case of a Lipschitz continuous moving boundary $x = \mu(t)$, but to ease the exposition we do not deal with this case. 
\item As it is fairly standard (see~\cite{Bressan,Dafermos,Serre1,Serre2}) in the study of one-dimensional system of conservation laws, in the statement of Theorem~\ref{t:main} we assume that every vector field is either genuinely nonlinear or linearly degenerate. This assumption is satisfied in most of the physicallly relevant examples, see \S\ref{ss:appl}, and eases the analysis as it simplifies the structure of the solution of the Riemann problem and of the so-called boundary Riemann problem. In the case of the Cauchy problem, the analysis  in~\cite{AnconaMarson} defines a wave front-tracking algorithm for general strictly hyperbolic system, with no genuine nonlinearity or linear degeneracy assumption. The extension of the analysis in the present work to the case of general systems is in principle possible, but it is likely to require heavy technicalities. 
\end{itemize}
\begin{remark} \label{r:bc}
We now compare our analysis with the one in the paper~\cite{Amadori} by Amadori: the main differences stem from the different way of assigning the boundary condition in the boundary characteristic case. We use Definition~\ref{d:equiv} and~\eqref{e:ibvp}, whereas Amadori follows~\cite{DuboisLeFloch} and introduces condition {\bf (C)} in~\cite[p.2]{Amadori}. To highlight the main differences between~\eqref{e:ibvp} and condition {\bf (C)} let us focus on the simplest case and assume that $\mf f$, $\mf g$ and $\mf D$ are such that one can assign a full boundary condition on~\eqref{e:vclaw} and~\eqref{e:bl}, that is $\boldsymbol{\widetilde \beta}(\mf v, \mf v_b) = \mf v - \mf v_b$. Condition 
 {\bf (C)} requires that for a.e. $t\in \R_+$ the trace $\mf v(t, 0)$ and the value $\mf v_b(t)$ are connected by waves (shocks, rarefactions or contact discontinuities) with nonpositive speed. Instead, in~\eqref{e:ibvp} we require that $\mf v(t, 0)$ and $\mf v_b(t)$  are connected by a $0$-speed shock or contact discontinuity and by a \emph{boundary layer} described by system~\eqref{e:bl}.  
To see that these two conditions do not coincide, we consider the linear case where up to a rescaling of the independent variables~\eqref{e:vclaw} boils down to  
$$
    \mf v_t + \mf F \mf v_x = \mf D \mf v_{xx}, \qquad \mf F, \mf D \in \mathbb{M}^{N \times N}
$$
and assume for simplicity that $\mf D$ is invertible and that $\mf F$ and $\mf D$ are symmetric. Given $\mf v_b \in \R^N$, condition {\bf (C)} dictates that $\mf v (t, 0) - \mf v_b$ lies in the eigenspace of $\mf F$ associated to nonpositive eigenvalues, whereas condition~\eqref{e:ibvp} is satisfied if and only if $\mf v (t, 0) - \mf v_b$ lies in the eigenspace of $\mf D^{-1} \mf F$ associated to nonpositive eigenvalues. Since in general 
$\mf F$ and $\mf D^{-1} \mf F$ have different eigenspaces, these two boundary conditions are different. Note in particular that condition~\eqref{e:ibvp} \emph{depends} on the viscosity matrix $\mf D$, whereas condition {\bf (C)} is unaffected by $\mf D$. 

Wrapping up, the main difference between condition {\bf (C)} and Definition~\ref{d:equiv} is that condition {\bf (C)} does not provide any information of the underlying viscous mechanism, whereas Definition~\ref{d:equiv} takes into account the relevant information concerning the transient behavior from the viscous approximation to the hyperbolic limit. 
This implies that, contrary to~\cite{Amadori}, our analysis of the behavior of the approximate solutions near the domain boundary has to handle the interaction between an \emph{hyperbolic} component given by rarefaction, shock and contact discontinuity 
waves and a \emph{parabolic} (or, more precisely, mixed hyperbolic-parabolic) component given by the boundary layers. In particular, 
we use a different solution of the so-called boundary Riemann problem, see \S\ref{s:rie}, which makes it much harder to obtain useful information (the so-called interaction estimates) on the behavior of the approximate solution near the domain boundary,
see for instance the fairly involved proof of Lemma~\ref{l:me}. 
\end{remark}
\begin{remark}\label{r:bcnc}
The non characteristic boundary case occurs when all the eigenvalues of the jacobian matrix $\mf D \mf f$ are bounded away from $0$. In this case the analysis in~\cite{Amadori} provides a proof of Theorem~\ref{t:main}; indeed, in the non characteristic case the boundary datum is assigned by the condition {\bf(NC)} in ~\cite[p.2]{Amadori}, which is consistent with~\eqref{e:ibvp}, namely boils down to~\eqref{e:ibvp} through a suitable choice of the function $b$. 
\end{remark}
\subsection{Paper outline} For the reader's convenience we provide a roadmap of the proof of Theorem~\ref{t:main} in \S\ref{ss:roadmap}. Also, in \S\ref{ss:notation} we collect the main notation used in the present paper.  In \S\ref{s:hyp} we precisely state the hypotheses of Theorem~\ref{t:main} and check that they are satisfied by the Navier-Stokes and viscous MHD equations, written in both Eulerian and Lagrangian coordinates. In \S\ref{s:rie} we recall the solution of the Riemann and of the so-called boundary Riemann problem, which constitute the building block for the definition of the wave front-tracking algorithm. In \S\ref{s:cinque} we define our wave front-tracking algorithm. Section \ref{s:ie} contains some of the most technically demanding and innovative results of the paper in the form of so-called interaction estimates. In \S\ref{s:functional} we provide uniform bounds on the total variation of the wave front-tracking approximation through the introduction of suitable (and new) functionals. In \S\ref{s:otto} we conclude the proof of the convergence of the wave front-tracking approximation. Finally, \S\ref{s:bc} contain the other most innovative part of the present paper as in that section we show that any limit of our wave front-tracking approximation satisfies the boundary condition in~\eqref{e:ibvp}. The proof is rather technical and involved, and unveils what we feel is interesting information on the limit solution at the domain boundary, see for instance Lemmas~\ref{l:tildeE},\ref{l:bl},\ref{l:bl2} and~\ref{l:tracciainterna}.

\subsection{Proof roadmap} \label{ss:roadmap}
The proof of Theorem~\ref{t:main} relies on the introduction of a new wave front-tracking algorithm. Wave front-tracking algorithms stem from the celebrated Glimm~\cite{Glimm} scheme and were introduced in the pivotal work by Dafermos~\cite{Dafermos72}. We refer to~\cite{Bressan,HoldenRisebro} for an extended overview. In the present paper we build upon the algorithm discussed in~\cite{Bressan}. 

Given $\ee>0$, we recall~\cite[Definition 7.1]{Bressan} of $\ee$-approximate solution of the Cauchy problem: heuristically speaking, it is a piecewise constant function with discontinuities occurring along finitely many straight lines in the $(t, x)$ plane, which solves equation~\eqref{e:claw} and attains the initial datum up to an error ``of size $\ee$". Also, the Lax admissibility condition is satisfied, again up to an error of size $\ee$. 
\begin{definition} \label{d:appsol}
Given $\ee>0$, we say that $\mf v_\ee: \R_+ \times \R_+ \to \R^N$ is an $\ee$-wave front tracking approximation of the initial-boundary value problem \eqref{e:claw},\eqref{e:ibvp} if it is an $\ee$-approximate solution in the sense of~\cite[Definition 7.1]{Bressan} and furthermore 
$$
    \| \mf v_\ee(0, \cdot) - \mf v_0 \|_{L^1} \leq \ee, \quad \mf v_\ee(t, 0) \sim_{\mf D} \mf{\widetilde v}_{\ee b} (t) \; \text{for a.e $t$ and for some $\mf{\widetilde v}_{ \ee b}  \in L^1_{\mathrm{loc}}(\R_+)$ s.t.} \; 
    \| \mf{\widetilde v}_{\ee b} - \mf v_b \|_{L^1} \leq \ee. 
$$
\end{definition}
Theorem~\ref{t:main} directly follows from Theorem~\ref{t:wft} below. 
\begin{theorem}\label{t:wft}
Under the same assumptions as in Theorem~\ref{t:main} for every $\ee>0$ there is an $\ee$-wave front tracking approximation of the initial-boundary value problem~\eqref{e:claw},~\eqref{e:ibvp}, which furthermore can be constructed in such a way that it satisfies the following properties. For every vanishing sequence $\{ \ee_m\}$ there is a subsequence (which we do not relabel) such that the corresponding sequence $\mf v_{\ee_m}$  satisfies 
\be \label{e:16bis}
     \mf v_{\ee m} (t, \cdot) \to  \mf v (t, \cdot)  \quad \text{in $L^1_\loc( \R_+)$ as $m\to + \infty$ for every $t>0$}.
\eq
In the above formula $\mf v \in BV ([0, T] \times \R_+)$ for every $T>0$, is a Lax admissible distributional solution of~\eqref{e:claw} and satisfies the initial and boundary condition~\eqref{e:ibvp} in the sense of traces. 
\end{theorem}
As pointed out before, as a limit of a wave front-tracking approximation, the solution $\mf v$ actually enjoys further regularity properties, see the analysis in~\cite[Ch.10]{Bressan}. Note furthermore that we state that the convergence of the wave front-tracking approximation occurs up to subsequences because technically speaking we have not established uniqueness results for solutions of the initial-boundary value problem~\eqref{e:claw},\eqref{e:ibvp}. As mentioned above, it is highly likely that one could establish uniqueness of a suitable boundary Standard Riemann Semigroup and this would imply that the whole family $\mf v_{\ee}$ converges as $\ee \to 0^+$. Finally, we recall Remark~\ref{r:bcnc} and conclude that the analysis in~\cite{Amadori} provides a proof of Theorem~\ref{t:wft} in the non characteristic boundary case. Hence, in the following we focus on the characteristic boundary case, namely we assume that one eigenvalue of $\mf D \mf f$ may attain the value $0$, see equation~\eqref{e:uast} for a more precise definition.  
      
The building blocks of our wave front-tracking algorithm are the solution of the Riemann problem, which is a Cauchy problem where the initial datum has a single jump discontinuity at the origin, and of the so-called boundary Riemann problem, which is the initial-boundary value problem~\eqref{e:claw},\eqref{e:ibvp} in the case where the data $\mf v_0$ and $\mf v_b$ are (in general, distinct) constant values in $\R^N$. An admissible solution of the Riemann problem is constructed in the celebrated work by Lax~\cite{Lax}, whereas the solution of the boundary Riemann problem is discussed in~\cite{AnconaBianchini,BianchiniSpinoloARMA,BianchiniSpinolo}. Note that~\cite{BianchiniSpinoloARMA,BianchiniSpinolo} deal with general hyperbolic systems, but the fact that we assume that the boundary characteristic field is either genuinely nonlinear or linearly degenerate leads to considerable simplifications that we use in our analysis. In \S\ref{s:rie} we provide a quite detailed overview of~\cite{BianchiniSpinoloARMA,BianchiniSpinolo} in the specific frameworks of genuinely nonlinear and linearly degenerate boundary characteristic field, and for the reader's convenience we outline the main ideas and the basic steps in the construction in \S\ref{ss:speroserva}. Note that the starting point of the construction of the solution of the boundary Riemann problem is the boundary layer equation at the first line of~\eqref{e:bl}, and indeed in \S\ref{s:bl} we have to spend some work to write this equation in a sufficienly explicit form, a nontrivial task owing to the singularity of the matrix $\mf D$. Note furthermore that a key point for the construction of the solution of the boundary Riemann solver is the definition of the so-called \emph{characteristic wave fan curve of admissible states} $\boldsymbol{\zeta}_k$. Loosely speaking, in the case of a genuinely nonlinear boundary characteristic vector field, the curve $\boldsymbol{\zeta}_k(\mf{\widetilde u}, \cdot)$ contains all the states that can be connected to the value $\mf{\widetilde u} \in \R^N$ by rarefactions and Lax admissible shocks with nonnegative speed, and by boundary layers lying on a suitable center manifold. In the case of a linearly degenerate boundary characteristic vector field, the curve $\boldsymbol{\zeta}_k(\mf{\widetilde u}, \cdot)$ contains all the states that can be connected to the value $\mf{\widetilde u} \in \R^N$ by contact discontinuities with nonnegative speed, and by boundary layers lying on a suitable center manifold. We refer to Lemma~\ref{l:cwfc} and \S\ref{ss:case2} for the detailed analysis.

After this preliminary analysis, in \S\ref{s:cinque} we define our wave front-tracking algorithm. The construction itself is fairly standard, but for the reader's convenience we overview it here. As a first step,  we approximate the initial and boundary data with piecewise constant functions with finitely many discontinuities. At each jump discontinuity of the initial datum we construct an approximate solution of the Riemann problem  through an algorithm termed \emph{accurate Riemann solver}, which yields a piecewise constant solution in the $(t, x)$ plane with finitely many jump discontinuities along straight lines (the wave fronts after which the whole algorithm is named) emanating from the origin. At each jump discontinuity point of the boundary datum we instead use an algorithm that we term \emph{accurate boundary Riemann solver.} We then juxtapose the solutions of the Riemann problems and consider the first \emph{interaction time}, i.e. the first time at which either two discontinuity lines interact (that is, they intersect),  or a discontinuity line reaches the boundary. In the first case, we solve the corresponding Riemann problem, in the second the corresponding boundary Riemann problem. We then extend the solution up to the next interaction time, solve the Riemann or boundary Riemann problem, and so on. To rule out the possibility that the total number of discontinuity lines blows up in finite time we introduce, as in~\cite{Bressan} a technical correction and in some situations, defined in \S\ref{ss:whenas}, we use so-called \emph{simplified Riemann} and \emph{boundary Riemann solvers} which introduce a minimal amount of new waves.

As a matter of fact, a key point in the proof of the convergence of a wave front-tracking algorithm is to establish the so-called \emph{interaction estimates}, i.e. to obtain very precise information on the behavior of the approximate solution after an interaction time. In the case of our algorithm, this is done in \S\ref{s:ie}. In particular, Lemmas~\ref{l:me} and~\ref{l:melindeg} provide very refined estimates on the structure of the solution when a wave front of the boundary characteristic family hits the boundary. These estimates are in our opinion fairly innovative and of independent interest and, to the best of our knowledge, are the first boundary interaction estimates  involving the positive part of the wave speed. Note furthermore that the proof of Lemma~\ref{l:me}, in turn, relies on estimate~\eqref{e:me1}, which accurately controls the difference between Lax's wave fan curve of admissible states and the characteristic wave fan curve $\boldsymbol{\zeta}_k$. The proof of these results is based on a very careful analysis of the structure of the curve $\boldsymbol{\zeta}_k$ and uses the assumption that the characteristic vector field is either linearly degenerate or genuinely nonlinear. 

In \S\ref{s:functional} we show that the total variation of the $\ee$-wave front-tracking approximation is uniformly bounded with respect to $\ee$ and time. As in~\cite{Bressan}, the proof relies on the introduction of suitable Glimm-type functionals. The explicit form of our functional, given in \S\ref{ss:deff}, is new and in particular involves the waves speed besides their strength.  In \S\ref{s:otto} we establish the convergence result in~\eqref{e:16bis}, and show that the limit is a distributional solution of~\eqref{e:claw}. The proof follows 
the same main steps as in~\cite{Bressan}, but the presence of the boundary and the consequent more complicated expression of our functionals accounts for much higher technicalities. 

In \S\ref{s:bc} we show that the limit function in~\eqref{e:16bis} attains the boundary condition in~\eqref{e:ibvp}.  From the technical viewpoint, one of the main challenges in the analysis is the fact that, since the boundary is characteristic, in  Definition~\ref{d:equiv} in general the boundary layer does not have as asymptotic state the trace $\mf{\bar v}$, but rather a possibly different state $\mf{\underline v}$ which is connected to $\mf{\bar v}$ through a $0$-speed Lax admissible shock (if the boundary characteristic field is genuinely nonlinear) or a $0$-speed contact discontinuity (if the boundary characteristic field is linearly degenerate). Let us now focus on the case of a genuinely nonlinear boundary characteristic field, which yields a richer and more interesting analysis. From the wave front-tracking approximation standpoint, the fact that $\mf{\bar v} \neq \mf{\underline v}$ and hence that there is a non-trivial $0$-speed Lax admissible shock exactly located at the domain boundary may be due to three different mechanisms: (i) in the wave front-tracking approximation there is a $0$-speed Lax admissible shock located at the domain boundary, which persists in the limit; (ii) in the wave front-tracking approximation there is a non-trivial center component of the boundary layer, which in the limit gives origin to a $0$-speed Lax admissible shock located at the boundary; (iii) in the wave front-tracking approximation there are one or more shocks with vanishing speed that are located inside the domain but in the limit approach the domain boundary. Mechanisms (i) and (ii) are taken into account by Lemma~\ref{l:bl2}, whereas the analysis of mechanism (iii) is the most complicated from the technical viewpoint and relies on the definition of the set $\widetilde E$ in~\eqref{e:tildeE}.  Up to negligible sets, $\widetilde E$ is loosely speaking the set of points at which mechanism (iii) occurs and to show that the boundary condition~\eqref{e:ibvp} is satisfied on $\widetilde E$ we rely on a series of technical results. In particular, Lemma~\ref{t:tvf} provides a uniform in $\ee$ bound on the total variation (with respect to time) of the flux function $\mf f(\mf v_{\ee})$ evaluated at the domain boundary. Note that this is a fairly non-trivial result as we do not expect that the traces $\mf v_{\ee} (\cdot, 0)$ and $\mf v(\cdot, 0)$ have uniformly bounded total variation, see for instance the related counter-example~\cite[\S4.3]{DMS}. To show that the boundary condition~\eqref{e:ibvp} is satisfied on the complement $\R \setminus \widetilde E$ we also rely on a considerably technical analysis; see in particular the proof of Lemma~\ref{l:tildeE}, which requires a careful control of the behavior of the wave front-tracking approximation close to the domain boundary.  To achieve this control we use among other things techniques introduces in~\cite[Ch10]{Bressan} to study the fine regularity properties of the limit of the wave front-tracking approximation. 
\subsection{Notation} \label{ss:notation}
\subsubsection*{Characteristic vs boundary characteristic vector field}
We term $i$-th \emph{characteristic field} the map $\mf u \mapsto \mf r_i(\mf u)$, where $\mf r_i$ is an eigenvalue of $\mf E^{-1} \mf A(\mf u)$ associated to the $i$-th eigenvalue. We say that the $k$-th vector field is \emph{boundary characteristic} if the eigenvalue $k$-th eigenvalue of $\mf E^{-1} \mf A(\mf u)$ may vanish.  
\subsubsection*{General mathematical symbols}
We denote matrices by capital bold letters, column vectors by lowercase bold letters, and scalars by lowercase normal letters, so $\mf C$ is a matrix, $\mf c$ is a column vector, $\mf c^t$ is a row vector and $c$ is a scalar. 
\begin{itemize}
\item $\R_+: = [0, + \infty[$ 
\item $BV$: the space of bounded total variation functions, see~\cite{AmbFuPal}
\item $\mathrm{Tot Var} \; \mf w$: the total variation of the function $\mf w: I \to \R^N$, where $I$ is an interval $I \subseteq \R$ 
\item $[\cdot]^-$: the negative part
\item $\mf 0_d $: the zero vector in $\R^d$
\item $\mathbb M^{n \times m}$: the space of $n \times m$ matrices
\item $\mf 0_{n \times m}$: the zero matrix in $\mathbb M^{n \times m}$
\item $C^{1,1}$: the space of continuously differentiable functions with Lipschitz continuous first derivative
\item $\mathcal M_{\mathrm{loc}}(\R_+)$: the space of Radon measures on $\R_+$
\item $\mf C^t$: the transpose of the matrix $\mf C$
\item $\mathcal L^1, \mathcal L^2$: the standard Lebesgue measure on $\R$ and $\R^2$, respectively
\end{itemize}
\subsubsection*{Symbols introduced in the present paper}
\begin{itemize}
\item $N$: the dimension of the system, i.e. in~\eqref{e:claw} $\mf v \in \R^N$
\item $\delta^\ast$: the same constant as in the statement of Theorem~\ref{t:main}
\item $\boldsymbol{\widetilde \beta}, \boldsymbol{\beta}$: the functions used to assign the boundary condition. The function $\boldsymbol{\beta}$ is defined in~\S\ref{ss:boucon} and $\boldsymbol{\widetilde \beta}$ is obtained from $\boldsymbol{\beta}$ by passing to the $\mf v$ coordinates
\item $h$: dimension of the kernel of $\mf B$, see~\eqref{e:B}
\item $k-1$: the number eigenvalues of $\mf E^{-1} \mf A$ that are negative and bounded away from $0$, see~\eqref{e:sh} and~\eqref{e:uast}
\item $\mf u_\ast$: the same value as in~\eqref{e:uast}
\item $\mf A_{11}$: see~\eqref{e:blockae}
\item $\ell$: the number non-positive eigenvalues of $\mf A_{11}$. In case B) in Hypothesis~\ref{h:eulerlag} it depends on $\mf u_b$ and it is denoted by $\ell (\mf u_b)$
\item $(\mf r_1, \lambda_1), \dots, (\mf r_N, \lambda_N)$: eigencouples of $\mf E^{-1}\mf A$ 
\item $\boldsymbol{\pi}_{\mf u}:$ the projection in~\eqref{e:proj}
\item $\mf i_i$: the rarefaction curve, see~\eqref{e:rarefdef} 
\item $\mf h_i$: the Hugoniot locus, see~\eqref{e:rh}
\item $\sigma_i (\mf{\widetilde u}, s)$: the shock speed satisfying~\eqref{e:rh}
\item $\mf t_i$: the $i$-th admissible wave fan curve, see~\eqref{e:laxc}
\item $\boldsymbol{\zeta}_k$: the characteristic wave fan curve of admissible states, see~\eqref{e:uguale1} and~\eqref{e:regcwfc} in the genuinely nonlinear case and~\eqref{e:mcld} in the linearly degenerate case
\item $\underline s$: see~\eqref{e:underlines}
\item $\bar s$: see~\eqref{e:baresse}
\item ${\boldsymbol \phi} $: the same function as in~\eqref{e:bl10}
\item $\mf u^\ee$: the $\ee$-wave front-tracking approximation constructed in \S\ref{s:cinque}
\item $\mf u_0^\ee$, $\mf u_0^\ee$: the initial and boundary condition of $\mf u^\ee$, see~\eqref{e:51} and~\eqref{e:52}
\item $r_\ee$: the bound on the maximal strength of new rarefaction wave-fronts, see~\cite[p.129]{Bressan}
\item $\omega_\ee$: threshold that discriminates between the use of the accurate and simplified Riemann and boundary Riemann solvers, see~\S\ref{ss:whenas}
\item $\varsigma_k (\mf{\widetilde u}, s_k)$: see~\eqref{e:speedwft}
\item $\mf{\hat u}$: see~\eqref{e:hatu}.
\item $\vartheta_\alpha$: see~\eqref{e:s}
\item $\Upsilon$: the functional defined in~\eqref{e:upsilon}
$\Delta \Upsilon (t) : = \Upsilon(t^+) - \Upsilon(t^-)$
\item $I, L_\alpha:$ the functionals in~\eqref{elle}
\end{itemize}
\vspace{1cm}

\subsubsection*{Constants}
\begin{itemize}
\item $\unpo$: a constant only depending on the data  $\mf g$, $\mf f$, $\mf D$ and on the change of variables $\mf u$, and on the value $\mf u^\ast$ in~\eqref{e:hp}. Its precise value can vary from occurrence to occurrence
\item $\delta^\ast$: the same constant as in~\eqref{e:hp}
\item $c$: the same constant as in~\eqref{e:sh}
\item $d$: the same constant as in~\eqref{e:gnl}
\item $C_1$: the same constant as in~\eqref{e:nc}
\item $C_2$: the same constant as in~\eqref{e:me} 
\item $C_3$: the same constant as in~\eqref{e:jd1}
\item $r_\ee$: the bound on the maximal strength of a rarefaction wave front, see~\S\ref{ss:arie}
\item $\hat \lambda$: the speed of non-physical fronts, see~\eqref{e:npspeed} 
\item $\omega_\ee$: threshold  to discriminate between accurate and simplified Riemann and boundary Riemann solvers, see \S\ref{ss:whenas}
\item $K_1$, $K_2$ and $K_3$: the same constants as in~\eqref{e:upsilon}
\item $A$: the same constant as in~\eqref{e:s} 
\item $\vartheta_\alpha$: the weighted signed strength of the wave front $\alpha$, see~\eqref{e:s}
\item $\delta$: bound on $\Upsilon$, see~\eqref{e:assumption}
\item $C_7$: the same constant as in~\eqref{e:cisette}. Note that $C_7>1$ by assumption
\item $C_6$: the same constant as in~\eqref{e:jd1}
\item $C_3$: the same constant as in~\eqref{e:dadim3}
\item $C_4$: the same constant as in~\eqref{e:msigma}
\item $C_5$: the same constant as in~\eqref{e:idv}
\item $C_8$: the same constant as in~\eqref{e:rarepiccola}
\item $C_9$: the same constant as in~\eqref{e:raref}
\item $C_{10}$: the same constant as in~\eqref{e:melindeg}
\item $C_{11}$: the same constant as in~\eqref{e:pallido}
\item $C_{12}$: the same constant as in~\eqref{e:pallido2}
\item $K_5$: the same constant as in~\eqref{elle}
\item $C_{13}$: the same constant as in~\eqref{e:tvf}
\item $C_{14}:$ the same constant as in~\eqref{e:settedue}
\item $K_4$: the same constant as in~\eqref{e:Lambda}
\end{itemize}
\subsubsection{Relation between the various constants} \label{sss:constants}
We now list the inequalities that in what follows we need to impose on the constants, and show that they are consistent. 
We point out that as it will be clear in the following the constants $C_i$ are given by the problem, whereas the constants $A$, $K_i$, $\delta$ and $\delta^\ast$ are to be chosen in the proof so that they satisfy suitable constraints.  More precisely, we require that the above constants satisfy the following relations:
\begin{equation}
\label{e:uno}
        \max \{C_7 [1 + K_1 \delta], C_1+ K_1[C_1 + C_{12} + C_1 C_{12} ] \delta  \} + 1 \leq A
\end{equation}
and 
\begin{equation}
\label{e:due}
        C_2 + K_1 [2 C_8+ C_2] \delta) + 1 \leq K_2, \qquad  C_2 + K_1[2 C_8+ C_2] \delta + 1 \leq  K_1.
\end{equation}
and 
\be \label{e:cisette3}
      C_2 C_7+ K_1 [2 C_8+ C_2C_7] \delta) + 1 \leq K_2, \qquad  C_2C_7+ K_1[2 C_8+ C_2C_7] \delta + 1 \leq  K_1.
\eq
and 
\be \label{e:4}
      C_6 +  K_1 \delta  \big( 2 C_6 + [1+ C_6] C_{11} \big) 
   +1 \leq  K_3
\eq
and
\be 
\label{e:tre}
    A C_5  + 1 +K_2 ( \delta  C_4 C_5 + 2 C_3  )  \leq  K_1  (1- 2 A \delta C_5- A^2\delta^2 C_5^2) 
\eq
and
\be \label{e:sole}
   C_2 \delta \leq 1, \quad C_8 \delta \leq \frac{1}{2}, \quad C_2 [1+ C_4] \delta \leq 1
\eq
and 
\be \label{e:delta1}
    \max\{ C_{14}, 1\}  \delta^\ast \leq \delta
\eq
To this end we first choose $A$ and $K_2$ in such a way that 
$$
    A \ge  \max \{2 C_7, C_1 + 1     \}+ 1
$$ 
and $K_2 \ge 2  C_7C_2  + 1$. Next, we choose $K_1$ and $K_3$ such that 
$$
    K_1 \ge \max \{  2[C_2C_7+1],  2A C_5  + 2  +2 K_2 ( 1+ 2 C_3  ) \}  , \quad 
K_3 \ge  C_6 + 2
$$
Finally, we choose $\delta$ sufficiently small to have 
\begin{equation*} \begin{split}
     &K_1 \delta \leq 1, \quad  K_1[C_1 + C_{12} + C_1 C_{12} ] \delta \leq 1, \quad  C_7C_2 \delta \leq \frac{1}{2}, \quad \delta  C_4 C_5 \leq 1, \\ & 2 A \delta C_5 \leq \frac{1}{4}, \quad 
     A^2 \delta^2 C^2_5 \leq \frac{1}{4}, \quad K_1 \delta  \big( 2 C_6 + [1+ C_6] C_{11} \big) \delta \leq 1.
 \end{split}
 \end{equation*}
and $\delta^\ast$ sufficiently small to have~\eqref{e:delta1}.

%
%

\section{Hypotheses}\label{s:hyp}
In this section we state the hypotheses imposed on system~\eqref{e:symmetric} in the present work. In \S\ref{ss:boucon} we define the function $\boldsymbol{\beta}$ used to assign the boundary condition on the mixed hyperbolic-parabolic system~\eqref{e:symmetric}. Finally, in \S\ref{ss:appl} we show that our assumptions are satisfied by the Navier-Stokes and viscous MHD equations, written in both Eulerian and Lagrangian coordinates.  
\subsection{Statement of the hypotheses}
To ease the exposition we state \emph{global} assumptions, i.e. assumptions that hold for every $\mf u \in \R^N$. However, all our analysis is \emph{local}, in the sense that the solution $\mf u$ and the approximate solutions $\mf u_{\ee_m}$ attain values in a small neighborhood of the value $\mf u(\mf v^\ast)$, where $\mf v^\ast$ is the same as in~\eqref{e:hp}, and hence it suffices to assume that our assumptions hold in that neighborhood. 
\begin{hyp}
\label{h:normal} 
System~\eqref{e:symmetric} is of the normal form, in the Kawashima Shizuta sense, see~\cite{KawashimaShizuta1}.  More precisely, the coefficients in~\eqref{e:symmetric} are smooth and satisfy the following assumptions:
\begin{itemize}
\item[i)] The matrix $\mf B$ satisfies the block decomposition 
\begin{equation}
\label{e:B}
    \mf B(\mf u) = 
    \left(
    \begin{array}{cc}
    \mf 0_{h \times h} & \mf 0_{N-h}^t \\
    \mf 0_{N-h} & \mf B_{22} (\mf u) \\
    \end{array}
    \right)
\end{equation}
for some symmetric and positive definite matrix $\mf B_{22} \in \mathbb M^{(N-h) \times (N-h)}$.
\item[ii)] For every $\mf u$, the matrix $\mf A (\mf u)$ is symmetric.
\item[iii)] For every $\mf u$, the matrix $\mf E (\mf u)$ is symmetric, positive definite and block diagonal, namely
\be 
\label{e:e}
    \mf E(\mf u) = 
    \left(
    \begin{array}{cc}
    \mf E_{11} (\mf u) & \mf 0_{N-h}^t \\
    \mf 0_{N-h} & \mf E_{22} (\mf u) \\
    \end{array}
    \right),
\end{equation}
where $\mf E_{11} \in \mathbb{M}^{h \times h}$ and $\mf E_{22} \in \mathbb{M}^{(N-h) \times (N- h)}$. 
\item[iv)] The second order term can be written in the form $\mf G(\mf u, \mf u_x) \mf u_x$, with  
$\mf G$ satisfying
\be
\label{e:G}
        \mf G(\mf u,  \mf u_x) = 
    \left(
    \begin{array}{cc}
    \mf 0_{h \times h} & \mf 0_{N-h}^t \\
    \mf G_1 & \mf G_2 \\
    \end{array}
    \right), \qquad 
\eq
for suitable functions $\mf G_1 \in \mathbb M^{h \times (N-h)}$, $\mf G_2
\in \mathbb  M^{(N-h) \times (N-h)}$. The function $\mf G_1$ depends on $\mf u$ and on the last $N-h$ components of $\mf u_x$ and vanishes when all these components vanish. The function $\mf G_2$ depends on $\mf u$ and $\mf u_x$ and vanishes when $\mf u_x = \mf 0_N$. 
 \end{itemize} 
\end{hyp}
Note that the point iv) implies that 
\be \label{e:G2} 
   \mf G_1 (\mf u, \mf 0_{N-h}) = \mf 0_{h \times (N-h)}, \quad 
   \mf G_2 (\mf u, \mf 0_N) = \mf 0_{(N-h) \times (N-h)}, \quad \text{for every $\mf u \in \R^N$}. 
\eq

\begin{hyp}[Kawashima-Shizuta condition]
\label{h:ks}
For every $\mf u$ the matrices $\mf E$, $\mf A$ and $\mf B$ satisfy
\be 
\label{e:ks}
   \Big\{ \text{eigenvectors of} \;  \mf E^{-1}(\mf u) \mf A (\mf u) \Big\} \cap  \text{kernel} \; \mf B (\mf u) = \emptyset. 
   \eq
\end{hyp}
Since the matrices $\mf A$ and $\mf E$ are both symmetric by Hypothesis~\ref{h:normal}, then owing to a classical result (see for instance~\cite[Lemma10.1]{BianchiniSpinolo}) the matrix $\mf E^{-1} (\mf u) \mf A(\mf u)$
has  $N$ real eigenvalues, provided each eigenvalue is counted according to its multiplicity. We now introduce the standard hypothesis that the system is \emph{strictly hyperbolic}. 
\begin{hyp}
\label{h:sh}
    For every $\mf u$ the matrix $\mf E^{-1} (\mf u) \mf A(\mf u)$ has $N$ distinct eigenvalues. 
\end{hyp}
We term $\lambda_1 (\mf u), \dots, \lambda_N (\mf u)$ the eigenvalues of $\mf E^{-1} (\mf u) \mf A(\mf u)$, and $\mf r_1(\mf u), \dots, \mf r_N(\mf u)$ a set of corresponding eigenvectors. In the following, we mostly focus on the \emph{boundary characteristic case}. More precisely, we assume that 
\begin{equation}
\label{e:sh}
    \lambda_1 (\mf u) < \dots < \lambda_{k-1} (\mf u)
    \leq - c< 0  <c \leq  \lambda_{k+1} (\mf u) < \dots 
    < \lambda_N(\mf u), \quad \text{for every $\mf u$}
\end{equation}
for some $k=1, \dots, N$ and $c>0$,  and that eigenvalue $\lambda_k(\mf u)$ can attain the value $0$, i.e. there is $\mf u^\ast$ such that 
\be \label{e:uast}
    \lambda_k (\mf u^\ast) =0.
\eq
In the following we assume that $\mf u^\ast$ is sufficiently close to our data, otherwise we can assume $\lambda_k \neq 0$ because all our analysis is confined in a a small neighborhood of the data. Just to fix the ideas, we assume that $\mf v^\ast = \mf v(\mf u^\ast)$  satisfies~\eqref{e:hp}. 
Next, we recall that the $k$-th vector field is termed \emph{genuinely nonlinear} if 
\be 
\label{e:gnl} 
        \nabla \lambda_k (\mf u) \cdot \mf r_k (\mf u) \geq d >0, \quad \text{for every $\mf u \in \R^N$} 
\eq
and a suitable constant $d>0$. The $k$-th vector field is \emph{linearly degenerate} if  
\be 
\label{e:lindeg} 
        \nabla \lambda_k (\mf u) \cdot \mf r_k (\mf u) = 0, \quad \text{for every $\mf u \in \R^N$}.
\eq
We now introduce the fairly standard assumption that every vector field is either linearly degenerate or genuinely nonlinear.
\begin{hyp}
\label{h:ldgnl}
For every $i=1, \dots, N$, the $i$-th vector field is either genuinely nonlinear or linearly degenerate. 
\end{hyp}
To state our last assumption, we introduce the block decomposition of $\mf A$ corresponding to the block decomposition~\eqref{e:B}, that is
\be
\label{e:blockae}
        \mf A (\mf u) = 
    \left(
    \begin{array}{cc}
    \mf A_{11} (\mf u)  & \mf A_{21}^t (\mf u) \\
    \mf  A_{21} (\mf u) & \mf A_{22} (\mf u), \\
    \end{array}
    \right) \qquad  \mf A_{11} \in \mathbb{M}^{h \times h}, \mf A_{22}  \in \mathbb{M}^{(N-h) \times (N-h)}, 
    \mf  A_{21} \in \mathbb{M}^{(N-h) \times h},
    \eq
    where $ \mf A_{21}^t $ denotes the transpose of the matrix $ \mf A_{21}$. 
\begin{hyp}
\label{h:eulerlag}
The matrix $\mf A$ satisfies one of the following conditions:
\begin{itemize}
\item[A)] we have 
\be \label{e:a11nz}
      \mathrm{det}  \ \mf A_{11}(\mf u) \neq 0, \quad \text{for every  $\mf u \in \R^N$}; 
\eq
\item[B)] we have
\begin{equation} \label{e:alpha}
      \mf A_{11} (\mf u) = \alpha (\mf u ) \mf E_{11}(\mf u)
\eq
for\footnote{Note that condition~\eqref{e:alpha} is only relevant if $h>1$. If $h=1$ then $\mf E_{11}$ is actually a scalar function which only attains strictly positive values since the matrix $\mf E$ is positive definite owing to Hypothesis~\ref{h:normal}, item (iii); hence, we can always set $\alpha (\mf u) : = \mf E_{11}^{-1} \mf A_{11} (\mf u)$.} some \emph{scalar} function $\alpha: \R^N \to \R$. Also, if for some $\mf u \in \R^N$ we have $\alpha (\mf u) =0$ then 
\be \label{e:lincomb}
   \nabla \alpha (\mf u) = (\mf 0_h^t, \mf{s}^t)
\eq
for some $\mf s \in \R^{N-h}$ that is a linear combination of the columns of $\mf A_{21} (\mf u)$.  
\item[C)] we have 
\be \label{e:lagr2}
     \mf A_{11} (\mf u) = \mf 0_{h \times h} , \quad \text{for every $\mf u \in \R^N$}.
\eq
\end{itemize}
\end{hyp}
Some remarks are here in order. First, conditions~\eqref{e:alpha} and~\eqref{e:lincomb} are Hypotheses 4 and 5 in~\cite[\S2.1]{BianchiniSpinolo}, respectively, and we refer to~\cite[\S2.1]{BianchiniSpinolo} for some further comment. Second, as pointed out in \S\ref{ss:appl} the Navier-Stokes and viscous MHD equations (with both null and positive electrical resistivity) in Lagrangian coordinates satisfy~\eqref{e:lagr2}, whereas the Navier-Stokes and viscous MHD equations (with both null and positive electrical resistivity) in Eulerian coordinates satisfy~\eqref{e:a11nz} if the fluid velocity is bounded away from $0$, and condition B) above otherwise. Third, Hypothesis~\ref{h:eulerlag} allows to write the boundary layers equation in a manageable form, see the analysis in~\S\ref{s:bl}. Finally, by combining the analysis in~\cite{BianchiniSpinoloARMA,BianchiniSpinolo} we could handle more general assumptions than Hypothesis~\ref{h:eulerlag}, but to ease the exposition we decided to focus on conditions i), ii) and iii) above. 
\subsection{The boundary condition for the mixed hyperbolic-parabolic system} \label{ss:boucon}
We recall the construction of the function $\boldsymbol{\beta}$ given in~\cite{BianchiniSpinoloARMA,BianchiniSpinolo} by distinguishing the 3 cases considered in the statement of Hypothesis~\ref{h:eulerlag}. The function  $\boldsymbol{\widetilde \beta}$ is then obtained from $\boldsymbol{\beta}$ by passing to the $\mf v$ coordinates. \\
{\sc Case A):} owing to item ii) in Hypothesis~\ref{h:normal}, the matrix $\mf A_{11}$ is symmetric and hence all its eigenvalues are real. We term $\ell \leq h$ the number of negative eigenvalues of $\mf A_{11}$ and  
$\boldsymbol{\pi}_{11}: \R^h \to \R^{h}$ the projection onto the subspace of $\R^h$ generated by the eigenvectors associated to positive eigenvalues (remember that $0$ cannot be an eigenvalue of $\mf A_{11}$ owing to~\eqref{e:a11nz}). To define $\boldsymbol{\beta}$, we introduce the decomposition 
\be \label{e:decompose}
   \mf u : = 
   \left( 
       \begin{array}{cc}
          \mf u_1 \\
          \mf u_2 \\
       \end{array}
   \right), \qquad  \mf u_1 \in \R^h, \; \mf u_2 \in \R^{N-h}
\eq
and set 
\begin{equation} \label{e:betagnl}
    \boldsymbol{\beta} : \R^N \times \R^N \to \R^{N}, \qquad \boldsymbol{\beta}(\mf u, \mf u_b) : =
     \left( 
       \begin{array}{cc}
          \boldsymbol{\pi}_{11}(\mf u_1 - \mf u_{1 b}) \\
          \mf u_2  - \mf u_{2b}\\
       \end{array}
   \right)    
\eq
{\sc Case B):} the function $\boldsymbol{\beta}$ is defined in~\cite[\S2.2]{BianchiniSpinolo}, see in particular~\cite[(2.17)]{BianchiniSpinolo}. \\
{\sc Case C):} we recall the decomposition~\eqref{e:decompose} and define $\boldsymbol{\beta}: \R^N \times \R^N\to \R^{N}$ by setting 
\be \label{e:betalindeg}
\boldsymbol{\beta}(\mf u, \mf u_b): =
    \left(
    \begin{array}{cc}
     \mf 0_h \\ 
    \mf u_2 - \mf u_{2b}
    \end{array}
    \right) 
\eq 
\begin{remark} \label{r:beta}
A word of warning. We explicitely point out, that even if the function $\boldsymbol{\beta}$ always attains values in $\R^N$, its explicit form implies that by imposing say $\boldsymbol{\beta} (\mf u(t, 0), \mf u_b)= \mf 0_N$ we are actually imposing on $\mf u(t, 0)$ the following number of 
conditions:  $N-\ell$ conditions in {\sc case A}) and  $N-h$ conditions in {\sc case C)}. In {\sc case B)} we are imposing $N - \ell(\mf u_b)$ conditions, where $\ell (\mf u_b)$ denotes the number of non-positive eigenvalues of $\mf A_{11}(\mf u_b)$ and depends on $\mf u_b$. 
\end{remark}
\subsection{Applications to the inviscid limit of the Navier-Stokes and viscous MHD equations in Eulerian and Lagrangian coordinates}\label{ss:appl}
In this paragraph we verify that Hypotheses~\ref{h:ks},$\dots$,\ref{h:eulerlag} are satisfied by the Navier-Stokes and viscous MHD equation written in both Lagrangian and Eulerian coordinates. We recall that the Navier-Stokes equations describe the motion of a viscous and compressible fluid, whereas the viscous MHD equations in one space dimension describe the propagation of plane waves in an electrically charged viscous and compressible fluid. 
\subsubsection{Navier-Stokes equation in Eulerian coordinates} \label{sss:appl1}
We refer to~\cite[\S2.3]{BianchiniSpinolo} and conclude that in this case $h=1$ and 
$$
    \mf E^{-1}\mf A (\mf u)=
    \left( 
    \begin{array}{ccc}
      u & \rho & 0 \\
      R \theta/ \rho & u & R \\
      0 & R\theta/ e_\theta & u \\
    \end{array}
    \right), \qquad \mf u = (\rho, u, \theta)
$$
In the above expression, $\rho$, $u$ and $\theta$ denote  the fluid density,
the fluid velocity, and
the absolute temperature, respectively. The internal energy $e$ is a function of $\theta$ and satisfies  
$e_\theta>0$. We are focusing on the case of a polytropic gas, with pressure law $p = R \theta \rho$. 
The eigencouples of $\mf E^{-1} \mf A$ are 
$$
   \lambda_1 (\mf u) = u-c, \; \mf r_1 (\mf u) = 
   \left(
   \begin{array}{ccc}
   \rho e_\theta \\
  - e_\theta c \\
   R \theta \\
   \end{array}
   \right), \quad 
   \lambda_2 (\mf u) = u, \; \mf r_2 (\mf u) = 
   \left(
   \begin{array}{ccc}
   - \rho \\
  0 \\
    \theta \\
   \end{array}
   \right), 
   \quad
   \lambda_3(\mf u) = u+c, \; \mf r_3 (\mf u) = 
  \left(
   \begin{array}{ccc}
   \rho e_\theta \\
  e_\theta c \\
   R \theta \\
   \end{array}
   \right)
$$
where $c = \sqrt{R \theta + R^2 \theta/e_\theta}$ denotes the sound speed. Since $e_\theta>0$, $c_\theta>0$, the first and third characteristic fields are genuinely nonlinear, the second is linearly degenerate. Summing up, when the fluid velocity is close to $c$ the boundary is characteristic and $k=1$, when it is close to $-c$ the boundary is again characteristic and $k=3$; in both cases,~\eqref{e:gnl},~\eqref{e:a11nz} are both satisfied. When the fluid velocity vanishes the boundary is characteristic with linearly degenerate boundary characteristic field; also, the system satisfies conditions~\eqref{e:alpha} and~\eqref{e:lincomb}, see~\cite[\S2.3]{BianchiniSpinolo} for the explicit computations. 
\subsubsection{Navier-Stokes equations in Lagrangian coordinates}
By using the analysis in~\cite[\S4.1]{Rousset} we conclude that Hypotheses \ref{h:normal},\ref{h:ks} and point C) in Hypothesis~\ref{h:eulerlag} are satisfied with $h=1$ and furthermore 
$$
    \mf E^{-1}\mf A (\mf u)=
    \left( 
    \begin{array}{ccc}
      0 & -1 & 0 \\
      -R\theta/\tau^2 & 0 & R/\tau \\
      0 & R \theta/\tau e_\theta &  0 \\
    \end{array}
    \right), \qquad \mf u = (\tau, u, \theta),
$$
where $\tau= \rho^{-1}$ is the specific volume.  The eigencouples of the above matrix are 
$$
   \lambda_1 = - \frac{c}{\tau}, \; \mf r_1 = \left(
   \begin{array}{ccc}
   e_\theta \\
   c e_\theta/\tau \\
   -R \theta/\tau
   \end{array}
   \right), \quad 
    \lambda_2 =0, \; \mf r_2 = \left(
   \begin{array}{ccc}
  - \tau/\theta \\
   0 \\
   1
   \end{array}
   \right),
   \quad 
   \lambda_3 =  \frac{c}{\tau}, \; \mf r_3 = \left(
   \begin{array}{ccc}
   - e_\theta \\
   c e_\theta/\tau \\
   R \theta /\tau
   \end{array}
   \right).
$$
The first and third characteristic fields are genuinely nonlinear, the second is linearly degenerate. 
\subsubsection{Viscous MHD equation in Eulerian coordinates} \label{sss:mhdeuler}We refer to~\cite[\S2.3]{BianchiniSpinolo} and conclude that Hypotheses~\ref{h:normal},\ref{h:ks} are both satisfied. Note that $h=1$ if the electrical resistivity is strictly positive, and $h=3$ when the electrical resistivity is $0$. Also, we have 
$$
    \mf E^{-1}\mf A (\mf u)=
    \left( 
    \begin{array}{ccccccc}
      u                   & \mf 0_2^t       & \rho     & \mf 0_2^t                           & 0 \\
      \mf 0_2          & u \mf I_2        &  \mf b  & - \beta \mf I_2                   & \mf 0_2 \\
      R \theta/ \rho &  \mf b^t /\rho     &  u        &         \mf 0_2^t                    & R \\
      \mf 0_2         & - \beta \mf I_2 & \mf 0_2 & u \mf I_2                        & \mf 0_2 \\
      0                  & \mf 0_2^t         &R\theta/ e_\theta & \mf 0_2^t &     u \\
    \end{array}
    \right), \qquad \mf u = (\rho, \mf b,  u, \mf w, \theta),
$$
where $(\beta, \mf b)$, $\beta>0$ denotes the magnetic field and $(u, \mf w)$ is the fluid velocity. 
By direct computations, one can verify that if $\mf b \neq \mf 0_2$ the matrix $   \mf E^{-1}\mf A$ satisfies Hypothesis~\ref{h:sh}. For almost every given value $\mf
u$, Hypothesis~\ref{h:ldgnl} is satisfied in a suitable neighborhood of $\mf u$.
Also, the fields corresponding to the eigenvectors $\lambda_2 (\mf u) = u - \beta \sqrt{\rho}$, $\lambda_4 (\mf u) =u$, $\lambda_6 (\mf u) = u+ \beta \sqrt{\rho}$ are linearly degenerate as the corresponding eigenvectors are 
$$
    \mf r_2 (\mf u) = 
   \left(
   \begin{array}{ccc}
  0 \\
 - b_2 \sqrt{\rho}\\
   b_1\sqrt{\rho} \\
  0\\
  -b_2 \\
  b_1 \\
  0
   \end{array}
   \right), \quad  \mf r_4 (\mf u) = 
   \left(
   \begin{array}{ccc}
   - \rho \\
  0 \\
  0 \\
  0 \\
  0\\
  0\\
    \theta \\
   \end{array}
   \right), 
   \quad \mf r_6 (\mf u) = 
   \left(
   \begin{array}{ccc}
  0 \\
  -b_2 \sqrt{\rho} \\
   b_1\sqrt{\rho}\\
  0\\
  -b_2 \\
  b_1 \\
  0
   \end{array}
   \right).
$$
If $u$ vanishes then the fourth characteristic field is boundary characteristic, it satisfies~\eqref{e:lindeg} and by relying on the analysis in~\cite[\S2.4]{BianchiniSpinolo} we conclude that the system satisfies conditions~\eqref{e:alpha} and~\eqref{e:lincomb} in both cases of null and positive electrical resistivity. If the eigenvalue $\lambda_i(\mf u)$ vanishes for some $i \neq 4$ then by relying on~\cite[\S2.4]{BianchiniSpinolo} one can show that~\eqref{e:a11nz} is satisfied in  both cases of null and positive electrical resistivity. 
\subsubsection{Viscous MHD equation in Lagrangian coordinates} 
By relying on~\cite[\S4.2]{Rousset} we get that Hypotheses~\ref{h:normal} and~\ref{h:ks} are satisfied  and moreover
$$
    \mf E^{-1}\mf A (\mf u)=
    \left( 
    \begin{array}{ccccccc}
      0                & \mf 0_2^t                     & -1              & \mf 0_2^t                                  & 0 \\
      \mf 0_2          & \mf 0_{2\times 2}       &  \mf b/\tau  & - \beta /\tau\mf I_2                   & \mf 0_2 \\
     -R \theta/\tau^2 &  \mf b^t                       &  0           &         \mf 0_2^t                    & R/\tau\\
      \mf 0_2         & - \beta \mf I_2 & \mf 0_2 & \mf 0_{2 \times 2}                     & \mf 0_2 \\
      0                  & \mf 0_2^t         &R \theta/ \tau e_\theta & \mf 0_2^t &     0 \\
    \end{array}
    \right), \qquad \mf u = (\tau, \mf b,  u, \mf w, \theta).
$$
Again for almost every given value $\mf
u$, Hypothesis~\ref{h:ldgnl} is satisfied in a suitable neighborhood of $\mf u$.
Note furthermore that the characteristic fields associated to the eigenvalues $\lambda_2 (\mf u)= - \beta/\sqrt{\tau}$, $\lambda_4 (\mf u)=0$ and $\lambda_6(\mf u) = \beta/\sqrt{\tau}$ are linearly degenerate as the corresponding eigenvectors are  
$$
    \mf r_2 (\mf u) = 
   \left(
   \begin{array}{ccc}
  0 \\
 - b_2 /\sqrt{\tau}\\
   b_1/\sqrt{\tau} \\
  0\\
  -b_2 \\
  b_1 \\
  0
   \end{array}
   \right) \quad 
    \mf r_4 (\mf u) = 
   \left(
   \begin{array}{ccc}
   \tau \\
  0 \\
  0 \\
  0 \\
  0\\
  0\\
    \theta \\
   \end{array}
   \right), 
    \quad 
  \mf r_6 (\mf u) = 
   \left(
   \begin{array}{ccc}
  0 \\
 b_2 /\sqrt{\tau}\\
 -  b_1/\sqrt{\tau} \\
  0\\
  -b_2 \\
  b_1 \\
  0
   \end{array}
   \right)
$$   
By relying on~\cite[\S4.2]{Rousset} we also get~\eqref{e:lagr2}   in both cases of null and positive electrical resistivity. 
 \section{Boundary layers analysis} \label{s:bl}
 Boundary layers are steady solutions of~\eqref{e:symmetric} and hence satisfy the equation  
\be \label{e:bltw}
    \mf A (\mf u) \mf u' = \mf B(\mf u) \mf u'' + \mf G(\mf u, \mf u') \mf u'. 
\eq 
Since the matrix $\mf B$ is singular, see~\eqref{e:B}, some work is required to write the above equation in an explicit form, which is done in this section by following the analysis in~\cite{BianchiniSpinoloARMA,BianchiniSpinolo} and separately considering cases A), B) and C) in Hypothesis~\ref{h:eulerlag}. We also provide the construction of the center manifold, which is relevant in our analysis because we deal with the boundary characteristic case and hence we have to take into account that in general the boundary layers have a component lying on the center manifold. The exposition is organized as follows: in \S\ref{ss:bld} we consider case A) in Hypothesis~\ref{h:eulerlag}, in \S\ref{ss:beuler} case B) and in \S\ref{ss:blagr} case C). 
\subsection{Case A) of Hypthesis~\ref{h:eulerlag}} \label{ss:bld}
\subsubsection{Boundary layers equation}
We assume~\eqref{e:gnl} and~\eqref{e:a11nz} and introduce the decomposition~\eqref{e:decompose} 
and by
using Hypothesis~\ref{h:normal} we recast~\eqref{e:bltw} as 
\be \label{e:bltw2} 
    \boldsymbol{\omega}' = \mf a( \boldsymbol{\omega}) 
\eq
provided $ \boldsymbol{\omega} \in \R^{2N-h}$ 
and $\mf a: \R^{2N-h} \to \R^{2N-h}$ are defined by setting 
\be \label{e:gamma} 
      \boldsymbol{\omega}\!: =\!\!
    \left(
    \begin{array}{cc}
    \mf u_1 \\
    \mf u_2 \\
    \mf z_2 \\
    \end{array}
    \right) \ 
    \mf a ( \boldsymbol{\omega})
   \!: =\!\!
    \left(
   \begin{array}{cc}
   -  \mf A_{11}^{-1} \mf A_{21}^t \mf z_2 \\
   \mf z_2 \\
   \mf B^{-1}_{22} \displaystyle{\big[ 
   ( \mf G_1 -\mf A_{21} ) \mf A_{11}^{-1} \mf A_{21}^t  + \mf A_{22}  - \mf G_2  \big] \mf z_2} \\
    \end{array}
    \right) \ \mf u_1 \in \R^h, \mf u_2 \in \R^{N-h}, \mf z_2 \in \R^{N-h}.
\eq
Since we need it in the following we introduce the projection 
\be \label{e:proj}
     \boldsymbol{\pi}_{\mf u}: \R^{2N-h} \to \R^N, \quad  \boldsymbol{\pi}_{\mf u} (\mf u, \mf z) : = \mf u
\eq
We linearize system~\eqref{e:bltw2} at $ \boldsymbol{\omega}^\ast: = (\mf u^\ast,  \mf 0_{N-h})$, where $\mf u^\ast$ is the same as in~\eqref{e:uast} and $\mf G_1$ and $\mf G_2$ satisfy~\eqref{e:G2}. To compute its center and stable space, we use the results that follow. 
\begin{lemma}
\label{l:kernel} 
If the vector  $\mf p (\mf u) \in \R^{N-h}$ and the scalar $\mu(\mf u)$ satisfy 
\be \label{e:eige}
   \big[ 
    -\mf A_{21}  \mf A_{11}^{-1} \mf A_{21}^t  (\mf u) + \mf A_{22}(\mf u) - \mu (\mf u)\mf B_{22}   (\mf u) \big] \mf p (\mf u)= \mf 0_{N-h}
\eq
then 
\be \label{e:errei}
    \boldsymbol{\rho}_j (\mf u): = 
    \left(
    \begin{array}{cc}
     -     \mf A_{11} ^{-1}\mf A_{21}^t (\mf u) \mf p(\mf u) \\
     \mf p (\mf u)\\
     \end{array} 
    \right)
\eq
satisfies $[\mf A(\mf u) - \mu (\mf u) \mf B(\mf u) ] \boldsymbol{\rho}(\mf u)= \mf 0_N$. Conversely, every vector belonging to the kernel of $[\mf A(\mf u) - \mu\mf B(\mf u)]$ can be written in the form~\eqref{e:errei}, with $\mf p$ satisfying~\eqref{e:eige}. 
\end{lemma}
Next, we point out that the matrix $\mf B_{22}^{-1}[
 -\mf A_{21} \mf A_{11}^{-1} \mf A_{21}^t  + \mf A_{22} ] \in \mathbb{M}^{(N-h)\times (N-h)}$ is symmetric and use a particular case of~\cite[Lemma 4.7]{BianchiniSpinoloARMA}. Note that owing to~\eqref{e:a11nz} the number $0$ cannot be an eigenvalue of $\mf A_{11}$ and that the following lemma implies in particular that $\ell \leq k-1$. 
\begin{lemma}
\label{l:4.7} Let $\ell$ denote the number of negative eigenvalues of $\mf A_{11}$\footnote{Owing to assumption~\eqref{e:a11nz}, $\ell$ does not depend on $\mf u$}, then 
the signature of the matrix $\mf B_{22}^{-1}[
 -\mf A_{21} \mf A_{11}^{-1} \mf A_{21}^t  + \mf A_{22} ](\mf u^\ast)$ is as follows: $k-1-\ell$ eigenvalues are strictly negative (each of them is counted according to its multiplicity), $0$ is an eigenvalue with multiplicity $1$, and $N-h-k+ \ell$ eigenvalues are strictly positive.  
\end{lemma}
\subsubsection{Center manifold construction}
We apply the above lemma and conclude that the  center space of $\mathbf D \mf g (\boldsymbol{\omega}^\ast)$ has dimension $N+1$ and it is given by 
\be 
\label{e:cs}
  M^c : = \R^N \times \mathrm{span} <\mf p_k > ,   
\eq 
 where $\mf p_k$ is a nonzero vector in the kernel of  $\mf B^{-1}_{22} \big[ -\mf A_{21} \mf A_{11}^{-1} \mf A_{21}^t  + \mf A_{22} \big] (\mf u^\ast)$. We apply the Center Manifold Theorem and fix a center manifold\footnote{The center manifold is in general not unique, but we can arbitrarily fix one.} $\mathcal M^c$. By arguing as in~\cite[\S4]{BianchiniBressan} we establish the following lemma. 
 \begin{lemma} \label{l:emmec}
 There is a map $\mf p_c$, defined in a sufficiently small neighborhood of $(\mf u^\ast, 0)$, attaining values in  $\R^{N-h}$ and satisfying the following properties: 
\begin{itemize}
\item[a)] $\mf p_c (\mf u^\ast, 0)$ is a nontrivial vector in the kernel of  $ \big[ -\mf A_{21} \mf A_{11}^{-1} \mf A_{21}^t  + \mf A_{22} \big] (\mf u^\ast)$. 
\item[b)] 
  in a small enough neighborhood of $(\mf u^\ast, 0)$, 
\be \label{e:emmec}
    (\mf u, \mf z_2 ) \in \mathcal M^c \iff \mf z_2 = z_c \mf p_c (\mf u, z_c) \quad \text{for some $z_c \in \R$}. 
\eq
\item[c)] We have 
\be \label{e:propc} \mf p^t_c \mf B_{22}  \mf p_c \equiv 1. \eq 
\end{itemize}
    Also, by restricting system~\eqref{e:bltw} on $\mathcal M^c$ we obtain
    \be 
    \label{e:bltw3}
    \left\{
    \begin{array}{ll}
    \mf u_1' =  - z_c \mf A_{11}^{-1} \mf A_{21}^t \mf p_c \\
    \mf u_2' = z_c \mf p_c \\
    z_c' = z_c  \theta_c  \\
    \end{array}
    \right.
    \eq
 where the function $\theta_c$ is defined in a sufficiently small neighborhood of $(\mf u^\ast, 0)$, attains real values and satisfies 
 \be
 \label{e:lambdac}
      \theta_c (\mf u^\ast, 0)=0. 
 \eq  
\end{lemma}
Since we need it in the following, we introduce the notation 
\be \label{e:errec}
      \mf r_c (\mf u, z_c)= \left(
     \begin{array}{cc}
                - \mf A_{11}^{-1} \mf A_{21}^t \mf p_c \\
                \mf p_c \\
     \end{array}
     \right)
\eq
and write down the explicit expression of $\theta_c$:  
\be 
\label{e:lambdac2} \begin{split}
        \theta_c (\mf u, z_c): = \frac{1}{1 +z_c  \mf p_c^t\,  \partial_{z_c} \mf p_c}
      \mf p_c^t \Big[ &
      ( \mf G_1 -\mf A_{21} ) \mf A_{11}^{-1} \mf A_{21}^t  + \mf A_{22}  - \mf G_2       -z_c  \mf D_{\mf u_1} \mf p_c  \mf A_{11}^{-1} \mf A_{21}^t \mf p_c  -
      z_c \mf D_{\mf u_2} \mf p_c \  
      \Big] \mf p_c 
\end{split}
\eq 
Note that by combining property a) above with~\eqref{e:G2} we get~\eqref{e:lambdac} from~\eqref{e:lambdac2}.

The next lemma collects some properties 
of $ \theta_c$ we use in the following. 
\begin{lemma}
\label{l:lambdac} We have 
\be 
\label{e:thetaczero}
    \lambda_k (\mf u)=   \theta_c (\mf u, 0) a (\mf u) \quad \text{for some smooth function $a (\mf u) \ge  a_0>0$},
\eq
where  $a_0 >0$ is a suitable constant. Also, if $\lambda_k (\mf u) =0$ then $\mf r_c (\mf u, 0)= \mf r_k (\mf u)$, where $\mf r_c$ and $\mf r_k$ are as in~\eqref{e:errec} and~\eqref{e:errei}, respectively. 
\end{lemma}
Note that, if $\lambda_k (\mf u)=0$ then $\mu (\mf u) =0$ satisfies the assumptions of Lemma~\ref{l:kernel}. 
\begin{proof}[Proof of Lemma~\ref{l:lambdac}]
We plug the equality $\mf z_2 = z_c \mf p_c$ into~\eqref{e:bltw2} and~\eqref{e:gamma}, use the equality $z'_c = \theta_c z_c$ and get 
$$
    \theta_c z_c \mf p_c + z_c \mf p'_c = \mf B^{-1}_{22} \big[ 
   ( \mf G_1 -\mf A_{21} ) \mf A_{11}^{-1} \mf A_{21}^t  + \mf A_{22}  - \mf G_2  \big]  z_c \mf p_c. 
$$
We divide both sides of the above equality by $z_c$, evaluate the result at $(\mf u, 0)$, recall~\eqref{e:G} and point out that, owing to~\eqref{e:bltw3}, $\mf p'_c = \mf 0_{N-h}$ when $z_c=0$. 
We arrive at $ \theta_c \mf B_{22}   \mf p_c =   \big[ 
   -\mf A_{21}  \mf A_{11}^{-1} \mf A_{21}^t  + \mf A_{22}\big]  \mf p_c$, which in turn by applying Lemma~\ref{l:kernel} with $\mu=\theta_c$ yields 
\be \label{e:ab}
     [\mf A (\mf u) - \theta_c (\mf u, 0)\mf B(\mf u) ] \mf r_c (\mf u, 0) = \mf 0_N, 
     \;
\eq
provided $\mf r_c $ is the same vector as in~\eqref{e:errec}. 
Owing to property a) in the statement of Lemma~\ref{l:emmec} and to Lemma~\ref{l:kernel},  $\mf r_c (\mf u^\ast, 0)$ is an eigenvector of $\mf A(\mf u^\ast)$ associated to the eigenvalue $\lambda_k(\mf u^\ast)=0$. We  then define the continuous function $\mf r_k (\mf u)$ in such a way that $\mf r_k (\mf u)$
is an eigenvector of $\mf A(\mf u)$ associated to the eigenvalue $\lambda_k(\mf u)$, and $\mf r_k (\mf u^\ast) = \mf r_c (\mf u^\ast, 0)$. Note that the corresponding 
vector $\mf p_k(\mf u)$ given by Lemma~\ref{l:kernel} satisfies $\mf p_k(\mf u^\ast)= \mf p_c (\mf u^\ast, 0). $
 
 We now use~\eqref{e:ab}, recall that all the matrices $\mf A$, $\mf B$ and $\mf E$ are symmetric and conclude that 
 \be \label{e:abe}
     \mf r^t_k (\mf u) \mf A(\mf u) \mf r_c (\mf u, 0) =  \theta_c (\mf u, 0) \mf r^t_k (\mf u) \mf B(\mf u) \mf r_c (\mf u, 0)=
     \lambda_k (\mf u)  \mf r^t_k (\mf u) \mf E(\mf u) \mf r_c (\mf u, 0). 
 \eq
 Note that 
 $$
     \mf r^t_k (\mf u^\ast) \mf B(\mf u^\ast) \mf r_c (\mf u^\ast, 0) \stackrel{\eqref{e:B},\eqref{e:errei},\eqref{e:errec}}{=}\mf p^t_k (\mf u^\ast) \mf B_{22} \mf p_c (\mf u^\ast, 0) \stackrel{\mf p_k(\mf u^\ast)= \mf p^t_c (\mf u^\ast, 0)}{=}\mf p^t_c  \mf B_{22} \mf p_c (\mf u^\ast, 0) \stackrel{\eqref{e:propc}}{=}1
 $$
 and that 
 $$
      \mf r^t_k (\mf u^\ast) \mf E(\mf u^\ast) \mf r_c (\mf u^\ast, 0) \stackrel{\mf r_k(\mf u^\ast)= \mf r^t_c (\mf u^\ast, 0)}{=}
        \mf r^t_k (\mf u^\ast) \mf E(\mf u^\ast) \mf r_k (\mf u^\ast) \stackrel{\text{Hypothesis~\ref{h:normal}}}{>} 0. 
 $$
 This implies that both $\mf r^t_k (\mf u) \mf B(\mf u) \mf r_c (\mf u, 0)$ and $ \mf r^t_k (\mf u) \mf E(\mf u) \mf r_c (\mf u, 0)$ are strictly positive if $\mf u$ is sufficiently close to $\mf u^\ast$ and hence that~\eqref{e:abe} implies~\eqref{e:thetaczero}, provided $a (\mf u) : = \mf r^t_k (\mf u) \mf B(\mf u) \mf r_c (\mf u, 0)/  \mf r^t_k (\mf u) \mf E(\mf u) \mf r_c (\mf u, 0) $. By combining~\eqref{e:thetaczero} with~\eqref{e:ab} we conclude that we can choose $\mf r_k$ in such a way that $\mf r_c (\mf u, 0) = \mf r_k(\mf u)$ when $\lambda_k (\mf u) =0$. 
 \end{proof} 
\subsection{Case B) of Hypothesis~\ref{h:eulerlag}} \label{ss:beuler}
We refer to~\cite{BianchiniSpinolo}. In particular, the boundary layers equation is discussed in~\cite[\S3.1]{BianchiniSpinolo}, whereas the construction of the center component of the boundary layers is given in~\cite[\S4]{BianchiniSpinolo}. 
\subsection{Case C) of Hypothesis~\ref{h:eulerlag}} \label{ss:blagr}
\subsubsection{Boundary layers equation}
\begin{lemma}
Under~\eqref{e:lagr2} we have that for every $\mf u \in \R^N$ the rank of the matrix $\mf A_{21}(\mf u)$ is $h$. 
\end{lemma}
\begin{proof}
Assume by contradiction that there is $\mf x \in \R^h$ such that $\mf A_{21}^t \mf x = \mf 0_{N-h}$ then 
$$
    \left(
    \begin{array}{cc}
    \mf 0_{h \times h} & \mf A_{21}^t \\
    \mf A_{21} & \mf A_{22} 
    \end{array}
    \right)
    \left(
    \begin{array}{cc}
    \mf x\\
    \mf 0_{N-h} \\
    \end{array}
    \right)=
    \mf 0_N
$$ 
and this contradicts the Kawashima-Shizuta condition~\eqref{e:ks}. 
\end{proof}
Up to a permutation of columns, we can assume that 
\be \label{e:C}
     \mf A_{21} = \left(
     \begin{array}{cc}
     \mf{\widetilde A_{21}} \\
     \mf{\widetilde A_{31}} \\
     \end{array}
     \right), \quad \mf{\widetilde A_{21}} \in \mathbb M^{h \times h}, \;  \mf{\widetilde A_{31}}  \in \mathbb M^{(N-2h) \times h}, \; 
      \mf{\widetilde A_{21}} \in \mathbb{GL}(h),
\eq
that is the matrix $  \mf{\widetilde A_{21}}(\mf u)$ is invertible for every $\mf u$. We now introduce the block decomposition of $\mf A$ and $\mf B$ corresponding to the above decomposition of $\mf A_{21}$ and we arrive at 
$$
   \mf A(\mf u) = 
    \left(
     \begin{array}{ccc}
     \mf 0_{h \times h} &     \mf{\widetilde A_{21}}^t &
     \mf{\widetilde A_{31}}^t \\
      \mf{\widetilde A_{21}} &  \mf{\widetilde A_{22}} &  \mf{\widetilde A_{32}}^t \\
     \mf{\widetilde A_{31}} & \mf{\widetilde A_{32}} &  \mf{\widetilde A_{33}}\\
     \end{array}
     \right) \; \mf{\widetilde A_{22}} \in \mathbb{M}^{h \times h}, \mf{\widetilde A_{32}} \in 
     \mathbb{M}^{(N-2h) \times h},  \mf{\widetilde A_{33}} \in \mathbb{M}^{(N-2h) \times (N -2h)}
$$
and  
$$
     \mf B(\mf u) = 
    \left(
     \begin{array}{ccc}
     \mf 0_{h \times h} &    \mf 0_{h \times h} & \mf 0_{h \times (N-2h)} \\
    \mf 0_{h \times h} &  \mf{\widetilde B_{22}} &  \mf{\widetilde B_{32}}^t \\
    \mf 0_{(N - 2h) \times h} & \mf{\widetilde B_{32}} &  \mf{\widetilde B_{33}}\\
     \end{array}
     \right), \; \mf{\widetilde B_{22}} \in \mathbb{M}^{h \times h}, \mf{\widetilde B_{32}} \in 
     \mathbb{M}^{(N-2h) \times h},  \mf{\widetilde B_{33}} \in \mathbb{M}^{(N-2h) \times (N -2h)}.
$$
Next, we write 
$$
   \mf u' = \left(
     \begin{array}{ccc}
     \mf z_1 \\
     \mf z_2 \\
     \mf z_3 \\
     \end{array}
   \right), \; \mf z_1 \in \R^h, \; \mf z_2 \in \R^h, \mf z_3 \in \R^{N-2h}, \quad 
    \mf G(\mf u, \mf u_x)=
    \left(
    \begin{array}{ccc}
    \mf 0_{h \times h} & \mf 0_h & \mf 0_{N-2h}^t \\
    \mf{\widetilde G_{21}} &  \mf{\widetilde G_{22}} &  \mf{\widetilde G_{23}} \\
     \mf{\widetilde G_{31}} & \mf{\widetilde G_{32}} &  \mf{\widetilde G_{33}}\\
    \end{array}
    \right)
$$
and point out that~\eqref{e:bltw} yields  
\be \label{e:systemtilde}
    \left\{
     \begin{array}{lll}
        \mf{\widetilde A_{21}}^t \mf z_2 + \mf{\widetilde A_{31}}^t       \mf z_3 = \mf 0_h  \\
     \mf{\widetilde A_{21}} \mf z_1 + \mf{\widetilde A_{22}}\mf z_2 + \mf{\widetilde A_{32}}^t \mf z_3= 
     \mf{\widetilde B_{22}} \mf z'_2 + \mf{\widetilde B_{32}}^t \mf z'_3 +
      \mf{\widetilde G_{21}} \mf z_{1} + \mf{\widetilde G_{22}}\mf z_{2} + \mf{\widetilde G_{23}} \mf z_{3}\\
     \mf{\widetilde A_{31}} \mf z_1 + \mf{\widetilde A_{32}}\mf z_2 + \mf{\widetilde A_{33}} \mf z_3= 
     \mf{\widetilde B_{32}} \mf z'_2 + \mf{\widetilde B_{33}} \mf z'_3 +
     \mf{\widetilde G_{31}} \mf z_{1} + \mf{\widetilde G_{32}}\mf z_{2} + \mf{\widetilde G_{33}}^t \mf z_{3}.
     \end{array}
     \right.
\eq
We solve the first equation for $\mf z_2$, which yields 
\be \label{e:z2}
   \mf z_2   =-  [\mf{\widetilde A_{21}}^t]^{-1}  \mf{\widetilde A_{31}}^t       \mf z_3 
   \implies \mf z_2'= -  [\mf{\widetilde A_{21}}^t]^{-1}  \mf{\widetilde A_{31}}^t       \mf z'_3
   + \mf H_1 (\mf u, \mf z_3) \mf z_1 +  \mf H_3 (\mf u, \mf z_3) \mf z_3   
\eq
for suitable functions $\mf H_1 \in \mathbb M^{h \times h}$, $\mf H_3 \in \mathbb M^{h \times (N-2h)}$. The exact expression of $\mf H_1$ and $\mf H_3$ is not relevant here, what is important is that 
\be \label{e:acca13}
    \mf H_1 (\mf u, \mf 0_{N-2h})= \mf 0_{h \times h}, \; 
   \mf H_3 (\mf u, \mf 0_{N-2h})= \mf 0_{h \times (2N-h)} \; \text{for every $\mf u$}. 
\eq
From the second equation in~\eqref{e:systemtilde} we get 
\begin{equation} \label{e:z1}
\begin{split}
    [ \mf{\widetilde A_{21}}+  \mf H_1+ \mf{\widetilde G_{21}} ] \mf z_1  = &\mf{\widetilde A_{22}}
    [\mf{\widetilde A_{21}}^t]^{-1}  \mf{\widetilde A_{31}}^t       \mf z_3  - \mf{\widetilde A_{32}}^t \mf z_3
    -
     \mf{\widetilde B_{22}}  [\mf{\widetilde A_{21}}^t]^{-1}  \mf{\widetilde A_{31}}^t       \mf z'_3\\&
    + \mf{\widetilde B_{22}}  \mf H_3 (\mf u, \mf z_3) \mf z_3 
    + \mf{\widetilde B_{32}}^t \mf z'_3 
     - \mf{\widetilde G_{22}} [\mf{\widetilde A_{21}}^t]^{-1}  \mf{\widetilde A_{31}}^t       \mf z_3 +
       \mf{\widetilde G_{23}} \mf z_{3}
\end{split}
\end{equation}
We recall that by assumption $\mf{\widetilde G_{21}}$ and $\mf{\widetilde G_{31}}$ only depend on $\mf u$, $\mf z_2$ and $\mf z_3$ and hence, owing to~\eqref{e:z2}, on $\mf u$ and $\mf z_3$, and vanish when $\mf z_3 = \mf 0_{N-2h}$. Owing to~\eqref{e:acca13}, this implies that in a neighborhood of $\mf z_3 = \mf 0_{N-2h}$ the matrix  
$  [ \mf{\widetilde A_{21}}+  \mf H_1+ \mf{\widetilde G_{21}} ]$ is invertible as $\mf{\widetilde A_{21}}$ and hence by using the above expression we can write $\mf z_1$ as a function of $\mf u$, $\mf z_3$ and $\mf z_3'$. By plugging this function into the third line of~\eqref{e:systemtilde} and using again~\eqref{e:z2} we get 
\begin{equation*}
\begin{split}
    \Big[\mf{\widetilde B_{33}}  &+  \mf{\widetilde A_{31}} [ \mf{\widetilde A_{21}}+  \mf H_1+ \mf{\widetilde G_{21}} ]^{-1}   [ \mf{\widetilde B_{22}} [ \mf{\widetilde A_{21}}^t]^{-1}  \mf{\widetilde A_{31}}^t -  \mf{\widetilde B_{32}}^t]
- \mf{\widetilde B_{32}}  [\mf{\widetilde A_{21}}^t]^{-1}  \mf{\widetilde A_{31}}^t   \Big] \mf z'_3\\ &
=  \mf{\widetilde A_{31}} [ \mf{\widetilde A_{21}}+  \mf H_1+ \mf{\widetilde G_{21}} ]^{-1}\Big[ \mf{\widetilde A_{22}}
    [\mf{\widetilde A_{21}}^t]^{-1}  \mf{\widetilde A_{31}}^t    - \mf{\widetilde A_{32}}^t + \mf{\widetilde B_{22}}  \mf H_3 (\mf u, \mf z_3)  - \mf{\widetilde G_{22}} [\mf{\widetilde A_{21}}^t]^{-1}  \mf{\widetilde A_{31}}^t   +
       \mf{\widetilde G_{23}} \Big] \mf z_3 \\
& \quad - \mf{\widetilde A_{32}}  [\mf{\widetilde A_{21}}^t]^{-1}  \mf{\widetilde A_{31}}^t \mf z_3+
     \mf{\widetilde A_{33}} \mf z_3 -  \mf{\widetilde B_{32}} \mf H_1 (\mf u, \mf z_3) \mf z_1 - \mf{\widetilde B_{32}}  \mf H_3 (\mf u, \mf z_3) \mf z_3 - \mf{\widetilde G_{31}} \mf z_{1} \\ & \quad+ \mf{\widetilde G_{32}}[\mf{\widetilde A_{21}}^t]^{-1}  \mf{\widetilde A_{31}}^t       \mf z_3 + \mf{\widetilde G_{33}}^t \mf z_{3} \phantom{\Big]}
\end{split}
\end{equation*}
 We have the following lemma
\begin{lemma}\label{l:invertible}
For every $\mf u$ the matrix 
$$
      \Big[\mf{\widetilde B_{33}} +  \mf{\widetilde A_{31}} [ \mf{\widetilde A_{21}} ]^{-1}  [  \mf{\widetilde B_{22}}  [\mf{\widetilde A_{21}}^t]^{-1}  \mf{\widetilde A_{31}}^t -  \mf{\widetilde B_{32}}^t]
- \mf{\widetilde B_{32}}  [\mf{\widetilde A_{21}}^t]^{-1}  \mf{\widetilde A_{31}}^t   \Big]
$$
is invertible. 
\end{lemma}
\begin{proof}
The lemma is a particular case of~\cite[Lemma 4.3]{BianchiniSpinoloARMA}, so we only provide a sketch of the proof. 
Assume by contradiction there is $\mf s \in \R^{N -2r}$, $\mf s \neq \mf 0_{N-2r}$ such that $$ \Big[\mf{\widetilde B_{33}} +  \mf{\widetilde A_{31}} [ \mf{\widetilde A_{21}} ]^{-1}  [  \mf{\widetilde B_{22}}  [\mf{\widetilde A_{21}}^t]^{-1}  \mf{\widetilde A_{31}}^t -  \mf{\widetilde B_{32}}^t]
- \mf{\widetilde B_{32}}  [\mf{\widetilde A_{21}}^t]^{-1}  \mf{\widetilde A_{31}}^t   \Big] \mf s = \mf 0_{N-2r}.$$ 
Set 
$$
   \mf q : = \left( 
   \begin{array}{cc}
   -  [\mf{\widetilde A_{21}}^t]^{-1}  \mf{\widetilde A_{31}}^t \mf s \\
   \mf s
\end{array}
   \right) \implies 
   \mf q^t \mf B_{22} \mf q = 
\Big(
  - \mf p^t \mf{\widetilde A_{31}} [\mf{\widetilde A_{21}}]^{-1}; \mf s^t \Big)
  \left( 
   \begin{array}{cc}
    \mf{\widetilde B_{22}} &  \mf{\widetilde B_{32}}^t \\
    \mf{\widetilde B_{32}} &  \mf{\widetilde B_{33}}\\
\end{array}
   \right) \left( 
   \begin{array}{cc}
   -  [\mf{\widetilde A_{21}}^t]^{-1}  \mf{\widetilde A_{31}}^t \mf s \\
   \mf s
\end{array}
   \right)=0
$$
and this contradicts the assumption that $\mf B_{22}$ is positive definite. 
\end{proof}
By using Lemma~\ref{l:invertible} and the fact that $\mf H_1$ and $\mf{\widetilde G}_{21}$ vanish when the derivative vanishes we can eventually write an explicit expression for $\mf z_3'$ in terms of $\mf u$ and $\mf z_3$. By arguing as in \S\ref{ss:bld} we can then rewrite~\eqref{e:bltw} as a first order ODE for $\mf u$ and $\mf z_3$, 
\begin{equation} \label{e:sylagr}
      \boldsymbol{\omega}' = \mf g( \boldsymbol{\omega}), \qquad 
       \boldsymbol{\omega}= \left(
      \begin{array}{cc}
             \mf u \\
             \mf z_3 \\
      \end{array}
       \right)
\end{equation}
The exact expression of the right-hand side $\mf g$ is not relevant here, what is important is that by linearizing at $ \boldsymbol{\omega}^\ast=(\mf u^\ast, \mf 0_{N-2h})$, where $\mf u^\ast$ satisfies~\eqref{e:uast}, we get  the matrix
$$
\mf D\mf g ( \boldsymbol{\omega}^\ast) =
\left(
\begin{array}{cc}
\mf 0_{N\times N} & \mf L \\
\mf 0_{(N-2h) \times N} & 
\mf N \\
\end{array}
\right), 
$$
where 
\begin{equation*}\begin{split}
    \mf N : = &  \Big[\mf{\widetilde B_{33}}  +  \mf{\widetilde A_{31}} [ \mf{\widetilde A_{21}} ]^{-1} [ \mf{\widetilde B_{22}} [ \mf{\widetilde A_{21}}^t]^{-1}  \mf{\widetilde A_{31}}^t -  \mf{\widetilde B_{32}}^t]
- \mf{\widetilde B_{32}}  [\mf{\widetilde A_{21}}^t]^{-1}  \mf{\widetilde A_{31}}^t   \Big]^{-1} \\ & \times\Big[
\mf{\widetilde A_{31}} [ \mf{\widetilde A_{21}} ]^{-1} \big[\mf{\widetilde A_{22}}
    [\mf{\widetilde A_{21}}^t]^{-1}  \mf{\widetilde A_{31}}^t    - \mf{\widetilde A_{32}}^t \big] - \mf{\widetilde A_{32}}  [\mf{\widetilde A_{21}}^t]^{-1}  \mf{\widetilde A_{31}}^t+
     \mf{\widetilde A_{33}} \Big].
\end{split}
\end{equation*}
and 
$$
    \mf L = 
    \left(
    \begin{array}{cc}
    \mf{\widetilde A_{21}}^{-1}\mf{\widetilde A_{22}}
    [\mf{\widetilde A_{21}}^t]^{-1}  \mf{\widetilde A_{31}}^t   - \mf{\widetilde A_{21}}^{-1} \mf{\widetilde A_{32}}^t \mf z_3
    -
     \mf{\widetilde B_{22}}  [\mf{\widetilde A_{21}}^t]^{-1}  \mf{\widetilde A_{31}}^t  \mf N \\
     -  [\mf{\widetilde A_{21}}^t]^{-1}  \mf{\widetilde A_{31}}^t\\
\mf I_{(N-2h) \times (N-2h)}\\
    \end{array}
    \right). 
$$
Next, we point out that the matrix $\mf N$ is symmetric, and that by direct computations one can check that $\mf N \boldsymbol{\xi}= \mu \boldsymbol{\xi}$ if and only if $(\mf A - \mu \mf B) \mf L \boldsymbol{\xi}= \mf 0_N$.  Next, we recall~\eqref{e:uast} and~\eqref{e:lagr2} and apply~\cite[Lemma 10.3]{BianchiniSpinolo} (see also~\cite[Lemma 4.7]{BianchiniSpinoloARMA}) to conclude that the matrix $\mf N$ has $k-1-h$ eigenvalues with strictly negative real part (each of them counted according to its multiplicity) and the eigenvalue $0$ has multiplicity $1$. 
This implies that the stable manifold of system~\eqref{e:sylagr} linearized at the equilibrium point $\boldsymbol{\omega}^\ast$ has dimension $k-1-h$. 
\section{Solutions of the Riemann and boundary Riemann problem} \label{s:rie}
In this section we go over the construction of the solutions of the so-called Riemann and boundary Riemann problems, which is pivotal to the construction of our wave front-tracking algorithm. The analysis relies on~\cite{BianchiniSpinoloARMA,BianchiniSpinolo,Lax}. Note however that the assumptions that the $k$-th (boundary characteristic) field is either linearly degenerate or genuinely nonlinear provides considerable simplifications with respect to the original analysis in~\cite{BianchiniSpinoloARMA,BianchiniSpinolo}, and hence for the reader's convenience we provide the complete construction of the solution of the boundary Riemann problem in these cases. Since the discussion in the present section is fairly technical, but also very important for the subsequent analysis, in \S\ref{ss:speroserva}  for the reader's convenience we provide an informal overview of the main ideas. In the next paragraphs we then provide the complete  construction with the technical details. 
\subsection{An informal roadmap to the construction of the solution of the boundary Riemann problem} \label{ss:speroserva}
We first recall that the Riemann problem is a Cauchy problem posed by coupling~\eqref{e:claw} with the initial datum 
\be \label{e:datirie}
    \mf v(0, x) = \left\{
    \begin{array}{ll}
    \mf v^- & x < 0 \\
     \mf v^+ & x>0 \\
    \end{array}
    \right., \quad \mf v^-, \mf v^+ \in \R^N. 
\eq
In his pioneering work~\cite{Lax}, Lax constructed an admissible solution of~\eqref{e:claw},\eqref{e:datirie} under the main assumptions that $|\mf v^+ - \mf v^-|$ is sufficiently small and that every vector field is either genuinely nonlinear or linearly degenerate. As a matter of fact, Lax's analysis implies that to solve the Riemann problem~\eqref{e:claw},\eqref{e:datirie} one has to determine $(s_1, \dots, s_N)\in \R^N$ in such a way that 
\be \label{e:risolvoRieP}
    \mf u^- = \mf t_1(\cdot, s_1) \circ \mf t_2 (\cdot, s_2) \circ \dots \circ \mf t_N (\mf u^+, s_N),
\eq
where $\mf u^\pm = \mf u(\mf v^{\pm})$ and we are using the $\mf u$ variables since it will be more convenient for what follows. Also, in~\eqref{e:risolvoRieP} $\mf t_i$ denotes the so-called \emph{wave fan curves of admissible states}\footnote{Note that for technical reasons we introduce a slight difference with respect to the original paper~\cite{Lax}. In particular, $\mf t_i (\widetilde{\mf u}, \cdot)$ is the wave fan curve of admissible \emph{left} states, whereas in the original work~\cite{Lax} Lax constructed the wave fan curve of admissible \emph{right} states.\label{foot}}, whose construction we review in~\S\ref{ss:lax}. 

Let us now move towards the so-called \emph{boundary Riemann problem}: to pose it, we consider the underlying viscous approximation~\eqref{e:vclaw} and correspondigly assign the boundary condition according to Definition~\ref{d:equiv}. The boundary Riemann problem is the initial-boundary value problem posed by coupling~\eqref{e:claw} with the data 
\be \label{e:briedata}
   \mf v(0, \cdot) = \mf v^+, \qquad \mf v(\cdot, 0) \sim_{\mf D} \mf v_b,
\eq
for given data $\mf v^+$, $\mf v_b \in \R^N$. The very basic idea in~\cite{BianchiniSpinoloARMA,BianchiniSpinolo} is to look for an identity like~\eqref{e:risolvoRieP}; to proceed step by step we first consider the simplest possible case and assume that the matrix $\mf B$ in~\eqref{e:symmetric} is positive definite (and henceforth invertible) and that the boundary is not characteristic, which means that all eigenvalues of the matrix $\mf E^{-1}\mf A$ in~\eqref{e:symmetric} are bounded away from $0$.
To fix the notation, we term $n$ the number of strictly negative eigenvalues, and hence $N-n$ is the number of strictly positive eigenvalues. Since the matrix $\mf B$ is positive definite, we can impose a full boundary condition on the solution of~\eqref{e:symmetric}. This is consistent with \S\ref{ss:boucon} and Remark~\ref{r:beta} since in this case $h=\ell=0$ and $\boldsymbol{\beta} (\mf u, \mf u_b) = \mf u- \mf u_b$. Wrapping up, the boundary condition in~\eqref{e:briedata} boils down to the requirement that there is a \emph{boundary layer} $\mf w: \R_+ \to \R$ such that (recall~\eqref{e:bltw})
\be \label{e:andiamoav}
\left\{
\begin{array}{ll}
         \mf A (\mf w) \mf w' = \mf B(\mf w) \mf w'' + \mf G(\mf w, \mf w') \mf w' \\
        \mf w(0) = \mf u_b, \quad \lim_{y \to + \infty} \mf w(y) = \mf{\bar u},
\end{array}
\right.
\eq
where $\mf u_b = \mf u(\mf v_b)$ and the trace $\mf{\bar u} = \mf u (\mf v(t, 0))$ does not depend on time since the solution we construct is self-similar. To study~\eqref{e:andiamoav}, we employ the standard Stable Manifold Theorem, which dictates that~\eqref{e:andiamoav} has a solution if and only if $\mf u_b$ lies on the so-called \emph{stable manifold} computed at the point $\mf{\bar u}$. A by-now classical lemma (see for instance~\cite{BGSZ} or also~\cite[Lemma 10.2]{BianchiniSpinolo}) states that the dimension of the stable manifold is $n$. In the following, we term $\boldsymbol{\psi}(\mf{\bar u}, \cdot)$ a map parameterizing the stable manifold. To solve the boundary Riemann problem we then replace~\eqref{e:risolvoRieP} with the identity 
\be \label{e:risolvoBRieP}
 \mf u_b = \boldsymbol{\psi}(\cdot, \xi_1, \dots, \xi_n)  \circ \mf t_{n+1} (\cdot, s_{n+1}) \dots \circ \mf t_N (\mf u^+, s_N),
\eq 
and by relying again on~\cite[Lemma 10.2]{BianchiniSpinolo} we conclude that this uniquely determines the values $(\xi_1, \dots, \xi_n, s_{n+1}, \dots, s_N)$. Once these values are known, we set $\mf{\bar u} : = \mf t_{n+1} (\cdot, s_{n+1}) \dots \circ \mf t_N (\mf u^+, s_N)$ and point out that, by construction, $\mf{\bar u}$ can be connected to $\mf u^+$ by rarefaction waves, shocks or contact discontinuities with \emph{positive} speed, and to $\mf{u_b}$ by a boundary layer satisfying~\eqref{e:andiamoav}. 

We now move forward and consider the case where the matrix $\mf B$ is still invertible, but the boundary is characteristic, namely one eigenvalue of $\mf E^{-1}\mf A$ can attain the value $0$. As in~\eqref{e:uast}, we assume that the vanishing eigenvalue is the $k$-th.  In this case, the basic idea of the construction is still that the trace $\mf{\bar u}$ must be connected to $\mf u_b$ by a boundary layer, but the analysis of~\eqref{e:andiamoav} is much more delicate. Loosely speaking, this is first of all due to the fact that by linearizing at the point $\mf{\bar u}$ one obtains a non-trivial center space. In other words, in general there are solutions of the equation at the first line of~\eqref{e:andiamoav} that lie on the stable manifold and hence converge exponentially fast, but there might also be converging solutions lying on a center manifold, and hence having a slower decay at infinity. Also, there might be shock or rarefaction waves (if the $k$-th vector field is genuinely nonlinear) or contact discontinuities (if the $k$-th vector field is linearly degenerate) of the $k$-th family with non-negative speed, which should be taken into account. To handle these challenges, we proceed as follows. 

First, we construct the so-called \emph{characteristic wave fan curve of admissible states} $\boldsymbol{\zeta}_k$. In \S\ref{ss:zetakgnl} we assume that the $k$-th vector field is genuinely nonlinear: in this case, the curve $\boldsymbol{\zeta}_k (\mf{\widetilde u}, \cdot)$ describes the states that can be connected to $\mf{\widetilde u}$ by a shock or a rarefaction with \emph{non-negative} speed close to $\lambda_k (\mf{\widetilde u})$, and by boundary layers (that is, solutions of~\eqref{e:bltw}) lying on the\footnote{The center manifold obtained by linearizing a given system at a given point is in general not unique, but me can arbitrarily fix one} center manifold. In \S\ref{ss:case2} we assume that the $k$-th vector field is linearly degenerate: the curve $\boldsymbol{\zeta}_k (\mf{\widetilde u}, \cdot)$ describes the states that can be connected to $\mf{\widetilde u}$ by a contact discontinuity with \emph{non-negative} speed equal to $\lambda_k (\mf{\widetilde u})$, or by boundary layers lying on the center manifold.

In \S\ref{sss:cbla} we then complete the analysis by considering \emph{all} boundary layers, not only those lying on the center manifold.  From the technical standpoint, the main tool we use is the Slaving Manifold Theorem, a classical result in dynamical systems theory, see~\cite{KatokH} for an extended discussion and~\cite[\S11]{BianchiniSpinolo} for a very brief exposition oriented towards the applications we need in the following. We now provide a very handwaving exposition of the main ideas underpinning the construction of the slaving manifold, and we refer to~\cite[\S11]{BianchiniSpinolo} for the precise statements. Fix an autonomous systems of ODEs\footnote{Note that we are using a different notation compared to~\cite[\S11]{BianchiniSpinolo} and here we denote by $ \boldsymbol{\omega}$ and $\mf a$ what in~\cite{BianchiniSpinolo} is termed $\mf v$ and $\mf g$, respectively. \label{foot3}} like~\eqref{e:bltw2}
an equilibrium point $\boldsymbol{\omega}^\ast$, $\mf a (\boldsymbol{\omega}^\ast) = \mf 0$, and an orbit $\boldsymbol{\omega}_0$ of~\eqref{e:bltw2} lying on a center manifold and confined in a sufficiently small neighborhood of $\boldsymbol{\omega} ^\ast$. Very loosely speaking, the Slaving Manifold Theorem characterizes the orbits of~\eqref{e:bltw2} that converge exponentially fast to $\boldsymbol{\omega}_0$ at $+ \infty$. Also, it provides a description of these orbits: each of them is given by $\boldsymbol{\omega}_0$, plus an orbit lying on the so-called uniformly stable manifold (whose dimension is the same as the number of eiegenvalues of $\mf D \mf a(\boldsymbol{\omega}^\ast)$ with strictly negative real part) and a ``perturbation term". The ``perturbation term" (which could also be called ``interaction term") is due to the nonlinearity of $\mf a$ and it is small, in a suitable quantified sense, with respect to the previous two, see~\cite[eq. (11.6)]{BianchiniSpinolo}. 
Note that in the applications to boundary layers discussed in the following the orbit $\boldsymbol{\omega}_0$ is obtained from the curve $\boldsymbol{\zeta}_k$, and can boil down to a single equilibrium point in some cases\footnote{We refer to the comments after~\eqref{e:festanatale2} below for some further details about the instances where the orbit $\boldsymbol{\omega}_0$ boils down to an equilibrium point.}. At the end of the day, by applying the Slaving Manifold Lemma we arrive at the formula 
\be \label{e:risolvoBRieP2}
   {\boldsymbol \phi} (\cdot, \xi_1, \dots, \xi_{k-1}, s_k)\circ \mf t_{k+1}(\cdot, s_{k+1}), \dots, t_N (\mf u^+, s_N)= \mf u_b,
\eq
which is the analogous of~\eqref{e:risolvoRieP} and~\eqref{e:risolvoBRieP}. In the above expression, 
${\boldsymbol \phi} $ is a suitable function whose precise expression is provided in \S\ref{sss:cbla}, see in particular formula~\eqref{e:bl10}. One can then show that the identity~\eqref{e:risolvoBRieP2} uniquely determines the values of $(\xi_1, \dots, \xi_{k-1}, s_{k}, \dots, s_N)$ and from them one reconstructs the solution of the boundary Riemann problem.

To complete our overview, we are left to add one last item to the picture, namely the fact that the matrix $\mf B$ in~\eqref{e:symmetric} is in general not invertible. As pointed out in \S\ref{ss:boucon}, this first of all requires a more complicated formulation of the boundary condition, which is expressed by the function $\boldsymbol{\beta}$ introduced in there. Note that this also applies to the boundary layers, so the system we have to consider is now 
\be
\label{e:andiamoav2}
 \left\{
\begin{array}{ll}
         \mf A (\mf w) \mf w' = \mf B(\mf w) \mf w'' + \mf G(\mf w, \mf w') \mf w' \\
        \boldsymbol{ \beta}(\mf w(0), \mf u_b) = \mf 0_N, \quad \lim_{y \to + \infty} \mf w(y) = \mf{\underline u},
\end{array}
\right.
\eq
rather than~\eqref{e:andiamoav}. In the above system, due to the fact that we are considering the boundary characteristic case the asymptotic state $ \mf{\underline u}$ is in general different from the trace $\mf{\bar u} = \mf u(t, 0)$. Note that Definition~\ref{d:equiv} dictates the relation between the asymptotic state
$\mf{\underline u}$ and $\mf{\bar u}$: either they coincide or there is a $0$-speed Lax admissible shock (if the $k$-th vector field is genuinely nonlinear) or $0$-speed contact discontinuity (if the $k$-th vector field is linearly degenerate) connecting $\mf{\underline u}$ (on the left) and $\mf{\bar u}$ (on the right). 

To write down the analogous of~\eqref{e:risolvoRieP},\eqref{e:risolvoBRieP} and~\eqref{e:risolvoBRieP2} it is convenient to split between {\sc cases A)} and {\sc C)} in Hypothesis~\ref{h:eulerlag} on one side and {\sc case B)} on the other side. 
In cases A) and C) we impose 
\be \label{e:solbrparma}
    \boldsymbol{\beta} ( {\boldsymbol \phi} (\cdot, \xi_{\ell +1}, \dots, \xi_{k-1}, s_k)\circ \mf t_{k+1}(\cdot, s_{k+1}), \dots, t_N (\mf u^+, s_N), \mf u_b) = \mf 0_{N}. 
\eq
In the above expression, $\ell$ denotes the number of negative eigenvalues of $\mf A_{11}$ in {\sc case A)}, and $\ell=h$ (the dimension of the kernel of $\mf B$) in {\sc case C)}. Note that this yields that the dimension of the vector $(\xi_{\ell +1}, \dots, \xi_{k-1}, s_{k}, \dots, s_N)$ is exactly the number of boundary conditions imposed through the function $\boldsymbol{\beta}$ given by~\eqref{e:betagnl} and \eqref{e:betalindeg}, respectively, see Remark~\ref{r:beta}. The analysis in~\cite{BianchiniSpinolo} dicates that~\eqref{e:solbrparma} uniquely determines the values of $(\xi_{\ell +1}, \dots, \xi_{k-1}, s_{k}, \dots, s_N)$ and henceforth the solution of the boundary Riemann problem. We now consider {\sc case B)} in Hypothesis~\ref{h:eulerlag}: in this case the number of non-positive eigenvalues of $\mf A_{11}$ \emph{depends} on $\mf u_b$, and is denoted by $\ell (\mf u_b)$. For instance, in the case of the Navier-Stokes and viscous MHD equations written in Eulerian coordinates this number depends on the sign of the fluid velocity, which is the second component of $\mf u_b$. In place of~\eqref{e:solbrparma} we have then 
 \be \label{e:solbrcpam}
    \boldsymbol{\beta} ( {\boldsymbol \phi} (\cdot, \xi_{\ell (\mf u_b) +1}, \dots, \xi_{k-1}, s_k)\circ \mf t_{k+1}(\cdot, s_{k+1}), \dots, t_N (\mf u^+, s_N), \mf u_b) = \mf 0_{N}, 
\eq  
and as a consequence of the analysis in~\cite{BianchiniSpinolo} we have that the above identity 
 uniquely determines the values of $(\xi_{\ell +1}, \dots, \xi_{k-1}, s_{k}, \dots, s_N)$ and henceforth the solution of the boundary Riemann problem. 
\subsection{The Lax wave fan curves of admissible states}\label{ss:lax} 
The key point in Lax's pioneering paper~\cite{Lax} on the solution of the Riemann problem is the construction of the so-called \emph{$i$-th wave fan curve of admissible states}, which we denote by $\mf t_i$. We now briefly recall its definition and fix some notation. We first consider the genuinely nonlinear case, next the linearly degenerate case. 
\subsubsection{Genuinely nonlinear vector field} 
Fix a state $\mf {\widetilde u}$ and $i=1, \dots, N$ and assume that the $i$-th vector field is genuinely nonlinear~\eqref{e:gnl}. We term $\mf i_i(\mf {\widetilde u}, \cdot)$ the rarefaction wave curve, which is the integral curve passing through $\mf {\widetilde u}$ of the vector field $\mf r_i$, namely the solution of the Cauchy problem
\begin{equation} \label{e:rarefdef}
          \displaystyle{\frac{\partial \mf i_i  (\mf {\widetilde u}, s) }{\partial s} = \mf r_i (\mf i_i(\mf {\widetilde u}, s))}
          \qquad 
          \mf i_i (\mf {\widetilde u}, 0) = \mf {\widetilde u}. 
\end{equation}
We also term $\mf h_i (\mf {\widetilde u}, \cdot)$ the curve of the states that can be connected to $\mf {\widetilde u}$ by a (non necessarily Lax admissible) shock  with speed close to $\lambda_i (\mf {\widetilde u})$, that is the set of points such that the Rankine-Hugoniot conditions 
\begin{equation} \label{e:rh}
    \mf f\circ \mf v (\mf h_i(\mf {\widetilde u}, s) ) - \mf f \circ \mf v(\mf {\widetilde u}) = \sigma_i (\mf {\widetilde u}, s) [\mf g \circ \mf v(\mf h_i) -\mf g \circ \mf v(\mf{\widetilde u})]
\end{equation}
hold for some $\sigma_i(\mf {\widetilde u}, s)$ close to $\lambda_i (\mf {\widetilde u})$. In the previous expression, $\mf v$ denotes the change of variables which allows to pass from~\eqref{e:symmetric} to~\eqref{e:vclaw}.  The analysis in~\cite{Lax} dictates that the \emph{$i$-th
wave fan curve of admissible states} is defined by setting\footnote{We recall footnote~\ref{foot} on page~\pageref{foot} and explicitely remark that, since we are defining the wave fan curve of admissible \emph{left} states, the \emph{negative} values of $s$ correspond to rarefactions, and the \emph{positive} to shocks.} 
\begin{equation} \label{e:laxc}
    \mf t_i (\mf {\widetilde u}, s) : = 
     \left\{
          \begin{array}{ll}
               \mf i_i (\mf {\widetilde u},s) & s <0  \\
               \mf h_i (\mf {\widetilde u},s) & s > 0                                             \\ 
         \end{array}
   \right.
\end{equation}
and comprises all the states $\mf u^-$ (on the left) that can be connected to $\mf {\widetilde u}$ on the right by either a Lax admissible shock (if $s>0$) or by a rarefaction (if $s<0$). 

Since we will need it in the following, we introduce some further notation: we recall that $\sigma_i (\mf {\widetilde u}, s)$ denotes the speed of the shock connecting $\mf {\widetilde u}$ with $\mf h_i (\mf {\widetilde u}, s)$, that is it satisfies~\eqref{e:rh}.
If the $k$-th characteristic field is genuinely nonlinear~\eqref{e:gnl}  the map $s \mapsto \sigma_k (\mf {\widetilde u}, s)$ is strictly monotone and henceforth invertible. 
We term $\underline s(\mf {\widetilde u})$ the unique value satisfying 
\be \label{e:underlines}
      \sigma_k (\mf {\widetilde u}, \underline s(\mf {\widetilde u})) =0. 
\eq
Note that, by the Implicit Function Theorem we have that the map $\mf {\widetilde u}  \mapsto  \underline s(\mf {\widetilde u})$ is Lipschitz continuous, namely 
\be \label{e:lipunders}
      | \underline s(\mf u_1) - \underline s (\mf u_2) | \leq \unpo |\mf u_1 - \mf u_2|. 
\eq
Owing again to genuine nonlinearity, the map $s \mapsto \lambda_k (\mf i_k(\mf{\widetilde u}, s))$ is also strictly monotone, we term $ \bar s (\mf{\widetilde u})$ the unique value satisfying 
\begin{equation} \label{e:baresse}
   \lambda_k (\mf i_k(\mf{\widetilde u}, \bar s (\mf{\widetilde u})) = 0. 
   \end{equation}
\subsubsection{Linearly degenerate vector field}
If the $i$-th vector field is linearly degenerate we set 
\begin{equation}
\label{e:lax:lindeg}
      \mf t_i (\mf {\widetilde u},s)=:  \mf i_i (\mf {\widetilde u},s) ,
\end{equation}
where $\mf i$ is the same as in~\eqref{e:rarefdef}. 
\subsection{The characteristic wave fan curve $\boldsymbol{\zeta}_k$ in the genuinely nonlinear case} \label{ss:zetakgnl}
To ease the exposition we assume~\eqref{e:a11nz}, that is case A) in Hypothesis~\ref{h:eulerlag} (the analysis in the other cases is similar).  Note this covers the case of the inviscid limit of the Navier-Stokes equations written in Eulerian coordinates when the fluid velocity is close to $\pm c$, where $c$ denotes the sound speed, see~\S\ref{sss:appl1}. 
\begin{lemma} \label{l:cwfc} Assume~\eqref{e:gnl} and let $\mf u^\ast$ be the same as in~\eqref{e:uast}. There is a sufficiently small constant $\nu>0$\footnote{The constant $\nu$ only depends on the functions $\mf E$, $\mf A$, $\mf B$, $\mf G$, and on the value $\mf u^\ast$} such that if $|\mf{\widetilde u} - \mf u^\ast| \leq \nu$, then there is curve $\boldsymbol{\zeta}_k (\mf{\widetilde u}, \cdot): ]-\nu, \nu[ \to \R^N$ satisfying the following properties:
\begin{itemize}
\item[1.] $\boldsymbol{\zeta}_k (\mf{\widetilde u}, 0) = \mf{\widetilde u}$. 
\item[2.] For every $s \in ]-\nu, \nu[$ one (and only one) of the following conditions holds true:
\begin{itemize}
\item[i)] $\mf{\widetilde u}$ (on the right) and $\boldsymbol{\zeta}_k (\mf{\widetilde u}, s)$ (on the left) are connected by a Lax admissible shock with nonnegative speed
\item[ii)] there is a value $\bar{\mf u}$ (possibly coinciding with either $\mf{\widetilde u}$ or $\boldsymbol{\zeta}_k (\mf{\widetilde u}, s)$) such that 
\begin{itemize}
\item[ii$_1$)] 
$\mf{\widetilde u}$ (on the right) and 
$\bar{\mf u}$ (on the left) are connected by a rarefaction wave with nonnegative speed 
\item[ii$_2$)] there is a \emph{boundary layer} $\mf w: [0, + \infty[ \to \R^N$ such that 
\be \label{e:bleq}
\left\{
\begin{array}{ll}
         \mf A (\mf w) \mf w' = \mf B(\mf w) \mf w'' + \mf G(\mf w, \mf w') \mf w' \\
         \boldsymbol{\beta}(\mf  w(0) , \boldsymbol{\zeta}_k (\mf{\widetilde u}, s)) = \mf 0, \quad \lim_{x \to + \infty} \mf w(x) = \bar{\mf u}. 
\end{array}
\right. 
\eq
\end{itemize}
\end{itemize}
\item[3.] The function $\boldsymbol{\zeta}_k$ is Lipschitz continuous with respect to both variables. It is also differentiable with respect to $s$ at $s=0$ and satisfies 
\be \label{e:derzeta}
   \left. \frac{d \boldsymbol{\zeta}_k (\mf{\widetilde u}, s)}{d s} \right|_{s =0}=
   \left\{
   \begin{array}{ll}
   \mf r_k (\mf{\widetilde u}) &  \lambda_k (\mf{\widetilde u}) \ge 0 \\
   \mf r_c (\mf{\widetilde u}, 0) & \lambda_k (\mf{\widetilde u}) < 0. \\ 
   \end{array}
   \right.
\eq
\end{itemize}
\end{lemma}
Note that in point ii) if $\bar{\mf u}$ coincides with $\mf{\widetilde u}$ then condition ii$_1)$ becomes trivial, if $\bar{\mf u}$ coincides with $\boldsymbol{\zeta}_k (\mf{\widetilde u}, s)$ then condition ii$_2)$ becomes trivial. A more technical observation is provided in Remark~\ref{r:fabio}.
\begin{remark} \label{r:fabio}
One might wonder why in item ii)$_1$ in the statement of Lemma~\ref{l:cwfc} one does not take into account the possibility that $\mf{\widetilde u}$ (on the right) and $\bar{\mf u}$ (on the left) are connected by a Lax admissible shock with nonnegative speed: this is due to the genuine nonlinearity assumption. Indeed, assume by contradiction that  $\mf{\widetilde u}$ (on the right) and $\bar{\mf u}$ (on the left) are connected by a Lax admissible shock with nonnegative speed; then $\lambda_k (\mf{\widetilde u})>0$ and hence there is no nontrivial orbit satisfying~\eqref{e:bleq} and lying on a center manifold. 
\end{remark}
We now provide the proof of Lemma~\ref{l:cwfc} by constructing  $\boldsymbol{\zeta}_k$. 
The exposition is organized as follows: in \S\ref{sss:bikappa} we define the function $\mf b_k$ describing the point connected to $\mf{\widetilde u}$ by a boundary layer, and describe some of its properties. In \S\ref{sss:defz} we define the function $\boldsymbol{\zeta}_k$ and discuss the heuristic ideas underpinning the definition. In \S\ref{sss:concl} we show that  $\boldsymbol{\zeta}_k$ satisfies the properties in the statement of Lemma~\ref{l:cwfc}. 
\subsubsection{The boundary layer function $\mf b_k (\mf{\widetilde u}, \cdot)$} \label{sss:bikappa}
We consider the Cauchy problem obtained by coupling 
\begin{equation} \label{e:rescala}
     \displaystyle{\frac{d \mf u}{d \tau} =  \mf r_c  (\mf u, z_c)} \qquad     \displaystyle{\frac{d z_c}{d \tau}=  \theta_c  (\mf u, z_c)} 
     \eq 
with the initial datum
\be
     \label{e:id}
  \mf u(\tau=0) = \mf{\widetilde u}, \qquad \qquad   z_c(\tau=0) = 0
\eq
where $\mf r_c$ is the same vector as in~\eqref{e:errec}. Note that~\eqref{e:rescala} is formally obtained from~\eqref{e:bltw3} by dividing by $z_c$. 
We define $ \mf b_k (\mf{\widetilde u}, s) $ by setting  
\be \label{e:bkappa}
    \mf b_k (\mf{\widetilde u}, s) : = \left(
    \begin{array}{cc}
         \mf u \\
          z_c \mf p_c (\mf u, z_c) \\ 
          \end{array}\right) (\tau=s),
\eq
where $\mf p_c$ is the same vector as in Lemma~\ref{l:emmec} and $(\mf u, z_c)$ is the solution of~\eqref{e:rescala},\eqref{e:id} at $\tau=s$. Note that the  function $\mf b_k$ attains values in $\R^{2N-h}$ and lies on the same center manifold 
$\mathcal M^c$ as in Lemma~\ref{l:emmec}. 
\begin{lemma}
The function $\mf b_k$ is of class $C^2$ with respect to both $\mf{\widetilde u}$ and $s$. Also,  
\be \label{e:siattacca}
       \lambda_k (\mf{\widetilde u}) =0 \implies \left. \frac{\partial \mf b_k (\mf{\widetilde u}, s)}{\partial s}\right|_{s=0} = 
       \left(
       \begin{array}{cc}
       \mf r_k (\mf{\widetilde u})\\
       \mf 0_{N-h} \\
       \end{array}
       \right),
\eq
where $\mf r_k$ is an eigenvector of $\mf E^{-1} \mf A$ associated to $\lambda_k$. 
\end{lemma}
\begin{proof}
The regularity of $\mf b_k$ follows from the classical theory of ODEs. The implication~\eqref{e:siattacca} follows from Lemma~\ref{l:lambdac} and the equalities
$\partial z_c (s=0)/\partial s =\theta_c (\mf{\widetilde u}, 0)$. 
\end{proof}
\begin{lemma}\label{l:convessa} Assume that $\mf u$ and $z_c$ solve the Cauchy problem~\eqref{e:rescala},\eqref{e:id}. There is a sufficiently small constant $\nu>0$ such that  if $|\mf u(s) - \mf u^\ast|<\nu$ and $|z_c (s)|<\nu$  then
\be \label{e:edadim}
   \left. \frac{d \lambda_k (\mf u (\tau))}{d\tau } \right|_{\tau =s} \ge   \frac{d}{2}>0, 
\eq
and
\be \label{e:giallo}
   \left. \frac{d \theta_c  (\mf u (\tau), z_c (\tau))}{d\tau } \right|_{\tau =s} \ge   \frac{d}{2 a_1}>0, 
\eq
provided $d$ is the same constant as in~\eqref{e:gnl} and $a_1$ is an upper bound for the function $a (\mf u)$ in the statement of Lemma~\ref{l:lambdac}.
\end{lemma}
\begin{proof}
To establish~\eqref{e:edadim} we point out that
\[ \begin{split}
   \left.  \frac{d \lambda_k (\mf u (\tau))}{d\tau } \right|_{\tau =s} & \stackrel{\eqref{e:rescala}}{=} 
    \nabla \lambda_k (\mf u(s)) \cdot \mf r_c (\mf u (s), z_c (s)) =
    \nabla \lambda_k (\mf u^\ast) \cdot \mf r_c (\mf u^\ast, 0) + \unpo
    |\mf u(s) - \mf u^\ast| + \unpo |z_c (s)| \\ &
    \stackrel{\eqref{e:uast}, \, \text{Lemma~\ref{l:lambdac}}}{=}
     \nabla \lambda_k (\mf u^\ast) \cdot \mf r_k (\mf u^\ast) 
     + \unpo
    |\mf u(s) - \mf u^\ast| + \unpo |z_c (s)| \\ & \stackrel{\eqref{e:gnl}}{\ge} d + 
    \unpo
    |\mf u(s) - \mf u^\ast| + \unpo |z_c (s)| \ge \frac{d}{2},
     \end{split}
\]
provided the constant $\nu$ is sufficiently small. We now establish~\eqref{e:giallo}:
\begin{equation*}
\begin{split}
           & \left.  \frac{d \theta_c  (\mf u (\tau), z_c (\tau))}{d\tau } \right|_{\tau =s}    =  
           \nabla_{\mf u} \theta_c  (\mf u (\tau), z_c (\tau))\frac{d \mf u}{d \tau} + \frac{\partial \theta_c}{\partial z_c } \frac{d z_c}{d \tau} \\ & 
           \stackrel{\eqref{e:rescala}}{=}   \nabla_{\mf u} \theta_c  (\mf u (\tau), 0) \cdot \mf r_c (\mf u, 0)+ \unpo |z_c (s)| + \unpo |\theta_c (\mf u(\tau), z_c(\tau))| \\
          & \stackrel{\eqref{e:thetaczero}}{=} a (\mf u)^{-1} \nabla \lambda_k (\mf u) \cdot \mf r_c (\mf u, 0) + 
             \unpo |\theta_c (\mf u,0) |   + \unpo |z_c (s)| + \unpo |\mf u(s) - \mf u^\ast| \\
          & \stackrel{\mathrm{Lemma}~\ref{l:lambdac}}{\ge} \frac{1}{ a_1}  \nabla \lambda_k (\mf u^\ast) \cdot \mf r_k (\mf u^\ast)
          + a_1 |z_c (s)| + \unpo |\mf u(s) - \mf u^\ast|  \ge  \frac{d}{2 a_1}.  \qedhere
\end{split} 
\end{equation*}
\end{proof}
\begin{lemma}
\label{l:min0} Assume $\lambda_k (\mf{\widetilde u})<0$ and that $z_c$ is the same as in~\eqref{e:rescala}; then there there is $\underline s > 0$, depending on $\mf{\widetilde u}$,  such that 
\be \label{e:robinia} 
z_c (\tau=\underline s)=0. 
\eq
Also, there is a $0$-speed traveling wave $\mf w: \R \to \R^N$ satisfying 
\be \label{e:etw}
  \left\{
\begin{array}{ll}
         \mf A (\mf w) \mf w' = \mf B(\mf w) \mf w'' + \mf G(\mf w, \mf w') \mf w' \\
          \lim_{x \to - \infty} \mf w(x) = \boldsymbol{\pi}_{\mf u} \circ \mf b_k(\mf{\widetilde u}, \underline s),  \quad \lim_{x \to + \infty} \mf w(x) =  \mf{\widetilde u},
\end{array}
\right. 
\eq
where $\mf b_k$ is as in~\eqref{e:bkappa}. 
If $s < \underline s$ then there is a boundary layer $\mf w: [0, + \infty[ \to \R^N$ satisfying~\eqref{e:bleq} with $\boldsymbol{\zeta}_k (\mf{\widetilde u}, s) = \boldsymbol{\pi}_{\mf u} \circ \mf b_k (\widetilde{\mf u}, s)$, where the projection $\boldsymbol{\pi}_{\mf u}$ is the same as in~\eqref{e:proj}, and $\bar{\mf u} = \widetilde{\mf u}.$ 
\end{lemma}
\begin{proof}We proceed according to the following steps.\\
{\sc Step 1:} we show the existence of $\underline s>0$ satisfying~\eqref{e:robinia}. 
First, we recall that $z_c (\tau=0) =0$ and point out that 
\be \label{e:dermin0}
     \left.  \frac{d z_c (\tau)}{d\tau } \right|_{\tau =0} \stackrel{\eqref{e:rescala},\eqref{e:id}}{=}
     \theta_c (\mf{\widetilde u}, 0) \stackrel{\eqref{e:thetaczero}}{=} a^{-1}(\mf{\widetilde u}) \lambda_k (\mf{\widetilde u}) <0. 
\eq
and hence $z_c(\tau) <0$ for $\tau>0$ sufficiently small.
On the other hand, 
\begin{equation}\label{e:esplode}
    z_c (\tau) = \int_0^{\tau} \theta_c (\mf u(\xi), z_c (\xi)) d\xi \stackrel{\eqref{e:giallo}}{\ge} \underbrace {\theta_c (\mf{\widetilde u}, 0)}_{= a(\mf{\widetilde u})\lambda_k (\mf{\widetilde u}) <0} \tau + \frac{d}{4 a_1} \tau^2,
\end{equation}
where as before $a_1$ is an upper bound for the function $a (\mf u)$ in the statement of Lemma~\ref{l:lambdac}. We can then conclude that $z_c$ attains the value $0$ for some $\underline s>0$. \\
{\sc Step 2:} we fix $s < \underline s$ and exhibit  a boundary layer $\mf w: [0, + \infty[ \to \R^N$ satisfying~\eqref{e:bleq} with $\boldsymbol{\zeta}_k (\mf{\widetilde u}, s) = \boldsymbol{\pi}_{\mf u} \circ \mf b_k (\widetilde{\mf u}, s)$, and $\bar{\mf u} = \widetilde{\mf u}.$ Just to fix the ideas we assume $s>0$ (the proof in the case $s<0$ is analogous). Owing to~\eqref{e:dermin0} and by definition~\eqref{e:robinia} of $\underline s$ we have $
    z_c(\tau) <0$ for every $\tau \in ]0, s]$.
We consider the Cauchy problem 
\be \label{e:caux}
\displaystyle{\frac{d x_c}{d \tau} = \frac{1}{z_c (\tau)} } \qquad 
x_c (s) = 0. 
\eq
Owing to~\eqref{e:esplode},  
\be \label{e:esplode2}
    \lim_{\tau\to 0^+} x_c(\tau) \stackrel{\eqref{e:caux}}{=}   \lim_{\tau\to 0^+} \left( -  \int_\tau^{s} \frac{1}{z_c (\xi)} d \xi \right) \stackrel{\eqref{e:esplode}}{\ge}  
    \lim_{\tau\to 0^+} \left( -  \int_\tau^{s} \frac{1}{  \theta_c (\mf{\widetilde u}, 0) \tau} d \xi \right) \stackrel{  \theta_c (\mf{\widetilde u}, 0) < 0}{=}+ \infty
\eq
and hence $x_c$ is an invertible map from $]0, s]$ onto $[0, + \infty[$. We term $x_c^{-1}$ its inverse and we consider the maps $\mf u(x): = \mf  u (x_c^{-1} (x))$, 
$z_c (x) := z_c (x_c^{-1} (x))$. By combining~\eqref{e:rescala} and~\eqref{e:caux} we conclude that $(\mf u, z_c)$ is a solution of~\eqref{e:bltw3} and hence 
by taking $\mf w= \mf u$ we obtain a solution of the equation at the first line of~\eqref{e:bleq}.  By recalling~\eqref{e:esplode2} and~\eqref{e:id}
we get $\mf w(0)= \boldsymbol{\pi}_{\mf u} \circ \mf b_k (\widetilde{\mf u}, s)$ and $\lim_{x \to + \infty} \mf w(y)= \widetilde{\mf u} $. \\
{\sc Step 3:} we conclude the proof of the lemma by considering the case $s=\underline s$.
We consider the Cauchy problem obtained by coupling the same ODE as in~\eqref{e:caux} with the condition $x_c (\underline s/2) =0$. By arguing as in~\eqref{e:esplode2} we get that $\lim_{\tau \to 0^+} x_c (\tau) = + \infty.$ Owing to~\eqref{e:giallo} the function $\tau \mapsto \theta_c (\mf u(\tau), z_c (\tau))$ is monotone increasing and, since $z_c (\underline s) =0$, then 
$$
    z_c (\tau) \ge \theta_c (\mf u(\underline s), z_c (\underline s)) (\tau - \underline s) \quad \text{for every $\tau \in [0, \underline s]$.}
$$
Note furthermore that $\theta_c (\mf u(\underline s), z_c (\underline s)) >0$. We conclude that
$$
       \lim_{\tau\to \underline s^-} x_c(\tau) \stackrel{\eqref{e:caux}}{=}   \lim_{\tau\to \underline s^-}   \int_{\underline s/2}^\tau \frac{1}{z_c (\xi)} d \xi \leq  
    \lim_{\tau\to \underline s^-}   \int_{\underline s/2}^\tau \frac{1}{\theta_c (\mf u(\underline s), z_c (\underline s))(\tau - \underline s) } d \xi = - \infty. 
$$
By arguing as before we establish~\eqref{e:etw}. 
\end{proof}
We recall that the value $\underline s(\mf{\widetilde u})$ is defined by~\eqref{e:underlines}. 
Note that, owing to~\eqref{e:etw}, there is a Lax admissible 0-speed shock connecting 
$\mf{\widetilde u}$ (on the right) with $\boldsymbol{\pi}_{\mf u} \circ\mf b_k (\mf{\widetilde u}, \underline s)$ (on the left). This implies that  
we can choose the parametrization of the curve $\mf t_k$ in such a way that $\underline s = \underline s (\mf{\widetilde u})$ and  that 
\be \label{e:pera}
     (\mf t_k (\mf{\widetilde u}, \underline s (\mf{\widetilde u})), \mf 0_{N-h}) = \mf b_k (\mf{\widetilde u}, \underline s (\mf{\widetilde u})) \quad \text{if $\lambda_k (\mf{\widetilde u}) \leq0$}. 
\eq
We also have the following result. 
\begin{lemma}
\label{l:min02} Assume $\lambda_k (\mf{\widetilde u})=0$, then for every $s<0$  there is a boundary layer $\mf w: [0, + \infty[ \to \R^N$ satisfying~\eqref{e:bleq} with $\boldsymbol{\zeta}_k (\mf{\widetilde u}, s) = \boldsymbol{\pi}_{\mf u} \circ \mf b_k (\widetilde{\mf u}, s)$ and $\bar{\mf u} = \widetilde{\mf u}.$ 
\end{lemma}
\begin{proof}
Owing to~\eqref{e:thetaczero} we have $\theta_c (\mf{\widetilde u}, 0)=0$. We now consider the Cauchy problem~\eqref{e:rescala},\eqref{e:id} and by arguing as in~\eqref{e:esplode} we conclude that 
$z_c(s)>0$ for $s<0$. We then consider the Cauchy problem~\eqref{e:caux} and by arguing as in {\sc Step 2} in the proof of Lemma~\ref{l:min0} we get the desired result. 
\end{proof}
\subsubsection{Definition of $\boldsymbol{\zeta}_k (\mf{\widetilde u}, \cdot)$} \label{sss:defz}
We now provide the definition of the curve $\boldsymbol{\zeta}_k (\mf{\widetilde u}, \cdot)$ by distinguishing the cases $\lambda_k (\mf{\widetilde u}) \ge 0$ and  
$\lambda_k (\mf{\widetilde u}) <0$. \\
{\sc Case $\lambda_k (\mf{\widetilde u}) \ge 0$} We recall that $\underline s (\mf{\widetilde u})$ and $\bar s (\mf{\widetilde u})$ are defined by~\eqref{e:underlines} and~\eqref{e:baresse}, respectively. If $\lambda_k (\mf{\widetilde u}) \ge 0$ then owing to genuine nonlinearity~\eqref{e:gnl} we have $\bar s (\mf{\widetilde u})\leq 0$. We then set 
\begin{equation} \label{e:uguale1}
    \boldsymbol{\zeta}_k (\mf{\widetilde u}, s):= 
    \left\{
    \begin{array}{ll}
    \mf t_k(\mf{\widetilde u}, s) & s \ge  \bar s (\mf{\widetilde u}) \\
     \boldsymbol{\pi}_{\mf u}\circ \mf b_k(\mf t_k (\mf{\widetilde u},  \bar s (\mf{\widetilde u})), s - \bar s (\mf{\widetilde u})) & s \leq  \bar s (\mf{\widetilde u}) \\
    \end{array}
    \right. \quad \text{if $\lambda_k (\mf{\widetilde u}) \ge 0$}.
\end{equation}
Note that  
\begin{equation} \label{e:uguale2}
    \left. \frac{\partial  \mf t_k(\mf{\widetilde u}, s))}{\partial s} \right|_{s=\bar s(\widetilde u)} = \mf r_k (   \mf t_k(\mf{\widetilde u}, \bar s(\mf{\widetilde u})))
    \stackrel{\eqref{e:baresse},\eqref{e:siattacca}}{ =}
     \left. \frac{\partial [ \boldsymbol{\pi}_{\mf u} \circ\mf b_k (\mf t_k (\mf{\widetilde u},  \bar s (\mf{\widetilde u})), s -  \bar s (\mf{\widetilde u}))] }{ \partial s}
    \right|_{s=\bar s(\widetilde u)} 
\end{equation}
and hence the curve is of class $C^{1, 1}$ in this case. 
Very loosely speaking, the basic idea underpinning definition~\eqref{e:uguale1} is the following: if $s>0$, $\mf{\widetilde u}$ (on the right) and $ \boldsymbol{\zeta}_k (\mf{\widetilde u}, s)= \mf t_k(\mf{\widetilde u}, s)$ (on the left) are connected by a Lax admissible shock that has positive speed since $\lambda_k (\mf{\widetilde u}) \ge 0$. For $s <0$, $ \boldsymbol{\zeta}_k (\mf{\widetilde u}, \cdot)$ ``follows" $\mf t_k(\mf{\widetilde u}, \cdot)$ as long as the latter describes a rarefaction with positive speed. When the speed would became negative, $ \boldsymbol{\zeta}_k (\mf{\widetilde u}, \cdot)$ ``switches" to the boundary layers curve.   \\
{\sc Case $\lambda_k (\mf{\widetilde u}) < 0$.} If $\lambda_k (\mf{\widetilde u}) < 0$ then owing again to genuine nonlinearity~\eqref{e:gnl} we have $\underline s (\mf{\widetilde u})> 0$. We recall~\eqref{e:pera}
and set 
\be
\label{e:regcwfc}    
      \boldsymbol{\zeta}_k (\mf{\widetilde u}, s): =
      \left\{
      \begin{array}{ll}
                  \boldsymbol{\pi}_{\mf u}\circ \mf b_k(\mf{\widetilde u}, s) & s \leq \underline s (\mf{\widetilde u}) \\
                 \mf t_k (\mf{\widetilde u}, s) & s \ge \underline s (\mf{\widetilde u}) \\
      \end{array}
      \right. \quad \text{if $\lambda_k (\mf{\widetilde u}) < 0$}.
      \eq
Note that in this case the curve is of class $C^{1,1}$ away from the point $s=  \underline s (\mf{\widetilde u})$, where the first derivative may be discontinuous. 
In a nutshell, the very basic idea behind definition~\eqref{e:regcwfc} is that the values $s < \underline s (\mf{\widetilde u})$ correspond to boundary layers, whereas the values $s > \underline s (\mf{\widetilde u})$ correspond to Lax admissible shocks with non-negative speed.
\subsubsection{Proof of Lemma~\ref{l:cwfc}}\label{sss:concl}
We now show that the curve  $\boldsymbol{\zeta}_k (\mf{\widetilde u}, \cdot)$ defined at the previous paragraph satisfies the properties in the statement of Lemma~\ref{l:cwfc}.

We first consider the case $\lambda_k (\widetilde{\mf u}) \ge 0$, which yields~\eqref{e:uguale1} and $\bar s (\widetilde{\mf u})<0$, and this implies that item i) in the statement of Lemma~\ref{l:cwfc} is satisfied. If $\bar{s} (\mf{\widetilde u}) \leq s<0$, then we set $\mf{\bar  u}: =  \boldsymbol{\zeta}_k (\mf{\widetilde u}, s)$ and item ii) in the statement of Lemma~\ref{l:cwfc} is satisfied with item ii$_2)$ trivially satisfied by the constant function $\mf w \equiv \mf{\bar  u} $.  The most interesting behavior occurs when $s< \bar{s} (\mf{\widetilde u})$. In this case we set 
$ \mf{\bar  u}: = \mf t_k (\mf{\widetilde u}, \bar s (\mf{\widetilde u}))$, and in this way item ii$_1)$ is satisfied. To establish ii$_2)$ one can then rely on the change of variables~\eqref{e:caux},\eqref{e:esplode2}.

We now assume $\lambda_k (\widetilde{\mf u}) < 0$. We recall that $\boldsymbol{\zeta}_k (\mf{\widetilde u}, \cdot)$ is defined by~\eqref{e:regcwfc} and that in this case $\underline s(\mf{\widetilde u})>0$. If $s \ge \underline s(\mf{\widetilde u})>0$ then item ii) is satisfied provided $\mf{\bar u} = \mf t_k (\mf{\widetilde u}, s)$ and $\mf w \equiv \mf{\bar u}$ in item ii)$_2$. If $s <  \underline s(\mf{\widetilde u})$, then 
we set $\mf{\bar u}:=\mf{\widetilde u}$. In this way item ii$_1)$ is trivially satisfied, whereas item ii$_2)$ is satisfied by Lemma~\ref{l:min0}. 
\subsection{Boundary layers analysis in the linearly degenerate case}
\label{ss:case2} We assume~\eqref{e:lindeg}  and to ease the exposition in the following we directly focus on the most interesting case, that is B) in Hypothesis~\eqref{h:eulerlag}.
This is the case of the Navier-Stokes and viscous MHD equations written in Eulerian coordinates when the fluid velocity vanishes. Note that another interesting case involving a linearly degenerate field is the case of the Navier-Stokes and viscous MHD equations written in Lagrangian coordinates. However, in that case $\lambda_k (\mf u) \equiv 0$ and the analysis trivializes: there are no boundary layers lying on the center manifold or $k$-th waves entering the domain and we set 
$
    \boldsymbol{\zeta}_k  (\mf{\widetilde u}, s) : = \mf t_k (\mf{\widetilde u}, s). 
$ 
Finally, we mention the case of the viscous MHD equations when $u = \pm \beta \sqrt{\rho}$: in this case the boundary characteristic field is also linearly degenerate and we have~\eqref{e:a11nz}. This case can be handled by combining the analysis at the previous paragraph with the proof of~\cite[Lemma 4.6]{BianchiniSpinolo}.  

The detailed construction of the curve $ \boldsymbol{\zeta}_k$ in case B) of Hypothesis~\ref{h:eulerlag} is provided in~\cite{BianchiniSpinolo}, see in particular~\cite[Lemma 4.6]{BianchiniSpinolo} for the construction of the characteristic wave fan curve. Since we need it in what follows, we now briefly overview~\cite[\S4.2]{BianchiniSpinolo} and in particular Lemma 4.6.  Note that 
\begin{equation}
\label{e:mcld}
        \boldsymbol{\zeta}_k  (\mf{\widetilde u}, s): = 
        \left\{
        \begin{array}{ll}
                   \mf t_k  (\mf{\widetilde u}, s)  & \lambda_k (\mf{\widetilde u}) \ge 0 \\
                   \boldsymbol{\pi}_{\mf u} \circ \mf b_k (\mf{\widetilde u}, s)  & \lambda_k (\mf{\widetilde u}) < 0, \\
        \end{array}
        \right.
\end{equation}
where $ \mf t_k  (\mf{\widetilde u}, s) $ is the same as in~\eqref{e:lax:lindeg} and $  \mf b_k(\mf{\widetilde u}, s) \in \R^{N+1}$ is defined in ~\cite[(4.4)]{BianchiniSpinolo} and is constructed from the solution, evaluated at $\tau =s$ of the  Cauchy problem 
\begin{equation}\label{e:cplindeg}
    \frac{d \mf u}{d\tau } = \mf r_c (\mf u, z_{00}, 0), 
    \quad   \frac{d z_{00}}{d\tau } = \theta_{00}  (\mf u, z_{00}, 0), \quad \mf u(0) = \mf{\widetilde u}, \quad z_{00}(0)=0,
\end{equation}
namely $  \mf b_k(\mf{\widetilde u}, s) = (\mf u(s), z_{00}(s)$.
In the above expression, we have set  
$$
   \mf r_c (\mf u, z_{00}, \sigma) : =
   \left(
   \begin{array}{cc}
        -e_{11}^{-1} \mf d^t \mf r_{00} \\
        \mf{R}_0 \mf r_{00} \\
   \end{array}
   \right) (\mf u, z_{00}, \sigma).
$$ 
and the functions  $e_{11}$, $\mf d$, $ \mf r_{00}$, $\theta_{00}$ are as in~\cite[Lemma 4.2]{BianchiniSpinolo}. Note that, owing to~\cite[Lemmas 4.4, 4.5]{BianchiniSpinolo} we have 
\begin{equation}
\label{e:errecl}
      \theta_{00}(\mf u, 0, \lambda_k (\mf u)) =0, \qquad    \mf r_c (\mf u, 0, \lambda_k (\mf u))= \mf r_k(\mf u). 
\end{equation}
Very loosely speaking, the rationale underpinning definition~\eqref{e:mcld} is the following. If $\lambda_k (\mf{\widetilde u}) \ge 0$ then there is a contact discontinuity with nonnegative speed joining $\mf{\widetilde u}$ and $ \boldsymbol{\zeta}_k  (\mf{\widetilde u}, s)$. If $\lambda_k (\mf{\widetilde u}) < 0$
then by~\cite[Lemma 4.6 Aii)]{BianchiniSpinolo} there is a boundary layer connecting $\mf{\widetilde u}$ and $ \boldsymbol{\zeta}_k  (\mf{\widetilde u}, s)$, i.e. there is a solution of~\eqref{e:bltw} such that 
\be \label{e:allafineserve}
   \mf u(0) =  \boldsymbol{\zeta}_k  (\mf{\widetilde u}, s), \qquad \lim_{y \to + \infty} \mf u(y) = \mf{\widetilde u}. 
\eq
\subsection{Complete boundary layers analysis} \label{sss:cbla}
The curve $\boldsymbol{\zeta}_k$ we discuss at the previous paragraphs describes, among other things, the boundary layers  (that is, solutions of~\eqref{e:bltw} with finite limit at $+ \infty$) lying on the center manifold. 
We now describe \emph{all} the boundary layers, not necessarily lying on the center manifold, and we refer to \S\ref{ss:speroserva} for the main ideas involved on the construction, which relies on the Slaving Manifold Lemma. Here instead we provide the technical details of the construction and to ease the exposition, we assume~\eqref{e:a11nz}. Case B) in Hypothesis~\ref{h:eulerlag} is discussed in~\cite{BianchiniSpinolo}, whereas the case of~\ref{e:lagr2} is discussed in~\cite{BianchiniSpinoloARMA} and is analogous to what follows. 

To complete the boundary layer analysis we recall the analysis in~\cite[\S11]{BianchiniSpinolo} and apply it to the case where the function $\mf a$ in the  ODE~\eqref{e:bltw2} is\footnote{Concerning the notation we recall the footnote \footref{foot3} at page \pageref{foot3}} is defined by~\eqref{e:gamma}, and hence~\eqref{e:bltw2} is equivalent to the boundary layers equation~\eqref{e:bltw}. We then proceed as follows. 
\begin{itemize}
\item We apply~\cite[Lemma 11.1]{BianchiniSpinolo} to~\eqref{e:bltw2} in the case where the function $\mf a$ is given by~\eqref{e:gamma}, and in this case one can show (see~\cite[Lemma 4.7]{BianchiniSpinoloARMA}) that $n_-= k-1-\ell$. We recall that $n_-$ denotes the number of eigenvalues of $\mf D \mf a (\boldsymbol{\omega}^\ast)$ with strictly negative real part, and $\ell$ is the number of negative eigenvalues of $\mf A_{11}$.  We also define the equilibrium points $\boldsymbol{\omega}^\ast$ and $\boldsymbol{\check \omega}$ by setting
\be \label{e:chicosa}
      \boldsymbol{\omega}^\ast: = (\mf u^\ast,  \mf 0_{N-h}), 
       \qquad 
      \boldsymbol{\check \omega}: = (\boldsymbol{\zeta}_k (\mf{\widetilde u}, s_k),  \mf 0_{N-h}),
\eq
where $\mf u^\ast$ and $\widetilde{\mf u}$ are the same as in the statement of Lemma~\ref{l:cwfc}. Lemma 11 in~\cite{BianchiniSpinolo} yields the existence of an $(k-1-\ell)$-dimensional  invariant manifold, the so-called uniformly stable manifold, which is parametrized by the function $\mf m_-( \boldsymbol{\check{\omega}}, \cdot)$. The solutions of~\eqref{e:bltw2} lying on the uniformly stable manifold converge exponentially fast to $  \boldsymbol{\check \omega}= (\boldsymbol{\zeta}_k (\mf{\widetilde u}, s_k),  \mf 0_{N-h})$.
 The Lipschitz continuous  function $\mf m_-$ attains values in $\R^{2N-h}$: we set  
\be \label{e:phiesse}
    \boldsymbol{\psi}_s (\boldsymbol{\zeta}_k (\mf{\widetilde u}, s_k), \xi_{\ell +1}, \dots, \xi_{k-1}) : = 
     \boldsymbol{\pi}_{\mf u}\circ \mf m_-(\boldsymbol{\check{\omega}}, \xi_{\ell +1}, \dots, \xi_{k-1})
\eq
where $\boldsymbol{\pi}_{\mf u}$ denotes the projection onto the first $N$ components. 
\item We now have to introduce some notation and we do so by first considering the genuinely nonlinear~\eqref{e:gnl} case. We recall the definition of $\boldsymbol{\zeta}_k(\mf{\widetilde u}, \cdot)$, namely~\eqref{e:uguale1} and~\eqref{e:regcwfc}, and conclude that if either $\lambda_k(\mf{\widetilde u}) <0$ and $s \ge \underline s (\mf{\widetilde u})$ or $\lambda_k(\mf{\widetilde u}) \ge 0$ and $s \ge \bar s (\mf{\widetilde u})$ then there is no boundary layer lying on the center manifold and therefore the function $\mf m_- $ describes all the boundary layers. 
If either $\lambda_k(\mf{\widetilde u}) <0$ and $s < \underline s (\mf{\widetilde u})$ or $\lambda_k(\mf{\widetilde u}) \ge 0$ and $s < \bar s (\mf{\widetilde u})$ then
there are boundary layers lying on the center manifold. To complete our construction we first  
introduce the function $\boldsymbol{\gamma}_k (\mf{\widetilde u}, \cdot)$  by setting 
\be \label{e:chicosa2}
     \boldsymbol{\gamma}_k (\mf{\widetilde u}, s_k) : =  
      \left\{ 
      \begin{array}{ll}
       \mf b_k(\mf{\widetilde u}, s_k)  &
       \text{if $\lambda_k (\mf{\widetilde u}) <0$ and $s_k < \underline s (\mf{\widetilde u})$ } \\
       (\mf t_k(\mf{\widetilde u}, s_k), \mf 0_{N-h}) & 
       \text{if $\lambda_k (\mf{\widetilde u}) <0$ and $s_k \ge \underline s (\mf{\widetilde u})$ } \\
        \mf b_k(\mf t_k (\mf{\widetilde u},  \bar s (\mf{\widetilde u})), s_k - \bar s (\mf{\widetilde u})) &
      \text{if $\lambda_k (\mf{\widetilde u}) \ge0$ and $s_k < \bar s (\mf{\widetilde u})$}\\
       (\mf t_k(\mf{\widetilde u}, s_k), \mf 0_{N-h}) &
        \text{if $\lambda_k (\mf{\widetilde u}) \ge0$ and $s_k \ge \bar s (\mf{\widetilde u})$}. \\
      \end{array}
     \right.
     \quad
\eq
Note that, owing to~\eqref{e:siattacca}, if $\lambda_k  (\mf{\widetilde u}) \ge0$ then the curve $\gamma_k ( \mf{\widetilde u}, \cdot)$ is of class $C^{1, 1}$.
Note furthermore that $ \boldsymbol{\gamma}_k (\mf{\widetilde u}, 0) = (\mf{\widetilde u}, \mf 0_{N-h})$ for every $\mf{\widetilde u}$, and 
\be \label{e:stessacosa}
     \boldsymbol{\pi}_{\mf u} \circ  \boldsymbol{\gamma}_k (\mf{\widetilde u}, s_k)  \stackrel{\eqref{e:uguale1},\eqref{e:regcwfc}}{=}
      \boldsymbol{\zeta}_k (\mf{\widetilde u}, s_k).
\eq
\item We now consider the linearly degenerate~\eqref{e:lindeg} case and we recall the definition of the curve $\boldsymbol{\zeta}_k (\mf{\widetilde u}, )$, namely~\eqref{e:mcld}. If $\lambda_k (\mf{\widetilde u}) \ge 0$, then there is no boundary layer lying on the center manifold, otherwise there is. In this case we set 
\be \label{e:servepure}
 \boldsymbol{\gamma}_k (\mf{\widetilde u}, s_k) : =  
      \left\{ 
      \begin{array}{ll}
       \mf b_k(\mf{\widetilde u}, s_k)  &
       \text{if $\lambda_k (\mf{\widetilde u}) <0$}  \\
       (\mf t_k(\mf{\widetilde u}, \xi_k), \mf 0_{N-h}) &
        \text{if $\lambda_k (\mf{\widetilde u}) \ge0$}. \\
      \end{array}
     \right.
\eq
\item We are now in a position to complete our analysis and provide a description of \emph{all} boundary layers, not necessarily lying on either the center manifold or the uniformly stable manifold. 
We rely on a Slaving Manifold Theorem. More precisely, in both the linearly degenerate and in the genuinely nonlinear case we apply~\cite[Lemma 11.2]{BianchiniSpinolo}  with $\boldsymbol{\omega}^\ast$ and $\boldsymbol{\check \omega}$ defined as in~\eqref{e:chicosa}, and 
\be \label{e:festanatale2}
 \boldsymbol{\omega}_0(0): =  \boldsymbol{\gamma}_k (\mf{\widetilde u}, s_k). 
 \eq
Note that owing to~\eqref{e:servepure} $ \boldsymbol{\omega}_0(0)$ is the initial datum for an orbit $\boldsymbol{\omega}_0$ of~\eqref{e:bltw2} lying on the center manifold. Note furthermore that by comparing~\eqref{e:chicosa2} and~\eqref{e:servepure} we get the equality $\boldsymbol{\pi}_{\mf u} ( \boldsymbol{\omega}_0(0))= \boldsymbol{\pi}_{\mf u} ( \boldsymbol{\check \omega})$. In all cases where we have the stronger equality $\boldsymbol{\omega}_0(0)=  \boldsymbol{\check \omega}$ (which holds for instance if $\lambda_k (\mf{\widetilde u}) <0$ and $s_k \ge \underline s (\mf{\widetilde u})$, see again~\eqref{e:servepure}) the orbit $\boldsymbol{\omega}_0$ of~\eqref{e:bltw2} with initial datum $\boldsymbol{\omega}_0(0)$ boils down to a single equilibrium point.
Lemma 11.2 in~\cite{BianchiniSpinolo}   yields the existence of a Lipschitz continuous  map $\mf m_p (\boldsymbol{\check \omega},  \boldsymbol{\omega}_0(0), \cdot): \R^{k - \ell -1}\to \R^{2N-h}$, which has the following property: let  $\boldsymbol{\omega}$ and $\boldsymbol{\omega}_0$ be the solutions of the Cauchy problem obtained by coupling~\eqref{e:bltw2} with the initial datum 
\be \label{e:festanatale}
    \boldsymbol{\omega}(0): = \mf m_-(\boldsymbol{\check{\omega}}, \xi_{\ell +1}, \dots, \xi_{k-1}) - \boldsymbol{\check{\omega}}+ \mf m_p (\boldsymbol{\check \omega},  \boldsymbol{\omega}_0(0),  \xi_{\ell +1}, \dots, \xi_{k-1})+ \boldsymbol{\omega}_0(0),
\eq
and $\boldsymbol{\omega}_0(0)$, respectively; then 
\be \label{e:slavingman}
\lim_{y \to + \infty}|\boldsymbol{\omega}_0 (y)- \boldsymbol{\omega}(y)|=0.
\eq
 We postpone to \S\ref{sss:casignl} and \S\ref{sss:casild} below an extended discussion on how this asymptotic behavior translates in the specific framework of the boundary layers system~\eqref{e:bltw}. We set 
\be \label{e:psipi}
    \boldsymbol{\psi}_p (\mf{\widetilde u} , \xi_{\ell +1}, \dots, \xi_{k-1}, s_k): = 
     \boldsymbol{\pi}_{\mf u}\circ \mf m_p ( \boldsymbol{\check \omega}, \boldsymbol{\omega}_0(0), \xi_{\ell +1}, \dots, \xi_{k-1}),
\eq
where as usual $ \boldsymbol{\pi}_{\mf u}$ denotes the projection onto the first $N$ coordinates. We also set  
\be \label{e:bl00}
    {\boldsymbol \phi} (\mf{\widetilde u}, \xi_{\ell +1}, \dots, \xi_{k-1}, s_k) : =\boldsymbol{\pi}_{\mf u}( \boldsymbol{\omega}(0)),
\eq
where $ \boldsymbol{\omega}(0)$ is the same as in~\eqref{e:festanatale}. By combining~\eqref{e:bl00} with~\eqref{e:festanatale},~\eqref{e:phiesse},\eqref{e:psipi} and the equality  $\boldsymbol{\pi}_{\mf u} ( \boldsymbol{\omega}_0(0))= \boldsymbol{\pi}_{\mf u} ( \boldsymbol{\check \omega})$ we arrive at 
\be 
\label{e:bl10}
     {\boldsymbol \phi} (\mf{\widetilde u}, \xi_{\ell +1}, \dots, \xi_{k-1}, s_k): = 
       {\boldsymbol \psi}_s  ( \boldsymbol{\zeta}_k (\widetilde{\mf u}, s_k), \xi_{\ell +1}, \dots, \xi_{k-1})    +       
      \boldsymbol{\psi}_p (\mf{\widetilde u}, \xi_{\ell +1}, \dots, \xi_{k-1}, s_k). 
\eq
Note that, if the $k$-th vector field is genuinely nonlinear, the function ${\boldsymbol \phi} (\mf{\widetilde u}, \cdot)$ describes the states that can be connected to $\mf{\widetilde u}$ by either Lax admissible shocks or rarefaction waves with non-negative speed, and by boundary layers. If the $k$-th vector field is linearly degenerate, the function ${\boldsymbol \phi} (\mf{\widetilde u}, \cdot)$ describes the states that can be connected to $\mf{\widetilde u}$ by contact discontinuities with non-negative speed, and by boundary layers. Note furthermore that, owing to~\cite[formula (11.6)]{BianchiniSpinolo} we have 
\be \label{e:p}
    |   \boldsymbol{\psi}_p ( \mf{\widetilde u}, \xi_{\ell +1}, \dots, \xi_{k-1}, s_k)| \leq \unpo 
    \left\{
    \begin{array}{lll}
    |s_k| \sum_{i=1}^{k-1} |\xi_i| &  \text{if $\lambda_k (\mf{\widetilde u}) <0$ and $s_k < \underline s (\mf{\widetilde u})$ } \\
    0 & \text{if $\lambda_k (\mf{\widetilde u}) <0$ and $s_k \ge \underline s (\mf{\widetilde u})$ } \\
   \left| s_k - \bar s (\mf{\widetilde u})\right| \sum_{i=1}^{k-1} |\xi_i|& \text{if $\lambda_k (\mf{\widetilde u}) \ge 0$ and $ s_k < \bar s (\mf{\widetilde u})$ } \\
     0 & \text{if $\lambda_k (\mf{\widetilde u}) \ge 0$ and $ s_k \ge \bar s (\mf{\widetilde u})$ } & 
    \end{array} 
    \right.  
\eq
in the genuinely nonlinear case, and 
\be \label{e:p2}
    |  \boldsymbol{\psi}_p ( \mf{\widetilde u}, \xi_{\ell +1}, \dots, \xi_{k-1}, s_k)| \leq \unpo 
    \left\{
    \begin{array}{lll}
    |s_k| \sum_{i=1}^{k-1} |\xi_i| &  \text{if $\lambda_k (\mf{\widetilde u}) <0$ } \\
    0 & \text{if $\lambda_k (\mf{\widetilde u}) \ge 0$  } \\
    \end{array} 
    \right.  
\eq
in the linearly degenerate case. 
\end{itemize}
\subsection{The solution of the boundary Riemann problem} \label{sss:blvero}
We now discuss the solution of the Riemann problem posed by coupling~\eqref{e:claw} with~\eqref{e:briedata}. As pointed out in \S\ref{ss:speroserva}, in either {\sc Case A)} or {\sc Case C)} in Hypothesis~\ref{h:eulerlag} we have to solve~\eqref{e:solbrparma} for $(\xi_{\ell+1}, \dots, \xi_{k-1}, s_{k}, \dots, s_N)$. In {\sc Case B)} in Hypothesis~\ref{h:eulerlag} we have to use~\eqref{e:solbrcpam}. 

For the reader's convenience we now go over the main properties of the solution of the boundary Riemann problem that we need in the following and we separately consider the case where the $k$-th vector field is genuinely nonlinear and the case where it is linearly degenerate.  
\subsubsection{Genuinely nonlinear $k$-th vector field} \label{sss:casignl}
To simplify the notation, we assume~\eqref{e:a11nz}, that is {\sc case A)} in Hypothesis~\ref{h:eulerlag}. As mentioned before, this applies to the inviscid limit of the Navier-Stokes equations written in Eulerian coordinates when the fluid velocity is close to $\pm c$, where $c$ denotes the sound speed, see~\S\ref{sss:appl1}.

To discuss the main properties of the solution of~\eqref{e:claw},~\eqref{e:briedata} consider $(\xi_{\ell+1}, \dots, \xi_{k-1}, s_{k}, \dots, s_N)$ satisfying~\eqref{e:solbrparma} and set 
\be \label{e:hatu2}
        \mf{\hat u} : = \mf t_{k+1} (\cdot, s_{k+1}) \circ \dots \mf t_N (\mf u^+, s_N),
\eq
see Figure~\ref{f:fabio} for a representation.
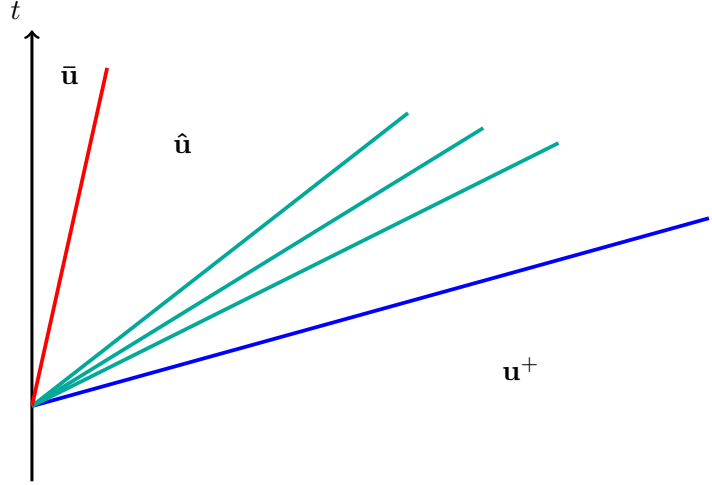
\begin{figure}
\begin{center}
\caption{The solution of the boundary Riemann problem} 
\label{f:fabio}
\begin{tikzpicture}
\draw[line width=0.4mm,->] (0,0) -- (0,6) node[anchor=south east] {$t$};
\draw[line width=0.5mm,blue]   (0, 1) -- (9, 3.5); 
\draw[line width=0.5mm,JungleGreen]   (0, 1) -- (6, 4.7); 
\draw[line width=0.5mm,JungleGreen]   (0, 1) -- (7, 4.5); 
\draw[line width=0.5mm,JungleGreen]   (0, 1) -- (5, 4.9); 
\draw[line width=0.5mm,red]   (0, 1) -- (1, 5.5); 
\draw  (2,4.5) node {$\mf{\hat u}$};
\draw  (0.5,5.4) node {$\mf{\bar u}$};
\draw  (6.5,1.5) node {$\mf{u}^+$};
\end{tikzpicture}
\end{center}
\end{figure}
Note that $\mf{u}^+$ and $\mf{\hat u}$ are connected by waves (shocks, rarefactions, contact discontinuities) of the families $k+1, \dots, N$ with strictly positive and bounded away from $0$ speed. 
To complete the description of the solution of the boundary Riemann solver we now consider the characteristic wave fan curve of admissible states $\boldsymbol{\zeta}_k (\mf{\hat u}, \cdot)$, whose construction is given in \S\ref{ss:zetakgnl}. Note that here compared with the notation in \S\ref{ss:zetakgnl} we have $\mf{\hat u}$ and $s_k$ in place of $\mf{\widetilde u}$ and $s$, respectively. We distinguish between the following cases. 
\begin{itemize} 
\item[i)] If $\lambda_k (\mf{\hat u}) \ge 0$ and  $s_k>0$ then owing to~\eqref{e:uguale1} we have $\boldsymbol{\zeta}_k (\mf{\hat u}, s_k)= \mf t_k (\mf{\hat u}, s_k)$ and the solution of the boundary Riemann problem contains a Lax
admissible shock between $\mf{\hat u}$ (on the right) and $\boldsymbol{\zeta}_k (\mf{\hat u}, s_k)$ (on the left).
In this case the trace $\mf{\bar u} = \mf u(t, 0)$ is $\mf{\bar u}= \boldsymbol{\zeta}_k (\mf{\hat u}, s_k)$ and satisfies $\lambda_k (\mf{\bar u}) >0$.  Note that we have 
\be 
\label{e:brs0}
   \boldsymbol{\beta} ( {\boldsymbol \phi} (\mf{\bar u}, \xi_{\ell+1}, \dots, \xi_{k}), \mf u_b) = \mf 0_{N}
\eq 
provided $\xi_k=0$. To establish the above identity, we recall that $\mf{\bar u} =\boldsymbol{\zeta}_k (\mf{\hat u}, s_k)$ and that $\lambda_k (\mf{\hat u})>0$, $s_k>0\ge  \bar s(\mf{\hat u})$, which owing to~\eqref{e:p} yields 
\be \label{e:psip}
    \boldsymbol{\psi}_p (\mf{\hat u}, \xi_{\ell+1}, \dots, \xi_{k-1}, s_k) = \mf 0_N 
\eq
and by~\eqref{e:bl10} this yields 
\be \label{e:2112} \begin{split}
    {\boldsymbol \phi} &(\mf{\hat u}, \xi_{\ell+1}, \dots, \xi_{k-1}, s_k) \stackrel{\eqref{e:bl10},\eqref{e:psip}}{=} 
    {\boldsymbol \psi_s} (\boldsymbol{\zeta}_k (\mf{\hat u}, s_k), \xi_{\ell+1}, \dots, \xi_{k-1})
   \\ & =  {\boldsymbol \psi_s} (\mf{\bar u}, \xi_{\ell+1}, \dots, \xi_{k-1})
    \stackrel{\eqref{e:bl10},\eqref{e:psip}}{=} {\boldsymbol \phi} (\mf{\bar u}, \xi_{\ell+1}, \dots, \xi_{k-1}, 0).
\end{split}
\eq
By combining~\eqref{e:solbrparma},\eqref{e:hatu2} and~\eqref{e:2112} we arrive at~\eqref{e:brs0} with $\xi_k=0$. 
Note furthermore that~\eqref{e:2112} combined with the definition~\eqref{e:phiesse} of $\boldsymbol{\psi}_s$ yields the existence of a solution of~\eqref{e:bltw} with initial datum 
$ {\boldsymbol \psi_s} (\mf{\bar u}, \xi_{\ell+1}, \dots, \xi_{k-1})$ and converging exponentially fast to $\mf{\bar u}$. This yields the existence 
of a boundary layer $\mf w$ satisfying~\eqref{e:andiamoav2} with asymptotic state $\mf{\underline u}= \mf{\bar u}$. 
\item[ii)] If $\lambda_k (\mf{\hat u}) \ge 0$ and $\bar s(\mf{\hat u}) \leq s_k \leq 0$, then again owing to~\eqref{e:uguale1} we have $\boldsymbol{\zeta}_k (\mf{\hat u}, s_k)= \mf t_k (\mf{\hat u}, s_k)$. The solution of the boundary Riemann problem contains a rarefaction wave with rightmost state $\mf{\hat u}$ and leftmost state $\boldsymbol{\zeta}_k (\mf{\hat u}, s_k)$ and the trace $\mf{\bar u} = \mf u(t, 0)$ is $\mf{\bar u} = \boldsymbol{\zeta}_k (\mf{\hat u}, s_k)$ and satisfies $\lambda_k (\mf{\bar u}) \ge 0$. By arguing as in case i) we get~\eqref{e:brs0} with $\xi_k=0$ and the existence of a boundary layer $\mf w$ satisfying~\eqref{e:andiamoav2} with asymptotic state $\mf{\underline u}= \mf{\bar u}$. 
\item[iii)]  If $\lambda_k (\mf{\hat u}) \ge 0$ and $s_k < \bar s (\mf{\hat u})$ we recall~\eqref{e:uguale1} and conclude that the  solution of the boundary Riemann problem contains a rarefaction wave with rightmost state $\mf{\hat u}$ and leftmost state $\boldsymbol{\zeta}_k (\mf{\hat u}, \bar s(\mf{\hat u}))= \mf t_k (\mf{\hat u}, \bar s(\mf{\hat u}))$. We set $\mf{\bar u} = \mf t_k (\mf{\hat u}, \bar s(\mf{\hat u}))$ and point out that $\mf{\bar u} = \mf u(t, 0)$, that is $\mf{\bar u}$ is the trace of the solution of the boundary Riemann problem and satisfies $\lambda_k(\mf{\bar u})=0$ and henceforth $\underline s(\mf{\bar u})=0=\bar s (\mf{\bar u})$, by the very definition of $\bar s(\mf{\hat u})$ and $\underline s (\mf{\hat u})$, see~\eqref{e:baresse} and~\eqref{e:underlines}, respectively. 
We now recall~\eqref{e:bl10} and point out that, under the assumptions of case iii), 
\begin{equation*}
\begin{split}
         {\boldsymbol \psi_s} &(\boldsymbol{\zeta}_k (\mf{\hat u}, s_k), \xi_{\ell+1}, \dots, \xi_{k-1})
    \stackrel{\eqref{e:uguale1}}{=}
    {\boldsymbol \psi_s} (\boldsymbol{\pi}_{\mf u} \circ\mf b_k (\mf{\bar u}, s_k-
    \bar s(\mf{\hat u})), \xi_{\ell+1}, \dots, \xi_{k-1})
     \\ & 
     \stackrel{\eqref{e:uguale1}}{=}
    {\boldsymbol \psi_s} (\boldsymbol{\zeta}_k (\mf{\bar  u}, s_k-
    \bar s(\mf{\hat u})), \xi_{\ell+1}, \dots, \xi_{k-1})
\end{split}
\end{equation*}
Concerning  ${\boldsymbol \psi_p}$, we recall~\eqref{e:psip} and that $\boldsymbol{\check \omega}$ and $\boldsymbol{\omega}_0(0)$ are given by 
$$ 
   \boldsymbol{\check \omega}\stackrel{\eqref{e:chicosa}}{=}
    \boldsymbol{\zeta}_k (\mf{\hat u  }, s_k) 
  \stackrel{\eqref{e:uguale1}}{=}
   \boldsymbol{\zeta}_k (\mf{\bar  u}, s_k-
    \bar s(\mf{\hat u}) \; \text{and} \;
   \boldsymbol{\omega}_0(0) \stackrel{\eqref{e:festanatale2}}{=}
   \boldsymbol{\gamma}_k (\mf{\hat u  }, s_k) 
    \stackrel{\eqref{e:chicosa2}}{=}
   \boldsymbol{\gamma}_k (\mf{\bar  u}, s_k-
    \bar s(\mf{\hat u}))
$$ 
respectively. By combining the above equalities with~\eqref{e:bl10} we arrive at 
\be \label{e:21122} 
    {\boldsymbol \phi} (\mf{\hat u}, \xi_{\ell+1}, \dots, \xi_{k-1}, s_k) =
 {\boldsymbol \phi} (\mf{\bar u}, \xi_{\ell+1}, \dots, \xi_{k-1}, s_k-
    \bar s(\mf{\hat u})).
\eq
and by combining~\eqref{e:solbrparma} with~\eqref{e:hatu2}this yields~\eqref{e:brs0} with $\xi_k = s_k-
\bar s(\mf{\hat u}) \leq 0 = \bar s(\mf{\bar u})$. We now recall~\eqref{e:bl00} and that $\boldsymbol{\omega}(0)$ in there is the initial point of an orbit of~\eqref{e:bltw2} satisfying~\eqref{e:slavingman}, 
where $\boldsymbol{\omega}_0$ is an orbit of~\eqref{e:bltw2} with initial datum given by~\eqref{e:festanatale2}. In our case, owing to~\eqref{e:chicosa2} the initial datum 
$\boldsymbol{\omega}_0(0)$ is $\mf b_k (\mf{\bar  u}, s_k-
    \bar s(\mf{\hat u}))$ and owing to Lemma~\ref{l:min02} we conclude that $\boldsymbol{\omega}_0$ converges to $(\mf{\bar u}, \mf 0_{N-h})$ as $y \to + \infty$. Wrapping up, \eqref{e:slavingman} yields the 
existence of a boundary layer $\mf w$ satisfying~\eqref{e:andiamoav2} with asymptotic state $\mf{\underline u}= \mf{\bar u}$. 
\item[iv)] If $\lambda_k (\mf{\hat u}) < 0$ and $s_k > \underline s (\mf{\hat u})$ then owing to~\eqref{e:regcwfc} we have $\boldsymbol{\zeta}_k (\mf{\hat u}, s_k)= \mf t_k (\mf{\hat u},s_k)$ and the solution of the boundary Riemann problem contains a Lax
admissible shock between $\mf{\hat u}$ (on the right) and $\boldsymbol{\zeta}_k (\mf{\hat u}, s_k)$ (on the left).
The trace is $\mf{\bar u}= \boldsymbol{\zeta}_k (\mf{\hat u}, s_k)$, satisfies $\lambda_k (\mf{\bar u}) >0$ and by arguing as in case i) we get~\eqref{e:brs0} with $\xi_k=0$ and the existence of a boundary layer $\mf w$ satisfying~\eqref{e:andiamoav2} with asymptotic state $\mf{\underline u}= \mf{\bar u}$. 
\item[v)] If $\lambda_k (\mf{\hat u}) < 0$ and $s_k = \underline s (\mf{\hat u})$ then owing to~\eqref{e:regcwfc} we have $\boldsymbol{\zeta}_k (\mf{\hat u}, s_k)= \mf t_k (\mf{\hat u}, s_k)$. We set $\mf{\underline u}: =  \mf t_k (\mf{\hat u}, s_k)$ and point out that $\mf{\hat u}$ (on the right) and 
$\mf{\underline u}$ (on the left) are connected by a $0$-speed Lax admissible shock\footnote{We explicitely point out that this shock does \emph{not} appear in the solution since it loosely sits at the domain boundary}. The solution of the boundary Riemann problem contains no wave of the $k$-th family  and the trace of the solution of the boundary Riemann problem is $\mf{\bar u}= \mf{\hat u}$, which satisfies $\lambda_k (\mf{\bar u}) <0$. Since $\mf{\hat u}=\mf{\bar u}$ by combining~\eqref{e:solbrparma} with~\eqref{e:hatu2} we arrive at~\eqref{e:brs0} with $\xi_k = s_k = \underline s (\mf{\hat u})$. Note furthermore that owing to~\eqref{e:p} we have~\eqref{e:psip}, which yields 
$$
    {\boldsymbol \phi} (\mf{\hat u}, \xi_{\ell+1}, \dots, \xi_{k-1}, s_k) \stackrel{\eqref{e:bl10},\eqref{e:psip}}{=} 
    {\boldsymbol \psi_s} (\boldsymbol{\zeta}_k (\mf{\hat u}, s_k), \xi_{\ell+1}, \dots, \xi_{k-1})
 $$
and hence, by arguing as in case i) we get the existence of a boundary layer $\mf w$ satisfying~\eqref{e:andiamoav2} with asymptotic state $\mf{\underline u}= \boldsymbol{\zeta}_k (\mf{\hat u}, s_k) = \mf t_k (\mf{\hat u}, \underline s(\mf{\hat u})) $. 
\item[vi)] If $\lambda_k (\mf{\hat u}) < 0$ and $s_k < \underline s (\mf{\hat u})$ then owing to~\eqref{e:regcwfc} the solution of the boundary Riemann problem contains no wave of the $k$-th family. The trace of the solution is $\mf{\bar u}= \mf{\hat u}$, satisfies $\lambda_k (\mf{\bar u}) <0$ and by recalling~\eqref{e:solbrparma} and~\eqref{e:hatu2} we arrive at~\eqref{e:brs0} with $\xi_k = s_k$. We also have 
\begin{equation*} \begin{split}
    {\boldsymbol \phi} &(\mf{\hat u}, \xi_{\ell+1}, \dots, \xi_{k-1}, s_k)
    \stackrel{\mf{\hat u}= \mf{\bar u}}{=}
     {\boldsymbol \phi} (\mf{\bar u}, \xi_{\ell+1}, \dots, \xi_{k-1}, s_k)\\ &
 \stackrel{\eqref{e:bl10}}{=} 
    {\boldsymbol \psi_s} (\underbrace{\boldsymbol{\zeta}_k (\mf{\bar u}, s_k)}_{= \boldsymbol{\pi}_{\mf u}
   \circ \mf b_k (\mf{\bar u}, s_k) \ \text{by~\eqref{e:regcwfc}}}, \xi_{\ell+1}, \dots, \xi_{k-1})+
      {\boldsymbol \psi_p} (\mf{\bar u}, \xi_{\ell+1}, \dots, \xi_{k-1}, s_k).
\end{split}
\end{equation*}
By arguing as just after formula~\eqref{e:21122} and using Lemma~\ref{l:min0} rather than Lemma~\ref{l:min02} we conclude that there is a boundary layer $\mf w$ solving~\eqref{e:andiamoav2} with asymptotic state $\mf{\underline u}= \mf{\hat u}= \mf{\bar u}$. 
\end{itemize}
We now highlight a property of the trace $\mf{\bar u}= \mf u(t, 0)$ of the solution of the boundary Riemann problem that we will use a lot in the following: if the $k$-th vector field is genuinely nonlinear and $\xi_k$ is the same as in~\eqref{e:brs0} by direct check on cases {\rm i)},$\dots$,{\rm vi)} above one can verify that 
\be \label{e:protra}
 \xi_k  \leq \underline s (\mf{\bar u} ) \; \text{if} \; \lambda_k  (\mf{\bar u} )  \leq 0, 
       \qquad 
       \xi_k =0  \; \text{if} \; \lambda_k  (\mf{\bar u}) )  > 0, 
\eq
whence 
\be \label{e:sigmatraccia}
    \sigma_k (\mf{\bar u}, \xi_k) \leq 0 \quad \text{if $\lambda_k  (\mf{\bar u} )  \leq 0$}.
\eq
\subsubsection{Linearly degenerate $k$-th vector field}\label{sss:casild}
To ease the exposition, we focus on the most interesting case and consider {\sc case B)} in Hypothesis~\ref{h:eulerlag}, which is the case of the inviscid limit of the Navier-Stokes and viscous MHD equations written in \emph{Eulerian} coordinates when the fluid velocity vanishes, see~\S\ref{sss:appl1} and~\S\ref{sss:mhdeuler}. Note that~\eqref{e:lagr2} occurs for instance in the case of the Navier-Stokes and viscous MHD equations written in \emph{Lagrangian} coordinates, which is much simpler since the $k$-th vector fields vanishes \emph{identically}. 

To discuss the structure of the solution of the boundary Riemann problem we recall~\eqref{e:solbrcpam} and \eqref{e:hatu2} and we separately consider the following cases.  
\begin{itemize}
\item If $\lambda_k (\mf{\hat u})>0$ then owing to~\eqref{e:mcld} we have $\boldsymbol{\zeta}_k (\mf{\hat u}, s_k)= \mf t_k (\mf{\hat u}, s_k)$ and the solution of the boundary Riemann problem contains a contact discontinuity with right state $\mf{\hat u}$ and left state $\boldsymbol{\zeta}_k (\mf{\hat u}, s_k)$. The trace of the solution of the boundary Riemann problem is $\mf{\bar u}= \boldsymbol{\zeta}_k (\mf{\hat u}, s_k)$, which satisfies 
\begin{equation}
 \label{e:brs0lindeg}
   \boldsymbol{\beta} ( {\boldsymbol \phi} (\mf{\bar u}, \xi_{\ell (\mf u_b)+1}, \dots, \xi_{k}), \mf u_b) = \mf 0_{N},
\eq 
provided $\xi_k=0$. By arguing as in case i) in \S\ref{sss:casignl} and using~\eqref{e:p2} we get the existence of a boundary layer $\mf w$ solving~\eqref{e:andiamoav2} with asymptotic state $\mf{\underline u}= \mf{\bar u}= \boldsymbol{\zeta}_k (\mf{\hat u}, s_k)$. 
\item If $\lambda_k (\mf{\hat u})=0$ then owing to~\eqref{e:mcld} we have $\boldsymbol{\zeta}_k (\mf{\hat u}, s_k)= \mf t_k (\mf{\hat u}, s_k)$ and $\mf{\hat u}$ and $\boldsymbol{\zeta}_k (\mf{\hat u}, s_k)$ are connected by a $0$-speed contact discontinuity\footnote{As in case v) in \S\ref{sss:casignl} we explicitely point out that this contact discontinuity does \emph{not} appear in the solution since loosely speaking it sits at the domain boundary}.  The solution of the boundary Riemann problem does not contain
any wave of the $k$-th family and its trace is $\mf{\bar u}= \mf{\hat u}$. By arguing as in case v) in \S\ref{sss:casignl} we get~\eqref{e:brs0lindeg} with $\xi_k=s_k$ and the existence of a boundary layer solving~\eqref{e:andiamoav2} with asymptotic state $\mf{\underline u} = \boldsymbol{\zeta}_k (\mf{\hat u}, s_k)$. 
\item If $\lambda_k (\mf{\hat u})<0$ then the solution of the boundary Riemann problem does not contain
any wave of the $k$-th family, its trace is $\mf{\bar u}= \mf{\hat u}$ and this yields~\eqref{e:brs0lindeg} with $\xi_k=s_k$. By arguing as in case vi) in \S\ref{sss:casignl} we also get  the existence of a boundary layer solving~\eqref{e:andiamoav2} with asymptotic state $\mf{\underline u} =\mf{\hat u}=\mf{\bar u}$.
\end{itemize}
Note that by the above analysis, the trace $\mf{\bar u}= \mf u(t, 0)$ of the solution of the boundary Riemann problem satisfies the following property: if $\xi_k$ is the same as in~\eqref{e:brs0lindeg} then 
\be \label{e:protra2}
 \xi_k =0  \quad \text{if} \; \lambda_k  (\mf{\bar u} )  > 0. 
\eq
\section{Construction of the wave front-tracking approximation} \label{s:cinque}
In this section we describe our wave front-tracking algorithm. In the following sections we will show that this algorithm yields for every $\ee>0$ an $\ee$-approximate solution of the initial-boundary value problem~\eqref{e:claw},~\eqref{e:ibvp} in the sense of Definition~\ref{d:appsol}, and that the other properties stated in Theorem~\ref{t:wft} are satisfied.  We basically follow the same construction as in~\cite[\S7.2]{Bressan} and in particular we introduce both \emph{accurate} and \emph{simplified} solvers. The only differences with the analysis in~\cite[\S7.2]{Bressan} are:
\begin{itemize}
\item[(i)] owing to the specific expression of our interaction functionals in \S\ref{s:functional} we have to put particular care in choosing the speed of the wave fronts of the $k$-th family;
\item[(ii)] we have to construct the solution of the boundary Riemann problems that occur at the origin, at discontinuity points of $\mf u_b^\ee$ and when a wave front hits the boundary.  
\end{itemize}
The exposition is organized as follows: in \S\ref{ss:id} we briefly discuss the construction of the approximate initial and boundary datum, in \S\ref{ss:ries} we define the accurate and simplified approximate solutions of the Riemann problem, in \S\ref{ss:abrie} the accurate approximate solution of the boundary Riemann problem, in \S\ref{ss:sbrie} the simplified solution of the boundary Riemann problem and finally in \S\ref{ss:whenas} we specify when we use the accurate solvers and when we use the simplified solvers. We do not outline the rest of the construction of the wave front-tracking approximation as one can exactly proceed as in~\cite[\S7.2]{Bressan}.
\subsection{Approximation of the initial datum and of the boundary datum} \label{ss:id} We fix $\ee>0$, recall~\eqref{e:hp} and set $\mf u_0: = \mf u(\mf v_0)$, $\mf u_b: = \mf u(\mf v_b)$. We conclude that there are two functions $\mf u_{0}^\ee, \mf u_{b}^\ee: \R_+ \to \R$ that are piecewise constant with a finite number of jump discontinuities and satisfy
\be \label{e:51}
    \| \mf u_0^\ee - \mf u_0 \|_{L^1} \leq \ee, \; \mathrm{TotVar}\ \mf u_0^\ee \leq  \mathrm{TotVar} \ \mf u_0, \quad 
     \| \mf u_b^\ee - \mf u_b \|_{L^1} \leq \ee, \; \mathrm{TotVar} \ \mf u_b^\ee \leq  \mathrm{TotVar} \ \mf u_b, 
\eq
and 
\be \label{e:52}
    |\mf u^\ee_0(0) - \mf u^\ee_b(0)| \leq  |\mf u_0(0) - \mf u_b(0)|. 
\eq
\subsection{Riemann solvers} \label{ss:ries}
We consider a Riemann problem for~\eqref{e:claw} and we construct two approximate solutions by using either the \emph{accurate} or the \emph{simplified} solver discussed below. In \S\ref{ss:whenas} we specify when we use the accurate and when the simplified solver.
\subsubsection{Accurate Riemann solver} \label{ss:arie}
As mentioned above, the only point where we cannot simply follow the construction in~\cite[\S 7.2]{Bressan} is that we have to put extra-care in choosing the speed of the wave fronts of the $k$-th family. To this end, we  recall~\cite[Remark 7.1 p.132]{Bressan} and~\cite[Remark 7.2 p.142]{Bressan}. We consider a wave-front $\alpha$ and we term $\mf u_\alpha$ its right state, by $s_\alpha$ its size, by $\mf t_k (\mf u_\alpha, s_\alpha)$ its left state and by $\mathrm{sp}_\alpha$ its speed. We want to choose $\mathrm{sp}_\alpha$ in such a way that 
\be \label{incidono}
    \mathrm{sp}_\alpha < 0 \implies  \varsigma_k (\mf{ u}_\alpha, s_\alpha) < 0 \; \text{if~\eqref{e:gnl}}, \qquad \quad \mathrm{sp}_\alpha < 0 \implies  \lambda_k (\mf{u}_\alpha) < 0 \; \text{if~\eqref{e:lindeg}},
\eq
where 
\be \label{e:speedwft} 
   \varsigma_k (\mf{u}, s_k) =
    \left\{
    \begin{array}{ll}
                \sigma_k (\mf{ u}, s_k)  & s_k \ge 0 \\
                \lambda_k ( \mf t_k (\mf{u}, s_k)) & s_k < 0, \\
    \end{array}
    \right.
\eq
and $ \sigma_k (\mf{u}, s_k) $ is the same as in~\eqref{e:rh}. 
To this end we set the following rules:
\begin{itemize}
\item if the $k$-th characteristic field is genuinely nonlinear~\eqref{e:gnl} and $s_\alpha>0$ (i.e. if $s_\alpha$ is a shock front) we can choose $\mathrm{sp}_\alpha$ in such a way that it satisfies both the Lax condition  
$
     \lambda_k (\mf{ u}_\alpha) < \mathrm{sp}_\alpha < 
   \lambda_k (\mf t_k (\mf{ u}_\alpha, s_\alpha))
$
and the implication~\eqref{incidono}\footnote{The reason why we cannot simply set $\mathrm{sp}_\alpha: = \sigma_k (\mf u_\alpha, s_\alpha)$ is because, as in~\cite{Bressan}, we want to avoid interactions of more than two wave-fronts, and this implies that we may have to very slightly modify the speed of some wave-fronts}.
\item if either the $k$-th characteristic field is linearly degenerate~\eqref{e:lindeg} or it is genuinely nonlinear and $s_\alpha<0$ (i.e., $\alpha$ is a rarefaction front)   we recall \cite[(FT4) p.143]{Bressan} and set 
\be \label{e:speedrare}
    \mathrm{sp}_\alpha   =  \lambda_k (\mf{ u}_\alpha).
\eq
\end{itemize}
We also recall that following the construction in~\cite[p.129]{Bressan}, we fix $r_\ee>0$\footnote{In~\cite[p.129]{Bressan} this threshold is actually termed $\delta$} and we make sure that each rarefaction front has strength bounded by $r_\ee$ by splitting, if needed, big rarefaction fronts into smaller ones. As in~\cite{Bressan}, we make an exception to this rule when a rarefaction front of the $i$-th family interacts with another wave and generates a rarefaction front of the $i$-th family: in this case the front exiting the interaction is not partitioned, even if its strength is greater than $r_\ee$.
\subsubsection{Simplified Riemann solver} \label{ss:srie} 
We follow the same construction as in~\cite{Bressan}[p.131] and we fix once and for all the speed of the non-physical wave fronts as $\hat \lambda>0$ satisfying 
\be \label{e:npspeed}
    \hat \lambda \ge \sup |\lambda_i (\mf u)| \quad \text{for every $i=1, \dots, N$.}
\eq 
\begin{remark} \label{r:destrasinistra}
A word of warning concerning the use of the simplified Riemann solver. We recall footnote~\footref{foot} at page~\pageref{foot} and the construction in~\cite[pp.131-132]{Bressan} and realize that when using the simplified Riemann solver we have to use (rather than the admissible wave-fan curves of \emph{left} states $\mf t_i$) the admissible wave-fan curves of 
\emph{right} states, the one introduced in the original paper by Lax~\cite{Lax}, which in the following we term $\boldsymbol{m}_i$, $i=1,\dots, N$. In this way, physical fronts can travel with very large and \emph{positive} speed. Note however that from the point of view of the strength of the wave fronts, the two constructions are completely equivalent since (up to a suitable choice of the parametrization) we have the relation 
\be \label{e:avantindietro}
      \mf t_i (\mf u^+, s) = \mf u^-  \iff  \boldsymbol{m}_i (\mf u^-, -s) = \mf u^+, \; \text{for every $s$, $\mf u^+$ and $i=1, \dots, N$}. 
\eq
\end{remark}
\subsection{Accurate boundary Riemann solver} \label{ss:abrie}
In this paragraph we consider the boundary Riemann problem obtained by coupling~\eqref{e:claw} with the boundary and initial conditions~\eqref{e:briedata}. We pass to the $\mf u$ coordinates and set $\mf u^+= \mf u(\mf v^+)$, $\mf u_b = \mf u(\mf v_b)$. We separately consider the cases of a genuinely nonlinear and a linearly degenerate $k$-th family. 
\subsubsection{Accurate boundary Riemann solver (genuinely nonlinear case)} \label{ss:abriegnl}
We assume~\eqref{e:gnl} and to fix the notation we assume~\eqref{e:a11nz}. Cases B) and C) in Hypothesis~\ref{h:eulerlag} can be handled analogously.   We consider $(\xi'_{\ell+1}, \dots, \xi'_{k-1}, s'_k, s'_{k+1}, \dots, s'_N)$ satisfying
\be 
\label{e:brs}
     \boldsymbol{\beta} ({\boldsymbol \phi} ( \cdot,  \xi'_{\ell+1}, \dots, \xi'_{k-1}, s'_k) \circ \mf t_{k+1} (\cdot, s'_{k+1}) 
\circ \dots \circ  \mf t_{N} (\mf u^+, s'_{N}), \mf u_b)= \mf 0_{N-\ell},   
\eq
where $ \boldsymbol{\beta}$ and ${\boldsymbol \phi}$ are the same as in~\eqref{e:betagnl} and~\eqref{e:bl10}, respectively. We recall cases i),$\dots$, vi) in \S\ref{sss:casignl} and define the accurate boundary Riemann solver as follows. \\
{\sc Step 1:} we define $\mf{\hat u}$ by setting
\be \label{e:hatu}
\mf{\hat u} : = \mf t_{k+1} (\cdot, s'_{k+1}) \circ \dots \mf t_N (\mf u^+, s'_N), 
\eq
see Figure~\ref{f:fabio}, and we use the accurate Riemann solver to solve the Riemann problem between $\mf u^+$ (on the right) and $\mf{\hat u}$. \\
{\sc Step 2} we separately consider the following cases 
\begin{itemize} 
\item[i)] If $\lambda_k (\mf{\hat u}) \ge 0$ and $s'_k>0$ we use the accurate Riemann solver to solve the Lax admissible shock between $\mf{\hat u}$ (on the right) and $\boldsymbol{\zeta}_k (\mf{\hat u}, s'_k)$ (on the left).
\item[ii)]  If $\lambda_k (\mf{\hat u}) \ge 0$ and $\bar s(\mf{\hat u}) \leq s'_k \leq 0$ we use the accurate Riemann solver to solve the rarefaction between $\mf{\hat u}$ (on the right) and $\boldsymbol{\zeta}_k (\mf{\hat u}, s'_k)$ (on the left). If the boundary Riemann problem~\eqref{e:claw},~\eqref{e:briedata} is originated at the interaction between a wave front and the boundary and the hitting wave front is a rarefaction of the $k$-th family  then the outgoing rarefaction front  between $\mf{\hat u}$ and $\boldsymbol{\zeta}_k (\mf{\hat u}, s_k)$ is not partitioned, even if its strength is greater than the threshold $r_\ee$. 
\item[iii)] If $\lambda_k (\mf{\hat u}) \ge 0$ and  $s'_k < \bar s (\mf{\hat u})$ we set $\mf{\bar u}: = \boldsymbol{\zeta}_k (\mf{\hat u}, \bar s(\mf{\hat u})$ and we use the accurate Riemann solver described at the previous point to solve the rarefaction between $\mf{\hat u}$ (on the right) and $\bar{\mf u}$ (on the left). If the boundary Riemann problem~\eqref{e:claw},~\eqref{e:briedata} is originated at the interaction between a wave front and the boundary and the hitting wave front is of the $k$-th family, then the rarefaction front  between $\mf{\hat u}$ and $\mf{\bar u}$ is not partitioned, even if its strength is greater than the threshold $r_\ee$. 
\item[iv)] If  $\lambda_k (\mf{\hat u}) < 0$ and $s'_k > \underline s (\mf{\hat u})$ we use the accurate Riemann solver to solve the Lax admissible shock between $\mf{\hat u}$ (on the right) and $\boldsymbol{\zeta}_k (\mf{\hat u}, s'_k)$ (on the left). 
\item[v)] If  $\lambda_k (\mf{\hat u}) < 0$ and $s'_k = \underline s (\mf{\hat u})$ there is no wave of the $k$-th family entering the domain after the interaction.
\item[vi)] If  $\lambda_k (\mf{\hat u}) <0$ and $s'_k < \underline s (\mf{\hat u})$ there is no wave of the $k$-th family entering the domain after the interaction. 
\end{itemize}
\subsubsection{Accurate boundary Riemann solver (linearly degenerate case)} \label{sss:532}
We assume~\eqref{e:lindeg}  and to fix the notation we focus on case B) in Hypothesis~\ref{h:eulerlag}. Cases A) and C) in Hypothesis~\ref{h:eulerlag} can be handled analogously. We recall the analysis in \S\ref{sss:casild}, solve 
\be 
\label{e:brslindeg}
     \boldsymbol{\beta} ({\boldsymbol \phi} ( \cdot,  \xi'_{\ell(\mf u_b)+1}, \dots, \xi'_{k-1}, s'_k) \circ \mf t_{k+1} (\cdot, s'_{k+1}) 
\circ \dots \circ  \mf t_{N} (\mf u^+, s'_{N}), \mf u_b)= \mf 0_{N},  
\eq
and consider the same value $\mf{\hat u}$ as in~\eqref{e:hatu}. We use the accurate Riemann solver to solve the Riemann problem between $\mf u^+$ (on the right) and $\mf{\hat u}$ (on the left) and 
\begin{itemize}
\item if $\lambda_k(\mf{\hat u})>0$ we accurately solve the Riemann problem with a contact discontinuity between $\mf u^+$ (on the right) and $\mf t_k (\mf{\hat u}, s'_k)$ (on the left);
\item  if $\lambda_k(\mf{\hat u})\leq 0$  there is no wave of the $k$-th family entering the domain after the interaction. 
\end{itemize}
\begin{remark}\label{r:1812}
Let $\bar{\mf u}'$ denote the trace of the wave front-tracking approximation after the interaction. Owing to the analysis in \S\ref{sss:casignl} and \S\ref{sss:casild} we have~\eqref{e:protra} and~\eqref{e:sigmatraccia}  (in the genuinely nonlinear case) and~\eqref{e:protra2} (in the linearly degenerate case) with 
$\xi'_k$ and $\bar{\mf u}'$ in place of $\xi_k$ and $\bar{\mf u}$, respectively. Here $\xi'_k$ satisfies  $\boldsymbol{\beta} ( {\boldsymbol \phi} (\mf{\bar u}', \xi'_{\ell+1}, \dots, \xi'_{k}), \mf u_b) = \mf 0_{N}$ and $\boldsymbol{\beta} ( {\boldsymbol \phi} (\mf{\bar u}', \xi'_{\ell (\mf u_b)+1}, \dots, \xi'_{k}), \mf u_b) = \mf 0_{N}$
in case A) and B) in Hypothesis~\ref{h:eulerlag}, respectively. The explicit expression of $\xi'_k$ in terms of $s'_k$ can be obtained by arguing as in \S\ref{sss:casignl} and \S\ref{sss:casild}.
\end{remark}
\subsection{Simplified boundary Riemann solver} \label{ss:sbrie}
To define the simplified boundary Riemann solver we consider a wave-front hitting the boundary and we term $\mf u^+$ and $\mf u^- $ the states at the right and left-hand side of the wave, respectively. We term $j$ the family of the hitting wave front and we separately consider the following cases. \\ 
{\sc Case $j<k$:} the simplified boundary Riemann solver is a non-physical wave front traveling with speed $\hat \lambda$ and with left and right states $\mf u^-$ and $\mf u^+$, respectively. \\
{\sc Case $j=k$:} we consider the solution of~\eqref{e:claw},~\eqref{e:briedata} and $\mf{\hat u}$ defined by~\eqref{e:hatu}. To define the simplified boundary Riemann solver 
\begin{enumerate}
\item we accurately solve the boundary Riemann problem~\eqref{e:claw},\eqref{e:briedata} with $\mf v^+$ replaced by $\mf v(\mf{\hat u})$. Note that by the definition of $\mf{\hat u}$ the solution only involves boundary layers and a wave front of the $k$-th family (if any);
\item we introduce a non-physical front between $\mf{\hat u}$ (on the left) and $\mf u^+$ (on the right);
\item we juxtapose the two.  
\end{enumerate}
Note that, if the $k$-th vector field is genuinely nonlinear, then the trace $\mf{\bar u}$ of the wave front-tracking approximation after the interaction (solved by using the simplified boundary Riemann solver) satisfies again~\eqref{e:protra}.
\subsection{Discrimination between accurate and simplified Riemann and boundary Riemann solvers} \label{ss:whenas}
We fix a parameter $\omega_\ee$, to be determined in the following, and set the following rules:
\begin{itemize}
\item Riemann problems involving the initial datum and boundary Riemann problems involving discontinuities in the boundary datum are always solved accurately;
\item at interactions between two wave fronts with strengths $|s_\alpha|$ and $|s_\beta|$ inside the domain we use the accurate Riemann solver if $|s_\alpha s_\beta| \ge \omega_\ee$, the simplified Riemann solver otherwise;
\item when a wave front of strength $|s|$ and family  $j<k$ hits the boundary, we use the accurate boundary Riemann solver if $|s|\ge \omega_\ee$, the simplified boundary Riemann solver otherwise;
\item assume that a $k$-th wave front of strength $|s|$ hits the boundary, and the $k$-th characteristic field is genuinely nonlinear~\ref{e:gnl}. In this case we use the accurate boundary Riemann solver if 
$|s|\big([\varsigma_k (\mf u^+, s)]^-+ |\xi_k|\big)\ge \omega_\ee$, the simplified boundary Riemann solver otherwise. In the previous expression, $\varsigma_k$ is the same as in~\eqref{e:speedwft}, $\mf u^+$ denotes the right state of the hitting wave front and $|\xi_k|$ the strength of the center component of the boundary layer. That is, if $\mf u_b$ denotes the boundary datum and if we have~\eqref{e:a11nz} then
$$
   \boldsymbol{\beta} ( {\boldsymbol \phi} (\cdot, \xi_{\ell +1}, \dots, \xi_{k})\circ \mf t_{k} (\mf u^+, s), \mf u_b) = \mf 0_{N- \ell}
$$
for some $ \xi_{\ell +1}, \dots, \xi_{k-1}$. 
\item when a $k$-th wave front hits the boundary, and the $k$-th characteristic field is linearly degenerate~\ref{e:lindeg} we use the accurate boundary Riemann solver if 
$|s| \big([\lambda_k (\mf u^+)+ |\xi_k|\big)\ge \omega_\ee$, the simplified boundary Riemann solver otherwise.
Here $\xi_k$ is the same as at the previous item. 
\end{itemize}
The motivation for the above choices comes from Lemma~\ref{l:nc},~\ref{l:me} and~\ref{l:melindeg} below, from the interaction estimates~\cite[(7.31),(7.32)]{Bressan}, and from the consequent expression of the functionals defined in \S\ref{s:functional}.
\section{Main interaction estimates}\label{s:ie}  
In this section we establish new interaction estimates: more precisely, we accurately track the strength of the new wave fronts that are generated when a wave front interacts with the boundary. From the technical viewpoint, the main estimate~\eqref{e:me} is one of the most innovative results of the paper. The exposition is organized as follows: in \S\ref{ss:iej} we focus on the interaction of a wave front of the family $j$, $j<k$, with the boundary. In \S~\ref{ss:ik} we state our main results concerning the interaction between a wave front of the $k$-th family with the boundary, and in \S\ref{ss:proofme} we give the proof in the case of a genuinely nonlinear $k$-th characteristic field, and in \S\ref{ss:prooflindeg} we deal with the linearly degenerate case. In \S\ref{ss:fie} we establish some further interaction estimates we need in the following sections. 
\subsection{Interaction between the boundary and a wave front of the $j$-th family, $j< k$} \label{ss:iej}
We fix $j < k$ and consider a wave front of the $j$-th family hitting the boundary $x=0$, see Figure~\ref{f} (left). 
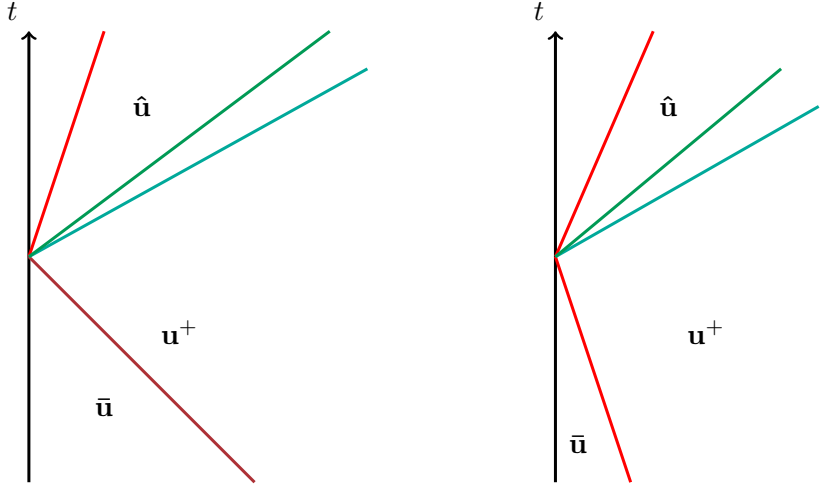
\begin{figure}
\begin{center}
\caption{A wave front of the $j$-th family, $j<k$, (left) and of the $k$-th family (right) hitting the boundary and the wave fronts that are generated at the interaction.} 
\label{f}
\begin{tikzpicture}
\draw[line width=0.4mm,->] (0,0) -- (0,6) node[anchor=south east] {$t$};
\draw[line width=0.4mm,Maroon]   (3, 0) -- (0, 3); 
\draw[line width=0.4mm,red]   (0, 3) -- (1, 6); 
\draw[line width=0.4mm,JungleGreen]   (0, 3) -- (4.5, 5.5); 
\draw[line width=0.4mm,ForestGreen]   (0, 3) -- (4, 6); 
\draw  (1,1) node {$\mf{\bar u}$};
\draw  (2,2) node {$\mf{u}^+$};
\draw  (1.5,5) node {$\mf{\hat u}$};
\draw[line width=0.4mm,->] (7,0) -- (7,6) node[anchor=south east] {$t$};
\draw[line width=0.4mm,red]   (8, 0) -- (7, 3); 
\draw[line width=0.4mm,red]   (7, 3) -- (8.3, 6); 
\draw[line width=0.4mm,JungleGreen]   (7, 3) -- (10.5, 5); 
\draw[line width=0.4mm,ForestGreen]   (7, 3) -- (10, 5.5); 
\draw  (7.3,0.5) node {$\mf{\bar u}$};
\draw  (9,2) node {$\mf{u}^+$};
\draw  (8.5,5) node {$\mf{\hat u}$};
\end{tikzpicture}
\end{center}
\end{figure}
To ease the exposition, we assume~\eqref{e:a11nz}, but the analysis in cases B) and C) in Hypothesis~\ref{h:eulerlag} is entirely analogous. 
We introduce the following notation: we term $\mf{\bar u}$ and $\mf u^+$ the left and the right state, respectively, of the hitting wave front of the $j$-th family, which yields 
\be
\label{e:lax2}
    \mf{\bar u} = \mf t_j (\mf u^+, s_j), \quad j<k,
\eq
for some suitable $s_j$. In the previous expression, $\mf t_j$ denotes the wave fan curve of admissible states, see~\cite{Lax} and \S\ref{ss:lax}. We also term  $\mf u_b$ the ``boundary datum", 
namely we assume $\mf v(\mf{\bar u})\sim_{\mf D}\mf v(\mf u_b)$ in the sense of Definition~\ref{d:equiv}. Note that owing to the analysis in \S\ref{sss:casignl} and \S\ref{sss:casild} this implies 
~\eqref{e:brs0} for some $(\xi_{\ell+1}, \dots, \xi_k)$ with $\xi_k$ satisfying~\eqref{e:protra} (if the $k$-th field in genuinely nonlinear) or~\eqref{e:protra2} (if the $k$-th vector field is linearly degenerate). To solve the boundary Riemann problem, we have to determine $\xi'_{\ell+1}, \dots, \xi'_{k-1}, s'_k, s'_{k+1}, \dots, s'_N$
in such a way that~\eqref{e:brs} holds. 
\begin{lemma}
\label{l:nc}
Assume~\eqref{e:brs0},\eqref{e:lax2} and~\eqref{e:brs}, then there is a constant $C_1>0$ only depending on the functions $\mf A$, $\mf E$, $\mf B$ and $\mf G$ such that 
\be
\label{e:nc}
      \sum_{i=\ell + 1}^{k-1} |\xi_i - \xi_i'| + |s'_k - \xi_k| + \sum_{i=k+1}^N |s'_i|  \leq C_1 |s_j|.
\eq
\end{lemma}
Note that the above Lemma holds in both cases of genuinely nonlinear and linearly degenerate $k$-th vector field.
\begin{proof}[Proof of Lemma~\ref{l:nc}]
First, we recall~\eqref{e:betagnl} and point out that $\boldsymbol{\beta} (\mf u, \mf u_b) = \mf 0_{N}$ if and only if $\mf u_{2}= \mf u_{2b}$ and $\boldsymbol{\pi}_{11} (\mf u) = \boldsymbol{\pi}_{11} (\mf u_b) $. In other words,   $\boldsymbol{\beta} (\mf u, \mf u_b) = \mf 0_{N}$ if and only if $\boldsymbol{\pi} (\mf u) = \boldsymbol{\pi} (\mf u_b) $, where $\boldsymbol{\pi}$ notes the projection onto a suitable $N - \ell$ dimensional subspace. Owing to the analysis in~\cite{AnconaBianchini,BianchiniSpinoloARMA,BianchiniSpinolo} we also have that the map at the left hand side of~\eqref{e:brs} satisfies the assumptions of~\cite[Theorem 1]{Clarke2}. Owing to the Lipschitz continuity of the inverse function, by combining~\eqref{e:brs0},~\eqref{e:lax2} and~\eqref{e:brs} we conclude 
\begin{equation*}
\begin{split}
     \sum_{i=\ell + 1}^{k-1}& |\xi_i - \xi_i'| + |s'_k - \xi_k| + \sum_{i=k+1}^N |s'_i|   \leq \unpo 
    |\boldsymbol{\pi} (\mf u_b ) - \boldsymbol{\pi}  \circ {\boldsymbol \phi} ( \mf u^+,   \xi_{\ell +1}, \dots, \xi_{k-1}, \xi_k)  |\\
    & \stackrel{\eqref{e:brs0},\eqref{e:lax2}}{=} \unpo
    | \boldsymbol{\pi} \circ  {\boldsymbol \phi} ( \cdot,  \xi_{\ell +1}, \dots, \xi_{k-1}, \xi_k) \circ \mf t_j (\mf u^+, s_j) -   \boldsymbol{\pi} \circ {\boldsymbol \phi} ( \mf u^+,   \xi_{\ell +1}, \dots, \xi_{k-1}, \xi_k)| \\
    & \leq  \unpo |  {\boldsymbol \phi} ( \cdot,   \xi_{\ell +1}, \dots, \xi_{k-1}, \xi_k) \circ \mf t_j (\mf u^+, s_j) -   {\boldsymbol \phi} ( \mf u^+,   \xi_{\ell +1}, \dots, \xi_{k-1}, \xi_k)|. \phantom{\int}
\end{split}
\end{equation*}
By the Lipschitz continuity of ${\boldsymbol \phi}$ with respect to $\mf u$ the above term is controlled by $\unpo |\mf t_j (\mf u^+, s_j)  - \mf u^+|$ and this in turn is controlled by $C_1 |s_j|$ for some suitable constant $C_1$.
\end{proof}
\subsection{Interaction of the boundary with a wave front of the $k$-th family} \label{ss:ik}
We now consider the case of a wave front of the $k$-th family hitting the boundary, see Figure~\ref{f} (right)
As in the previous paragraph we term $\mf{\bar u}$ and $\mf u^+$ the left and the right state of the hitting wave front, respectively, which implies 
\be
\label{e:lax}
    \mf{\bar u} = \mf t_k (\mf u^+, s_k),
\eq
for some value $s_k$. The estimates we need in the case the hitting wave front belongs to the $k$-th family are much more delicate and precise than~\eqref{e:nc}, and we obtain them by distinguishing between the genuinely nonlinear case~\eqref{e:gnl} and the linearly degenerate case~\eqref{e:lindeg}. 
\subsubsection{Genuinely nonlinear case}
We consider the genuinely nonlinear case~\eqref{e:gnl} and to fix the notation we also assume~\eqref{e:a11nz}.  The argument in case B) or C) in Hypothesis~\ref{h:eulerlag} is entirely 
analogous. Note that we have~\eqref{e:brs0} for some $(\xi_{\ell+1}, \dots, \xi_k)$ with $\xi_k$ satisfying~\eqref{e:protra} and~\eqref{e:sigmatraccia}, and by combining these inequalities with~\eqref{e:lax}  we obtain 
\be \label{e:ceblayer}
       \xi_k  \leq \underline s (\mf t_k (\mf u^+, s_k) ) \; \text{if} \; \lambda_k  (\mf t_k (\mf u^+, s_k) )  \leq 0, 
       \qquad 
       \xi_k =0  \; \text{if} \; \lambda_k  (\mf t_k (\mf u^+, s_k) )  > 0 \eq
and  $\varsigma_k (\mf u^+, s_k) <0$. We now consider the values $(\xi'_{\ell+1}, \dots, \xi'_{k-1}, s'_k, \dots, s'_N)$ satisfying~\eqref{e:brs} and state our main interaction estimate.
\begin{lemma}
\label{l:me} 
Assume~\eqref{e:gnl},~\eqref{e:brs0},~\eqref{e:brs},~\eqref{e:lax} with $\varsigma_k (\mf u^+, s_k) <0$, and~\eqref{e:ceblayer}. There is a constant $C_2>0$ only depending on the functions $\mf A$, $\mf E$, $\mf B$ and $\mf G$ and on the value $\mf u^\ast$ such that 
\be 
\label{e:me}
     \sum_{i=\ell+ 1}^{k-1} |\xi_i - \xi_i'| + |s'_k - (\xi_k + s_k)| + \sum_{i=k+1}^N |s'_i|  \leq C_2 |s_k| \Big(   [\varsigma_k (\mf u^+, s_k)]^-  + |\xi_k|   \Big),
\eq
where the symbol $[\cdot]^-$ denotes the negative part.
 \end{lemma}
 As mentioned before, from the technical viewpoint estimate~\eqref{e:me} is one of the key points of our analysis. Its proof is rather long and involved and it is postponed to \S\ref{ss:proofme}. 
\begin{remark}
Note that~\eqref{e:nc} is basically ``a first order estimate", because the right hand side is of the same order of the strength $s_j$.  Conversely,~\eqref{e:me} can be regarded as ``a second order estimate" because the right hand side is much smaller than the strength $s_k$.  The reason why in the case of the waves of the $j$-th family, $j<k$, a ``first order estimate" like~\eqref{e:nc} suffices is because these waves do not appear after the interaction and hence can be differently weighted, in the same spirit as in~\cite{Amadori}, see~\eqref{e:s}. 
\end{remark}
\begin{remark}
Note that the right-hand side of~\eqref{e:me} does not contain any term of the form $|s_k| | \sum_{i=\ell+1}^{k-1} |\xi_i|$. If it did, then our analysis in the next sections would fail. Very loosing speaking, the reason is because the potential $\Upsilon$ defined by~\eqref{e:upsilon} should then contain a term like  $|s_k| | \sum_{i=1}^{k-1} |\xi_i|$, which in general does not disappear when a $k$-wave hits the domain boundary $x=0$, and hence our proof of the monotonicity of $\Upsilon$ would break down. 
\end{remark}
\subsubsection{Linearly degenerate case}
We now focus on the linearly degenerate case~\eqref{e:lindeg}. To fix the notation, we also assume case B) in Hypothesis~\ref{h:eulerlag}, which is satisfied by the Navier-Stokes and viscous MHD equations in Eulerian coordinates when the fluid velocity vanishes. Note that case C) corresponds to the same equations written in lagrangian coordinates, but in this case no $k$-th wave front can hit the boundary since $\lambda_k (\mf u)\equiv 0$. 

We consider the values $\xi_{\ell(\mf u_b) +1}, \dots, \xi_k, s_k$ and $\xi'_{\ell(\mf u_b) +1}, \dots, \xi'_{k-1}, s'_k, \dots, s'_N$ satisfying \eqref{e:lax},~\eqref{e:brs0lindeg} and~\eqref{e:brslindeg}, respectively, and state our main interaction estimate.
\begin{lemma}\label{l:melindeg}
\label{l:me2} Assume~\eqref{e:lindeg},~\eqref{e:lax2}, \eqref{e:brs0lindeg} and~\eqref{e:brslindeg}. Then 
\begin{equation} \label{e:melindeg}
 \sum_{\ell(\mf u_b)+1}^{k-1} |\xi_i - \xi_i'| + |s'_k - (\xi_k + s_k)| + \sum_{i=k+1}^N |s'_i|  \leq C_{10} |s_k|\big(  |\xi_k| + [\lambda_k (\mf u^+)]^- \big)
\end{equation}
for a suitable constant $C_{10}>0$ only depending on $\mf A$, $\mf E$, $\mf B$ and $\mf G$ and on the value $\mf u^\ast$.  
\end{lemma}
The proof of Lemma~\ref{l:melindeg} relies on the proof of Lemma~\ref{l:me} and is given in \S\ref{ss:prooflindeg}. 
\subsection{Proof of Lemma~\ref{l:me} (interaction estimates in the genuinely nonlinear case)} \label{ss:proofme}
We first establish three preliminary results. 
\begin{lemma}
\label{l:gammak}
Let $\boldsymbol{\gamma}_k$ be the same function as in~\eqref{e:chicosa2}, then $\boldsymbol{\gamma}_k$ is Lipschitz continuous with respect to both the variables $\bf{\widetilde u}$ and $\xi_k$. If $\lambda_k (\bf{\widetilde u}) \ge 0$ it is also of class $C^{1,1}$ with respect to $\xi$, if
  $\lambda_k (\bf{\widetilde u}) < 0$ it is of class $C^{1,1}$ on each side of the value $\xi_k = 
\underline s  (\bf{\widetilde u}).$
\end{lemma}
\begin{proof}
All the properties directly follow from the definition~\eqref{e:chicosa2}, except for the  $C^{1,1}$ regularity in the case $\lambda_k (\bf{\widetilde u}) \ge 0$: this follows by combining~\eqref{e:uguale1} with~\eqref{e:uguale2}. 
\end{proof}

\begin{lemma}
\label{l:speed}
Fix $\mf{\widetilde u} \in \R^N$ and let $ \mf t_k (\mf{\widetilde u}, \cdot)$ and $ \boldsymbol{\gamma}_k (\mf{\widetilde u}, \cdot)$ the curves defined by~\eqref{e:laxc} and by~\eqref{e:chicosa2}, respectively.  Then 
\be \label{e:me1}
    |  (\mf t_k (\mf{\widetilde u}, s_k), \mf 0_{N-h}) -
         \boldsymbol{\gamma}_k (\mf{\widetilde u}, s_k) | \leq \unpo |s_k| [\varsigma_k (\mf{\widetilde u}, s_k)]^-
\eq
provided $\varsigma_k (\mf{\widetilde u}, s_k)$ is the same as in~\eqref{e:speedwft}.
\end{lemma}
The proof of Lemma~\ref{l:speed} is a key step in the proof of~\eqref{e:me} and it is rather technical, but it is based on two basic ingredients: first, the $k$-th vector field is genuinely nonlinear and hence the value $\varsigma_k$ defined by~\eqref{e:speedwft} can be used to parametrize both  
$\mf t_k$ and $\boldsymbol{\zeta}_k$. Second, the proof of~\eqref{e:me1} would be easier if $\boldsymbol{\gamma}_k$ were of class $C^{1, 1}$. The function $ \boldsymbol{\gamma}_k $ is unfortunately only Lipschitz continuous, but it is ``piecewise of class $C^{1,1}$" and hence one can circumvent the lack of regularity by separately considering different cases (see 
{\sc Case 1, 2A} and {\sc 2B} below). Note however that the lack of regularity accounts for much of the technicality of the proof.   
\begin{proof}[Proof of Lemma~\ref{l:speed}] To simplify the exposition, in the proof of the lemma we use the shorthand notation $\varsigma_k$ for $\varsigma_k (\mf{\widetilde u}, s_k)$. We also recall the definitions~\eqref{e:underlines} and~\eqref{e:baresse} of $\underline s (\mf{\widetilde u})$ and $\bar s (\mf{\widetilde u})$, respectively.
We now separately consider the following cases. \\
{\sc Case 1: $\lambda_k  (\mf{\widetilde u}) \ge 0$}. We recall the definition~\eqref{e:baresse} of $\bar s (\mf{\widetilde u})$ and conclude that $\bar s (\mf{\widetilde u}) \leq 0$. 
We also recall~\eqref{e:chicosa2} and point out that, if $s_k \ge \bar s (\mf{\widetilde u})$, then the left-hand side of~\eqref{e:me1} vanishes and~\eqref{e:me1} is trivially satisfied.  By definition of $\bar s$ we have $\lambda_k (\mf t_k (\mf{\widetilde u}, \bar s (\mf{\widetilde u})  )=0$ and by combining~\eqref{e:siattacca} with the Taylor expansion formula we arrive at 
 \begin{equation} \label{e:taylor}
     | ( \mf t_k (\mf{\widetilde u}, s_k), \mf 0_{N-h}) -
        \boldsymbol{\gamma}_k (\mf{\widetilde u}, s_k) | \leq \unpo ( \bar s  (\mf{\widetilde u}) - s_k)^2
      \stackrel{s_k \leq \bar s( \mf{\widetilde u}) \leq 0}{\leq} \unpo ( \bar s  (\mf{\widetilde u}) - s_k) |s_k| 
         \quad \text{if $s_k \leq \bar s  (\mf{\widetilde u})$}.
\eq
We recall~\eqref{e:rarefdef},~\eqref{e:baresse} and \eqref{e:speedwft} and point out that for $s_k\leq 0$
$$
    [\varsigma_k]^- \stackrel{\eqref{e:baresse},\eqref{e:speedwft} }{=} \int_{s_k}^{\bar s (\mf{\widetilde u})} \frac{\partial \lambda_k (\mf i_k (\mf{\widetilde u}, \tau))}{\partial \tau} d \tau
    \stackrel{\eqref{e:gnl} }{\ge} d  ( \bar s  (\mf{\widetilde u}) - s_k).
$$
By plugging the inequality $( \bar s  (\mf{\widetilde u}) - s_k) \leq [\varsigma_k]^-/d$ into~\eqref{e:taylor} we arrive at~\eqref{e:me1}. \\
{\sc Case 2: $\lambda_k  (\mf{\widetilde u}) < 0$}, which implies $\bar s (\mf{\widetilde u})  >0$. Owing to~\eqref{e:chicosa2}, if $s_k \ge \underline s (\mf{\widetilde u})$ then the left-hand side of~\eqref{e:me1} vanishes and~\eqref{e:me1} is trivially satisfied. We now distinguish between the cases $s_k \leq 0$ and $0 < s_k <  \underline s (\mf{\widetilde u})$. \\
{\sc Case 2A: $s_k \leq 0$}. We have 
$$
    \varsigma_k      \stackrel{\eqref{e:baresse}}{=} - \int_{s_k}^{\bar s (\mf{\widetilde u})    } \frac{\partial \lambda_k (\mf i_k (\mf{\widetilde u}, \tau)   )}{\partial \tau} d \tau
      \stackrel{\eqref{e:gnl}}{\leq} - d   (\bar s (\mf{\widetilde u}) - s_k)
$$
which yields 
\be \label{e:pranzo} 
    \max \{|s_k|, \bar s (\mf{\widetilde u}) \} 
    \stackrel{\bar s (\mf{\widetilde u}) >0, s_k \leq 0}{\leq}
    |\bar s (\mf{\widetilde u}) - s_k | \leq \frac{1}{d} [\varsigma_k]^-. 
\eq
We set   
\be \label{e:a}
     \mf a: = \mf i_k(\mf{\widetilde u}, \bar s(\mf{\widetilde u}))  \; \text{which implies} \;  \lambda_k (\mf a) =0, \; \bar s(\mf a) =0, \;
     \mf i_k(\mf a, - \bar s(\mf{\widetilde u})) = \mf{\widetilde u}
\eq
and we refer to Figure~\ref{f2} for a representation. 
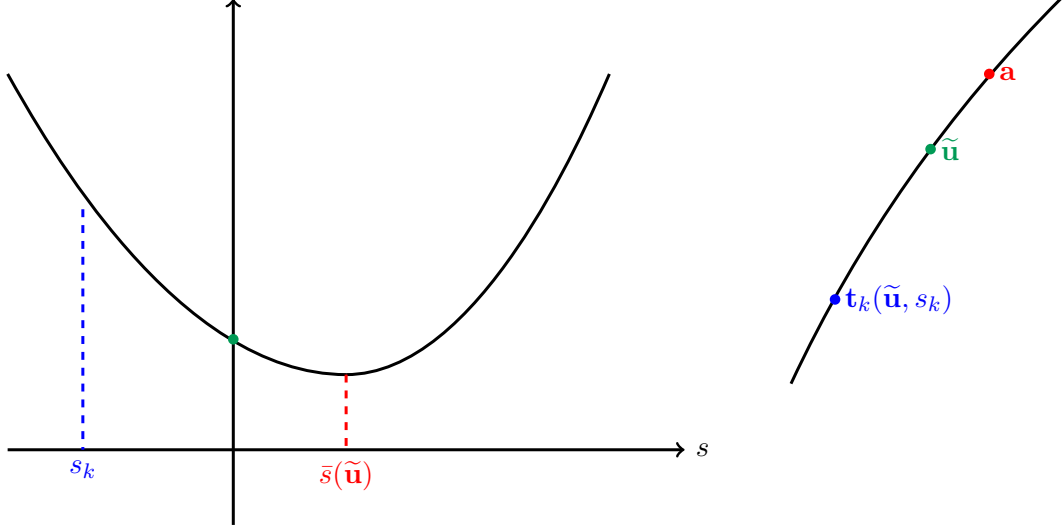
\begin{figure}
\begin{center}
\caption{The function $s \mapsto \lambda_k (\mf i_k (\mf{\widetilde u}, s))$ (left) and a representation of the curve $\mf i_k (\mf{\widetilde u}, s)$ in the phase space (right) in {\sc Case 2A}}
\label{f2}
\begin{tikzpicture}
\draw[line width=0.4mm,->] (0,-1) -- (0,6);
\draw[line width=0.4mm,->] (-3,0) -- (6,0) node[anchor=west] {$s$};
\draw[line width=0.4mm] (1.5,1) parabola (5,5);
\draw[line width=0.4mm] (1.5,1) parabola (-3,5);
\draw[line width=0.4mm, red, dashed] (1.5,1) --(1.5,0) node[anchor=north] {$\bar s(\mf{\widetilde u})$};
\draw[line width=0.4mm, blue, dashed] (-2,3.2) --(-2,0) node[anchor=north] {$s_k$};
\draw[line width=0.4mm]  (11,6) arc (135:155:18cm);
\fill[red] (10.05,5) circle (2pt) node[anchor=west] {$\mf{a}$};
\fill[ForestGreen] (9.27,4) circle (2pt) node[anchor=west] {$\mf{\widetilde u}$};
\fill[ForestGreen] (0,1.47) circle (2pt);
\fill[blue] (8,2) circle (2pt) node[anchor=west] {$\mf t_k(\mf{\widetilde u}, s_k)$};
\end{tikzpicture}
\end{center}
\end{figure}
We also set 
\be \label{e:gpiccolo}
   \boldsymbol{\varphi}(\tau, s) : =
    \big(   \mf t_k ( \mf i_k(\mf a, \tau), s), \mf 0_{N-h}    \big)   
    -
  \boldsymbol{\gamma}_k\big(
    \mf i_k(\mf a, \tau), s\big)
    \implies 
     \boldsymbol{\varphi} (- \bar s(\mf{\widetilde u}), s_k) =  
    (\mf t_k (\mf{\widetilde u}, s_k), \mf 0_{N-h}) -
         \boldsymbol{\gamma}_k (\mf{\widetilde u}, s_k)
\eq
We now recall that $\lambda_k (\mf a )=0$, $\bar s(\mf a)=0$ and by using~\eqref{e:taylor} we get 
\be \label{e:taylor2}
   |  \boldsymbol{\varphi}(0, s) |= 
    |\big(   \mf t_k ( \mf a, s), \mf 0_{N-h}    \big)   
    -
  \boldsymbol{\gamma}_k\big(
   \mf a, s\big)|
   \leq \unpo |s|^2,
\eq 
Also,
\be \label{e:vanish}
   \boldsymbol{\varphi}(\tau, 0) = \mf 0_{2N-h} \; \text{for every $\tau$} \implies 
   \frac{\partial  \boldsymbol{\varphi}}{\partial \tau} (\tau, 0) =  \mf 0_{2N-h} \; \text{for every $\tau$}. 
\eq
Since the function $ \boldsymbol{\varphi}$ is of class $C^{1,1}$ on the set $\{(\tau, s):  s \leq 0 \}$, we have  
\[ 
    \begin{split}  
   | &  \boldsymbol{\varphi} (- \bar s(\mf{\widetilde u}), s_k)|  =
   \left|   \boldsymbol{\varphi} (0, s_k) - \bar s(\mf{\widetilde u}) 
    \int_0^1  \frac{\partial  \boldsymbol{\varphi}}{\partial \tau} (- \theta  \bar s(\mf{\widetilde u}) , s_k) d \theta
   \right|  \\ & \stackrel{\eqref{e:taylor2},\eqref{e:vanish}}{\leq}
    \unpo |s_k|^2 + 
     \bar s(\mf{\widetilde u}) 
    \int_0^1   \left| \frac{\partial  \boldsymbol{\varphi}}{\partial \tau} (- \theta  \bar s(\mf{\widetilde u}) , s_k) -
    \frac{\partial  \boldsymbol{\varphi}}{\partial \tau} (- \theta  \bar s(\mf{\widetilde u}) , 0)
     \right| d \theta \leq  \unpo |s_k|^2  +  \unpo \bar s(\mf{\widetilde u})  |s_k| \\ &
    \stackrel{\eqref{e:pranzo}}{\leq} \frac{\unpo}{d}|s_k| [\varsigma_k]^-+
     \frac{\unpo}{d}|s_k| [\varsigma_k]^-
   \end{split}
\]
and owing to~\eqref{e:gpiccolo} this yields~\eqref{e:me1}. \\
{\sc Case 2B:} $0< s_k <  \underline s ( \mf{\widetilde u})$.  Owing to~\eqref{e:chicosa2} we have
\be \label{e:semplifica}
     (\mf t_k (\mf{\widetilde u}, s_k), \mf 0_{N-h}) -
         \boldsymbol{\gamma}_k (\mf{\widetilde u}, s_k)=
      (\mf t_k (\mf{\widetilde u}, s_k), \mf 0_{N-h}) -
        \mf b_k (\mf{\widetilde u}, s_k)
\eq
We recall that $\sigma_k (\mf{\widetilde u}, s)$ denotes the speed of the shock joining $ (\mf{\widetilde u}$ (on the right) with $\mf h_k (\mf{\widetilde u}, s)$ (on the left) and that the function $\mf h_k$ is defined by~\eqref{e:rh}. Next, we consider the function 
$$
    \boldsymbol{\ell} (\tau, s) : =
    \left(
    \begin{array}{cc}
                 \sigma_k \big( \mf i_k (\mf a,  \tau ), s  \big) \\
                 s \\
    \end{array}
    \right),
$$
where $\mf a$ is the same as in~\eqref{e:a}. 
Note that 
\be \label{e:li}
    \boldsymbol{\ell} (- \bar s (\mf{\widetilde u}), s) : =
    \left(
    \begin{array}{cc}
                 \sigma_k (\mf{\widetilde u}, s ) \\
                 s \\
    \end{array}
    \right), \qquad \qquad 
    \mf D  \boldsymbol{\ell} (0, 0) = 
      \left(
    \begin{array}{cc}
                \displaystyle{\frac{\partial \sigma_k}{\partial \tau}   } &   \displaystyle{\frac{\partial \sigma}{\partial s}   }\\
                0 &  1 \\
    \end{array}
    \right).
\eq
Since $ \sigma_k \big( \mf i_k (\mf a,  \tau ), 0  \big)=  \lambda_k (  \mf i_k (\mf a,  \tau ))$, then 
\be \label{e:mela}
    \left. \frac{\partial \sigma_k}{\partial \tau} (\mf i_k (\mf a,  \tau), 0) \right|_{\tau =0}= \nabla \lambda_k ( \mf a) \cdot \mf r_k ( \mf a) \stackrel{\eqref{e:gnl}}{\ge} d >0 
\eq
and hence the function $\boldsymbol{\ell}$ is locally invertible in a neighborhood of the origin $ (0, 0)$. Next, we set 
\be \label{e:621}
    \widetilde{\boldsymbol{\varphi}}(\tau, s) : = \big(\mf t_k ( \mf i_k (\mf a, \tau ), s), \mf 0_{N-h}\big) -
        \mf b_k \big( \mf i_k (\mf a, \tau ), s\big)
\eq
and point out that
\be \label{e:li2}
     \widetilde{\boldsymbol{\varphi}}(- \bar s (\mf{\widetilde u}), s)  \stackrel{\eqref{e:a}}{=} \big(\mf t_k (\mf{\widetilde u}, s), \mf 0_{N-h} \big) - \mf b_k (\mf{\widetilde u}, s),
\eq
which is the right-hand side of~\eqref{e:semplifica}. Also, we set
\be \label{e:q}
       \underline{\boldsymbol{\varphi}}(\sigma, s) : =   \widetilde{\boldsymbol{\varphi}}\circ \boldsymbol{\ell}^{-1} ( \sigma, s)
\eq 
and point out that
\be \label{e:qli}
       \boldsymbol{\ell}^{-1} ( \sigma_k (\mf{\widetilde u}, s ), s) \stackrel{\eqref{e:li}}{=} 
       (- \bar s (\mf{\widetilde u}), s)
        \stackrel{\eqref{e:li2}}{\implies}  \underline{\boldsymbol{\varphi}}( \sigma_k (\mf{\widetilde u}, s), s) =  \big(\mf t_k (\mf{\widetilde u}, s), \mf 0_{N-h} \big) - \mf b_k (\mf{\widetilde u}, s). 
\eq 
Assume for the time being that we have established the equalities 
\be \label{e:derzero}
    \underline{\boldsymbol{\varphi}} (0, s) =\mf 0_{2N-h} \; \text{for every $s \ge 0$}, \quad  \mf{\underline g}(\sigma, 0) = \mf 0_{2N-h}  \; \text{for every $\sigma$}.
\eq
The function $ \underline{\boldsymbol{\varphi}}$ is of class $C^2$ since both $ \widetilde{\boldsymbol{\varphi}}$ and $ \boldsymbol{\ell}$ are of class $C^2$. We can then argue as in the proof of~\cite[Lemma 2.5]{Bressan} and conclude that 
$$
   | \underline{\boldsymbol{\varphi}}(\sigma, s) | \leq \unpo |\sigma||s| , \; \text{for every $s \ge 0$}. 
$$
Owing to~\eqref{e:qli} and to the inequality $s_k \ge0$, this implies that the modulus of the right-hand side of~\eqref{e:semplifica} is bounded by $\unpo |\sigma_k (\mf{\widetilde u}, s_k)||s_k|$. By recalling that $\sigma_k (\mf{\widetilde u}, s_k)<0$ because $s_k < \underline s(\mf{\widetilde u})$ and by using~\eqref{e:speedwft}  this concludes the proof of~\eqref{e:me1}. We are thus left to establish~\eqref{e:derzero}. To establish the first equality in~\eqref{e:derzero}, we recall that $\lambda_k( \mf a)=0$, which in turn yields 
\be \label{e:uva}
   \sigma_k (\mf i_k (\mf a, 0), s)= \sigma_k ( \mf a, s)  \ge 0 \; \text{for every $s \ge 0$}.
\eq
Owing to~\eqref{e:mela}, the partial derivative $\partial \sigma_k (\mf i_k (\mf a, \tau), s) / \partial \tau$ is strictly positive in a neighborhood of the origin.  Owing to~\eqref{e:uva}, this  implies that if $\sigma_k (\mf i_k (\mf a, \tau), s) =0$ for some $s \ge 0$, then $\tau \leq 0$, which in turn implies $\lambda_k (\mf i_k (\mf a, \tau)) \leq 0$ and $s = \underline s (\mf i_k (\mf a, \tau))$. We now recall~\eqref{e:621} and~\eqref{e:q} and infer that 
\begin{equation*} \begin{split}
    \underline{\boldsymbol{\varphi}}(0, s)  =
    \widetilde{\boldsymbol{\varphi}}(\tau, \underline s (\mf i_k (\mf a, \tau))) & =  \big( \mf t_k (\mf i_k ( \mf a, \tau), \underline s (\mf i_k (\mf a, \tau)), \mf 0_{N-h}\big) - \mf b_k (\mf i_k ( \mf a, \tau), \underline s (\mf i_k (\mf a, \tau)) \stackrel{\eqref{e:pera}}{=} \mf 0_{2N-h} \\
     & \quad    \text{for every $s \ge 0$},
     \end{split}
\end{equation*}
which establishes the first equality in~\eqref{e:derzero}. To establish the second one we point out that 
\begin{equation*}
\begin{split}
     \underline{  \boldsymbol{\varphi}}(\sigma, 0)& = \widetilde{\boldsymbol{\varphi}}(\tau, 0)   
    =  \big( \mf t_k (\mf i_k ( \mf a, \tau), 0), \mf 0_{N-h} \big) - \mf b_k (\mf i_k ( \mf a, \tau), 0)=
    \big( \mf i_k ( \mf a, \tau), \mf 0_{N-h} \big) - \big( \mf i_k ( \mf a, \tau), \mf 0_{N-h} \big) 
    \\ & = \mf 0_{2N-h }\quad  \text{for $\tau$ such that
    $\lambda_k (\mf i_k ( \mf a, \tau))= \sigma$.}
\end{split}
\end{equation*}
This concludes the proof of~\eqref{e:derzero} and hence of~\eqref{e:me1}. 
\end{proof}
\begin{lemma} \label{l:proseguo}
Under the same assumptions as in the statement of Lemma~\ref{l:me} we have  
\be \label{e:anguria2}
   | \boldsymbol{\gamma}_k   (\mf t_k (\mf u^+, s_k), \xi_k) -    \boldsymbol{\gamma}_k (\mf u^+, s_k + \xi_k)  |
   \leq \unpo |s_k| ( [\varsigma_k(\mf u^+, s_k)]^- + |\xi_k| )
\eq 
provided the curve $\boldsymbol{\gamma}_k$ is the same as in~\eqref{e:chicosa2}. 
\end{lemma}
Note that by combining~\eqref{e:anguria2} with~\eqref{e:stessacosa} we get in particular
\be \label{e:anguria}
   | \boldsymbol{\zeta}_k   (\mf t_k (\mf u^+, s_k), \xi_k) -    \boldsymbol{\zeta}_k (\mf u^+, s_k + \xi_k)  |
   \leq \unpo |s_k| ( [\varsigma_k(\mf u^+, s_k)]^- + |\xi_k| )
\eq 
\begin{proof}[Proof of Lemma~\ref{l:proseguo}]
To simplify the exposition, in the proof of the lemma we use the shorthand notation $\varsigma_k$ for $\varsigma_k (\mf u^+, s_k)$. 
We separately consider the following cases.\\
{\sc Case 1:} $\lambda_k (\mf{u}^+) \ge 0$. By combining the definition~\eqref{e:speedwft} of $\varsigma_k$ with the inequality $\varsigma_k<0$ we conclude that $\lambda_k (\mf t_k (\mf u^+, s_k))<0$ and hence that $s_k < \bar s (\mf u^+ ) \leq 0$.   
 Owing to~\eqref{e:ceblayer} this implies $\xi_k \leq \underline s (\mf t_k (\mf u^+, s_k)).$   We set 
\begin{equation} \label{e:pi}
    \mf q(s, \xi) : = \boldsymbol{\gamma}_k  ( \mf t_k(\mf u^+, s), \xi) -    \boldsymbol{\gamma}_k (\mf u^+, s + \xi)
\end{equation}
and we consider  the function $\mf q$ defined on the set 
\be \label{e:EEEE}
    E : = \{ (s, \xi) : \; s \leq \bar s (\mf u^+ ), \; \xi \leq \underline s (\mf t_k (\mf u^+, s)) \},
\eq
see Figure~\ref{f3}.
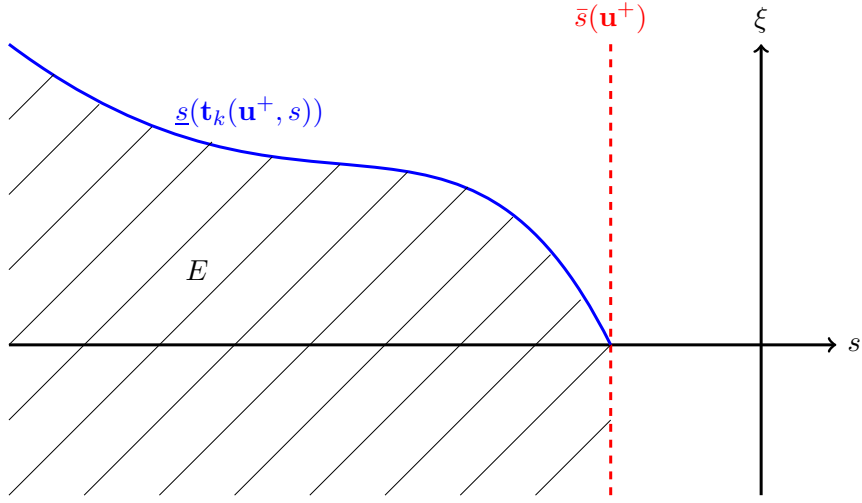
\begin{figure}
\begin{center}
\caption{The set $E$ in~\eqref{e:EEEE}} 
\label{f3}
\begin{tikzpicture}
\draw[line width=0.4mm,->] (-10,0) -- (1,0) node[anchor=west] {$s$};
\draw[line width=0.4mm,red,dashed] (-2,-2) -- (-2,4) node[anchor=south] {$\bar s(\mf{u}^+)$};
\draw[line width=0.4mm,->] (0,-2) -- (0,4) node[anchor=south] {$\xi$};
\draw[line width=0.4mm,blue] (-2,0) .. controls (-4,4) and (-6,1) .. (-10,4);
\draw[line width=0.1mm] (-10,3) -- (-9.4,3.6); 
\draw[line width=0.1mm] (-10,2) -- (-8.8,3.2); 
\draw[line width=0.1mm] (-10,1) -- (-8.1,2.9); 
\draw[line width=0.1mm] (-10,0) -- (-7.3,2.7); 
\draw[line width=0.1mm] (-10,-1) -- (-6.5,2.5); 
\draw[line width=0.1mm] (-10,-2) -- (-5.6,2.4); 
\draw[line width=0.1mm] (-9,-2) -- (-4.7,2.3); 
\draw[line width=0.1mm] (-8,-2) -- (-3.9,2.1); 
\draw[line width=0.1mm] (-7,-2) -- (-3.3,1.7); 
\draw[line width=0.1mm] (-6,-2) -- (-2.8,1.2); 
\draw[line width=0.1mm] (-5,-2) -- (-2.4,0.6); 
\draw[line width=0.1mm] (-4,-2) -- (-2,0); 
\draw[line width=0.1mm] (-3,-2) -- (-2, -1); 
\draw (-7.5, 1) node {$E$};
\draw[blue]  (-6.8,3.1) node {$\underline s(\mf t_k(\mf{u}^+, s))$};
\end{tikzpicture}
\end{center}
\end{figure}
Note that  $ \boldsymbol{\gamma}_k (\mf u^+, \cdot)$ is of class  $C^{1,1}$ since $\lambda_k (\mf u^+) \ge 0$ by assumption, and hence $\mf q$ is of class $C^{1,1}$ with respect to the variable $\xi$ on $E$ owing to Lemma~\ref{l:gammak}.  
We now evaluate $ \mf q(\bar s (\mf u^+), \xi) $. 
By definition~\eqref{e:baresse} of $\bar s$, $\lambda_k (\mf t_k(\mf u^+, \bar s (\mf u^+)) )=0$ and this yields 
\be \label{e:autunno}
    \underline s (\mf t_k (\mf u^+,  \bar s (\mf u^+)))=0=\bar s (\mf t_k (\mf u^+,  \bar s (\mf u^+)))
\eq   
If  $\xi \leq \underline s (\mf t_k (\mf u^+,  \bar s (\mf u^+)))=0$ then by the third definition at the right hand-side of~\eqref{e:chicosa2} and because of~\eqref{e:autunno} we have $\boldsymbol{\gamma}_k  ( \mf t_k(\mf u^+, \bar s (\mf u^+)), \xi)= \mf b_k  ( \mf t_k(\mf u^+, \bar s (\mf u^+)), \xi)$. Also, if $\xi \leq 0$ then $\bar s (\mf u^+) + \xi \leq \bar s (\mf u^+) $ and hence~\eqref{e:chicosa2} yields $\boldsymbol{\gamma}_k (\mf u^+, \bar s (\mf u^+) + \xi) = \mf b_k  ( \mf t_k(\mf u^+, \bar s (\mf u^+)), \xi)$. 
We conclude that 
\be \label{e:liguria}
    \mf q(\bar s (\mf u^+), \xi) = \mf 0_{2N-h} \quad \text{for every 
   $\xi \leq  \underline s (\mf t_k (\mf u^+, \bar s (\mf u^+))=0$}, 
\eq
and this yields 
\be \label{e:kiwi}
    \frac{\partial  \mf q  }{\partial \xi} (\bar s (\mf u^+), \xi ) = \mf 0_{2N-h}
    \quad \text{for every $\xi \leq \underline s (\mf t_k (\mf u^+, \bar s (\mf u^+))=0$.}
\eq
If $(s_k, \xi_k) \in E$ and $\xi_k <  0$ we then have
\begin{equation} \label{e:ananas}
\begin{split}
           |\mf q (s_k, \xi_k) |  & = \left| \mf q (s_k, 0) + \xi_k
            \int_0^1  \frac{\partial  \mf q }{\partial \xi} 
          (s_k, \theta \xi_k) d \theta \right| \\
          & \stackrel{\eqref{e:pi},\eqref{e:kiwi}}{\leq}
        | \underbrace{ \boldsymbol{\gamma}_k( \mf t_k(\mf u^+, s_k), 0)}_{\displaystyle{=  ( \mf t_k(\mf u^+, s_k), \mf 0_{N-h})} } - \boldsymbol{\gamma}_k(\mf u^+, s_k) |
           +  \left|  \xi_k 
            \int_0^1 \left(  \frac{\partial   \mf q  }{\partial \xi} 
          (s_k, \theta \xi_k) -  \frac{\partial   \mf q }{\partial \xi} 
          (\bar s(\mf u^+), \theta \xi_k)  \right) d \theta  
          \right| \\ & \stackrel{\eqref{e:me1}}{\leq}
         \unpo |s_k| [\varsigma_k]^- + \unpo |\xi_k| |s_k -\bar s(\mf u^+) |
         \stackrel{s_k \leq  \bar s(\mf u^+) \leq 0}{\leq} \unpo |s_k| ([\varsigma_k]^- + |\xi_k|),
\end{split}
\end{equation}
that is~\eqref{e:anguria}. For $(s_k, \xi_k) \in E$, $\xi_k \ge 0$ we have by $C^{1, 1}$ regularity 
\begin{equation} \label{e:ananas2}
\begin{split}
           |\mf q (s_k, \xi_k) | & = \left| \mf q (s_k, 0) + \xi_k 
            \int_0^1  \frac{\partial \mf q }{\partial \xi} 
          (s_k, \theta \xi_k) d \theta \right| \\ & \stackrel{\eqref{e:pi}}{\leq}
        |( \mf t_k(\mf u^+, s_k), \mf 0_{N-h}) - \boldsymbol{\gamma}_k(\mf u^+, s_k) |
           +  \left|  \xi_k
            \int_0^1 \! \! \! \left(  \frac{\partial \mf q  }{\partial \xi} 
          (s_k, \theta \xi_k) -  \frac{\partial \mf q  }{\partial \xi} 
          (s_k, 0)  \right) d \theta + \xi_k\frac{\partial  \mf q }{\partial \xi} 
          (s_k, 0) \right|\\& \stackrel{\eqref{e:me1}}{\leq}
          \unpo |s_k| [\varsigma_k]^- + \unpo |\xi_k |^2 +
         |\xi_k | \left| 
         \frac{\partial \mf q  }{\partial \xi} 
          (s_k, 0) \right|  \\ &
        \stackrel{\eqref{e:kiwi}}{\leq}
          \unpo |s_k| [\varsigma_k]^- + \unpo |\xi_k|^2 +
         |\xi_k | \left| 
         \frac{\partial \mf q  }{\partial \xi} 
          (s_k, 0)   - \frac{\partial \mf q }{\partial \xi}  (\bar s (\mf u^+), 0)  \right|
        \\ &  \leq \unpo ( |s_k| [\varsigma_k]^- + |\xi_k|^2 + |\xi_k| |s_k -\bar s (\mf u^+) | ) 
       \stackrel{s_k \leq  \bar s(\mf u^+) \leq 0}{\leq} 
      \unpo ( |s_k| [\varsigma_k]^- + |\xi_k|^2 + |\xi_k| |s_k | ) 
\end{split}
\end{equation}
To conclude, we point out that we are considering the case
\begin{equation*}
\begin{split}
    0 \leq  \xi_k \leq \underline s (\mf t_k (\mf u^+, s_k)) \implies 
     |\xi_k| & \leq | \underline s (\mf t_k (\mf u^+, s_k))| 
     \stackrel{\eqref{e:autunno}}{=} | \underline s (\mf t_k (\mf u^+, s_k)) - 
    \underline s (\mf t_k (\mf u^+, \bar s (\mf u^+ ))  |  \\ & 
    \stackrel{\eqref{e:lipunders}}{\leq} \unpo |s_k - \bar s (\mf u^+ )| \stackrel{s_k < \bar s (\mf u^+ ) \leq 0}{\leq} \unpo |s_k|
\end{split}
\end{equation*}
and by plugging the above inequality into~\eqref{e:ananas2} we arrive at~\eqref{e:anguria2}. \\
{\sc Case 2:} $\lambda_k (\mf u^+) < 0$ and $\lambda_k (\mf t_k (\mf u^+, s_k)) >0$. Owing to~\eqref{e:ceblayer} we have $\xi_k=0$, which yields
$$
    | \boldsymbol{\gamma}_k   (\mf t_k (\mf u^+, s_k), \xi_k) -    \boldsymbol{\gamma}_k (\mf u^+, s_k + \xi_k)  |
    \stackrel{\xi_k=0}{=}| (\mf t_k (\mf u^+, s_k), \mf 0_{N-h}) -    \boldsymbol{\gamma}_k (\mf u^+, s_k )  |
     \stackrel{\eqref{e:me1}}{\leq} \unpo |s_k| [\varsigma_k]^-,
$$
which in turn implies~\eqref{e:anguria2}. \\
{\sc Case 3:} $\lambda_k (\mf u^+) < 0$ and $\lambda_k (\mf t_k (\mf u^+, s_k)) \leq 0$, which yields $s_k \leq \underline{s} (\mf u^+)$ and hence, owing to~\eqref{e:chicosa2},  
\be \label{e:fabioandrea}
     \boldsymbol{\gamma}_k  (\mf u^+, s_k) = \mf b_k (\mf u^+, s_k).
\eq
Owing to~\eqref{e:ceblayer} we have 
$\xi_k \leq \underline s (\mf t_k (\mf u^+, s_k))$ and hence, owing to~\eqref{e:chicosa2}, 
\be \label{e:sonobl}
     \boldsymbol{\gamma}_k   (\mf t_k (\mf u^+, s_k), \xi_k)  = \mf b_k    (\mf t_k (\mf u^+, s_k), \xi_k). 
\eq
We now separately consider the following two sub-cases.\\
{\sc Case 3A:} $s_k + \xi_k \leq \underline s (\mf u^+).$ Owing to~\eqref{e:chicosa2} and recalling~\eqref{e:sonobl} this implies $\mf q(s_k, \xi_k) = 
\mf{\widetilde q}(s_k, \xi_k)$, where $\mf q$ is the same as in~\eqref{e:pi} and 
\be \label{e:tildeq}
   \mf{\widetilde q}(s, \xi): = \mf b_k (\mf t_k (\mf u^+, s), \xi)- \mf b_k  (\mf u^+, s + \xi)
\eq
Note that $\mf{\widetilde q}$ is a function of class $C^2$ satisfying $\mf{\widetilde q} (0, \xi)=0$ for every $\xi$, and hence $\frac{\partial \mf{\widetilde q} }{\partial \xi} (0, \xi) =0$ for every $\xi$. This yields  
\begin{equation} \label{e:betulla}
\begin{split}
           |\mf{\widetilde q}(s_k, \xi_k) |  & = \left| \mf{\widetilde q} (s_k, 0) + \xi_k
            \int_0^1  \frac{\partial \mf{\widetilde q}  }{\partial \xi} 
          (s_k, \theta \xi_k) d \theta \right| \\ & \stackrel{\eqref{e:tildeq}}{=}
          |( \mf t_k(\mf u^+, s_k), \mf 0_{N-h}) - \mf b_k (\mf u^+, s_k) |         
            +  \left|  \xi_k
            \int_0^1 \left(  \frac{\partial  \mf{\widetilde q}  }{\partial \xi} 
          (s_k, \theta \xi_k) -  \frac{\partial \mf{\widetilde q} }{\partial \xi} 
          (0, \theta \xi_k)  \right) d \theta  
          \right| \\ & \stackrel{\eqref{e:me1},\eqref{e:fabioandrea}}{\leq}
         \unpo |s| [\varsigma_k]^- + \unpo |\xi_k| |s_k|,
         \end{split}
\end{equation}
and this implies~\eqref{e:anguria2}.\\
{\sc Case 3B:}  $s_k + \xi_k > \underline s (\mf u^+).$ Owing to~\eqref{e:chicosa2}, this implies $ \boldsymbol{\gamma}_k   (\mf u^+, s_k+ \xi_k)= (\mf t_k  (\mf u^+, s_k+ \xi_k), \mf 0_{N-h})$ and yields 
\be \label{e:dopous}
\begin{split}
          | \boldsymbol{\gamma}_k & ( \mf t_k(\mf u^+, s_k), \xi_k) -    \boldsymbol{\gamma}_k (\mf u^+, s_k + \xi_k)| =
        |  \boldsymbol{\gamma}_k  ( \mf t_k(\mf u^+, s_k), \xi_k) -    (\mf t_k  (\mf u^+, s_k+ \xi_k), \mf 0_{N-h})| \\
   & \leq |  \boldsymbol{\gamma}_k  ( \mf t_k(\mf u^+, s_k), \xi_k) -    \big(\mf t_k (\mf t_k  (\mf u^+, s_k), \xi_k), \mf 0_{N-h}  \big)| + |\mf t_k (\mf t_k  (\mf u^+, s_k), \xi_k)- \mf t_k  (\mf u^+, s_k+ \xi_k) | \\
& \stackrel{\eqref{e:me1}}{\leq}
    \unpo [\varsigma_k (\mf t_k(\mf u^+, s_k), \xi_k)]^- |\xi_k|  +  |\mf t_k (\mf t_k  (\mf u^+, s_k), \xi_k)- \mf t_k  (\mf u^+, s_k+ \xi_k) | \\
    & \stackrel{\eqref{e:speedwft}}{=}
    \unpo [\sigma_k (\mf t_k(\mf u^+, s_k), \xi_k)]^- |\xi_k|  +  |\mf t_k (\mf t_k  (\mf u^+, s_k), \xi_k)- \mf t_k  (\mf u^+, s_k+ \xi_k) |. 
\end{split}
\eq
To establish the last equality, we have used the fact that $\xi_k \ge 0$ since  $\xi_k \ge \underline s(\mf u^+)- s_k$ (because we are dealing with {\sc Step 3B}) and $ \underline s(\mf u^+)- s_k \ge 0$ (because $\varsigma_k \leq 0$).

To control the right-hand side of~\eqref{e:dopous}, we first point out that, owing to genuine nonlinearity~\eqref{e:gnl}, the map $\xi  \mapsto [  \sigma_k (\mf t_k(\mf u^+, s_k), \xi)  ]^- $ is monotone decreasing. Since $\xi_k \ge \underline s(\mf u^+)- s_k$ then 
\be \label{e:leccio}
      [  \sigma_k (\mf t_k(\mf u^+, s_k), \xi_k)  ]^-  \leq [  \sigma_k (\mf t_k(\mf u^+, s_k), 
      \underline s(\mf u^+)- s_k )  ]^-. 
\eq
We now set $q(s) : = [  \sigma_k (\mf t_k(\mf u^+, s), 
      \underline s(\mf u^+)- s )  ]^- $ and we point out that $q$ is a Lipschitz continuous function which, owing 
to the definition~\eqref{e:underlines} of $\underline s$, satisfies  
$q(0)=0$. This implies that $|q(s)| \leq \unpo |s|$ and by plugging this inequality into~\eqref{e:leccio} we arrive at
\be \label{e:cactus}
       [  \sigma_k (\mf t_k(\mf u^+, s_k), \xi_k)  ]^- \leq \unpo |s_k|.
\eq
To control the second term at the right-hand side of~\eqref{e:dopous} we set $\mf{\underline q} (s, \xi) : = \mf t_k (\mf t_k  (\mf u^+, s), \xi)- \mf t_k  (\mf u^+, s+ \xi)$ and point out that $\mf{\underline q} $ is a $C^2$ function satisfying $\mf{\underline q} (s, 0)= \mf 0_N$ for every $s$ and  $\mf{\underline q} (0, \xi)= \mf 0_N$ for every $\xi$. Owing to~\cite[Lemma 2.5]{Bressan} this implies $|\mf{\underline q} (s, \xi)| \leq \unpo |s| |\xi|$ and by plugging this inequality into~\eqref{e:dopous} and recalling~\eqref{e:cactus} we arrive at~\eqref{e:anguria2}. 
\end{proof}
\subsubsection{Proof of Lemma~\ref{l:me}} \label{sss:pme}
By arguing as in the proof of Lemma~\ref{l:nc} we arrive at 
\begin{equation*}
\begin{split}
     \sum_{i=\ell + 1}^{k-1}& |\xi_i - \xi_i'| + |s'_k - (\xi_k + s_k)| + \sum_{i=k+1}^N |s'_i|   \\ &\stackrel{\eqref{e:brs0},\eqref{e:lax}}{\leq}
    | {\boldsymbol \phi} ( \cdot,  \xi_{\ell +1}, \dots, \xi_{k-1}, \xi_k) \circ \mf t_k (\mf u^+, s_k) -  {\boldsymbol \phi} ( \mf u^+,  \xi_{\ell +1}, \dots, \xi_{k-1}, \xi_k+s_k)|\\
    &  \stackrel{\eqref{e:bl10}}{\leq}
     | {\boldsymbol \psi}_s ( \cdot,  \xi_{\ell +1}, \dots, \xi_{k-1}) \circ \boldsymbol{\zeta}_k (\cdot, \xi_k)  \circ \mf t_k (\mf u^+, s_k) -  {\boldsymbol \psi}_s 
      ( \cdot,  \xi_{\ell +1}, \dots, \xi_{k-1}) \circ \boldsymbol{\zeta}_k (\mf u^+, \xi_k+s_k) | \\
      &  \quad + | {\boldsymbol \psi}_p (\mf t_k (\mf u^+, s_k),  \xi_{\ell +1}, \dots, \xi_{k-1}, \xi_k)  -  {\boldsymbol \psi}_p 
      ( \mf u^+,  \xi_{\ell +1}, \dots, \xi_{k-1}, \xi_k+s_k) |  : = \mathcal E. \phantom{\int}
\end{split}
\end{equation*}
By using the Lipschitz continuity of both $ {\boldsymbol \psi}_s$ and $ {\boldsymbol \psi}_p $ we then get 
\[
\begin{split}
       \mathcal  E & \leq \unpo | \boldsymbol{\zeta}_k (\cdot, \xi_k)  \circ \mf t_k (\mf u^+, s_k)-
       \boldsymbol{\zeta}_k (\mf u^+, \xi_k + s_k)  | +
        |\boldsymbol{\gamma}_k (\cdot, \xi_k) \circ \mf t_k (\mf u^+, s_k) - \boldsymbol{\gamma}_k  (\mf u^+, \xi_k+s_k)| \\
       & \stackrel{\eqref{e:anguria2},\eqref{e:anguria}}{\leq} \unpo |s_k| ( [\varsigma_k (\mf u^+, s_k)]^- + |\xi_k| ),
\end{split} 
\]
that is~\eqref{e:me}. 
\subsection{Further boundary interaction estimates in the genuinely nonlinear case} \label{ss:fie}
In this paragraph we collect some further results (Lemmas~\ref{l:shockaftershock},~\ref{l:enterare} and~\ref{l:nebbia2})concerning the interaction between  a wave front of the $k$-th family and the boundary. We will use them in the following sections.  
\begin{lemma}\label{l:shockaftershock}
Under the same assumptions as in Lemma~\ref{l:me}, there is a sufficiently small constant $\nu>0$ such that, if $[\sigma_k (\mf u^+, s_k)]^- + |\xi_k| \leq \nu$, then the following holds. Assume $s_k > 0$
and let $\mf{\hat u}$ be the same as in~\eqref{e:hatu}; 
 then $\lambda_k (\mf{\hat u}) < 0$. 
\end{lemma}
\begin{remark}\label{r:entrashock}
The meaning of Lemma~\ref{l:shockaftershock} is the following. Assume that a $k$-th family wave front hits the boundary, and that the wave front is a shock, which implies $s_k>0$.  Then, if there is a $k$-th family wave front entering the domain after the interaction, this wave front must be a shock. Indeed, $\mf{\hat u}$ is the right state of the wave front and the condition $\lambda_k (\mf{\hat u}) < 0$ dictates that $\mf{\hat u}$ cannot be the right state of a $k$-th family rarefaction wave with positive speed. 
\end{remark}
The reason why we need the above Lemma~\ref{l:shockaftershock} (or, more precisely, its consequence discussed in Remark~\ref{r:entrashock})  is because we want to have that every wave front can be uniquely continued forward in time, unless it gets canceled by the interaction with a front of the same family with opposite side or it is absorbed by the boundary. If a shock could hit the boundary and be reflected as a rarefaction, then by the construction in~\S\ref{ss:abrie} this rarefaction could be partitioned and hence the original wave front could not be uniquely continued forward in time. 
\begin{proof}[Proof of Lemma~\ref{l:shockaftershock}]
We proceed according to the following steps. \\
{\sc Step 1:} we show that 
\be \label{e:bdsulk}
     [\lambda_k (\mf u^+)]^- \ge \frac{d}{2} s_k,
\eq
provided $d$ is the same constant as in~\eqref{e:gnl}. To this end, we first of all point out that the Lax entropy condition combined with the inequality $\sigma_k (\mf u^+, s_k)<0$ yields $\lambda_k (\mf u^+) <0$ and also $s_k < \underline s(\mf u^+)$. Recalling that $s_k >0$ by assumption, this implies that to establish~\eqref{e:bdsulk} it suffices to show that 
\be \label{e:bdsulk2}
     [\lambda_k (\mf u^+)]^- \ge \frac{d}{2} \underline s(\mf u^+). 
\eq
Owing to~\eqref{e:gnl} we have $\partial \sigma_k (\mf u^+, s)/ds \ge d/2$ provided $s$ is small enough. This implies 
 $$
   0= \sigma_k(\mf u^+, \underline s(\mf u^+)) = \sigma_k(\mf u^+, 0)+  \int_0^{\underline s(\mf u^+)} \frac{\partial \sigma_k}{\partial s} ds = 
   \lambda_k (\mf u^+)+  \int_0^{\underline s(\mf u^+)} \frac{\partial \sigma_k}{\partial s} ds \ge  \lambda_k (\mf u^+)+  \frac{d}{2}\underline s(\mf u^+)
$$
and this yields~\eqref {e:bdsulk2} and henceforth~\eqref{e:bdsulk}. \\
{\sc Step 2:} we conclude the proof. We point out that 
\[
   \begin{split}
   \lambda_k (\mf{\hat u}) & =  \lambda_k (\mf u^+) +  \lambda_k (\mf{\hat u}) -  \lambda_k (\mf u^+) \stackrel{\eqref{e:bdsulk}}{\leq} 
   - \frac{d}{2} s_k + \unpo |\mf{\hat u} - \mf u^+| \stackrel{\eqref{e:hatu}}{\leq} - \frac{d}{2} s_k + \unpo \sum_{i=k+1}^N |s'_i|  \\ &
   \stackrel{\eqref{e:me}}{\leq} s_k \left(- \frac{d}{2} + \unpo ([\varsigma_k(\mf u^+, s_k)]^- + |\xi_k|   ) \right) <0 
   \end{split}
\]
provided $[\varsigma_k(\mf u^+, s_k)]^-+ |\xi_k| $ is sufficiently small. By recalling~\eqref{e:speedwft} and the inequality $s_k>0$ we get $\varsigma_k(\mf u^+, s_k)= \sigma_k(\mf u^+, s_k)$ and this concludes the proof of the lemma. 
\end{proof}
\begin{lemma}\label{l:enterare}Under the same assumptions as in the statement of Lemma~\ref{l:me}, there is a sufficiently small constant $\nu>0$ such that, if $[\varsigma_k(\mf u^+, s_k)]^- + |\xi_k| \leq \nu$, then the following holds. Assume that $\lambda_k (\mf u^+) <0$, $\lambda_k (\mf{\hat u})>0$, where $\mf{\hat u}$ is the same as in~\eqref{e:hatu},  and $s'_k < 0$; then 
\be \label{e:rarepiccola}
      \min \{ |s'_k|, |\bar s(\hat{\mf u}) |      \}  \leq C_8 |s_k|(|\xi_k| + [\varsigma_k(\mf u^+, s_k)]^-)
\eq
for a suitable constant $C_8 >0$ only depending on the functions $\mf A$, $\mf E$, $\mf B$ and $\mf G$ and on the value $\mf u^\ast$.  
\end{lemma} 
The meaning of the above lemma is the following. Assume that there is a $k$-th family wave front hitting the boundary and that after the interaction there is a $k$-th family rarefaction wave front enterning the domain. Owing to Lemma~\ref{l:shockaftershock}, this dictates that the hitting wave is a rarefaction front and owing to~\eqref{e:speedrare} yields  $\lambda_k (\mf u^+) <0$. Lemma~\ref{l:enterare} dictates that the left-hand side of~\eqref{e:rarepiccola}, which is the size of the rarefaction wave front entering the domain after the interaction, is actually fairly small as it is controlled by the right hand side of~\eqref{e:rarepiccola}.
\begin{proof} 
It suffices to show that $|\bar s(\hat{\mf u}) | \leq C_8 |s_k|(|\xi_k| + [\varsigma_k(\mf u^+, s_k)]^-)$. We have 
\begin{equation}\label{e:neve}
\begin{split}
        |\bar s(\hat{\mf u}) | & \stackrel{\eqref{e:gnl}}{\leq} \frac{2}{d} \big( \lambda_k (\hat{\mf u}) - \underbrace{\lambda_k (\mf t_k (\hat{\mf u }, \bar s(\hat{\mf u})  ))}_{=0 \;\text{by~\eqref{e:baresse}}} \big)
         \stackrel{\lambda_k (\mf u^+) <0 }{\leq}  \frac{2}{d} \big( \lambda_k (\hat{\mf  u}) - \lambda_k (\mf u^+) \big) \leq \unpo |\hat{\mf  u} - \mf u^+|
         \\ & \stackrel{\eqref{e:hatu}}{\leq} \unpo \sum_{i=k+1}^N |s'_i| \stackrel{\eqref{e:me}}{\leq} 
          C_8 |s_k|(|\xi_k| +[\varsigma_k(\mf u^+, s_k)]^-) , 
\end{split}
\end{equation}
and this yields~\eqref{e:rarepiccola}. 
\end{proof}
We now consider an interaction where the wave hitting the boundary belongs to the $j$-th family, $j<k$, and combine~\eqref{e:lax2} with~\eqref{e:protra} to get 
\be \label{e:pioggia2}
       \xi_k  \leq \underline s (\mf t_j (\mf u^+, s_j)) \; \text{if} \; \lambda_k  (\mf t_j (\mf u^+, s_j))  \leq 0, 
       \qquad 
       \xi_k =0  \; \text{if} \; \lambda_k  (\mf t_j (\mf u^+, s_j))  > 0 .
\eq
\begin{lemma}
\label{l:nebbia2}
Under the same assumptions as in the statement of Lemma~\ref{l:nc}, assume~\eqref{e:pioggia2}, let $\mf{\hat u}$ be as in~\eqref{e:hatu} and assume that 
\be \label{e:amaro2}
       \lambda_k (\mf{\hat u})\ge 0, \quad s'_k \leq \bar s (\mf{\hat u}). 
\eq
Then 
\be \label{e:pallido2}
    | \bar s (\mf{\hat u})| \leq C_{12}   |s_j|. 
\eq
for a suitable constant $C_{12}>0$ only depending on the functions $\mf A$, $\mf E$, $\mf B$ and $\mf G$ and on the value $\mf u^\ast$. 
\end{lemma}
The heuristic meaning of the above lemma is the following. Owing to the analysis of case iii) in \S\ref{sss:casignl}, assumption~\eqref{e:amaro2} implies that there is a rarefaction wave front of the $k$-th family entering the domain with size $ \bar s (\mf{\hat u})$. Estimate~\eqref{e:pallido2} dictates that, no matter how large $\xi_k$ is, the size of this rarefaction wave front is controlled by the strength of the hitting wave. 
\begin{proof}[Proof of Lemma~\ref{l:nebbia2}]
We separately consider the following cases. \\
{\sc Case 1:} if $\lambda_k (\mf t_j (\mf u^+, s_j)) \leq 0$ then we can follow the same computations as in~\eqref{e:neve} and conclude that 
\begin{equation*} \begin{split}
    | \bar s (\mf{\hat u})| & \leq \unpo | \mf t_j (\mf u^+, s_j) - \mf{\hat u} | \leq \unpo | \mf t_j (\mf u^+, s_j) - \mf u^+ | + 
    \unpo | \mf u^+ - \mf{\hat u} |   \\ & \leq \unpo |s_j| +
    \unpo |\sum_{j=k+1}^N |s'_j|
    \stackrel{\eqref{e:nc}}{\leq} \unpo  |s_j|.
    \end{split}
\end{equation*}
{\sc Case 2:}  if $\lambda_k (\mf t_j (\mf u^+, s_j)) > 0$ then $\xi_k=0$ by~\eqref{e:pioggia2}, which owing to~\eqref{e:nc} yields 
\be \label{e:galaverna2}
     |s'_k| \leq C_1 |s_j|.
\eq
We recall that $ \lambda_k (\mf{\hat u})\ge 0$ by~\eqref{e:amaro2}, 
and by~\eqref{e:baresse} this implies $\bar s (\mf{\hat u}) \leq 0$, which combined with the second inequality in~\eqref{e:amaro2} yields
$$
    |\bar s (\mf{\hat u})| \leq  |s'_k| \stackrel{\eqref{e:galaverna2}}{\leq} C_1 |s_j|, 
$$
and this eventually concludes the proof of~\eqref{e:pallido2}.
\end{proof}
\subsection{Proof of Lemma~\ref{l:melindeg} (interaction estimate in the linearly degenerate case)}\label{ss:prooflindeg}
Lemma~\ref{l:melindeg} can be established by combining the proof of Lemma~\ref{l:me} in \S\ref{sss:pme} with Lemma~\ref{l:novembre} below, which can be viewed as the linearly degenerate equivalent of Lemma~\ref{l:proseguo}. We recall~\eqref{e:mcld}, the construction in \S\ref{ss:case2} and in particular that $\mf b_k (\mf{\widetilde u}, s)$ is the solution of the Cauchy problem~\eqref{e:cplindeg} evaluated at $\tau=s$. Next, we set   
\begin{equation} \label{e:gammak2}
   \boldsymbol{\gamma}_k (\mf{\widetilde u}, s): =
   \left\{
        \begin{array}{ll}
                   (\mf t_k  (\mf{\widetilde u}, s), 0)  & \lambda_k (\mf{\widetilde u}) \ge 0 \\
                    \mf b_k  (\mf{\widetilde u}, s)  & \lambda_k (\mf{\widetilde u}) < 0, \\
        \end{array}
        \right.
\end{equation}
\begin{lemma}
\label{l:novembre}
Assume that the $k$-th characteristic field is linearly degenerate~\eqref{e:lindeg} and let $\boldsymbol{\gamma}_k$ be the same as in~\eqref{e:gammak2};
then 
\begin{equation}
\label{e:novembre}
      |\boldsymbol{\gamma}_k \big( \mf t_k (\mf{\widetilde u}, s), \xi \big) - \boldsymbol{\gamma}_k (\mf{\widetilde u}, s + \xi) | \leq \unpo 
      |s| \big( |\xi| + [\lambda_k (\mf{\widetilde u}) ]^-). 
\end{equation}
\end{lemma}
\begin{proof}
We proceed according to the following steps. \\
{\sc Step 1:} we assume $\lambda_k  (\mf{\widetilde u}) \ge 0$.  This implies that  $\lambda_k ( \mf t_k (\mf{\widetilde u}, s)) \ge 0$ because the $k$-th characteristic field is linearly degenerate 
and by~\eqref{e:gammak2} the left-hand side of~\eqref{e:novembre} boils down to 
$$
    |(\mf t_k  \big( \mf t_k (\mf{\widetilde u}, s), \xi \big), 0) - (\mf t_k (\mf{\widetilde u}, s + \xi), 0) | \stackrel{\eqref{e:lax:lindeg}}{=}0,
$$
which implies that~\eqref{e:novembre} is trivially satisfied. \\
{\sc Step 2:} we assume $\lambda_k  (\mf{\widetilde u}) < 0$. Owing to~\eqref{e:mcld} in this case the left-hand side of~\eqref{e:novembre} boils down to 
$| \mf b_k  \big( \mf t_k (\mf{\widetilde u}, s), \xi \big) -  \mf b_k (\mf{\widetilde u}, s + \xi) |$. 
Note that by classical properties of the Cauchy problem the function $\mf b_k$ is of class $C^2$ with respect to both its arguments. We conclude that the function 
$$
   \mf h (s, \xi) : =  \mf b_k  \big( \mf t_k (\mf{\widetilde u}, s), \xi \big) -  \mf b_k (\mf{\widetilde u}, s + \xi)
$$
is of class $C^2$. Also, it satisfies the equalities 
\be \label{e:mfacca}
\mf h (0, \xi)= \mf 0_{N+1} \;  \text{for every $\xi$} \implies \frac{\partial \mf h}{\partial \xi} ( 0, \xi) = \mf 0_{N+1}
  \; \text{for every $\xi$},
\eq
which in turn yield
\begin{equation}\label{e:sei56}
\begin{split}
        |\boldsymbol{\gamma}_k & \big( \mf t_k (\mf{\widetilde u}, s), \xi \big) - \boldsymbol{\gamma}_k (\mf{\widetilde u}, s + \xi) | =
        |\mf h (s, \xi) |  = \left| \mf h(s, 0) + \xi \int_0^1  \frac{\partial \mf h}{\partial \xi} (s, \theta \xi) d \theta \right| \\ &
        \stackrel{\eqref{e:mfacca}}{\leq}
        |\mf h(s, 0) | + \left| \xi \int_0^1  \frac{\partial \mf h}{\partial \xi} ( s, \theta \xi) - \frac{\partial \mf h}{\partial \xi} (0, \theta \xi) d \theta \right|
        \leq  |\mf h(s, 0) |  + \unpo |s| |\xi|. 
\end{split}
\end{equation}
Assume that we have established the inequality 
\be \label{e:sei57}
      |\mf h(s, 0) | = | (\mf t_k (\mf{\widetilde u}, s), 0)-  \mf b_k (\mf{\widetilde u}, s) | \leq \unpo |s| |\lambda_k(\mf{\widetilde u}) |
\eq
then by plugging~\eqref{e:sei57} into~\eqref{e:sei56} and recalling that we are handling the case $\lambda_k(\mf{\widetilde u})<0$ we establish~\eqref{e:novembre}. \\
{\sc Step 3:} we establish~\eqref{e:sei57}. First, we recall~\eqref{e:cplindeg} and point out that 
\be \label{e:seitanto}
\begin{split}
   \frac{d |\mf u (\tau) - \mf t_k (\mf{\widetilde u}, \tau)|}{d \tau} &
   \stackrel{\eqref{e:lax:lindeg},\eqref{e:cplindeg}}{\leq} 
  | \mf r_c (\mf u, z_{00}, 0) - \mf r_k (\mf t_k (\mf{\widetilde u}, \cdot)) |\\&
   \stackrel{\eqref{e:errecl}}{\leq} \unpo |\mf u - \mf t_k (\mf{\widetilde u}, \cdot) | + \unpo |z_{00}| + \unpo |\lambda_k (\mf t_k (\mf{\widetilde u}, \cdot))|
   \\
   &
   \stackrel{\eqref{e:lindeg}}{=} \unpo |\mf u - \mf t_k (\mf{\widetilde u}, \cdot) | + \unpo |z_{00}| + \unpo |\lambda_k (\mf{\widetilde u})|
   \end{split}
   \eq
   Next, 
\be \label{e:seitanto2}
\begin{split}
   \frac{d |z_{00}|(\tau)|}{d \tau} &
   \stackrel{\eqref{e:cplindeg}}{\leq} 
  | \theta_{00} (\mf u, z_{00}, 0) |
   \stackrel{\eqref{e:errecl}}{\leq}  \unpo |z_{00}| + \unpo |\lambda_k (\mf u)| \\&
   \stackrel{\eqref{e:lindeg}}{\leq } 
    \unpo |z_{00}| + \unpo |\mf u - \mf t_k (\mf{\widetilde u}, \cdot) |+ 
   \unpo |\lambda_k (\mf{\widetilde u})|
   \end{split}
   \eq   
   and by combining~\eqref{e:seitanto} and~\eqref{e:seitanto2} with the Comparison Theorem for solutions of ODEs we arrive at 
  $$
       |\mf u (s) - \mf t_k (\mf{\widetilde u}, s)|+ |z_{00}(s)| \leq \unpo |\lambda_k (\mf{\widetilde u})| [\exp (\unpo |s|)-1] \leq \unpo |\lambda_k (\mf{\widetilde u})||s|,
  $$ 
   that is~\eqref{e:sei57}. 
\end{proof}

\section{Total variation bounds}\label{s:functional}
In this section we show a uniform in time and $\ee$ bound on the space total variation of the $\ee$-wave front-tracking approximation constructed in \S\ref{s:cinque}.
As it is by now fairly customary in the wave front-tracking business, the proof of the total variation bounds relies on the introduction of a Glimm-type functional. Note however that while our functional $\Upsilon$ contains some terms that have been already introduced in~\cite[p.136-7]{Bressan}, we also introduce completely new terms accounting for the interaction of wave fronts of the $k$-th (boundary characteristic) family with the boundary.
As a matter of fact, the explicit form of $\Upsilon$ is a consequence of the interaction estimates established in the previous section and is one of the main novelties of the present paper.

Our functional is well-defined and piecewise constant in $t$ as long as the total number of wave fronts is finite. Lemma~\ref{l:numberwf} in the next section states that this is actually the case for every $t$. In this section we show that, as long as it well-defined, $\Upsilon$ is a monotone non-increasing function with respect to time. More precisely, we proceed as follows:  we fix a time $t$ at which either two wave fronts cross each other or a wave front hits the boundary, or the boundary datum $\mf u_b^\ee$ has a jump discontinuity\footnote{By changing of an arbitrary small amount the speed of the wave fronts, we can assume (i) that at a given time at most two wave fronts can cross each other (no triple or multiple interactions); (ii) that at a given time at most one of the following events can occur: two wave fronts cross each other, a wave front hits the boundary and the boundary datum has a discontinuity. 
See also Remark 7.1 at page 132 and property (FT1) at page 142 in \cite{Bressan}}.  We assume that 
\be \label{e:assumption}
      \Upsilon (t^-): =\lim_{\tau \to t^-} \Upsilon(\tau) \leq \delta
\eq
and we show that if the constant $\delta>0$ is sufficiently small then $\Upsilon (t^+): =\lim_{\tau \to t^+} \Upsilon(\tau) \leq \Upsilon(t^-)$. We actually establish a more precise estimate and we exactly quantify the difference $\Upsilon(t^+)-\Upsilon (t^-)$, as we need this more precise information in the following sections. Note that showing that $\Upsilon $ is monotone non-increasing immediately yields a uniform control on the total variation owing to the estimate
\be \label{e:settedue}  
    \mathrm{Tot Var} \ \mf u_\ee (t, \cdot) \leq \unpo \Upsilon (t) \stackrel{\text{$\Upsilon$ monotone}}{\leq} \unpo \Upsilon (0) \leq C_{14} \delta^\ast
      \stackrel{\eqref{e:delta1}}{\leq}  \delta, 
\eq
where $\delta^\ast$  is the same as in~\eqref{e:hp} and $C_{14}>0$ is a suitable constant only depending on $\mf g$, $\mf f$, $\mf D$ and on the change of variables $\mf u$, and on the value $\mf u^\ast$ in~\eqref{e:uast}. The first and last inequalities in the above estimate directly follow from the explicit form of $\Upsilon$ given below, see~\eqref{e:upsilon}. 

The exposition in this section is organized as follows. In \S\ref{ss:deff} we define the functional $\Upsilon$. In \ref{ss:accuratem} we consider the case of a wave front of the family $j$, $j<k$, that hits the boundary, in \S\ref{ss:accuratek} we deal with a wave front of the $k$-th family hitting the boundary and in \S\ref{ss:jump} with a jump discontinuity of $\mf u_b^\ee$. In \S\ref{ss:inside} we discuss the interaction of two wave fronts inside the domain. 
In this section we use the notation $\Delta \Upsilon (t) : =\Upsilon(t^+) - \Upsilon(t^-)$. 
\begin{remark}
\label{r:lindeginter}
To ease the exposition, in the following  we carry on the explicit computations in the case of a genuinely nonlinear boundary characteristic field~\eqref{e:gnl} only. The case of a linearly degenerate boundary characteristic field~\eqref{e:lindeg} is analogous (if not simpler), the main difference is that in the definition~\eqref{e:S} of the function $S$ (involved in the definition of $\Upsilon$) one has to use the second line rather than the first, and also use estimate~\eqref{e:melindeg} instead of~\eqref{e:me}. 
\end{remark}
\subsection{Definition of the interaction functionals} \label{ss:deff}
 First, we introduce the following notation: consider a wave $\alpha$ connecting the right state $\mf u_\alpha^+$ with the left state $\mf t_j (\mf u^+_\alpha, s_\alpha)$, $j=1, \dots, N$. We term $\vartheta_\alpha$ the weighted signed strength of $\alpha$, namely 
\be \label{e:s}
    \vartheta_\alpha \footnote{Note that we are using a different weight than in~\cite{Amadori}} = 
    \left\{
    \begin{array}{ll}
   s_\alpha  & j_\alpha  \ge  k \\
    A \ s_\alpha  & j_\alpha <k, \\
    \end{array}
    \right.
\eq
where $A>1$ is a suitable constant, to be determined in what follows. If $\alpha$ is a non-physical front between the states $\mf u^-$ and $\mf u^+$, then we set $\vartheta_\alpha : = |\mf u^+ - \mf u^-|$. We define the interaction functional $\Upsilon$ by setting 
\be
\label{e:upsilon}
     \Upsilon (t) : = V(t) + K_1 Q(t) + K_1 R(t) + K_2 S(t) + K_3 Z(t), 
\eq
where $K_1$, $K_2$ and $K_3$ are suitable constants, to be determined in the following, and the quantities $Q(t), R(t)$, $S(t)$ and $Z(t)$ are defined as \be 
\label{e:V}
      V(t) : = |\xi_k| +\sum_{\text{waves}} |\vartheta_\alpha|,
\eq 
where $\xi_k$ is loosely speaking the center component of the boundary layer and is rigorously defined as in cases i),$\dots$,v)  in~\S\ref{ss:abrie}. Also, 
\be 
\label{e:Q}
     Q(t) : = \sum_{ \text{waves} \; \alpha, \beta} w_{\alpha \beta} |\vartheta_\alpha \, \vartheta_\beta|,
\eq
where the the weight $w_{\alpha \beta}=1$ if the wave fronts $\alpha$ and $\beta$ are approaching, and  $w_{\alpha \beta}=0$ otherwise.  
The definition of ``approaching waves" is the same as in the Cauchy problem, see~\cite[p.137]{Bressan}, but for the fact that two wave fronts of the $k$-th family are always considered approaching, even if they are both rarefaction wave fronts. We comment on this choice in Remark~\ref{r:approach} below. We also set 
\be \label{e:R}
    R(t) : = |\xi_k| \sum_{ \text{waves} } |\vartheta_\beta| w_{\beta}
\eq
where $w_\beta =1$ if $\beta$ belongs to the family $i_\beta \leq k$ and $w_\beta=0$ otherwise. 
 Also, 
\be 
\label{e:S}
    S(t) : = \left\{
    \begin{array}{ll}
     \displaystyle{\sum_{k-\text{waves} } |s_\alpha| \big[ \varsigma_\alpha (\mf u_\alpha, s_\alpha) \big]^- } & 
     \text{if~\eqref{e:gnl}} \\
      \displaystyle{\sum_{k-\text{waves} } |s_\alpha| \big[ \lambda_k (\mf u_\alpha) \big]^- }
      & 
     \text{if~\eqref{e:lindeg}} \\
      \end{array}
    \right.
\eq
In the previous expression, $ \varsigma_\alpha$ is defined as in~\eqref{e:speedwft}, $\mf u_\alpha$ is the right state of the wave front  and $\big[ \cdot \big]^- $ denotes the negative part. Finally, 
\be 
\label{e:Z} 
     Z(t) = \mathrm{TotVar}_{[t, + \infty[} \mf u_b^\nu. 
\eq
\subsection{A wave front of the $j$-th family hits the boundary, $j<k$}\label{ss:accuratem}
We assume that a wave front $\gamma$ with right state $\mf u^+$ and left state $\mf t_j(\mf u^+, s)$, $j<k$, hits the boundary at time $t$. In \S\ref{sss:accuratem1} we consider the case where we use the accurate boundary Riemann solver,  see~\S\ref{ss:abrie}, in \S\ref{sss:accuratem2} the case where we use the simplified boundary Riemann solver, see~\ref{ss:sbrie}. 
\begin{remark}\label{r:approach}
Very loosely speaking, the reason why we have to modify the definition of approaching wave fronts in~\cite[p.137]{Bressan} and say that two wave fronts of the $k$-th family are \emph{always} approaching (even if they are both rarefaction wave fronts) is the following. Owing to the particular form of our interaction estimate~\eqref{e:me} the center part of the boundary layer, which is parameterized by $\xi_k$, must be basically treated as a wave front of the $k$-th family, and this is accounted for in the term $R$ defined by~\eqref{e:R}. Note that the explicit form of $R$ takes into account that the center part of the boundary layer can interact with both rarefaction and shock wave fronts of the $k$-th family. This implies that rarefaction wave fronts of the $k$-th family can interact through the following mechanism: a rarefaction wave front can hit the boundary, get absorbed by the center part of the boundary layer and then interact with any other wave front of the $k$-th family. To take into account this mechanism, we have to add to the functional $Q$ the terms accounting for the interaction between rarefaction wave fronts of the $k$-th family.  From the technical viewpoint, the fact that wave fronts of the $k$-th family are always approaching yields the term $|\vartheta_\beta||s|$ at the first line of~\eqref{e:servequi}, which is necessary to establish~\eqref{e:c:dupsilon}. 
\end{remark}
\subsubsection{Accurate Riemann solver} \label{sss:accuratem1}
To evaluate $\Delta \Upsilon (t) = \Upsilon (t^+) - \Upsilon(t^-)$ we separately consider the behavior of $V$, $Q+R$, $S$ and $Z$. To this end, we first make a remark about notation.
\begin{remark}\label{r:notation}
As in the previous section, in the following, when writing estimates involving the accurate or simplified boundary Riemann solver, we always use the same notation as in~\eqref{e:brs}. This implies, in particular, that $s'_k$ does \emph{not} necessarily represent the size of the wave front of the $k$-th family (if any) entering the domain after the interaction. Instead, by separately considering cases i),\dots,v) in~\ref{ss:abrie} we realize that $s'_k = \xi'_k + \tilde s'_k$, where $\xi'_k$ is the same as in  cases i),\dots,v) and $\tilde s'_k$ is the  size of the wave front of the $k$-th family (if any) entering the domain. Note that either $\xi'_k$ or $\tilde s'_k$ identically vanish in some of cases i),\dots,v) in~\ref{ss:abrie}. Note furthermore that we always have 
\be \label{e:comprendetutto}
|s'_k| = |\xi'_k| +  |\tilde s'_k|.
\eq  
\end{remark}
To control $\Delta \Upsilon (t)$, first, note that 
\be \label{e:deltav}
   \Delta V = V(t^+) - V(t-)\stackrel{\eqref{e:s},\eqref{e:V},\eqref{e:comprendetutto}}{=} |s'_k | + \sum_{i=k+1}^{N} |s'_i| -  
    |\xi_k | - A|s| \stackrel{\eqref{e:nc}}{\leq}|s| (C_1-A) . 
\eq 
To discuss the behavior of $Q+R$  at the interaction we proceed as follows: first, we fix a wave front $\beta$, located at $x_\beta>0$, we consider the interaction potentials $Q_\beta$ and $R_\beta$ defined by setting 
\begin{equation} \label{e:qbrb}
    Q_\beta (t) : =  | \vartheta_\beta |\sum_{\alpha} w_{\alpha \beta } |\vartheta_\alpha|, 
    \qquad R_\beta (t) : = |\xi_k| |\vartheta_\beta| w_\beta.
\end{equation}
We now discuss the behavior of $Q_\beta + R_\beta$ at the interaction time $t$. 
Note that
\be \label{e:qbrbm}
     Q_\beta (t^-) + R_\beta (t^-) =w_{\gamma \beta} | \vartheta_\beta    |    |s|+| \vartheta_\beta |\sum_{\alpha \neq \gamma} w_{\alpha \beta } 
      |\vartheta_\alpha | + |\xi_k| |\vartheta_\beta| w_\beta \ge  |\vartheta_\beta |\sum_{\alpha \neq \gamma} w_{\alpha \beta } 
      |\vartheta_\alpha | + |\xi_k| |\vartheta_\beta| w_\beta
\eq
and that 
\be \label{e:qbrbp}
     Q_\beta (t^+) + R_\beta (t^+) \leq 
      | \vartheta_\beta |\sum_{\alpha \neq \gamma} w_{\alpha \beta } |\vartheta_\alpha| + 
    |\vartheta_\beta| \left[ \sum_{i=k+1}^N |s'_i| + |s'_k| w_\beta \right]. 
\eq
Note that the term $|\vartheta_\beta| |s'_k| w_\beta$ in the above expression can contribute to either $Q_\beta (t^+)$ (cases i), ii) and iv) in \S\ref{ss:abrie})
or $R_\beta (t^+)$ (case v) in \S\ref{ss:abrie}), or also to both in case iii) in \S\ref{ss:abrie}. In any case we have 

$$
    Q_\beta (t^+) + R_\beta (t^+) - Q_\beta (t^-) + R_\beta (t^-) \stackrel{\eqref{e:qbrbm},\eqref{e:qbrbp}}{\leq} 
     |\vartheta_\beta| \left[ \sum_{i=k+1}^N |s'_i| + ( |s'_k|- |\xi_k|) 
 w_\beta \right] \stackrel{\eqref{e:nc}}{\leq} 
   C_1 |\vartheta_\beta| |s|.
$$
Next, we sum up over $\beta$ and we point out that $R(t^+)$ may contain a term due to the interaction between the new boundary layer and the new $k$-th wave entering the domain (if any). This can happen only if we are in case iii) discussed in \S\ref{ss:abrie}. Note that in this case the strength of the new wave entering the domain is precisely $\bar s (\mf {\hat u})$, where $\mf{\hat u}$ is the same as in~\eqref{e:hatu}. Also, we have
\be \label{e:stimaxip}
    |\xi'_k| \leq |s'_k| \stackrel{\eqref{e:nc}}{\leq} |\xi_k| + C_1 |s | 
    \stackrel{\eqref{e:assumption}}{\leq} \delta + C_1 \delta 
  =[C_1 +1 ] \delta,   
\eq 
which owing to Lemma~\ref{l:nebbia2} yields  
\be \label{e:deltarq} \begin{split}
    \Delta Q (t) + \Delta R(t)& = Q(t^+) - Q(t^-) + R(t^+) - R(t^-) \stackrel{\eqref{e:pallido2},\eqref{e:stimaxip}}{\leq} C_1  |s| \sum_\beta |\vartheta_\beta| +C_{12} [1 + C_1] \delta |s| 
    \\ &  \stackrel{\eqref{e:V}}{\leq}  C_1  |s| V(t^-)  + C_{12} [1 + C_1] \delta |s|   \stackrel{\eqref{e:assumption},\eqref{e:upsilon}}{\leq}  
     [C_1+  C_{12}+ C_{12} C_1]  |s|\delta  . 
\end{split}
\eq
Since the quantities $S$ and $Z$ are constant at the interaction, we conclude that 
\be \label{e:upsgiu1}
\begin{split}
    \Delta \Upsilon (t)&\stackrel{\eqref{e:upsilon}}{=} \Delta V (t) + K_1 \Delta Q (t) + K_1 \Delta R (t) + K_2 \Delta S (t) \\
    &
    \stackrel{\eqref{e:deltav},\eqref{e:deltarq}}{\leq} 
    |s| \left[ C_1 + K_1 [C_1+  C_{12}+ C_{12} C_1] \delta  - A\right] \stackrel{\eqref{e:uno}}{\leq}-  |s|. 
    \end{split}
\eq
\subsubsection{Simplified Riemann solver} \label{sss:accuratem2}
We consider the case where the interaction is solved with the simplified boundary Riemann solver defined in \S\ref{ss:sbrie} and we show that $\Delta \Upsilon (t) \leq 0$. As in~\S\ref{ss:sbrie}, we term $\mf u^-$ and $\mf u^+$ the left and right state of the hitting wave front, respectively.  Note that 
\be \label{e:cisette}
     |\mf u^+ - \mf u^-| \leq C_7 |s|
\eq
for a suitable constant $C_7 >0$ only depending on the data  $\mf E$, $\mf A$, $\mf B$, $\mf G$, and on the value $\mf u^\ast$ in~\eqref{e:uast}. 
We recall the construction in \S\ref{ss:sbrie} and conclude in particular that after the interaction the only front entering the domain is a non-physical front of strength $ |\mf u^+ - \mf u^-| $. We have 
$$
     \Delta V(t) = - A|s| + |\mf u^+ - \mf u^-| \stackrel{\eqref{e:cisette}}{\leq}
     (-A + C_7) |s| 
$$
and 
$$
      \Delta Q(t) \leq |\mf u^+ - \mf u^-| \sum_{\text{waves}}|\vartheta_\alpha|
      \stackrel{\eqref{e:assumption},\eqref{e:cisette}}{\leq}
      C_7 \delta |s|, \qquad \Delta R(t) \leq - A|\xi_k||s|, \qquad \Delta S(t) = \Delta Z(t)=0
$$
and by recalling~\eqref{e:upsilon} this yields 
\be \label{e:np:upsgiu}
    \Delta \Upsilon (t) \leq   \big[-A + C_7 (1 + K_1  \delta) \big]|s | \stackrel{\eqref{e:uno}}{\leq}
   - |s| \leq 0. 
\eq

\subsection{A wave front of the $k$-th family hits the boundary} \label{ss:accuratek} 
We assume that a wave front $\gamma$ with right state $\mf u^+$ and left state $\mf t_k(\mf u^+, s)$ hits the boundary at time $t$. In \S\ref{sss:iabrs} we consider the case where we use the accurate boundary Riemann solver,  see~\S\ref{ss:abrie}, in \S\ref{ss:simplified} the case where we use the simplified boundary Riemann solver, see~\ref{ss:sbrie}. In this paragraph we use the shorthand notation $\varsigma_k$ for the quantity defined by~\eqref{e:speedwft}.
\subsubsection{Accurate boundary Riemann solver} \label{sss:iabrs}
We have 
\be \label{e:c:tv}
\begin{split}
    \Delta V (t)&= V(t^+) - V(t-)\stackrel{\eqref{e:V}}{=}  |s'_k | + \sum_{i=k+1}^{N} |s'_i|   
    - |\xi_k | - |s_k|  \stackrel{\eqref{e:me}}{\leq}
     C_2 |s| \Big( [\varsigma_k]^- + |\xi_k|  \Big). 
\end{split}
\eq     
Next, we proceed as in the previous paragraph: we fix a wave front $\beta$ of the $i_\beta$-th family located at $x_\beta>0$ and we consider the same functionals $Q_\beta$ and $R_\beta$ as in~\eqref{e:qbrb}. We first focus on the case $i_\beta > k$: since $w_\beta=0$ then $R_\beta(t^-) = R_\beta (t^+) =0$. Also, 
$$
   Q_\beta (t^-)  =   | \vartheta_\beta |\sum_{\alpha } w_{\alpha \beta } |\vartheta_\alpha|,\qquad
   Q_\beta (t^+)   \leq | \vartheta_\beta |\sum_{\alpha } w_{\alpha \beta } |\vartheta_\alpha|+    |\vartheta_\beta| \sum_{i=k+1}^{N} |s'_i|
    \quad \text{if} \; i_\beta > k, 
$$
which yields 
\be \label{e:dqdr1}
    \Delta Q_\beta (t) + \Delta R_\beta (t) \leq |\vartheta_\beta|  \sum_{i=k}^{N} |s'_i|  
    \stackrel{\eqref{e:me}}{\leq}  C_2 |\vartheta_\beta|  |s| \Big( [\varsigma_k]^- + |\xi_k| \Big). 
\eq
Next, we assume $i_\beta \leq k$, recall that according to our definition two wave fronts of the  $k$-th family are always approaching and arrive at 
 \be \label{e:servequi}
 \begin{split}
   & Q_\beta (t^-) + R_\beta (t^-) =  | \vartheta_\beta |\sum_{\alpha \neq \gamma} w_{\alpha \beta } |\vartheta_\alpha| +|\vartheta_\beta| |s| +  |\vartheta_\beta| |\xi_k|, \\  
   & Q_\beta (t^+) + R_\beta (t^+)  = | \vartheta_\beta |\sum_{\alpha \neq \gamma} w_{\alpha \beta } |\vartheta_\alpha|  +
    |\vartheta_\beta| \left( \sum_{i=k+1}^{N} |s'_i|  + |s'_k| \right) \qquad \text{if $i_\beta \leq k$}, 
\end{split}
   \end{equation}
and hence 
\be \label{e:dqdr2}
    \Delta Q_\beta (t) + \Delta R_\beta (t)  = |\vartheta_\beta| \left(  \sum_{i=k}^{N} |s'_i| + |s'_k |
    - |s| - |\xi_k| \right)  \stackrel{\eqref{e:me}}{\leq}   C_2 |\vartheta_\beta | |s|\Big( [\varsigma_k]^- +
     |\xi_k| \Big). 
\eq
We now point out that 
$$
   R(t^-) =  |\xi_k| |s| + \sum_\beta  R_\beta (t)
$$   
and we observe that the only case where after the interaction we have both 
$\xi'_k \neq 0$ and there is wave front entering the domain is if we are in case iii) in \S\ref{ss:abrie}. We can then apply Lemma~\ref{l:enterare} and conclude that the contribution of this specific term to the functional $R(t^+)$ is controlled by $C_8 |\xi'_k| |s| (|\xi_k| + [\varsigma_k]^-)$. By using estimate~\eqref{e:stimaxip} we then 
arrive at
\be \label{e:cqr} \begin{split}
    \Delta Q (t) & + \Delta R(t)  \stackrel{\eqref{e:rarepiccola},\eqref{e:stimaxip}}{=}  - |\xi_k| |s| + 2 C_8 \delta |s| (|\xi_k| + [\varsigma_k]^-) +   \sum_\beta   \Delta Q_\beta (t)+ \Delta R_\beta(t) \\ &
     \stackrel{\eqref{e:dqdr1},\eqref{e:dqdr2}}{\leq} - |\xi_k| |s| +
        |s|\Big( [\varsigma_k]^- +
     |\xi_k| \Big) \left( 2 C_8 \delta + C_2 \sum_\beta |\vartheta_\beta | \right)\\
& \stackrel{\eqref{e:upsilon},\eqref{e:assumption}}{\leq}
    - |\xi_k| |s|  + [2 C_8+ C_2]  |s|\Big( [\varsigma_k]^- +
     |\xi_k| \Big) \delta =[2 C_8+ C_2]  \delta |s| [\varsigma_k]^- + |s| |\xi_k| ([2 C_8+ C_2]  \delta -1). 
\end{split}
\eq 
Next, we point out that 
\be \label{e:cds}
        S(t^-) =   \sum_{k-\text{wave} \neq \gamma} |s_\alpha| \big[ \varsigma_\alpha \big]^- + 
        |s| [\varsigma_k]^- ,\quad 
       S(t^+) =   \sum_{k-\text{wave} \neq \gamma } |s_\alpha| \big[ \varsigma_\alpha \big]^- 
       \implies \Delta S(t) = -  |s| \big[ \varsigma_k \big]^-
\eq
which yields (since the quantity $Z$ defined by~\eqref{e:Z} is  constant at the interaction)
\be \label{e:c:dupsilon}
\begin{split}
    \Delta \Upsilon(t)  & \stackrel{\eqref{e:upsilon}}{=} \Delta V(t) + K_1 \Delta Q(t) +
   K_1 \Delta R(t) + K_2 \Delta S(t) \\ & 
   \stackrel{\eqref{e:c:tv},\eqref{e:cqr},\eqref{e:cds}}{\leq} 
    |s| [\varsigma_k]^-  [C_2 + K_1 [2 C_8+ C_2]  \delta - K_2]   + |s||\xi_k|   [C_2 + K_1 ( [2 C_8+ C_2]  \delta -1) ] \\
    &\stackrel{\eqref{e:due}}{\leq}  - |s|([\varsigma_k]^- + |\xi_k|)      . 
\end{split}
\eq
\subsubsection{Simplified boundary Riemann solver} \label{ss:simplified}
We recall the construction in \S\ref{ss:sbrie} and the definition~\eqref{e:hatu} of $\mf{\hat u}$ and by increasing, if needed, the value of $C_7$ we obtain 
\be \label{e:cisette2}
     |\mf u^+ - \mf{\hat u}| \leq C_7 \sum_{i=k+1}^N |s'_i|
\eq
We can then repeat the same computations as in \S\ref{ss:accuratek} replacing $\sum_{i=k+1}^N |s'_i|$ with $ C_7 \sum_{i=k+1}^N |s_i|$. By~\eqref{e:cisette3}, we obtain $\Delta \Upsilon (t) \leq 0.$
\subsection{The boundary datum $\mf u_b^\ee$ has a jump discontinuity} \label{ss:jump}
Assume that the piecewise constant function $\mf u^\ee_b$ experiences a jump discontinuity at $t$. We set 
$$
    \mf u_b^- : = \lim_{s \to t^-} \mf u_b^\ee (s), \quad  \mf u_b^+ : = \lim_{s \to t^+} \mf u_b^\ee (s),
$$
which yields 
\be \label{e:jdz}
      \Delta Z(t) = Z(t^+) - Z(t^-) \stackrel{\eqref{e:Z}}{=} - |\mf u_b^+ - \mf u_b^-|. 
\eq
We also set $\mf{\bar u_\ast}: = \lim_{x\to 0^+} \mf u^\ee(t, x)$, where $\mf u^\ee$ is the $\ee$-wave front-tracking approximation. 
We recall the analysis in\S\ref{ss:abrie} and for simplicity assume~\eqref{e:a11nz} (the analysis in {\sc Cases B)} and {\sc C)} in Hypothesis~\ref{h:eulerlag} is analogous). We conclude that 
\be \label{e:freddo}
     \boldsymbol{\beta} (\boldsymbol{\phi} (\mf{\bar u_\ast}, \xi_{\ell+1}, \dots, \xi_k), \mf u^-_b) = \mf 0_{N- \ell}, \quad 
    \boldsymbol{\beta} ( \boldsymbol{\phi} (\cdot, \xi'_{\ell+1}, \dots, \xi'_{k-1}, s'_k)  \circ \mf t_{k+1} (\cdot, s'_{k+1})
      \dots \circ \mf t_{N} (\mf{\bar u}_\ast, s'_{N}),  \mf u^+_b) = \mf 0_{N- \ell}.
\eq
By the Lipschitz continuity of the implicit  function we conclude that 
\be \label{e:jd1}
      \sum_{i=\ell + 1}^{k-1} |\xi'_i - \xi_i| + |s'_k - \xi_k|+\sum_{j=k+1}^N |s'_j| \leq C_6 |\mf u_b^+ - \mf u_b^-|
\eq
for a suitable constant $C_6>0$. 
This yields 
\be \label{e:jdv}
    \Delta V(t) = V(t^+) - V(t^-) \stackrel{\eqref{e:V},\eqref{e:jd1}}{\leq}  C_6 |\mf u_b^+ - \mf u_b^-|. 
\eq
In the sequel we need the following lemma. 
\begin{lemma}
\label{l:nebbia}
Under~\eqref{e:freddo}, assume that
\be \label{e:pioggia}
       \xi_k  \leq \underline s (\mf{\bar u}_\ast) \; \text{if} \; \lambda_k  (\mf{\bar u}_\ast )  \leq 0, 
       \qquad 
       \xi_k =0  \; \text{if} \; \lambda_k  (\mf{\bar u}_\ast)  > 0 .
\eq
Let $\mf{\hat u}$ be as in~\eqref{e:hatu} with $\mf{\bar u_\ast}$ in place of $\mf u^+$ and assume that 
\be \label{e:amaro}
       \lambda_k (\mf{\hat u})\ge 0, \quad s'_k < \bar s (\mf{\hat u}). 
\eq
Then 
\be \label{e:pallido}
    | \bar s (\mf{\hat u})| \leq C_{11}   |\mf u_b^+ - \mf u_b^-|. 
\eq
for a suitable constant $C_{11}>0$. 
\end{lemma}
\begin{proof}
We can follow the proof of Lemma~\ref{l:nebbia2} and replace $\mf t_j (\mf u^+, s_j)$ with $\mf{\bar u}$, and $s_j$ with $|\mf u_b^+ - \mf u_b^-|$. The details are omitted. 
\end{proof}

The reason why we need the above lemma is the following: we want to evaluate $\Delta Q(t) + \Delta R(t)$. To this end, we recall the analysis in \S\ref{ss:abrie} and we point out that after the interaction the term $R$ may contain a contribution in the form $|\xi'_k||\bar s (\mf{\hat u})|$, where $\xi'_k$ is defined as in \S\ref{ss:abrie} and $\mf{\hat u}$ is as in~\eqref{e:hatu}, but only if we are in case iii) in \S\ref{ss:abrie}. We recall that in this case $\xi'_k = s'_k - \bar s (\mf{\hat u})$. By applying Lemma~\ref{l:nebbia}  we conclude that 
\begin{equation}\label{e:brina}
\begin{split}
   |\xi'_k||\bar s (\mf{\hat u})| & = | s'_k - \bar s (\mf{\hat u})| |\bar s (\mf{\hat u})| \stackrel{\bar s (\mf{\hat u}) <0 \; \text{in case iii)}}{\leq} |s'_k | |\bar s (\mf{\hat u})|
   \stackrel{\eqref{e:jd1},\eqref{e:pallido}}{\leq} [ |\xi_k| + C_6 | |\mf u_b^+ - \mf u_b^-|]  C_{11}  |\mf u_b^+ - \mf u_b^-|\\
   & \stackrel{\eqref{e:assumption}}{\leq}
   [1+ C_6] C_{11} \delta  |\mf u_b^+ - \mf u_b^-|. 
   \end{split}
\end{equation}
Wrapping up, we have 
\be \label{e:jdqr}\begin{split}
     \Delta Q(t) + \Delta R(t) & \stackrel{\eqref{e:Q},\eqref{e:R}}{\leq} 
     \left( \sum_{j=k+1}^N |s'_j|  \right)
     \sum_{\beta} |\vartheta_\beta| + |s'_k|  \sum_{\beta}w_\beta |\vartheta_\beta|
     + |\xi'_k||\bar s (\mf{\hat u})| - |\xi_k|  \sum_{\beta}w_\beta |\vartheta_\beta| \\& 
     \stackrel{\eqref{e:jd1},\eqref{e:brina}}{\leq}
     2 C_6 |\mf u_b^+ - \mf u_b^-|  \sum_{\beta} |\vartheta_\beta| +  [1+ C_6] C_{11} \delta  |\mf u_b^+ - \mf u_b^-|\\ &
     \stackrel{\eqref{e:assumption}}{\leq}
   \big( 2 C_6 + [1+ C_6] C_{11} \big) \delta  |\mf u_b^+ - \mf u_b^-|  
\end{split}
\eq
Since $\Delta S (t) =0$ (because new wave fronts of the $k$-th family  that enter the domain, if any, have positive speed) 
we conclude that 
\be \label{e:jdupsilon}\begin{split}
    \Delta \Upsilon (t) & \stackrel{\eqref{e:upsilon}}{=} \Delta V (t) + K_1\Delta Q(t) + K_1 \Delta R(t)+ K_3 \Delta Z(t)
      \\ & \stackrel{\eqref{e:jdz},\eqref{e:jdv},\eqref{e:jdqr}}{=}
    |\mf u_b^+ - \mf u_b^-| \Big( C_6 +  K_1 \delta  \big( 2 C_6 + [1+ C_6] C_{11} \big) 
   - K_3  \Big)   \stackrel{\eqref{e:4}}{\leq}   -  |\mf u_b^+ - \mf u_b^-|. 
\end{split}
\eq
\subsection{Two wave fronts interacts inside the domain}
\label{ss:inside}
In this paragraph we assume that at time $t$ two wave fronts $\alpha$ and a $\beta$ intersect. We assume that $\alpha$ and $\beta$ belong to the family 
$i_\alpha$ and $i_\beta$, respectively, and with a slight abuse of notation we say that if $\alpha$ is a non-physical wave front then $i_\alpha=N+1$. We also term $\mf u^+_\alpha$ and $\mf u^+_\beta$ the right state of $\alpha$ and $\beta$, respectively, and by $\mf t_{i_\alpha} (\mf u_\alpha^+, s_\alpha)$ and $\mf t_{i_\beta} (\mf u_\beta^+, s_\beta)$ the left state. If $\alpha$ is a non-physical wave front we set $s_\alpha: = |\mf u^+_\alpha - \mf u^-_\alpha|$, where $\mf u^-_\alpha$ is the left state. We also term $s'_\alpha$ and $s'_\beta$ the sizes of the $i_\alpha$ and $i_\beta$ wave front (if any) outgoing the interaction. 

The exposition is organized as follows: in \S\ref{sss:751} we give some preliminary results, in \S\ref{sss:deltaupsid} we assume that the interaction is solved by the accurate Riemann solver, see~\S\ref{ss:arie}, in \S\ref{sss:753} we assume that the interaction is solved by the simplified Riemann solver, see~\S\ref{ss:srie}, and finally in \S\ref{sss:754} we discuss the interaction between a physical and a non-physical wave front. 
\subsubsection{Preliminary results}\label{sss:751}
We first recall~\cite[eq (7.31),(732) p. 133]{Bressan} and~\cite[eq (7.33),(7.34) p. 134]{Bressan} and conclude\footnote{Note that, strictly speaking, estimates (7.31)-(7.34) in~\cite{Bressan} involve the wave strengths along the admissible wave fan curves of \emph{right} states $\boldsymbol{m}_i$, rather than the admissible wave fan curves of \emph{left} states $\mf t_i$. See also footnote~\footref{foot} at page~\pageref{foot} and Remark~\ref{r:destrasinistra} at page~\pageref{r:destrasinistra}. However, owing to the relation~\eqref{e:avantindietro}, these estimates  straightforwardly extend to the strengths of the waves along the curves $\mf t_i$, $i=1, \dots, N$, that is the strengths involved in the definition of the functionals in \S\ref{ss:deff}. 
} that there is a constant $C_5>0$ only depending on $\mf E$, $\mf A$ and $\mf u^\ast$ such that
\be \label{e:iebressan}
     |s'_\alpha - s_\alpha| +   |s'_\beta - s_\beta | + \sum_{j \neq i_\alpha, i_\beta} |s'_j| \leq C_5 |s_\alpha s_\beta| \quad \text{for $i_\alpha, i_\beta \in \{ 1, \dots, N+1 \}$, $i_\alpha \neq i_\beta$} 
\eq
and 
\be \label{e:iebressan2}
     |s'_\alpha - (s_\alpha + s_\beta)| +   \sum_{j \neq i_\alpha} |s'_j| \leq C_5 |s_\alpha s_\beta| \quad \text{for $i_\alpha, i_\beta \in \{ 1, \dots, N+1 \}$, $i_\alpha= i_\beta$}. 
\eq
Next, we control the change in the quantity defined by~\eqref{e:speedwft} at the interaction. 
\begin{lemma}
\label{l:sigmaba} 
Assume $i_\alpha =k$ and let $\mf u_\alpha$ and $\mf u'_\alpha$ denote the right state of the wave front of family $i_\alpha$ before and after the interaction, respectively. Then there are constants $\nu>0$ and $C_3>0$  only depending on $\mf E$, $\mf A$ and $\mf u^\ast$ such that, if $|s_\alpha| \leq \nu$ and $|s_\beta| \leq \nu$ then  
\be 
\label{e:dadim3}
      | - \varsigma_k (\mf u_\alpha, s_\alpha) + \varsigma_k (\mf u'_\alpha, s'_\alpha)  | \leq C_3 |s_\beta|,
\eq
provided $\varsigma_k $ is the same as in~\eqref{e:speedwft}. 
\end{lemma} 
\begin{proof}[Proof of Lemma~\ref{l:sigmaba}]
It suffices to point out that the function $\varsigma_k $ defined by~\eqref{e:speedwft} is Lipschitz continuous with respect to both $\mf{\widetilde u}$ and $s_k$, and that owing to~\eqref{e:iebressan} and~\eqref{e:iebressan2}
$$
    |\mf u_\alpha - \mf u'_\alpha| \leq \unpo |s_\beta| + \unpo |s_\alpha s_\beta| \stackrel{\eqref{e:assumption}}{\leq}
    \unpo |s_\beta| (1 + \delta). 
$$
\end{proof}
Note that an elementary computation shows that~\eqref{e:dadim3} implies 
\be \label{e:dadim3bis}
     \Big|[ \varsigma_k (\mf u_\alpha, s_\alpha) ]^- -[ \varsigma_k (\mf u'_\alpha, s'_\alpha)]^-   \Big| \leq C_3 |s_\beta|.
\eq  
We now recall~\eqref{e:hp},\eqref{e:uast} and~\eqref{e:assumption}, and point out that by our construction of the $\ee$-wave front-tracking approximation $\mf u_\ee$ we have $\lim_{x \to + \infty} \mf u_\ee (t, x) =\lim_{x \to + \infty} \mf u_0 (x)$  for every $t \ge 0$. We conclude that if $\delta^\ast \leq \delta$ then 
there is a constant $C_4$ only depending on $\mf E$, $\mf A$ and $\mf u^\ast$ such that
\be \label{e:msigma}
      \sup_{\text{$k$-wave}} {[\varsigma_k (\mf u_\alpha, s_\alpha)]^-} \leq 
       \sup_{\text{$k$-wave}} {|\varsigma_k (\mf u_\alpha, s_\alpha)|} \leq  C_4 \delta. 
\eq 
\subsubsection{Accurate Riemann solver} \label{sss:deltaupsid}
We consider the case that the interaction between the wave fronts $\alpha$ and $\beta$ is solved by the accurate Riemann solver. 
By applying either~\eqref{e:iebressan} or~\eqref{e:iebressan2}  we conclude that 
\be \label{e:idv}
    \Delta V(t) \stackrel{\eqref{e:V}}{\leq} A C_5 |s_\alpha s_\beta|
\eq
To control $\Delta Q(t)$ we basically proceed as in~\cite[p.137]{Bressan}, the only difference is that our definition of approaching waves is slightly different. More precisely, according to our definition two wave fronts of the $k$-th family are always approaching, even if they are both rarefaction wave fronts. 
We recall that rarefaction fronts of the same family of incoming wave-fronts are never partitioned (see property~\cite[(P)]{Bressan}) and we realize that the difference with the analysis in~\cite[p.137]{Bressan} occurs when $i_\alpha \neq k \neq i_\beta$ and some rarefaction wave fronts of the $k$-th family are generated at the interaction. Let us denote by $m$ the number of the 
outgoing rarefaction wave fronts of the $k$-th family, then the increment of $Q$ due to the interaction among these waves can be controlled by
$$
    \underbrace{\left( \frac{A C_5  |s_\alpha s_\beta|}{m} \right)}_{\text{maximal strength of the single wave front}} \! \! \! \! \! \! \! \! \! \! \! \! \! \! \! \! \! \! \! \! \! \! \! \! \! 
\times \left( \frac{AC_5  |s_\alpha s_\beta|}{m} \right) 
\! \! \! \! \! \! \! \! \! \! \! \! \! \! \! \! \! \! \! \! \! 
\underbrace{ \sum_{h=1}^{m-1} h}_{\text{number of possible interactions}}
\! \! \! \! \! \! \! \! \! \! \! \! \! \! \! \! \! \! \! \! \! \! \! \! \! 
 \leq \frac{A^2 C_5^2 |s_\alpha s_\beta|^2}{m^2} \frac{m (m-1)}{2}  \! \! 
   \stackrel{\eqref{e:s},\eqref{e:V},\eqref{e:upsilon},\eqref{e:assumption}}{\leq} \! \!
    A^2 C^2_5 \delta^2  |s_\alpha s_\beta|. 
$$
By recalling~\eqref{e:iebressan} and~\eqref{e:iebressan2}  this yields 
\begin{equation}
\label{e:idq}
     \Delta Q (t)\leq - |\vartheta_\alpha  \vartheta_\beta| + A C_5 |s_\alpha s_\beta| \sum_{\omega \neq \alpha, \beta} |\vartheta_\omega| + A^2
      C^2_5 \delta^2  |s_\alpha s_\beta| \! \! 
     \stackrel{\eqref{e:assumption},\eqref{e:upsilon},\eqref{e:V}}{\leq}  \! \! - |\vartheta_\alpha  \vartheta_\beta| + A C_5 \delta ( 1 + \delta AC_5)  |s_\alpha s_\beta| 
\end{equation}
and that
\be \label{e:deltar}
      \Delta R(t) \leq A |\xi_k| C_5 |s_\alpha s_\beta|  \stackrel{\eqref{e:assumption},\eqref{e:upsilon},\eqref{e:V}}{\leq} A  C_5 \delta |s_\alpha s_\beta|
       \eq
We now control $\Delta S(t):$ if $i_\alpha \neq k$ and $i_\beta \neq k$, then we combine~\eqref{e:msigma} with either~\eqref{e:iebressan} or~\eqref{e:iebressan2} and conclude that 
\be \label{e:deltas1}
      \Delta S (t)\leq C_4 C_5 \delta |s_\alpha s_\beta| \qquad \text{if} \; i_\alpha \neq k, \; i_\beta \neq k. 
\eq
If $i_\alpha =k$ and $i_\beta \neq k$, then 
\be \label{e:deltas2}
\begin{split}
    &\Delta S (t) \! \! = \! \! |s'_\alpha| [\varsigma_k (\mf u'_\alpha, s'_\alpha)]^-\! \! - \! \! |s_\alpha| [\varsigma_k (\mf u_\alpha, s_\alpha)]^-\! \!\! =\! \!
    [\varsigma_k (\mf u'_\alpha, s'_\alpha)]^-\big(|s'_\alpha| \! \! - \! \!|s_\alpha|  \big) \! + \! |s_\alpha| \big(   [\varsigma_k (\mf u'_\alpha, s'_\alpha)]^- -  [\varsigma_k (\mf u_\alpha, s_\alpha)]^-\big) \\
    & \! \!\! \!\stackrel{\eqref{e:msigma},\eqref{e:iebressan}}{\leq} \! \!\! \!C_4 C_5 \delta |s_\alpha| |s_\beta| 
    + |s_\alpha| \big(   [\varsigma_k (\mf u'_\alpha, s'_\alpha)]^- -  [\varsigma_k (\mf u_\alpha, s_\alpha)]^-\big)\! \! \stackrel{\eqref{e:dadim3bis}}{\leq}\! \!
    C_4 C_5 \delta |s_\alpha s_\beta| 
 + C_3 |s_\alpha s_\beta| 
\end{split}
\eq
If $i_\alpha =k$ and $i_\beta = k$, then
\be \label{e:deltas3}
\begin{split}
    \Delta S (t)& = |s'_\alpha|    [\varsigma_k (\mf u'_\alpha, s'_\alpha)]^-  - |s_\alpha|  [\varsigma_k (\mf u_\alpha, s_\alpha)]^-- |s_\beta| [\varsigma_k (\mf u_\beta, s_\beta)]^-\\ &=
    [\varsigma_k (\mf u'_\alpha, s'_\alpha)]^- \big(|s'_\alpha| -|s_\alpha| - |s_\beta| \big) +
      \big( |s_\alpha| + |s_\beta| \big)   [\varsigma_k (\mf u'_\alpha, s'_\alpha)]^- 
    - |s_\alpha|  [\varsigma_k (\mf u_\alpha, s_\alpha)]^-- |s_\beta| [\varsigma_k (\mf u_\beta, s_\beta)]^- \\
    & \stackrel{\eqref{e:msigma},\eqref{e:iebressan2}}{\leq}C_4 C_5 \delta |s_\alpha| |s_\beta| 
    + |s_\alpha|  \big(   [\varsigma_k (\mf u'_\alpha, s'_\alpha)]^- -  [\varsigma_k (\mf u_\alpha, s_\alpha)]^-\big)+  
    |s_\beta|  \big(   [\varsigma_k (\mf u'_\alpha, s'_\alpha)]^- -  [\varsigma_k (\mf u_\beta, s_\beta)]^-\big)\\
    & \stackrel{\eqref{e:dadim3bis}}{\leq}
    C_4 C_5 \delta |s_\alpha s_\beta| 
 + 2 C_3 |s_\alpha s_\beta|
\end{split}
\eq
Since $\Delta Z(t)=0$ we conclude that 
 \begin{equation} \label{e:upsgiu2} \begin{split}
    \Delta \Upsilon (t)& \stackrel{\eqref{e:upsilon}}{=}
    \Delta V(t) + K_1 \Delta Q (t) + K_1 \Delta R(t) + K_2 \Delta S(t) \\ &
   \stackrel{\eqref{e:idv},\eqref{e:idq},\eqref{e:deltar}}{\leq}
   |s_\alpha s_\beta| \Big[ A C_5 + \delta K_1 AC_5 (2  + \delta  A C_5)\Big] - 
   K_1 |\vartheta_\alpha \vartheta_\beta |+ K_2 \Delta S(t) \\ &
    \stackrel{\eqref{e:deltas1},\eqref{e:deltas2},\eqref{e:deltas3}}{\leq}
   |s_\alpha s_\beta| \Big[ A C_5 + \delta K_1 A C_5 (2  + \delta A C_5)  + K_2 C_4 C_5 \delta 
 + 2 K_2 C_3  \Big] - 
   K_1 |\vartheta_\alpha \vartheta_\beta |\\& 
 \stackrel{\eqref{e:s}, \ A>1}{\leq} 
   |s_\alpha s_\beta| \Big( A C_5   - K_1  + 2 K_1 A \delta C_5 + K_1 A^2C_5^2 \delta^2+ 
    K_2   C_4 C_5 \delta  + 2 K_2 C_3 \Big) \stackrel{\eqref{e:tre}}{\leq} - |s_\alpha s_\beta|.
    \end{split}
\end{equation}
\subsubsection{Simplified Riemann solver}\label{sss:753}
The analysis is analogous to the one in \S\ref{sss:deltaupsid}, but simpler. We have $\Delta V (t) \leq C_5 |s_\alpha s_\beta|$, $\Delta Q(t) \leq - |s_\alpha s_\beta| +
AC_5 \delta |s_\alpha s_\beta|$, $\Delta R (t)=0$. We also have~\eqref{e:deltas1},\eqref{e:deltas2} and~\eqref{e:deltas3} and by following~\eqref{e:upsgiu2} we arrive at $\Delta \Upsilon(t) \leq  - |s_\alpha s_\beta|$.  
\subsubsection{Interaction between a physical and a non-physical front}\label{sss:754}
In this case we have $\Delta V (t) \leq C_5 |s_\alpha s_\beta|$, $\Delta Q (t)= - |s_\alpha s_\beta|$, $\Delta R(t) =0$. We also have either $\Delta S(t)=0$ (if the physical front is not a $k$-wave) or \eqref{e:deltas2} (if the physical front is a $k$-wave) and by following~\eqref{e:upsgiu2} we arrive at $\Delta \Upsilon (t)\leq  - |s_\alpha s_\beta|$.   
\section{Convergence proof}\label{s:otto}
In this section we establish~\eqref{e:16bis} for some function $\mf v$ such that $\mf v \in BV(]0, T[ \times \R_+)$ for every $T>0$. What we actually show is that, for every vanishing 
sequence $\{ \ee_m \}_{m \in \mathbb N }$ there is a subsequence (which to ease the notation we do not relabel) such that 
\be \label{e:81}
    \mf u_{\ee_m} (t, \cdot) \to \mf u (t, \cdot)  \quad \text{in $L^1_{\loc} (\R_+)$ as $m \to + \infty$ for every $t>0$},
\eq
where $\mf u_{\ee_m}$ is the wave front tracking approximation constructed as in~\S\ref{s:cinque} and $\mf u$ is a suitable limit function. To obtain~\eqref{e:16bis} we then simply apply the change of variables $\mf u \leftrightarrow \mf v$. The main steps in the proof of~\eqref{e:81} are the same as in the proof of~\cite[Theorem 7.1]{Bressan}, so we do not repeat every detail, but only focus on the parts where the proof is different. 
More precisely, we establish Lemmas~\ref{l:numberwf},~\ref{l:raref},~\ref{l:np}, which heuristically speaking state that the total number of wave fronts is finite, that the strength of each rarefaction front is small and that the total strength of non-physical wave fronts is small, respectively. 

Note that the analysis in~\cite[Chapters 7,10]{Bressan} yields that $\mf v$ being the limit of the wave front-tracking approximation enjoys several further properties: in particular it has the regularity properties dictated by~\cite[Theorem 10.4]{Bressan} and if the system~\eqref{e:claw} has a convex entropy then $\mf v$ is entropy admissible. Conversely, showing that $\mf v$ satisfies the boundary condition $\mf v_b$ in the sense of Definition~\ref{d:equiv} requires quite some new analysis, which is provided in section~\S\ref{s:bc}.  
\begin{lemma}\label{l:numberwf} Let $\mf u_\ee$ be the wave front-tracking approximation constructed in \S\ref{s:cinque}, then the total number of wave fronts and interactions is finite and bounded by a constant that only depends on $N$, on the number of discontinuities in the initial and boundary conditions $\mf u^\ee_0$ and $\mf u^\ee_b$, and on the threshold parameter $\omega_\ee$ in \S\ref{ss:whenas}. 
\end{lemma}
\begin{proof}
We recall the rules stated in \S\ref{ss:whenas} and we proceed according to the following steps.\\
{\sc Step 1:} we show that the total number of physical fronts is finite. The number of physical fronts created at $t=0$ is finite because so is the number of discontinuities of the approximate initial datum. By also recalling the analysis in \S\ref{ss:arie}, \S\ref{ss:abrie}, \S~\ref{ss:srie}, \S~\ref{ss:sbrie} we conclude that the number of physical fronts can only increase
\begin{itemize}
\item at times where the approximate boundary datum has a discontinuity. These times are finitely many by construction of the approximate boundary datum;
\item at times where the accurate Riemann solver is used to solve interactions inside the domain. By using the analysis in \S\ref{ss:inside} 
and in particular~\eqref{e:upsgiu2} we conclude that any time this happens, we have $\Delta \Upsilon \leq - \omega_\ee$. Since $0 \leq \Upsilon (t)\leq \Upsilon (0)$, this can happen at most finitely many times;
\item at times where the accurate boundary Riemann solver is used to solve boundary Riemann problems at the domain boundary. Owing to~\eqref{e:upsgiu1} and~\eqref{e:c:dupsilon}, any times this happens we have $\Delta \Upsilon \leq - \omega_\ee$ and hence this can happen at most finitely many times;
\item the number of physical fronts does not increase at times where the simplified Riemann or boundary Riemann solver is used. 
\end{itemize}
{\sc Step 2:} we show that the total number of non-physical fronts is finite, which is the most delicate point. 
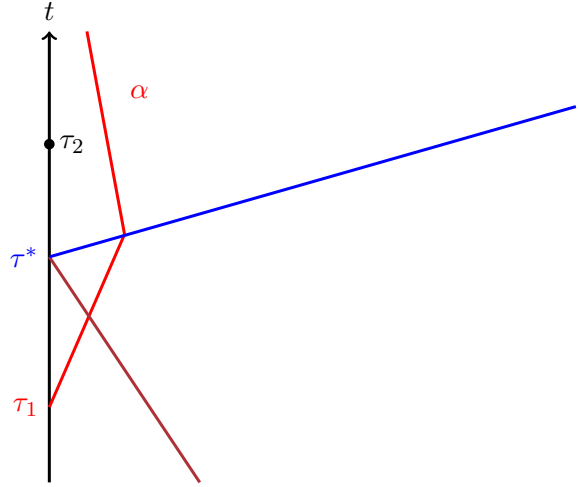
\begin{figure}
\begin{center}
\caption{The argument in {\sc Step 2} of the proof of Lemma~\ref{l:numberwf}} 
\label{f:finite}
\begin{tikzpicture}
\draw[line width=0.4mm,->] (0,1) -- (0, 7) node[anchor=south] {$t$};
\draw[line width=0.4mm,red] (1,4.3) -- (0, 2)  node[anchor=east] {$\tau_1$};
\draw[line width=0.4mm,red] (1,4.3) -- (0.5, 7) ;
\draw[red]  (1.2,6.2) node {$\alpha$};
\draw[line width=0.4mm,Maroon] (2, 1) -- (0, 4);
\draw[line width=0.4mm,blue] (0, 4) node[anchor=east] {$\tau^\ast$} -- (7, 6);
\fill (0,5.5) circle (2pt) node[anchor=west] {$\tau_2$};
\end{tikzpicture}\end{center}
\end{figure} Non-physical fronts are not generated at time $t=0$. By the analysis in  \S\ref{ss:srie}, \S\ref{ss:sbrie} we conclude that  the number of non-physical fronts can only increase 
\begin{itemize}
\item at times where an interaction between two physical fronts occurs inside the domain. Since two given physical fronts can interact only once, then by the analysis at {\sc Step 1} this can happen at most finitely many times.  
\item at times where a physical front hits the boundary. Wave fronts of the family $j<k$ can only hit the boundary once and hence this kind of interactions can occur at most finitely many times. However, it may in principle happen that a single wave front of the $k$-th family  is bounced back and forth between the boundary and other waves inside the domain infinitely many times.  To rule out this possibility, fix a given wave front $\alpha$ of the $k$-th family which leaves the boundary (with positive speed) at $t =\tau_1$. This wave front is marked in red in Figure~\ref{f:finite}. Assume that $\alpha$ does not cross any physical front on the time interval $]\tau_1, \tau_2[$; we now show that it cannot cross non-physical fronts either. Indeed, assume by contradiction that is does. Since the speed of a non-physical front (marked in blue in Figure~\ref{f:finite}) is larger than the speed of $\alpha$, the non-physical front must have been created at a time $\tau_\ast > \tau_1$ when a physical wave front (marked in maroon in Figure~\ref{f:finite}) had hit the boundary, and this physical front must have crossed $\alpha$ on the interval $]\tau_1, \tau_\ast[\subseteq ]\tau_1, \tau_2[$, which yields a contradiction. Summing up, if at time $\tau_2>\tau_1$ the wave front $\alpha$ crosses another wave front and as a result the speed of $\alpha$ becomes negative, then $\alpha$ must have crossed a physical front in the interval $]\tau_1, \tau_2]$. Since  two given physical fronts can interact only once, then by the analysis at {\sc Step 1} this can happen at most finitely many times and hence a given wave front can be bounced back and forth between the boundary and other waves at most finitely many times. 
\end{itemize}
This implies that the total number of non-physical fronts is also finite. \\
{\sc Step 3:} the above analysis also implies that the total number of interactions is finite. 
\end{proof}
We recall form \S\ref{ss:arie} that $r_\ee$ represents the threshold for the strength of rarefaction waves 
\begin{lemma}
\label{l:raref}
Let $\alpha$ be a rarefaction wave front in the $\ee$-wave front-tracking approximation $\mf u_\ee$; then the strength of $\alpha$ satisfies  
\be \label{e:raref}
     |s_\alpha | \leq C_9 r_\ee 
\eq
for some suitable constant $C_9>0$ only depending on the functions $\mf A$, $\mf E$, $\mf B$ and $\mf G$ and on the value $\mf u^\ast$.
\end{lemma}
\begin{proof}
We fix a rarefaction wave front $\alpha$, which is created at time $\tau_\alpha$ and exists till time $T_\alpha$ (with $T_\alpha$ possibly equal to $+ \infty$). We term $s_\alpha(t)$ its size at time $t$ and $i_\alpha$ the family it belongs to. 
Note that $s_\alpha(\cdot)$ is a piecewise constant function defined for $t \ge \tau_\alpha$ and that $|s_\alpha(\tau_\alpha)|\leq r_\ee$ by construction. To control the growth of $s_\alpha$ we adapt an argument given in~\cite[p.139]{Bressan} and we set 
\be \label{e:yj=k}
    Y_\alpha (t) : = \sum_{\beta \in \mathcal A(\alpha)} |\vartheta_\beta| + \sum_{\beta \in \mathcal I }  |\vartheta_\beta|  + |\xi_k|,
\eq
where $ \mathcal A(\alpha)$ is the set of waves approaching $\alpha$ at the time $t$ and $\mathcal I$ is the set of waves of the families $1, \dots, k-1$. The reason why we need to include the term with $\mathcal I$ is to handle  {\sc Case 3} below when the hitting wave is not characteristic. 
Finally, we set 
\be \label{elle}
    I (t) : =   K_1 Q(t) +  K_1 R(t) + K_2 S(t) +  K_3 Z(t), \qquad
    L_\alpha (t) : = |s_\alpha (t)| \exp \big( K_5 [ Y_\alpha(t)+ I (t)] \big),
\eq
where $K_5$ is a suitable constant satisfying 
\be \label{k5}
     K_5 \ge \max \{ 1, 4 C_5\}, \qquad K_5 \delta \leq \frac{1}{2}. 
\eq
Assume for a moment that we have shown that the function $L$ is monotone non-increasing, then 
\begin{equation*}\begin{split}
    |s_\alpha (t)| & \leq |s_\alpha (t)| \exp \big( K_5 [ Y_\alpha(t)+ I (t)] \big) \leq 
      |s_\alpha (\tau_\alpha)| \exp \big( K_5 [ Y_\alpha(\tau_\alpha)+ I (\tau_\alpha)] \big)\\
     & \stackrel{\eqref{e:assumption}}{\leq} r_\ee  \underbrace{\exp \big( K_5 \delta \big)}_{: = C_9} 
     \quad \text{for every $t \in ]\tau_\alpha, T_\alpha[$}
\end{split}
\end{equation*}
that is~\eqref{e:raref}. To show that the function $L$ is monotone non-increasing we assume that a given time $\tau$ an interaction or a discontinuity in the boundary datum occurs and by using an elementary computation based on the convexity of the exponential function we get 
\be \label{expcon}\begin{split}
    \Delta L (\tau) & \leq  \exp \big( K_5  [ Y_\alpha (\tau^+)  + I(\tau^+)] \big)
     \Big[ \Delta |s_\alpha| (\tau) +  K_5 |s_\alpha (\tau^-)| [\Delta Y_\alpha (\tau)  + \Delta I(\tau)]  \Big]
\end{split}     
\eq 
We now prove $\Delta L (\tau) \leq 0$ by  separately considering the following cases. \\
{\sc Case 1:} the interaction occurs inside the domain and does not involve $\alpha$, but rather two wave fronts with sizes $s_\beta$ and $s_\gamma$. We have 
$$
      \Delta |s_\alpha| (\tau) = 0. 
$$
Since $\xi_k (\tau^+) = \xi_k (\tau^-)$ then by \eqref{e:iebressan} and \eqref{e:iebressan2} we get 
$ \Delta Y_\alpha (\tau) \leq AC_5 |s_\beta| |s_\gamma|$ and by following the analysis in \S\ref{ss:inside} and in particular~\eqref{e:upsgiu2} we conclude 
$$
      \Delta Y_\alpha (\tau) + \Delta I(t) \leq 0
$$
and by combining the above inequalities with~\eqref{expcon} we conclude that $\Delta L(\tau) \leq 0$. \\ 
{\sc Case 2:} the interaction occurs inside the domain and involves $\alpha$, which interacts with a  physical front $\beta$ with strength $s_\beta$. 
If $\alpha$ and $\beta$ belong to the same family, $i_\alpha= i_\beta$, then $\beta$ must be a shock front (because rarefaction wave fronts of the same family do not cross each other) and then $\Delta |s_\alpha| (\tau) \leq 0$ owing to cancellation. If $i_\beta \neq i_\alpha$ then owing to~\eqref{e:iebressan} we have $
      \Delta |s_\alpha| (\tau) \leq      C_5 |s_\alpha s_\beta|$. 
Also, 
$$
     \Delta Y_\alpha (\tau) \stackrel{\eqref{e:iebressan}}{\leq} - |s_\beta| + A C_5 |s_\alpha| |s_\beta| + 
      \stackrel{A C_5 \delta \leq 8^{-1}}{\leq} -\frac{1}{4} |s_\beta|
$$
By following again the analysis in \S\ref{sss:deltaupsid} and in particular~\eqref{e:upsgiu2} we get $\Delta I(\tau) \leq 0$. 
Note that the same estimates hold if the front that interacts with $\alpha$ is a non-physical one (in this case we actually have the stronger estimate $\Delta s_\alpha (\tau) =0$). Either case, by combining the above inequalities with~\eqref{k5} and~\eqref{expcon} we conclude that $\Delta L (\tau) \leq 0.$\\
{\sc Case 3:} the interaction occurs at the domain boundary and does not involve $\alpha$;  in this case we have $\Delta |s_\alpha| (\tau) =0$ and, owing to the analysis in \S\ref{ss:accuratem},\S\ref{ss:accuratek} and in particular to~\eqref{e:upsgiu1},\eqref{e:c:dupsilon},\eqref{e:np:upsgiu}, we have $\Delta Y_\alpha (\tau) + \Delta I(\tau) \leq 0$ and owing to~\eqref{expcon} this yields $\Delta L(\tau) \leq 0$. \\
{\sc Case 4:}  the interaction occurs at the domain boundary and involves $\alpha$; we separately consider the following cases.\\
{\sc Case 4A:} $\alpha$ belongs to the family $j_\alpha =k$.  If there is no wave front of the $k$-th family entering the domain after the interaction, or if the wave front of the $k$-th family is a a shock, then the rarefaction wave front ceases to exist at time $\tau$ and there is nothing to prove. If there is a rarefaction wave front entering the domain after the interaction then we are in case ii) or iii) in \S\ref{ss:abrie}. Also, we can apply Lemma~\ref{l:enterare} and arrive at 
\be \label{e:alphabordo}
 \Delta |s_\alpha| (\tau) \stackrel{\eqref{e:rarepiccola}}{\leq} 
       |s_\alpha| (\tau^-) \big[ -1 + C_8  ([\varsigma_k]^- + |\xi_k|) \big] \stackrel{\eqref{e:sole}}{\leq} -\frac{1}{2} |s_\alpha| (\tau^-), \quad 
     \quad \Delta I(\tau) 
     \stackrel{\eqref{e:c:dupsilon}}{\leq}  0. 
\eq
Also, $\Delta Y_\alpha (t) \leq |\xi'_k| - |\xi_k| \leq  | \xi'_k | \leq \delta$ and by using~\eqref{k5} and~\eqref{expcon} this yields $\Delta L(\tau) \leq 0.$ Note that the previous considerations apply in both the case of the accurate Riemann solver and the simplified Riemann solver. \\
{\sc Case 4B:} $\alpha$ belongs to the family $i_\alpha <k$. In this case the wave front $\alpha$ ceases to exist at time $\tau$ and there is nothing to prove. \\
{\sc Case 5:} the boundary datum $\mf u_b^{\ee}$ is discontinuous at $\tau$. We have $\Delta |s_\alpha| (\tau) =0$ and by recalling the analysis in \S\ref{ss:jump} and in particular~\eqref{e:jdupsilon} we infer $\Delta Y_\alpha (\tau) + \Delta I(\tau) \leq 0$ and owing to~\eqref{expcon} this yields $\Delta L(\tau) \leq 0$.\end{proof}
We recall that $\omega_\ee$ represents the threshold to discriminate between the accurate and the simplified Riemann and boundary Riemann problem, see~\ref{ss:whenas}. 
\begin{lemma} \label{l:np}
For every $\nu>0$, there is $\chi_\nu> 0$ such that if $\omega_\ee \leq \chi_\nu$ then 
\begin{equation}
     \sum_{\alpha \in \mathcal{NP}} |s_\alpha| \leq \nu,
\end{equation}
where $\mathcal{NP}$ denotes the set of non-physical fronts. 
\end{lemma}
\begin{proof}
The proof is based on the same argument as in~\cite[pp.139-142]{Bressan}, so instead of providing all the details we only focus on the points where the argument is different. We proceed according to the following steps. \\
{\sc Step 1:} we show that 
\be \label{npmax}
       |s_\alpha| \leq \unpo  \omega_\ee, \quad \text{for every $\alpha \in \mathcal{NP}$}. 
\eq
We recall that, by construction, when a non-physical front is created its strength is bounded by $\unpo \omega_\ee$. To conclude with~\eqref{npmax} we can then use exactly the same argument as in~\cite[p.140]{Bressan}. \\
{\sc Step 2:} we assign to each wave-front $\alpha$ a \emph{generation order} $n_\alpha$. To this end, we proceed as follows 
\begin{itemize}
\item all the waves generated at $t=0$ and at discontinuity times of $\mf u^\ee_b$ have generation order $1$;
\item when two waves with generation orders $n_\alpha$ and $n_\beta$ interact inside the domain, the generation order of the outgoing waves is assigned as 
in~\cite[pp.140]{Bressan}; 
\item if a wave-front $\alpha$ with generation order $n_\alpha$ hits the boundary and $i_\alpha =k$, then wave front of the $k$-th family entering the domain after the interaction (if any) has generation order $n_\alpha$, the other wave fronts entering the domain have generation order $n_\alpha+1$;
\item if a wave-front $\alpha$ with generation order $n_\alpha$ hits the boundary and $i_\alpha <k$, then all wave fronts entering the domain after the interaction have generation order $n_\alpha +1$.
\end{itemize}
Note that we are heuristically speaking regarding the boundary as a wave front with generation order~$1$. \\
{\sc Step 3:} we fix $n \ge 2$ and we set
\begin{equation*} \begin{split}
   & V^+_n (t) : = \sum_{\substack{n_\alpha \ge n \\ i_\alpha \ge k}} |s_\alpha|, \quad V^-_n (t) : = \sum_{\substack{n_\alpha \ge n \\ i_\alpha < k}} |\vartheta_\alpha|, 
     \quad 
   Q_n(t) : = \!  \!  \!  \!  \!  \!  \!  \!  \! \!  \!  \! \sum_{ \max\{n_\alpha,  n_\beta \} \ge n}\!  \!  \!  \!\!  \!  \!  \! \!   \!  \! w_{\alpha \beta}  |\vartheta_\alpha \vartheta_\beta|, \quad R_n(t): =
    |\xi_k |   \!  \! \sum_{\substack{ i_\alpha \leq k, \\ n_\alpha \ge n}} |\vartheta_\alpha|, \\
   &
   W_n (t) : =  V^+_{n+1} (t) + V^-_n(t), 
   \end{split}
\end{equation*}
where we formally count non-physical fronts as belonging to the family $N+1$. We also introduce the notation 
$$
     S_n (t) : = \left\{
    \begin{array}{ll}
     \displaystyle{ \sum_{\substack{ i_\alpha = k, \\ n_\alpha \ge n}} |s_\alpha| [\varsigma_k (\mf u_\alpha, s_\alpha)]^- } & 
     \text{if~\eqref{e:gnl}} \\ \phantom{ciao} \\
      \displaystyle{ \sum_{\substack{ i_\alpha = k, \\ n_\alpha \ge n}} |s_\alpha| \big[ \lambda_k (\mf u_\alpha) \big]^- }
      & 
     \text{if~\eqref{e:lindeg}} \\
      \end{array}
    \right.
$$
and 
$$ 
    I_n(t) : =   K_1 Q_n(t) + K_1 R_n(t) + K_2 S_n(t), \qquad \text{for $n \ge 2$}
$$
and we set 
$$
    V_1^+(t) : = |\xi_k | +  \sum_{ i_\alpha \ge k} |s_\alpha|, \qquad 
    V^-_1 (t) : = \sum_{ i_\alpha < k} |\vartheta_\alpha|, \qquad W_1 (t) : = V_2^+(t) + V_1^-(t), \qquad I_1(t): = I(t), 
$$
where $I$ is as in~\eqref{elle}.  
We also term 
\begin{itemize}
\item $J^0_n$ the set of times at which two wave fronts $\alpha$ and $\beta$ with $\max \{ n_\alpha, n_\beta \}=n$ interact;
\item  $J_n^k$ the set of times at which  a wave-front $\alpha$ with $i_\alpha=k$ and $n_\alpha =n$ hits the boundary;
\item $J^-_n$ the set of times where a wave front $\alpha$ with $i_\alpha <k$ and $n_\alpha =n$ hits the boundary.  
\end{itemize}
In the remaining part of the proof we will always assume $n \ge 3$. 
We have 
\begin{equation} \label{stimew}
\begin{array}{rl}
\Delta W_n ( t) =0 & 
t \in J^0_1 \cup\dots \cup J^0_{n-2} \cup J^k_1 \cup \dots \cup J^k_{n-1}  \cup J^-_1 \cup \dots \cup J^-_{n-1} ;\\
\Delta W_n (t)+ \Delta I_{n-1} (t) \stackrel{\S\ref{ss:inside}}{\leq} 0 & t \in \displaystyle{\bigcup_{m\ge n-1} J^0_m} ;\\
\Delta W_n (t)+ \Delta I_{n-1} (t)  \stackrel{\S\ref{ss:accuratek}}{\leq} 0 &   t \in \displaystyle{\bigcup_{m\ge n} J^k_m} ;\\ 
\Delta W_n (t) +  \Delta I_{n-1} (t) \stackrel{\S\ref{ss:accuratem}}{\leq}  0 &  t \in \displaystyle{\bigcup_{m\ge n} J^-_m.}
\end{array}
\end{equation} 
If $t \in  J^0_1 \cup\dots \cup J^0_{n-2}$ then $\Delta R_n (t)= \Delta S_n (t) =0$, so $\Delta I_n (t) = K_1  \Delta Q_n (t)$, which owing to~\eqref{e:idv},~\eqref{e:upsgiu2} and to the inequality $W_n \leq W_{n-1}$ yields 
$$
     \Delta I_n (t) + K_1 C_5 A \Delta \Upsilon W_{n-1} (t^-) \leq  \Delta I_n (t) + K_1 C_5 A \Delta \Upsilon W_{n} (t^-)\leq 0.
$$
If $t \in  J^k_1 \cup\dots \cup J^k_{n-1} $ we have  $\Delta S_n (t) =0$, whereas $R_n$ can change because the strength $\xi_k$ can change at $t$. This yields 
 $\Delta I_n (t) = K_1  \Delta Q_n (t) + K_1 \Delta R_n (t)$, which owing to~\eqref{e:c:tv} and~\eqref{e:c:dupsilon} yields 
$$
    \Delta I_n (t) +  K_1 C_2 A \Delta \Upsilon W_{n-1}(t^-) \leq 0.
$$
If $t \in  J^-_1 \cup\dots \cup J^-_{n-1} $ we again have $\Delta I_n (t) = K_1  \Delta Q_n (t) + K_1 \Delta R_n(t)$, which owing to~\eqref{e:deltav} and~\eqref{e:upsgiu1} yields 
$$
    \Delta I_n (t) +  K_1 C_1 A \Delta \Upsilon W_{n-1}(t^-) \leq 0.
$$  
If $ t \in J^0_{n-1}$ then owing to~\eqref{e:iebressan},~\eqref{e:iebressan2} and~\eqref{e:upsgiu2} we have  
\begin{eqnarray*}
    \Delta Q_n (t) + A C_5 V(t^-) \Delta I_{n-1} (t) \leq 0, \qquad 
    \Delta R_n (t) + A C_5 |\xi_k(t^-)| \Delta I_{n-1} (t) \leq 0 
\end{eqnarray*} 
We now point out that there are two possible contributions to $\Delta S_n (t)$: i) new $k$-family waves may be created at the interaction; ii) $k$-family waves with generation order greater or equal than $n$ change their strength and speed. Owing to~\eqref{e:iebressan},\eqref{e:iebressan2},\eqref{e:dadim3},~\eqref{e:msigma} and~\eqref{e:upsgiu2} yields 
$$
    \Delta S_n (t) + C_4C_5 \delta \Delta I_{n-1} (t) +C_3 C_5 \delta \Delta I_{n-1} (t) \leq 0. 
$$
If $ t \in  \displaystyle{\bigcup_{m\ge n} J^0_m}$ then owing to~\eqref{e:upsgiu2} we have $\Delta I_{n} \leq 0$ owing to~\eqref{e:c:dupsilon}. 
If $ t \in  \displaystyle{\bigcup_{m\ge n} J^k_m}$ then $\Delta I_n \leq 0$. If  $ t \in  \displaystyle{\bigcup_{m\ge n} J^-_m}$ then $\Delta S_n =0$ (because the new $k$-family waves created at the interaction, if any, have positive speed) 
and, owing to~\eqref{e:deltav} and~\eqref{e:deltarq} 
$$
     \Delta Q_n (t) +   \Delta R_n (t) + \frac{C_1+  C_{12}+ C_{12} C_1}{A-C_1}  \Delta W_{n-1} (t) \delta \leq 0. 
$$
Summing up, we have
\begin{equation}
        \label{stimei}
        \begin{array}{rl}
         \Delta I_n (t) + \unpo \Delta \Upsilon (t) W_{n-1} (t^-)\leq 0 & t \in J^0_1 \cup\dots \cup J^0_{n-2} \cup  J^k_1 \cup\dots \cup J^k_{n-1} \cup  J^-_1 \cup\dots 
              \cup J^-_{n-1} \phantom{\displaystyle{\bigcup}}\\
                          \Delta I_n (t)+ \unpo \delta  \Delta I_{n-1} (t) \leq 0 & t \in   J^0_{n-1} \phantom{\displaystyle{\bigcup}}\\
                           \Delta I_{n}(t) \leq 0 &  t \in  \displaystyle{\bigcup_{m\ge n} J^k_m} \\
                            \Delta I_{n} (t) + \unpo \Delta W_{n-1} (t) V(t^-)\leq 0 &  t \in  \displaystyle{\bigcup_{m\ge n} J^-_m} \\
        \end{array}
\end{equation}
Note that~\eqref{stimew} yields 
\be \label{stimew2}
    W_n (t^+ ) \leq \sum_{0\leq \tau \leq t} [\Delta I_{n-1} (\tau) ]_-,
\eq
where in the previous expression the sum is over the times $\tau$ of interaction, and that owing to~\eqref{stimei} we also have 
\begin{equation}\label{stimai2}
\begin{split}
    I_n (t) & \leq   \sum_{0\leq \tau \leq t} [\Delta I_n (\tau) ]_+ \leq \unpo 
    \sup_{\tau} W_{n-1} (\tau) \underbrace{\sum_{0\leq \tau \leq t} [\Delta \Upsilon (\tau)]_- }_{\leq \Upsilon (0) \leq \delta}
     + \unpo \delta \sum_{0\leq \tau \leq t} [\Delta I_{n-1} (\tau)]_- 
     \\ & + \unpo \delta  \sum_{0\leq \tau \leq t} [\Delta W_{n-1} (\tau)]_-
    \end{split}
\end{equation}
Note that  $W_{n-1}(0)=0$ since $n \ge 3$ and this implies 
$$
     \sum_{0\leq \tau \leq t} [\Delta W_{n-1} (\tau) ]_- = \sum_{0\leq \tau \leq t} [\Delta W_{n-1} (\tau) ]_+ - W_{n-1} (t) 
     \stackrel{\eqref{stimew}}{\leq} 
      \sum_{0\leq \tau \leq t} [\Delta I_{n-2} (\tau) ]_-- W_{n-1} (t) 
     \leq   \sum_{0\leq \tau \leq t} [\Delta I_{n-2} (\tau) ]_-
$$
and by plugging the above inequality into~\eqref{stimai2} and recalling~\eqref{stimew2} we arrive at 
\begin{equation*}
\begin{split}
      \sum_{0\leq \tau \leq t} [\Delta I_n (\tau) ]_+ \leq
     \unpo \delta \sum_{0\leq \tau \leq t} [\Delta I_{n-1} (\tau)]_- 
      + \unpo \delta  \sum_{0\leq \tau \leq t} [\Delta I_{n-2} (\tau)]_-
    \end{split}
\end{equation*}
and since 
$$
    \sum_{0\leq \tau \leq t} [\Delta I_{n-1} (\tau)]_- \stackrel{I_{n-1}(0)=0}{=} \sum_{0\leq \tau \leq t} [\Delta I_{n-1} (\tau)]_+ - I_{n-1}(t) \leq \sum_{0\leq \tau \leq t} [\Delta I_{n-1} (\tau)]_+
$$
by setting 
$$
    \tilde I_n : = \sum_{0\leq \tau \leq t} [\Delta I_n (\tau) ]_+ 
$$
we arrive at 
\be \label{stimaIind}
    \tilde I_n \leq \unpo \delta \tilde I_{n-1} + \unpo \delta \tilde I_{n-2}
\eq
From~\eqref{stimaIind} a fairly standard induction argument implies that, if $\delta$ is sufficiently small then $\tilde I_n \leq (2 \unpo \delta)^{e_n}$, with $e_n = [(n-3)/2]+1$ and $[\cdot]$ denoting the integer part. From this we can then argue as in~\cite[p. 141-142]{Bressan}, control $I_n$ and $W_n$ and  whence (by arguing as in~\cite[(7.77)]{Bressan}) the total strength of non-physical waves. This concludes the proof of Lemma~\ref{l:np}.
\end{proof}

\section{Boundary condition}\label{s:bc}
The goal of this section is to conclude the proof of Theorem~\ref{t:wft} (and hence of Theorem~\ref{t:main}) by establishing the following result.
\begin{theorem}
\label{t:maintr} Under the same assumption as in the statement of Theorem~\ref{t:main}, let $\mf v: \R_+ \times \R_+ \to \R$ be the same as in~\eqref{e:16bis}. Then for a.e. $t \in \R_+$ the trace $\mf{\bar v} (t) : = \lim_{x \to + 0^+} \mf v(t, x)$ satisfies $\mf{\bar v}(t) \sim_{\mf D} \mf v_b(t)$ in the sense of Definition~\ref{d:equiv}.  
\end{theorem}
The proof of Theorem~\ref{t:maintr} is rather technical and involved, and represents one of the most innovative parts of the present paper. 
 Before giving some technical details, a word of warning: in the previous sections, we have always worked with the variable $\mf u$ only because in the boundary layers analysis we have used the fact that the viscous system~\eqref{e:symmetric} is in the \emph{normal} form, in the Kawashima Shizuta sense~\cite{KawashimaShizuta1}. In this section, instead, we investigate the boundary condition and hence we have to use the variable $\mf v$ as well, because Definition~\ref{d:equiv} involves the flux function $\mf f$. For the reader's convenience we now provide a sketchy overview of the proof of Theorem~\ref{t:maintr}. 
 \subsection{Roadmap of the proof of Theorem~\ref{t:maintr}}
 For the most part, in what follows we assume that the $k$-th characteristic field is genuinely nonlinear, namely we assume~\eqref{e:gnl}, and only in \S\ref{ss:bclindeg} we discuss how to adapt the proof to the case of a linearly degenerate~\eqref{e:lindeg} $k$-th characteristic field.

In \S\ref{ss:tvf} we show that 
\begin{equation}
\label{e:flupointwise}
      \lim_{m \to + \infty}  \mf f \left(  \lim_{x \to 0^+}   \mf v_{\ee_m}(t, x) \right) =  \lim_{x \to 0^+} \mf f \circ \mf v(t, x) , \qquad \text{for a.e. $t \in \R_+$,}
\eq  
where $\mf v_{\ee_m} = \mf v (\mf u_{\ee_m})$ and $\mf u_{\ee_m}$ is the same as in~\eqref{e:81}\footnote{Note that in~\eqref{e:81} $\ee_m$ is a specific subsequence and, in principle, the limit $\mf u$ depends on the particular choice of the subsequence (as we mentioned in the introduction, we are confident that one could establish uniqueness results, but this goes beyond the scopes of the present paper). However, in the remaining of this section we will always work with the same sequence $\ee_m$ and this uniquely identifies the limit $\mf u$.\label{foot2}}.
To this end, the main issue is establishing Lemma~\ref{t:tvf}, which provides a uniform in $\ee$ total variation bound on $\mf f \left(  \lim_{x \to 0^+}   \mf v_{\ee_m}(\cdot, x) \right)$. In the proof of Lemma~\ref{t:tvf} we introduce the functional $\Lambda$, defined as in~\eqref{e:Lambda}, and prove that it is monotone non-increasing by carefully tracking how it varies any time either a wave front hits the boundary, or the approximate boundary datum has a jump discontinuity. 

In \S\ref{ss:tildeE} we construct the set $\widetilde E \subseteq \R_+$, which plays a main role in our analysis and is defined by setting
 \be \label{e:tildeE}
    \widetilde E : = F  \cup G \cup H \cup D.
\eq
In the previous expression, the sets $F$ and $H$  are rigorously defined in~\eqref{e:F} and~\eqref{e:H} below, but, very loosely speaking, a time $\tau$ belongs to $F$ if there are a lot of wave fronts of the wave front-tracking approximation $\mf u_{\ee_m}$ hitting the domain at times that accumulate around $\tau$ as $m \to + \infty$. Conversely,  an instant $\tau$ belongs to $H$ if there are a lot of waves of the wave front-tracking approximation $\mf u_{\ee_m}$ that interact in proximity of the domain boundary at times that accumulate around $\tau$ as $m \to + \infty$. By relying on the rigorous definitions~\eqref{e:F} and~\eqref{e:H} it is fairly easy to show that $F$ and $H$ are both $\mathcal L^1$-negligible.  The set $D$ is the set where the boundary datum $\mf v_b$ is discontinuous, and it is also $\mathcal L^1$-negligible owing to the assumption $\mathrm{TotVar} \, \mf v_{b} < + \infty$. Finally, the (not necessarily negligible) set $G$ in formula~\eqref{e:tildeE}  is rigorously defined by~\eqref{e:Gi} and heuristically speaking represents the set of times where there is a sequence of  Lax admissible shocks for $\mf v_{\ee_m}$ with vanishing speed that as $m \to + \infty$ accumulate at the domain boundary. This heuristic interpretation is consistent with  Lemma~\ref{l:tildeE}, which dictates that if $t \notin \widetilde E$, the limit of the traces of $\mf v_{\ee_m}$ coincides with the trace of the limit $\mf v$. 

In~\S\ref{ss:notinE} we show that, for a.e. $t \notin \widetilde E$, the trace $\mf{\bar v}(t)= \lim_{x \to 0^+} \mf v(t, x)$ satisfies $\mf{\bar v}(t) \sim_{\mf D} \mf v_b(t)$ in the sense of Definition~\ref{d:equiv}, whereas in~\S\ref{ss:tinE} by using~\eqref{e:flupointwise} we establish the same property for a.e. $t \in \widetilde E$. Finally, in  \S\ref{ss:bclindeg} we establish the proof of Theorem~\ref{t:maintr} in the case where the $k$-th vector field is linearly degenerate.  
\subsection{Proof of~\eqref{e:flupointwise}}\label{ss:tvf}
The proof of~\eqref{e:flupointwise} relies on the following total variation bound.
\begin{lemma}
\label{t:tvf}
Under the same assumptions as in Theorem~\ref{t:main} we have  
\begin{equation}\label{e:tvf}
    \mathrm{TotVar} \, \mf f \circ \mf v_{\ee_m}  (\cdot, 0) \leq C_{13} \delta, \quad \text{for every $m \in \mathbb N$}
\eq
and for a suitable constant $C_{13}>0$  only depending on $\mf g$, $\mf f$, $\mf D$ and on the change of variables $\mf u$, and on the value $\mf u^\ast$. In the above expression, the constant $\delta$ is the same as in~\eqref{e:assumption}. 
\end{lemma}
The proof of the above theorem uses the following result. 
\begin{lemma}
\label{l:dflux}
Under the same assumptions as in Theorem~\ref{t:main}, if the $k$-the vector field is genuinely nonlinear~\eqref{e:gnl} we have
\begin{equation}
\label{e:dflux}
      | \mf f \circ \mf v (\mf{\widetilde u}) -  \mf f \circ \mf v (\mf t_k (\mf{\widetilde u}, s))|
      \leq
      \left\{ 
     \begin{array}{ll}
          \unpo  |s| | \sigma_k (\mf{\widetilde u}, s)| & s >0 \\
          \unpo |s |   \max \{ |\lambda_k (\mf t_k  (\mf{\widetilde u}, s))|  , |\lambda_k (\mf{\widetilde u})|    \} & s < 0. \\
      \end{array}
      \right.
\eq
\end{lemma}
\begin{proof}[Proof of Lemma~\ref{l:dflux}]
We separately consider the following cases. \\
{\sc Case 1:} $s\ge 0$. We have  
 $$
    | \mf f \circ \mf v (\mf{\widetilde u}) -  \mf f \circ \mf v (\mf t_k (\mf{\widetilde u}, s))|
   \stackrel{\eqref{e:rh}}{=} |\sigma_k (\mf{\widetilde u}, s)|
    | \mf g \circ \mf v (\mf{\widetilde u}) -  \mf g \circ \mf v (\mf t_k (\mf{\widetilde u}, s))|
\leq \unpo |s| |\sigma_k (\mf{\widetilde u}, s)| 
$$
and this establishes~\eqref{e:dflux}. \\
{\sc Case 2:} $s<0$. We have 
\begin{equation*} \begin{split}
    | \mf f \circ \mf v (\mf{\widetilde u}) & -  \mf f \circ \mf v (\mf t_k (\mf{\widetilde u}, s))|
     \stackrel{\eqref{e:laxc}}{=}  | \mf f \circ \mf v (\mf{\widetilde u}) -  \mf f \circ \mf v (\mf i_k (\mf{\widetilde u}, s))| 
     \stackrel{\eqref{e:raref}}{=}
    \left| \int_s^0 \mf D \mf f \, \mf D\mf v \ \mf r_k (\mf i_k (\mf{\widetilde u}, z)) dz  \right| \\
    & 
     =  \left| \int_s^0 \lambda_k (\mf i_k (\mf{\widetilde u}, z)) \mf D \mf g \, \mf D\mf v \ \mf r_k (\mf i_k (\mf{\widetilde u}, z)) dz  \right| \leq \unpo   \int_s^0  |\lambda_k  (\mf i_k (\mf{\widetilde u}, z)) | dz  \\ & \leq 
     \unpo |s| \max\{|\lambda_k (\mf i_k (\mf{\widetilde u}, s))|, | \lambda_k (\mf{\widetilde u}) | \}
 \stackrel{\eqref{e:laxc}}{=} \unpo |s| 
     \max\{|\lambda_k (\mf t_k (\mf{\widetilde u}, s))|, | \lambda_k (\mf{\widetilde u}) | \},
\end{split}
\end{equation*}
i.e.~\eqref{e:dflux}. To establish the last inequality at the second line of the above expression we have used the fact that, owing to~\eqref{e:gnl}, the function $z \mapsto \lambda_k (\mf i_k (\mf{\widetilde u}, z))$ is monotone non-decreasing and hence the maximum of its modulus on an interval is attained at one the two extrema. 
\end{proof}
We are now ready to give the 
\begin{proof}[Proof of Lemma~\ref{t:tvf} (genuinely nonlinear case)]We assume that the $k$-th vector field is genuinely nonlinear~\eqref{e:gnl} and set
\be \label{e:Lambda}
    \Lambda (t) : = F(t) +K_4 \Upsilon(t) 
\eq
where 
\be 
\label{e:Vt}
      F(t) = \mathrm{Tot Var}_{[0, t[}  \mf f \circ \mf v_{\ee_m}(\cdot, 0),
      \eq 
$\Upsilon$ is the same as in~\eqref{e:upsilon} and $K_4$ is a suitable constant, to be determined in the following. 
Assume for a moment that we have shown that the functional $\Lambda$ is monotone non-increasing; then 
$$
    \mathrm{Tot Var}  \; \mf f \circ \mf v_{\ee_m} (\cdot, 0) \leq \lim_{t \to + \infty} \Lambda(t) \leq \Lambda (0) \leq \unpo \delta 
$$
and this establishes~\eqref{e:tvf}. To show that $\Lambda$ is monotone non-increasing, we first of all recall that owing to the analysis in \S\ref{s:functional} the term $\Upsilon$ is monotone non-increasing. We now separately analyze the three instances where the term $F$ can vary. \\
{\sc Case I:} at time $\tau$ a wave front $\alpha$ of the family $j_\alpha < k$ with size $s_j$ hits the boundary. 
Assume for a moment that we have established the inequality
\be \label{e:deltaF2}
    \Delta F(\tau)  \leq \unpo |s_j|. 
\eq
Then owing to~\eqref{e:upsgiu1} we can choose the constant $K_4$ in~\eqref{e:Lambda} in such a way that $\Delta \Lambda(\tau) \leq 0$. To establish~\eqref{e:deltaF2} we first have to introduce some notation. We term $\mf u^-$ and 
$\mf u^+$ the left and the right state of $\alpha$, and denote by $\mf{\hat u}$ the same state as in~\eqref{e:hatu}. We recall the construction in \S\ref{ss:srie} and conclude that if we solve the boundary Riemann problem~\eqref{e:claw},\eqref{e:briedata} by using the simplified boundary Riemann problem then $\Delta F (\tau) =0$ and~\eqref{e:deltaF2} is trivially satisfied. We are left to consider the case where the boundary Riemann problem~\eqref{e:claw},\eqref{e:briedata} is solved by using the accurate boundary Riemann solver: we recall the construction in \S\ref{sss:casignl} and \S\ref{ss:abriegnl} and we separately consider cases i),$\dots$,\ vi) in there.
\begin{itemize}
\item[i)] we have  
   \begin{equation}     \label{e:deltaF} 
\begin{split}
    \Delta F(\tau) & = | \mf f (\mf v (\mf u^-))-  \mf f (\mf v (\boldsymbol{\zeta}_k (\mf{\hat u}, s'_k) ))   | \\ &  \leq 
    | \mf f (\mf v (\mf u^-))-  \mf f (\mf v (\mf u^+))   | + | \mf f (\mf v (\mf u^+))-  \mf f (\mf v (\mf{\hat u}))| +
    |\mf f (\mf v (\mf{\hat u})) - \mf f (\mf v (\boldsymbol{\zeta}_k (\mf{\hat u}, s'_k) ))| \\
   & \leq  \unpo |\mf u^- - \mf u^+| + \unpo |\mf u^+ - \mf{\hat u}| +  
   |\mf f (\mf v (\mf{\hat u})) - \mf f (\mf v (\boldsymbol{\zeta}_k (\mf{\hat u}, s'_k) ))| \\ &
   \stackrel{\eqref{e:nc}}{\leq} \unpo|s_j| +  |\mf f (\mf v (\mf{\hat u})) - \mf f (\mf v (\boldsymbol{\zeta}_k (\mf{\hat u}, s'_k) ))| 
    \end{split}
\end{equation}
To control the second term at the right-hand side of the above formula we point out that (recalling that $\sigma_k (\mf{\hat u}, s'_k)>0$ by definition of case i))
  \begin{equation} \label{e:cvelocita}  
\begin{split}
       |\mf f (\mf v (\mf{\hat u})) & - \mf f (\mf v (\boldsymbol{\zeta}_k (\mf{\hat u}, s'_k) ))| 
        \stackrel{\eqref{e:rh}}{\leq} [\sigma_k (\mf{\hat u}, s'_k)]^+
        |\mf g (\mf v (\mf{\hat u})) - \mf g (\mf v (\boldsymbol{\zeta}_k (\mf{\hat u}, s'_k) ))|
         \\ & \leq \unpo |s'_k|
        [\sigma_k (\mf{\hat u}, s'_k)]^+ \leq \unpo \delta [\sigma_k (\mf{\hat u}, s'_k)]^+ \\ &
        \leq\unpo \delta| [ \sigma_k (\mf{\hat u}, s'_k)]^+- [\sigma_k (\mf u^+, s'_k)]^+| + \unpo  \delta| 
         [\sigma_k (\mf u^+, s'_k)]^+ - [\sigma_k (\mf u^-, s'_k)]^+   | \\ & \quad +\unpo  \delta  | [\sigma_k (\mf u^-, s'_k)]^+ - [\sigma_k (\mf u^-, \xi_k)]^+ |
         \\ & \leq \unpo \delta | \mf{\hat u} - \mf u^+ | + \unpo \delta |\mf u^+ - \mf u^- | + \unpo \delta |s'_k - \xi_k| 
         \stackrel{\eqref{e:nc}}{\leq} \unpo \delta |s_j|
  \end{split}
\end{equation}
To establish the third to last inequality we have used the fact that $ [\sigma_k (\mf u^-, \xi_k)]^+ =0$, obtained by applying~\eqref{e:sigmatraccia} with $\mf u^-$ and $\xi_k$ in place of $\mf{\bar u}$ and $\xi_k$, respectively. By plugging the previous chains of inequalities into~\eqref{e:deltaF} we get~\eqref{e:deltaF2}
\item[ii)] we 
use again~\eqref{e:deltaF} and to control the second term at the right hand side we proceed as follows. If $\xi_k=0$ then owing to~\eqref{e:nc} we have 
$|s'_k| \leq C_1 |s_j|$ and this yields
$$
    |\mf f (\mf v (\mf{\hat u})) - \mf f (\mf v (\boldsymbol{\zeta}_k (\mf{\hat u}, s'_k) ))|  \leq \unpo |s'_k| \leq \unpo |s_j|. 
$$
If $\xi_k >0$ then owing to~\eqref{e:ceblayer} we have $\lambda_k (\mf u^-) \leq 0$. We then have 
\begin{equation*} \begin{split}
     |\mf f (\mf v (\mf{\hat u}))  - \mf f (\mf v (\boldsymbol{\zeta}_k (\mf{\hat u}, s'_k) ))| &=
      |\mf f (\mf v (\mf{\hat u})) - \mf f (\mf v (\mf t_k (\mf{\hat u}, s'_k) ))| \\ & =
      |\mf g \circ \mf g^{-1} \circ \mf f (\mf v (\mf{\hat u})) - \mf g \circ \mf g^{-1} \circ \mf f (\mf v (\mf t_k (\mf{\hat u}, s'_k) ))|\\
      & \leq \unpo  | \mf g^{-1} \circ \mf f (\mf v (\mf{\hat u})) - \mf g^{-1} \circ \mf f (\mf v (\mf t_k (\mf{\hat u}, s'_k) ))|
      \leq \left| \int_{s'_k}^0 \lambda_k \mf r_k (\mf t_k (\mf{\hat u}, s)  ds \right| \\
      & \leq \unpo \int_{s'_k}^0 | \lambda_k  (\mf t_k (\mf{\hat u}, s) | ds \stackrel{\eqref{e:gnl}}{\leq}
       \unpo \int_{s'_k}^0 \lambda_k (\mf{\hat u}) ds \leq \unpo |s'_k| \lambda_k (\mf{\hat u}) 
       \\ & \stackrel{\lambda_k (\mf u^-) \leq 0}{\leq}  
       \unpo |s'_k| [ \lambda_k (\mf{\hat u}) - \lambda_k (\mf u^-)] \leq  \unpo |s'_k| |\mf{\hat u} - \mf u^-|
      \\ &  \stackrel{\eqref{e:nc}}{\leq} \unpo |s'_k| |s_j| \leq \unpo \delta |s_j|
\end{split}     
\end{equation*} 
and by plugging the above inequality into~\eqref{e:deltaF} we arrive at~\eqref{e:deltaF2}. 
\item[iii)] the proof is the same as at the previous point, it suffices to replace $s'_k$ with $\bar s (\mf{\hat u})$ and use~\eqref{e:pallido2}. 
\item[iv)] the proof of~\eqref{e:deltaF2} is the same as in case i).
\item[v), vi)] since there is no wave of the $k$-th family entering the domain after the interaction  we have $ \Delta F(\tau)= | \mf f (\mf v (\mf u^-))-  \mf f (\mf v (\mf{\hat u})   |$ and one can then proceed  
as in~\eqref{e:deltaF}.  
\end{itemize}
{\sc Case II:} at time $\tau$ a wave front $\alpha$ of the $k$-th family with (signed) strength $s$ hits the boundary. 
Assume for a moment that we have established the inequality 
\be \label{e:deltaFc2}
    \Delta F(\tau) \leq \unpo |s| ([\varsigma_k (\mf u^+, s) ]^- + |\xi_k|   ). 
\eq
Then owing to~\eqref{e:c:dupsilon} we can choose the constant $K_4$ in~\eqref{e:Lambda} in such a way that $\Delta \Lambda(\tau) \leq 0$. 
To establish~\eqref{e:deltaFc2} we first assume that the boundary Riemann problem~\eqref{e:claw},~\eqref{e:briedata} is solved by using the accurate boundary Riemann solver, we recall the analysis in \S\ref{ss:abrie} and we separately consider the following cases.  As in {\sc Case I}, we term $\mf u^-$ and 
$\mf u^+$ the left and the right state of $\alpha$, respectively, and denote by $\mf{\hat u}$ the same state as in~\eqref{e:hatu}. Note that, owing to~\eqref{e:speedrare} we have $\lambda_k (\mf u^+) <0$. We now separately consider cases i),$\dots$,vi) in \S\ref{sss:casignl} and \S\ref{ss:abriegnl}.
\begin{itemize}
\item[i)]  We point out that 
\begin{equation}     \label{e:deltaFc} 
\begin{split}
    \Delta F(\tau) & = | \mf f (\mf v (\mf u^-))-  \mf f (\mf v (\boldsymbol{\zeta}_k (\mf{\hat u}, s'_k) ))   | \\ &  \leq 
    | \mf f (\mf v (\mf u^-))-  \mf f (\mf v (\mf u^+))   | + | \mf f (\mf v (\mf u^+))-  \mf f (\mf v (\mf{\hat u}))| +
    |\mf f (\mf v (\mf{\hat u})) - \mf f (\mf v (\boldsymbol{\zeta}_k (\mf{\hat u}, s'_k) ))| 
    \end{split}
\end{equation}
To control first term at the right-hand side of~\eqref{e:deltaFc} we recall~\eqref{e:dflux}. In the case $s<0$ we recall that owing to~\eqref{e:speedrare} the speed of the wave-front is $\lambda_k (\mf u^+)$ and that $\lambda_k (\mf u^+)<0$ (otherwise the wave could not hit the boundary). Since the map $s \mapsto \lambda_k (\mf i_k (\mf u^+, s))$ is monotone non-decreasing owing to~\eqref{e:gnl}, we conclude that 
$$
    \max \{ |\lambda_k (\mf t_k  (\mf u^+, s))|  , |\lambda_k (\mf u^+)|    \}=
    |\lambda_k (\mf t_k  (\mf u^+, s))| \stackrel{\eqref{e:speedwft}}{=}
   | \varsigma_k   (\mf u^+, s)| = [\varsigma_k   (\mf u^+, s)]^- \quad \text{if $s<0$}.
$$ 
Owing to~\eqref{e:dflux} and recalling~\eqref{e:speedwft} we conclude that, regardless the sign of $s$, we have  
\be \label{e:firstt}
     | \mf f (\mf v (\mf u^-))-  \mf f (\mf v (\mf u^+ ))   | \leq 
\unpo |s| [ \varsigma_k   (\mf u^+, s)]^-.  
\eq
To control the second term at the right-hand side of~\eqref{e:deltaFc} we simply point out that 
\be \label{e:secondt}
     | \mf f (\mf v (\mf u^+))-  \mf f (\mf v (\mf{\hat u}))| \leq \unpo |\mf u^+ - \mf{\hat{u}}|
    \stackrel{\eqref{e:me},\eqref{e:hatu}}{\leq} \unpo |s| ([\varsigma_k (\mf u^+, s) ]^- + |\xi_k|   ).
\eq
We now recall~\eqref{e:speedwft} and control the third term at the right-hand side of~\eqref{e:deltaFc}:
\begin{equation} \label{e:thirdt}
\begin{split}
        |\mf f (\mf v (\mf{\hat u})) - & \mf f (\mf v (\boldsymbol{\zeta}_k (\mf{\hat u}, s'_k) ))| 
       \stackrel{\eqref{e:dflux}, s'_k >0}{\leq} \unpo |s'_k| [\varsigma_k (\mf{\hat u}, s'_k)]^+
       \\ & \leq
      \unpo\delta  [\varsigma_k (\mf{\hat u}, s'_k)]^+ -  [\varsigma_k (\mf u^+, s'_k)]^+    |+\unpo \delta 
     | [\varsigma_k (\mf u^+, s'_k)]^+    -  [\varsigma_k (\mf u^+, s +\xi_k)]^+| \\ & \quad +
    \unpo |s'_k|  [\varsigma_k (\mf u^+, s +\xi_k)]^+ 
       \\ & \stackrel{\eqref{e:me}}{\leq} 
     \unpo \delta |  \mf{\hat u}- \mf u^+ | + 
    \unpo \delta |s| ([\varsigma_k (\mf u^+, s) ]^- + |\xi_k|   ) +
     \unpo |s'_k|  [\varsigma_k (\mf u^+, s +\xi_k)]^+  \\ &
    \stackrel{\eqref{e:hatu},\eqref{e:me}, \eqref{e:msigma}}\leq 
     \unpo \delta |s| ([\varsigma_k (\mf u^+, s) ]^- + |\xi_k|   ) +
     \unpo |s + \xi_k|  [\varsigma_k (\mf u^+, s +\xi_k)]^+ 
\end{split}
\end{equation}
When $\xi_k=0$, the above estimate yields~\eqref{e:deltaFc2}, hence to control the last term at the right-hand side of the above expression we assume $\xi_k \neq 0$ and we separately consider the following cases: if $\varsigma_k (\mf u^+, s +\xi_k)\leq 0$ then the term vanishes and the estimate is complete. We are thus left to consider the case  $\varsigma_k (\mf u^+, s +\xi_k) \ge 0$:   recalling that $\lambda_k (\mf u^+) <0$ and since $\varsigma_k(\mf u^+, s) \leq 0$ then $s \leq \underline s (\mf u^+)$ and $\xi_k \ge \underline s(\mf u^+) - s \ge 0$.  We also point out that $\varsigma_k (\mf t_k (\mf u^+, s), \xi_k) \leq 0$ owing to~\eqref{e:ceblayer}.  
We thus have
\be \label{e:gi}
       [\varsigma_k (\mf u^+, s +\xi_k)]^+ = \varsigma_k (\mf u^+, s +\xi_k  \leq
      \varsigma_k (\mf u^+, s +\xi_k)-\varsigma_k (\mf t_k (\mf u^+, s), \xi_k) := g(s, \xi_k)
\eq
We now point out that the function $g$ is of class $C^2$ and satisfies 
\be \label{e:gi2}
    g(0, \xi) =0 \;
    \text{for every $\xi$}, \qquad  
    g(s, 0) =   \varsigma_k (\mf u^+, s)-\lambda_k (\mf t_k(\mf u^+, s)) \stackrel{\text{Lax}}{\leq} 0 \; \text{for $s \ge  0$}. 
\eq
If $s\ge 0$ we can then argue as follows: since $g (s, \xi_k) \ge 0$ and $g(s, 0) \leq 0$ then there must be $\xi(s) \in [0, \xi_k]$ such that 
$g(s, \xi(s))=0$. By using~\eqref{e:gi2} and arguing as in the proof of~\cite[Lemma 2.5]{Bressan} we then get 
\begin{equation*} \begin{split}
   | g (s, \xi_k)| & \leq \int_{ \xi(s)}^{\xi_k} 
   \left| \frac{\partial g}{\partial \xi} (s, z)\right| dz   \leq \unpo  \int_{ \xi(s)}^{\xi_k} |s| dz \leq \unpo |s| |\xi_k -\xi(s)| \leq \unpo |s| |\xi_k| \; \text{for $s\ge 0$}.
\end{split}
\end{equation*}
This allows us to control the last term at the right-hand side of~\eqref{e:thirdt} and concludes the proof of~\eqref{e:deltaFc2} in the case $s \ge 0$.
If $s <0$ we recall that $s'_k>0$ because we are in case i). Owing to~\eqref{e:me}, this implies that, if $s + \xi_k \leq 0$, then 
$$
   |s_k + \xi_k| \leq |s| ([\varsigma_k (\mf u^+, s) ]^- + |\xi_k|   ) 
$$ 
and by plugging this estimate into~\eqref{e:thirdt} we arrive at~\eqref{e:deltaFc2}. We are thus left to control the last term at the right hand-side of~\eqref{e:thirdt} in the case $s <0$, $s+ \xi_k>0$, which yields $\xi_k > -s >0.$ We now consider the function 
$$
    h(s, \xi): = [s + \xi] g(s, \xi),
$$
where $g$ is the same as in~\eqref{e:gi} and point out that $h$ satisfies 
$$
    h(0, \xi)=0, \quad h(s, -s) = 0 \quad \text{for every $\xi$, $s$}
$$
and by arguing as in the proof of~\cite[Lemma 2.5]{Bressan} we then get 
$$
    |h(s, \xi)| \leq \unpo |s| |\xi +s| \leq \unpo |s| |\xi|  \quad \text{for $\xi >0>s$}
$$
and this allows us to control the last term at the right-hand side of~\eqref{e:thirdt} and concludes the proof of~\eqref{e:deltaFc2}.
\item[ii)] we use~\eqref{e:deltaFc} and to control the first and second term at the right-hand side  we use~\eqref{e:firstt} and~\eqref{e:secondt}, respectively. To control the third term at the right hand-side of~\eqref{e:deltaFc} we recall that, since we are in case ii),  $\bar s (\mf{\hat u}) \leq s'_k <0$ and  hence 
$$
   0 \leq \lambda_k (\mf t_k (\mf{\hat u}, s'_k) ) \leq \lambda_k (\mf{\hat u})
$$
and by using~\eqref{e:dflux} we get 
\begin{equation*}
\begin{split}
      |\mf f (\mf v (\mf{\hat u})) & - \mf f (\mf v (\boldsymbol{\zeta}_k (\mf{\hat u}, s'_k) ))| 
      \leq \unpo |s'_k| [\lambda_k (\mf{\hat u})]^+\\ &
     = \unpo |s'_k| \big| [\lambda_k (\mf{\hat u})]^+ -[ \lambda_k (\mf u^+)]^+ \big|
     \leq \unpo  |s'_k| |\mf{\hat u} - \mf u^+|  \\ & \stackrel{\eqref{e:hatu},\eqref{e:me}}{\leq} \unpo \delta
    |s| ([\varsigma_k (\mf u^+, s) ]^- + |\xi_k|   )
\end{split}
\end{equation*}
In the above expression we have used the equality $[ \lambda_k (\mf u^+)]^+=0$: this follows from the fact that the wave-front $\alpha$ has negative speed and the Lax entropy condition (if $s>0$) and the equality~\eqref{e:speedrare} (if $s<0$). 
\item[iii)]  the proof is the same as at the previous point, it suffices to replace $s'_k$ with $\bar s (\mf{\hat u})$ and use~\eqref{e:pallido2}.
\item[iv)] we proceed as in point i)
\item[v), vi)] in these cases there is no wave front of the $k$-th family entering the domain after the interaction and hence we have  $\Delta F(\tau)  = | \mf f (\mf v (\mf u^-))-  \mf f (\mf v (\hat{\mf u})) |$ and we can argue as in~\eqref{e:deltaFc}.  
\end{itemize}
To establish~\eqref{e:deltaFc2} in the case where we use the simplified boundary Riemann solver instead of the accurate Riemann solver we can use the exact same analysis since the trace of the approximate solution after the interaction is the same as in the case where we use the accurate Riemann solver. \\
{\sc Case III:} at time $\tau$ the datum $\mf u_{bm}$ has a discontinuity. We term $\mf u_{bm}^+$ and $\mf u_{bm}^-$ the right and left limits $\mf u_{bm}(\tau^+)$ and $\mf u_{bm}(\tau^-)$, respectively, $\mf u^+$ and $\mf{\bar u}$ the traces of $\mf u_m(t, \cdot)$ for $t= \tau^-$ and $t=\tau^+$, respectively, and $\mf{\hat u}$ 
the same state as in~\eqref{e:hatu}. Assume we have established the inequality 
\be \label{e:deltaFd2}
      \Delta F(\tau) \leq \unpo |\mf u_{bm}^+ - \mf{u}_{bm}^-|;  
\eq
then owing to~\eqref{e:jdupsilon} we can choose the constant $K_4$ in~\eqref{e:Lambda} in such a way that $\Delta \Upsilon (\tau) \leq 0$. To establish~\eqref{e:deltaFd2} we point out that  
\begin{equation}     \label{e:deltaFd} 
\begin{split}
    \Delta F(\tau) & = | \mf f (\mf v (\mf u^+))-  \mf f (\mf v (\mf{\bar u}))   |   \leq 
   | \mf f (\mf v (\mf u^+))-  \mf f (\mf v (\mf{\hat u}))| +
    |\mf f (\mf v (\mf{\hat u})) -  \mf f (\mf v (\mf{\bar u}))  )| \\
    & \stackrel{\eqref{e:jd1}}{\leq} \unpo |\mf u_{bm}^+ - \mf{u}_{bm}^-| +    |\mf f (\mf v (\mf{\hat u})) - \mf f (\mf v (\mf{\bar u}))  |
    \end{split}
\end{equation}
To control the second term at the right-hand side of the above expression we point out that, by separately considering cases i),$\dots$,vi) in \S\ref{sss:casignl} and \S\ref{ss:abriegnl} we conclude that $\mf{\bar u} = \mf t_k (\mf{\hat u}, \widetilde s_k)$, where the value attained by $\widetilde s_k$ in the various cases is as follows:
\begin{itemize}
\item[i)] $\widetilde s_k = s'_k>0$;
\item[ii)] $\widetilde s_k = s'_k \leq 0$;
\item[iii)] $\widetilde s_k = \bar s(\mf{\hat u}) \leq 0$, and hence $s'_k < \widetilde s_k \leq 0$;
\item[iv)] $\widetilde s_k = s'_k>0$;
\item[v), vi)] $\widetilde s_k =0$;
\end{itemize}
If $\widetilde s_k =0$ then the second term at the right-hand side of~\eqref{e:deltaFd} vanishes and this completes the proof of~\eqref{e:deltaFd2}. 
If $\widetilde s_k > 0$, then $\widetilde s_k = s'_k$ and by applying Lemma~\ref{l:dflux}  we get 
\begin{equation}   \label{e:luce}
\begin{split}
       | \mf f (\mf v (\mf{\hat u})) - & \mf f (\mf v (\bar{\mf u}))| \leq 
     \unpo   [\sigma_k (\mf{\hat u}, \widetilde s_k) ]^+ |\widetilde s_k|  \stackrel{\widetilde s_k = s'_k}{=}  \unpo   [\sigma_k (\mf{\hat u}, s'_k) ]^+ |s'_k|  \\ &
     \leq \unpo \delta | [\sigma_k (\mf{\hat u}, s'_k) ]^+ -  [\sigma_k (\mf u^+, s'_k) ]^+ |   +
     \unpo \delta |  [\sigma_k (\mf u^+, s'_k) ]^+-  [\sigma_k (\mf u^+, \xi_k) ]^+ | \\
     & \stackrel{\eqref{e:jd1}}{\leq} \unpo \delta  |\mf u_{bm}^+ - \mf{u}_{bm}^-|. 
 \end{split}
\end{equation}
To establish the second inequality in the previous expression, we have used the equality $ [\sigma_k (\mf u^+, \xi_k) ]^+=0$,  which follows from~\eqref{e:protra} applied with $\mf u^+$ and $\xi_k$ in place of $\mf{\bar u}$ and $\xi'_k$, respectively. By plugging~\eqref{e:luce} into~\eqref{e:deltaFd} we establish~\eqref{e:deltaFd2}.
If $\widetilde s_k <0$, then by applying Lemma~\ref{l:dflux} we get 
\be \label{e:widetildestime}
    | \mf f (\mf v (\mf{\hat u})) -  \mf f (\mf v (\mf {\bar u}))| \leq 
     \unpo |\widetilde s_k| \max \{ |\lambda_k (\mf{\hat u})|, |\lambda_k  (\boldsymbol{\zeta}_k (\mf{\hat u}, \widetilde s_k) )| \} .
\eq
Next, we recall that the inequality  $\widetilde s_k <0$ implies that we are in either case ii) or case iii) in \S\ref{ss:arie}, and this in turn yields 
\be \label{e:sonpiccini}
   |\widetilde s_k| \leq |s'_k|, \quad \max \{ |\lambda_k (\mf{\hat u})|, |\lambda_k  (\boldsymbol{\zeta}_k (\mf{\hat u}, \widetilde s_k) )| \} = |\lambda_k (\mf{\hat u})| = 
   \lambda_k (\mf{\hat u}) >0. 
\eq
If $\lambda_k (\mf u^+)\leq 0$ we can then control the right-hand side of~\eqref{e:widetildestime} by arguing as follows: 
$$
   |s'_k| [\lambda_k (\mf{\hat u})]^+\leq |s'_k| |[\lambda_k (\mf{\hat u})]^+ - [ \lambda_k (\mf u^+)]^+ | \leq 
   \unpo \delta   |\hat{\mf u} -\mf u^+ | \stackrel{\eqref{e:jd1}}{\leq}
   \unpo \delta |\mf u_{bm}^+ - \mf{u}_{bm}^-|. 
$$
If $\lambda_k (\mf u^+)>0$ then $\xi_k=0$ by~\eqref{e:protra} applied with $\mf u^+$ and $\xi_k$ in place of $\mf{\bar u}$ and $\xi'_k$, respectively, and then owing to~\eqref{e:jd1} we get $|s'_k| \leq C_6  |\mf u_{bm}^+ - \mf{u}_{bm}^-|$ and by~\eqref{e:sonpiccini} this allows us to control the right hand-side of~\eqref{e:widetildestime}. This concludes the proof of~\eqref{e:deltaFd2} and hence of Lemma~\ref{t:tvf} in the genuinely nonlinear case. 
\end{proof}
\begin{corol}
Under the same assumptions as in the statement of Theorem~\ref{t:main} we have~\eqref{e:flupointwise}. 
\end{corol}
\begin{proof}
 We point out that that $\mf f \circ \mf v_{\ee_m} (\cdot, 0)$ is uniformly bounded because so is $\mf u^\ee$ by~\eqref{e:settedue}. Owing to Helly-Kolmogorov Compactness Theorem, the total variation bound~\eqref{e:tvf} implies that $\mf f \circ \mf v_{\ee_m} (\cdot, 0)$ converges up to subsequences in $L^1_{\mathrm{loc}}(\R_+)$ to some limit function $\boldsymbol{\rho}$. By the analysis in~\cite[p.144-5]{Bressan}, we have 
$$
   \lim_{m \to + \infty} \left( \int_0^{+\infty}  \int_0^{+\infty} \partial_t \varphi \ \mf g \circ  \mf v_{\ee_m} (t, x) + \partial_x \varphi \ \mf f \circ \mf v_{\ee_m} (t, x) dx dt + 
   \int_0^{+\infty}  \varphi (t, 0) \ \mf f \circ \mf v_{\ee_m} (t, 0)  dt \right) =0
$$
for every $\varphi \in C^\infty_c (]0, + \infty[ \times \R_+)$ and this yields $\boldsymbol{\rho} = \mf f \circ \mf v(\cdot, 0)$ using the fact that $\mf v$ is of bounded total variation and hence has a trace at $x=0$, where $\mf v = \mf v \circ \mf u$ and $\mf u$ is the same as in~\eqref{e:81}. Also, up to subsequences we have~\eqref{e:flupointwise}.  
\end{proof}
\subsection{Construction and analysis of the set $\widetilde E$ in the genuinely nonlinear case} \label{ss:tildeE}
The set $\widetilde E \subseteq \R_+$ is defined as in~\eqref{e:tildeE}. We recall that the set $D$ is the set where the boundary datum $\mf v_b$ is discontinuous and we now rigorously define the sets
$F$, $H$ and $G$.  
\subsubsection{Construction of the set $F$}\label{ss:F} 
For every $m \in \mathbb N$ we construct a sequence of measures $\mu_m \in \mathcal M_\mathrm{loc}(\R_+)$ by proceeding as follows: first, 
we term $\tau_1, \dots, \tau_{n_m}$ the times at which a wave front of the $\ee_m$-wave front-tracking approximation $\mf u^{\ee_m}$ hits the boundary;
next,  for every $\tau_\alpha$, $\alpha=1, \dots, n_m$, we term $j_\alpha$ and $s_\alpha$ the family and the size of the hitting wave, and set 
\be \label{e:etalpha}
     \chi_\alpha: = 
     \left\{
     \begin{array}{ll}
     A|s_\alpha| & j_\alpha <k \\
     |s_\alpha| (   [\varsigma_\alpha(\mf u_\alpha, s_\alpha)]^- + |\xi_k|) & j_\alpha =k. \\
     \end{array}
     \right.
\eq
Finally,  we set 
\be \label{e:mum}
   \mu_m : = \sum_{\alpha=1}^{n_m} \chi_\alpha \delta_{t = \tau_\alpha},
\eq
where $ \delta_{t = \tau_\alpha}$ denotes the Dirac delta concentrated at $\tau_\alpha$. 
Note that by using the specific piecewise constant structure of $\mf u^\ee$ and~\eqref{e:msigma} we get  
$$
    |\mu_m| (]0, T[) \leq \unpo \mathrm{D \mf u^\ee} (]0, T[\times \R_+) 
    \stackrel{\eqref{e:settedue}}{\leq}
   \unpo  T \delta, \quad \text{for every $T>0$}.  
$$
This implies that, up to subsequences, $\mu_m$ converges weakly-$^\ast$ in the sense of measures on $\R_+$ to some limit measure $\mu\in \mathcal M_\mathrm{loc}(\R_+)$ as $m \to + \infty$. We term $F$ the set of atoms of $\mu$, that is 
\be \label{e:F}
    F : = \{ t \in \R_+: \mu(\{ t\}) >0 \}.
\eq
Note that $F$ is at most countable (and henceforth $\mathcal{L}^1$-negligible) by the local finiteness of $\mu$. 
\subsubsection{Construction of the set $G$} \label{sss:G} 
We employ the same construction of \emph{maximal $\nu_{n}$-shock fronts} of the $k$-th family as in~\cite[page 219, step 2]{Bressan}. In particular, this yields the construction of a countable family $\{ y_n \}_{n \in \mathbb N}$ of shock curves  for the limit function  $\mf u$, obtained as limits of \emph{maximal $\nu_{n}$-shock fronts} of the approximate solution, which has strength larger of equal than $\nu_n /2$. Note that this shock curves are contained in the \emph{closed} set $[0, + \infty[ \times [0, + \infty[$ and that there may be $0$-speed shocks exactly located at the origin. 
We set 
\be \label{e:Gi}
    G: = \big\{t \in \R_+:  \, y_n (t) =0, \; y'_n(t) =0 \; \text{for some $n \in \mathbb N$} \big\}. 
\eq
\subsubsection{Construction of the set $H$} \label{sss:H}
We consider the same measure $\mu^{IC} \in \mathcal M_{loc} (\R_+ \times [0, + \infty[)$ of interaction and cancellation as in~\cite[formula (10.68)]{Bressan}, constructed by using~\cite[formula (7.96)]{Bressan}. 
We term $H$ the set of atoms of $\mu^{IC}$ on the $x=0$ axis, namely  
\be \label{e:H}
     H : = \{ t \in \R_+:  \; \mu^{IC} (\{ (t, 0) \})>0 \}. 
\eq
Note that $H$ is at most countable (and henceforth $\mathcal{L}^1$-negligible) by the local finiteness of the measure $\mu^{IC}$.
\subsubsection{Analysis of the set $\widetilde E$}
 \begin{lemma} \label{l:tildeE}
Let $\widetilde E$ be the same as in~\eqref{e:tildeE} and set 
\be \label{e:traccialimite}
     \mf{\bar u}(t) := \lim_{x \to 0^+} \mf u(t, x),
\eq
where $\mf u$ is the same as in~\eqref{e:81}, and 
\be \label{e:baruemme}
 \mf{\bar u}_m (t): = \lim_{x \to 0} \mf u_{\ee_m}(t, x). 
\eq
If $t \notin \widetilde E$ then  
\be \label{e:Etilde} 
    \lim_{m \to + \infty} \mf{\bar u}_m (t) =   \mf{\bar u} (t). 
\eq
\end{lemma}
\begin{proof}
We fix $t \notin \widetilde E$, recall footnote~\footref{foot2} at page~\pageref{foot2}, argue by contradiction and assume that there is a constant $\iota$ such that 
\be \label{e:iota}
    |  \lim_{m \to + \infty} \mf{\bar u}_m (t)  - \lim_{x \to 0} \mf u(t, x)   | \ge \iota >0. 
\eq
Next, we proceed according to the following steps. \\
{\sc Step 1:} we point out that~\eqref{e:81} implies that, up to subsequences, $\mf u_{\ee_m} (t, x)$ converges to $\mf u(t, x)$ for a.e. $x \in \R_+$. Owing to~\eqref{e:iota}, this in turn implies that there are two sequences $\{ x_m \} $ and $\{ y_m\}$ such that $\mf u_{\ee_m}$ is continuous at both $x_m$ and $y_m$ and 
\be \label{e:iota2}
    y_m > x_m \ge 0, \;  |\mf u_{\ee_m} (t, x_m) - \mf u_{\ee_m}(t, y_m) | \ge \frac{\iota}{2} \; \text{for every $m$}, \quad \lim_{m \to + \infty} y_m =  \lim_{m \to + \infty} x_m =0. 
\eq
{\sc Step 2:} by recalling the structure of the function $\mf u_{\ee_m}$ we conclude that the last inequality in~\eqref{e:iota2} implies that the space interval $]x_m, y_m[$
is crossed at time $t$ by an amount of wave fronts such that their combined strength satisfies 
$$
   \sum_{\substack{\text{$\alpha$ wave in}\\ \, ]x_m, y_m[} } |s_\alpha| \ge  \frac{L \iota}{2} 
$$
for a suitable constant $L>0$. 
Up to subsequences, we can then assume that at least one of the following cases is verified. \\
{\sc Case 1:} the wave fronts of the families $1, \dots, k-1$ satisfy
$$
   \sum_{\substack{\text{$\alpha$ wave in}\\ \, ]x_m, y_m[ \\ i_\alpha =1, \dots, k-1} } |s_\alpha| \ge  \frac{L \iota}{10} .
$$
Then there are two possibilities:
\begin{itemize}
\item[i)] the combined strength of the wave fronts hitting the boundary in the interval $]t, t +y_m/ c[$, where $c$ is the same as in~\eqref{e:sh}, is at least 
$L\iota/20$. By recalling~\eqref{e:etalpha} and~\eqref{e:mum} this implies
\be \label{e:abbagrande}
    \mu_m \left( \left[t, t + \frac{y_m}{c}\right[\right) \ge  \frac{L A \iota}{20} \quad \text{for every $m \in \mathbb N$}.
\eq
We claim that~\eqref{e:abbagrande} yields $\mu(\{t\})>0$, that is $t \in F \subseteq \widetilde E$, and thus contradicts the assumption $t \notin \widetilde E$. To show that $\mu(\{t\})>0$ we 
use the outer regularity of $ \mu$ (see~\cite[Proposition 1.43 (ii)]{AmbFuPal}), which in particular implies that to show that $\mu(\{t\})>0$ it suffices to show that 
\be 
\label{e:outreg}
       \mu(\mathcal O) \ge  \frac{L A \iota}{20}, \quad \text{for every open set $\mathcal O$ such that $t \in \mathcal O$}. 
\eq
We now fix $\widetilde m$ such that  $[t, t + y_{\widetilde m}/c] \subseteq \mathcal O$, which implies $\mu ([t, t + y_{\widetilde m}/c]) \leq \mu (\mathcal O)$. 
Owing to the upper semicontinuity of the measure of compact sets with respect to weak$^\ast$ convergence (see~\cite[formula (1.8)]{AmbFuPal}) we also have that 
\be \label{e:limsup}
    \limsup_{m \to + \infty} \mu_m ([t, t + y_{\widetilde m}/c]) \leq \mu ([t, t + y_{\widetilde m}/c]) \leq \mu (\mathcal O). 
\eq
By extracting, if needed, a subsequence from $\{ y_m \}$, we can assume that  $[t, t + y_m /c] \subseteq [t, t + y_{\widetilde m}/c]$ for every $m \ge \widetilde m$, which implies 
$$
     \frac{L A \iota}{20}  \stackrel{\eqref{e:abbagrande}}{\leq} \limsup_{m \to + \infty} \mu_m ([t, t + y_{ m}/c]) \leq   \limsup_{m \to + \infty} \mu_m ([t, t + y_{\widetilde m}/c]) 
     \stackrel{\eqref{e:limsup}}{\leq}
    \mu (\mathcal O),
$$
which eventually yields~\eqref{e:outreg}.
\item[ii)] an amount of wave fronts of total strength at least $L\iota/20$ is cancelled by the interactions with other waves on the set $]t, t+ y_m /c[\times ]0, y_m[$. By recalling~\cite[formula (7.96)]{Bressan} this implies that 
\be \label{e:andrea}
\mu^{IC}_m ([t, t+ y_m /c]\times [0, y_m])\ge L\iota/20.  
\eq
and by arguing as in item i) above this implies $\mu^{IC}(\{ (t, 0) \}) >0$ and hence $t \in H \subseteq \widetilde E$, which contradicts the assumption $t \notin \widetilde E$. 
\end{itemize}
{\sc Case 2:} the wave fronts of the families $k+1, \dots, N$ satisfy
$$
   \sum_{\substack{\text{$\alpha$ wave in}\\ \, ]x_m, y_m[ \\ i_\alpha =k+1, \dots, N} } |s_\alpha| \ge  \frac{L \iota}{10} .
$$
Then there are two possibilities:
\begin{itemize}
\item[i)] an amount of waves of combined strength at least $L \iota/20$ is generated on the interval $]t-y_m/c, t[$, where $c$ is the same as in~\eqref{e:sh}, either at discontinuity points of $\mf u_{bm}$ or when a wave hits the boundary. By arguing as in {\sc Case 1}, item i) we conclude that either $\mu^{IC}(\{ (t, 0) \}) >0$ or $t$ is a discontinuity point of $\mf u_b$. Either way, we find a contradiction with  the definition~\eqref{e:tildeE} of $\widetilde E$ and the assumption $t \notin \widetilde E$. 
\item[ii)] an amount of wave fronts of total strength at least $L\iota/20$ is generated by the interactions with other waves on the set $]t - y_m/c, t[\times ]0, y_m[$. By recalling~\cite[formula (7.96)]{Bressan} this implies that $\mu^{IC}_m ([t - y_m/c, t]\times [0, y_m])\ge L\iota/20$ and by arguing as in item i) of {\sc Step 1} this in turn implies $\mu^{IC}(\{ (t, 0) \}) >0$, which contradicts the assumption $t \notin \widetilde E$. 
\end{itemize}
{\sc Case 3:} there is $\upsilon>0$ such that the combined amount of the waves of the $k$-th family with negative speed $\upsilon$-bounded away from $0$ satisfies  
$$
   \sum_{\substack{\text{$\alpha$ wave in}\\ \, ]x_m, y_m[ \\ i_\alpha =k, \, \varsigma_\alpha <-\upsilon} } |s_\alpha| \ge  \frac{L \iota}{10} .
$$
Then there are three possibilities:
\begin{itemize}
\item[i)] the combined strength of the wave fronts hitting the boundary in the interval $]t, t +2y_m/\upsilon [$ is at least 
$L\iota/30$. This is ruled out by arguing as in {\sc Case 1}, item i). 
\item[ii)] an amount of wave fronts of total strength at least $L\iota/30$ is cancelled by the interactions with other waves on the set $]t, t+ 2 y_m /\upsilon[\times ]0, y_m[$. This is ruled out by arguing as in {\sc Case 1}, item i). 
\item[iii)] in the time interval $]t, t+ 2 y_m/\upsilon[$ an amount of waves of total strength at least $L\iota/30$ increases its speed from at most $-\upsilon$ to at least $-\upsilon/2$ owing to the interaction with other waves. By the Lipschitz continuity of the map $\mf u \mapsto \lambda_k (\mf u)$ this implies that 
$\mu^{IC}_m([t, t+ 2 y_m/\upsilon] ~\times~ [0, y_m ])~\ge~\unpo L\iota \upsilon~>~0$ and by arguing as in item i) of {\sc Step 1} this in turn implies $\mu^{IC}(\{ (t, 0) \}) >0$, which contradicts the assumption $t \notin \widetilde E$. 
\end{itemize}
{\sc Case 4:} there is $\upsilon>0$ such that the combined amount of the waves of the $k$-th family with positive speed $\upsilon$-bounded away from $0$ satisfies  
$$
   \sum_{\substack{\text{$\alpha$ wave in}\\ \, ]x_m, y_m[ \\ i_\alpha =k, \, \varsigma_\alpha>\upsilon} } |s_\alpha| \ge  \frac{L \iota}{10} .
$$
This case is ruled out by combining the arguments in {\sc Case 2} and {\sc Case 3}.\\
{\sc Case 5:} there is a sequence $\eta_m$ such that $\eta_m>0$ for every $m$, $\lim_{m \to + \infty} \eta_m=0$ and 
\be \label{e:case5}
   \sum_{\substack{\text{$\alpha$ wave in}\\ \, ]x_m, y_m[ \\ i_\alpha =k, \, |\varsigma_\alpha|<\eta_m} } |s_\alpha| \ge  \frac{L \iota}{10} .
\eq
Assume for a moment that there are $\eta>0$,  $\nu>0$ such that 
$$
   \sum_{\substack{\text{$\alpha$ wave in}\\ \, ]x_m, y_m[ \\ i_\alpha =k+1, \dots, N} } |s_\alpha|
   + \sum_{\substack{\text{$\alpha$ wave in}\\ \, ]x_m, y_m[ \\ i_\alpha =1, \dots, k-1} } |s_\alpha| 
   + \sum_{\substack{\text{$\alpha$ wave in}\\ \, ]x_m, y_m[ \\ i_\alpha =k, |\varsigma_\alpha| \ge \eta} } |s_\alpha| \ge \nu,
$$
then we can argue as in {\sc Cases $1, \dots, 4$} before. Hence, we can assume with no loss of generality that the 
strength of all the waves different from the one in~\eqref{e:case5} vanishes in the $m \to + \infty$ limit.  
Up to subsequences, we know that one of the following two possibilities is verified. \\
{\sc Case 5A:} there is $ \nu>0$ such that for every $m \in \mathbb N$ there is at least one front comprised between $x_m$ and $y_m$ with strength greater or equal than $ \nu>0$ and speed in modulus smaller or equal than $\eta_m$. 
Since owing to Lemma~\ref{l:raref} the maximal strength $C_9 r_{\ee_m}$ of the rarefaction fronts vanishes in the $m \to + \infty$ limit, this front must be a shock front.
By the construction of maximal $\nu$-shock fronts given in~\cite[p. 219]{Bressan} and the definition~\eqref{e:Gi} of $G$, this implies $t \in G \subseteq \widetilde E$ and hence contradicts the assumption $t \notin \widetilde E$. \\
{\sc Case 5B:} the maximal strength of the wave-fronts confined between $x_m$ and $y_m$ vanishes in the $m \to + \infty$ limit. 
Since the strength of all the waves of the $j$-th family, $j \neq k$, vanishes in the  $m \to + \infty$ limit, the values $\mf u_{\ee_m} (t, x)$ with $x \in ]x_m, y_m[$  get 
closer and closer to the curve $\mf t_k (\mf u_{\ee_m} (t, x_m), \cdot)$ in the  $m \to + \infty$ limit.  Since also the  maximal speed of the wave fronts of $\mf u_{\ee_m}(t, \cdot)$ comprised between  $x_m$ and $y_m$ vanishes in the $m \to + \infty$, we  conclude that 
$$
    \lim_{m \to + \infty} |\mf u_{\ee_m} (t, x) - \mf t_k \big(\mf u_{\ee_m} (t, x_m), \bar s (\mf u_{\ee_m} (t, x_m)\big)   | =0 \quad \text{for every $x \in [x_m, y_m] $}, 
$$
which contradicts~\eqref{e:iota2}. \qedhere 
\end{proof}
\subsection{Definition~\ref{d:equiv} in the case $t \notin \widetilde E$}\label{ss:notinE}
In this paragraph we verify the boundary condition in the statement of Theorem~\ref{t:maintr} on the set $\R_+ \setminus \widetilde E$, where $\widetilde E$ is the same as in~\eqref{e:tildeE}. In other words, we show that for a.e. $t \in \R_+$, $t \notin \widetilde E$, we have  $\mf{\bar v}(t) \sim_{\mf D} \mf v_b(t)$ in the sense of Definition~\ref{d:equiv}, provided $\mf{\bar v}(t)$ is the  trace of the limit function $\mf v$ as in the statement of Theorem~\ref{t:maintr}. This is a consequence of Lemma~\ref{l:bl} and Lemma~\ref{l:bl2} below.

For simplicity in this paragraph and in the following one we assume~\eqref{e:a11nz}. The case of assumption C) in Hypothesis~\ref{h:eulerlag} is analogous. The case of assumption B) is slightly more delicate and is explicitly discussed in~\S\ref{sss:952} since the most prominent example where assumption B) is satisfied are the Navier-Stokes and viscous MHD equations in Eulerian coordinates when the fluid velocity vanishes, and in that case the $k$-th characteristic field is linearly degenerate.

Let $\mf{\bar u}_m$ be the same as in~\eqref{e:baruemme} then owing to~\eqref{e:brs0} we have 
$$
   \boldsymbol{\beta} (\boldsymbol{\phi} (\mf{\bar u}_m(t), \xi^m_{\ell +1}(t), \dots, \xi^m_k(t)), \mf u_{b}^{\ee_m}(t) )= \mf 0_{N} 
   \quad \text{for a.e. $t \in \R_+$}
$$
and suitable functions $\xi^m_{\ell+ 1} (t), \dots, \xi^m_k(t)$ satisfying $|\xi^m_i| \leq \delta$ for every $i=\ell+ 1, \dots, k$ and $m \in \mathbb N$. Note that owing to~\eqref{e:51} we have that, up to subsequences, 
 \be \label{e:converge2}
    \mf u_{b}^{\ee_m}(t) \to \mf u_b (t), \quad \text{for a.e. $t \in \R_+$}. 
\eq
By combining~\eqref{e:converge2} and~\eqref{e:Etilde} with the Implicit Function Theorem we conclude that for a.e. $t \notin \widetilde E$ we have that $\xi^m_{\ell+ 1}(t), \dots, \xi^m_k(t)$ converge as $m \to + \infty$ to suitable functions $\xi_{\ell+1}(t), \dots, \xi_k(t)$ satisfying 
\be \label{e:ubiti}
   \boldsymbol{\beta} (\boldsymbol{\phi} (\mf{\bar u}(t), \xi_{\ell +1} (t), \dots, \xi_k(t)), \mf u_{b}(t) )=  \mf 0_{N}  
   \quad \text{for a.e. $t \in \R_+$}, 
\eq
provided $\mf{\bar u}$ is the same as in~\eqref{e:traccialimite}.
\begin{lemma}
\label{l:bl}
 Let $\mf{\bar u}(t)$ be the same as in~\eqref{e:traccialimite} and $\xi_k(t)$ the same as in~\eqref{e:ubiti}. For a.e. $t \in \R_+$, $t \notin \widetilde E$, satisfying $\xi_k (t)\neq \underline s (\mf{\bar u}(t))$, we have that $\mf{\bar v}(t): = \mf v(\mf{\bar u}(t))$ satisfies $\mf{\bar v} (t)\sim_{\mf D} \mf v_b(t)$ in the sense of Definition~\ref{d:equiv}.In particular, property ii) is satisfied provided $\mf{\underline v} = \mf{\bar v} (t)$.
\end{lemma}
\begin{proof}
We recall Remark \ref{r:1812} and in particular that the approximation $\mf u_{\ee_m}$ satisfies~\eqref{e:protra} with $\mf u^+= \mf{\bar u}_m(t)$ and $\xi_k= \xi_k^m(t)$ respectively. Up to subsequences we can assume that one of the following condition holds true: 
\begin{itemize}
\item[a)] $\lambda_k (\mf{\bar u}_m(t)) > 0$ and $\xi^m_k(t) =0$ for every $m$;
\item[b)] $\lambda_k (\mf{\bar u}_m(t)) \leq 0$ and $\xi^m_k (t)\leq \underline s (\mf{\bar u}_m)(t)$ or every $m$.
\end{itemize}
Condition a) above implies $\xi_k(t)=0$, which by~\eqref{e:p} yields $\boldsymbol{\psi}_p (\mf{\bar u}(t), \xi_{\ell +1} (t), \dots, \xi_{k-1}(t), 0)= \mf 0_N$ and by plugging this equality into~\eqref{e:bl10} and recalling~\eqref{e:ubiti} we arrive at 
\be \label{e:9392012}
     \boldsymbol{\beta} (\boldsymbol{\psi}_s (\mf{\bar u}(t), \xi_{\ell +1} (t), \dots, \xi_{k-1}(t)), \mf u_{b}(t) )=  \mf 0_{N},
\eq
where $\boldsymbol{\psi}_s$ is defined  in~\eqref{e:phiesse}. By the construction of $\boldsymbol{\psi}_s$,  there is a solution of~\eqref{e:bltw} with initial datum $\boldsymbol{\psi}_s (\mf{\bar u}(t), \xi_{\ell +1} (t), \dots, \xi_{k-1}(t))$  which converges exponentially fast to $\mf{\bar u}(t)$ as $y \to + \infty$. Owing to~\eqref{e:9392012}, this implies that there is a boundary layer $\mf w$  satisfying~\eqref{e:andiamoav2} with $\mf{\underline u} = \mf{\bar u} (t)$ and going back to the variable $\mf v$ we conclude that $\mf{\bar v} (t):= \mf{ v} (\mf{\bar u}(t))$ satisfies property ii) in Definition~\ref{d:equiv} with 
$\mf{\underline v} = \mf{\bar v}(t)$. 

Owing to~\eqref{e:Etilde} condition b) above yields  $\lambda_k (\mf{\bar u}(t)) \leq 0$ and $\xi_k(t) \leq \underline s (\mf{\bar u}(t))$. Since $\xi_k(t) \neq \underline s (\mf{\bar u}(t))$ by assumption,  we conclude that $\xi_k (t)< \underline s (\mf{\bar u}(t))$. If $\lambda_k (\mf{\bar u}(t)) < 0$ then by recalling case vi) in \S\ref{sss:casignl} we conclude that $\mf{\bar v} (t):= \mf{ v} (\mf{\bar u}(t))$ satisfies property ii) in Definition~\ref{d:equiv} with $\mf{\underline v} = \mf{\bar v}(t)$. If $\lambda_k (\mf{\bar u}(t)) = 0$ then  $\underline s (\mf{\bar u}(t))= \bar s (\mf{\bar u}(t))=0$ by~\eqref{e:underlines} and~\eqref{e:baresse}, and by recalling case iii) in \S\ref{sss:casignl}  we conclude the proof of the lemma.
\end{proof}
\begin{lemma}
\label{l:bl2}
Let $\mf{\bar u}(t)$ be the same as in~\eqref{e:traccialimite} and $\xi_k(t)$ the same as in~\eqref{e:ubiti}. For a.e. $t \in \R_+$,  $t \notin \widetilde E$, satisfying $\xi_k (t) = \underline s (\mf{\bar u}(t))$ we have that $\mf{\bar v}(t): = \mf v(\mf{\bar u}(t))$ satisfies $\mf{\bar v} (t)\sim_{\mf D} \mf v_b(t)$ in the sense of Definition~\ref{d:equiv}.In particular, property ii) is satisfied provided $\mf{\underline v} (t)=
\mf v(\boldsymbol{\zeta}_k (\mf{ \bar u}(t), \underline s(\mf{\bar u}(t))))$.  
\end{lemma}
\begin{proof}
We recall cases a) and b) in the proof of Lemma~\ref{l:bl}. In case a), we obtain $\xi_k(t)=0$, which implies $\underline s (\mf{\bar u}(t))=\xi_k(t)=0 $ and hence $\boldsymbol{\zeta}_k (\mf{ \bar u}(t), \underline s(\mf{\bar u}(t)))) = \mf{ \bar u}(t)$. Owing to~\eqref{e:9392012}, and arguing as in the proof of Lemma~\ref{l:bl}, this yields the desired conclusion. 

We now focus on case b). If $\lambda_k (\mf{\bar u}(t)) = 0$ then $\underline s(\mf{\bar u})=0$, which yields on the one hand the equality $\boldsymbol{\zeta}_k (\mf{ \bar u}(t), \underline s(\mf{\bar u}(t)))= \mf{ \bar u}(t)$, on the other 
the equalities $\xi_k(t) =\underline s(\mf{\bar u}(t))=0$ and hence (by arguing as in the proof of Lemma~\ref{l:bl}) the equality~\eqref{e:9392012} and from this we arrive at the desired conclusion. Conversely, if $\lambda_k (\mf{\bar u}(t)) <0$ we can combine~\eqref{e:ubiti} with the same argument as in case v) in \S\ref{sss:casignl} and conclude the proof of the lemma.
\end{proof}
\subsection{Definition~\ref{d:equiv} in the case $t \in \widetilde E$} \label{ss:tinE}
In this paragraph we verify the boundary condition in the statement of Theorem~\ref{t:maintr} on the set $\widetilde E$. Namely, we show that for a.e. $t \in \widetilde E$, we have  $\mf{\bar v}(t) \sim_{\mf D} \mf v_b(t)$ in the sense of Definition~\ref{d:equiv}, provided $\mf{\bar v}(t)$ is the same as in the statement of Theorem~\ref{t:maintr}. Since the sets $D, F$ and $H$ in the definition~\eqref{e:tildeE} of $\widetilde E$ are negligible, this is a consequence of Lemma~\ref{l:bc} below. We recall that in this paragraph, as in the previous one, to ease the exposition we assume~\eqref{e:a11nz}.

 Owing to the construction in \S\ref{sss:G} we have that if $t \in G$ then there is $\nu>0$ and a sequence of maximal $\nu$-shock fronts $\{ y_m \}$ such that $y_m \to y$ in $C^0$ as $m \to + \infty$ for some Lipschitz continuous curve $y$ such that $y(t)=0$, $y'(t)=0$. We now fix $t \in G$ and term $\mf u^-_m(t)$ and $\mf u^+_m(t)$ the left and right state of the shock curve $y_m$ at time $t$. We recall~\eqref{e:traccialimite} and by arguing as in the proof of~\cite[formula (10.80)]{Bressan} and of~\cite[formula (10.66)]{Bressan} we get  
\be \label{e:vallatraccia} 
      \lim_{m \to + \infty} \mf u^+_m (t) = \mf{\bar u}(t), \quad  \lambda_k ( \mf{\bar u}(t)) \leq 0 \; \text{for a.e. $t \in G$.}
\eq
Owing to the second condition in~\eqref{e:vallatraccia}, for a.e. $t \in G$ by setting
\be \label{e:tracciainternalimite}
     \mf{\underline u}(t) = \mf t_k (\mf{\bar u}(t), \underline s(\mf{\bar u}(t)))
\eq
we get that $\mf{\underline u}(t)$ and $ \mf{\bar u}(t)$ are the left and the right state, respectively, of a $0$-speed Lax admissible shock.
\begin{lemma}
\label{l:limitedasx}
Let $\mf{\underline u}(t)$ be the same as in~\eqref{e:tracciainternalimite}, then
\be \label{e:limitedasx}
       \mf{\underline u}(t) =   \lim_{m \to + \infty} \mf u^-_m (t) \; \text{for a.e. $t \in G$.}
\eq 
\end{lemma}
\begin{proof}
By the definition of $\mf u^+_m(t)$ and $\mf u^-_m(t)$ we have 
\be \label{e:sonoconnessi}
      \mf u^-_m (t) = \mf t_k (\mf u^+_m (t), s_{m}(t)) \; \text{for a.e. $t \in G$}
\eq
and for some $s_{m} (t)\in [-\delta, \delta].$ Note that, owing to the definition of $G$ and of $\nu$-shock front,  $s_m(t) \ge \nu(t)/2$ for every $m \in \mathbb N$ and some $\nu(t)>0$. 
We now fix any such $t \in G$ and point out that, up to subsequences (which to ease the notation we do not relabel), $s_{m}(t)$ converges to to some limit $s(t)\ge \nu(t) /2>0$. By the continuity of the function $\sigma_k$ defined by~\eqref{e:rh} and by recalling~\eqref{e:vallatraccia} we arrive at 
$$
    \lim_{m \to + \infty } \sigma_k (\mf u^+_m (t), s_{m}(t)) = \sigma_k (\mf{\bar u}(t), s(t)). 
$$
On the other hand, by the analysis in~\cite[p.227]{Bressan} we get that $  \lim_{m \to + \infty } \sigma_k (\mf u^+_m (t), s_{m}(t))=0$, which combined with~\eqref{e:vallatraccia} yields  
 $s(t) =\underline s(\mf{\bar u}(t))$. By passing to the limit in~\eqref{e:sonoconnessi} and recalling~\eqref{e:tracciainternalimite} we eventually arrive at~\eqref{e:limitedasx}. 
\end{proof}
\begin{lemma}
\label{l:tracciainterna} 
Let $\mf{\bar u}_m (t)$ and $ \mf{\underline u}(t)$ be the same as in~\eqref{e:baruemme} and in~\eqref{e:tracciainternalimite}, respectively; then 
\be \label{e:tracciainterna}
      \mf{\underline u}(t) = \lim_{m \to + \infty} \mf{\bar u}_m (t)  \quad \text{for a.e. $t \in G$}. 
\eq
\end{lemma}
\begin{proof}
We fix $t \in G \setminus F \cup H \cup D$ satisfying~\eqref{e:limitedasx} and establish~\eqref{e:tracciainterna}. Since the sets $F$, $H$ and $D$ are all negligible, this provides the proof of the lemma. We argue by contradiction and assume that there is $\iota>0$ such that 
$$| \lim_{m \to + \infty} \mf{\bar u}_m (t) -  \lim_{m \to + \infty} \mf u^-_m (t)| \ge \iota.
$$ 
In particular, this implies that there is a sequence of points $\{ \widetilde y_m\} $ (to ease the notation, we do not indicate the dependence on $t$) such that 
$\lim_{m \to + \infty } \widetilde y_m= 0 $, $0 < \widetilde y_m< y_m$, $\mf u_{\ee_m}$ is continuous at $(t, \widetilde y_m)$, $|\mf u_{\ee_m} (t, \widetilde y_m) - \mf u^-_m|$ vanishes in the $m \to + \infty$ limit,  and therefore
$$
   |\mf u_{\ee_m} (t, \widetilde y_m) - \lim_{m \to + \infty} \mf{\bar u}_m (t)| \ge \iota/2 \quad \text{for every $m$}. 
$$   
   We can then construct a sequence of points $\{ x_m \}$  such that $0< x_m < \widetilde y_m$, $\mf u_{\ee_m}(t, \cdot)$ is continuous at $x_m$ and 
$\mf u_{\ee_m} (t,  x_m)$ is arbitrarily close to $\mf{\bar u}_m(t)$, and hence 
$$
    |\mf u_{\ee_m} (t, \widetilde y_m) - \mf u_{\ee_m}(t, x_m)| \ge \iota/4 \quad \text{for every $m$.}
$$    
To show that this implies $t \in F \cup H \cup D$ and hence find a contradiction with the assumptions of the present lemma we can then argue as in the proof of Lemma~\ref{l:tildeE}, the only point that requires a different argument is the handling of {\sc Case 5A}. 
Assume therefore that we are in {\sc Case 5}, that all the waves different from the one in \eqref{e:case5} vanish in the $m \to + \infty$ limit and that there is $\widetilde \nu>0$ such that for every $m \in \mathbb N$  there is at least one shock-front of the $k$-th family comprised between $x_m$ (on the left) and $\widetilde y_m$ (on the right) and with strength at least $\widetilde \nu$ and with speed vanishing in the $m \to + \infty$ limit. Up to subsequences, we can then assume that one of the following conditions is satisfied. 
\begin{itemize}
\item The speed of the shock located at $y_m$ is bounded away from $0$, uniformly in $m$. Since the strength of this shock is at least $\nu /2$ by the construction of maximal $\nu$-shock front in~\cite[p. 219]{Bressan} we can handle this case by arguing as in {\sc Case 3} and {\sc Case 4} in the proof of Lemma~\ref{l:tildeE} to conclude that in this case $t \in F \cup H \cup D$, which yields a contradiction. 
\item The speed of the shock located at $y_m$ vanishes in the $m \to + \infty$ limit.  Since $\mf u^-_m$ is the left state of a shock with strength at least $\nu/2$ and speed vanishing in the $m \to + \infty$ limit, the genuine nonlinearity of the $k$-th characteristic field implies that $\lambda_k (\mf u^-_m) \ge  \unpo \nu >0$ for $m$ sufficiently large and this in turn yields  $\lambda_k (\mf u_{\ee_m}(t, \widetilde y_m)) \ge  \unpo \nu > 0$. On the other side, by assumption between $x_m$ and $\widetilde y_m$
 there is a shock with strength at least $\widetilde \nu>0$ and speed vanishing in the $m \to + \infty$ limit. By the Lax entropy condition, this implies that there is a sequence $\{ z_m\} \subseteq ]x_m, \widetilde y_m[$ such that $z_m$ is just at the right of the shock and $ \lambda_k (\mf u_{\ee_m}(t, z_m) \leq  - \unpo \widetilde \nu < 0$ for every $m \in \mathbb N$.  We claim that this contradicts the assumption 
that the total strength of all the waves different from the ones in~\eqref{e:case5} vanishes in the $m \to + \infty$ limit. Indeed, since $z_m < \widetilde y_m$,  $ \lambda_k (\mf u_{\ee_m}(t, z_m) \leq  - \unpo \widetilde \nu < 0$, $\lambda_k (\mf u_{\ee_m}(t, \widetilde y_m)) \ge  \unpo \nu > 0$ and the total strength of the wave-fronts of the $j$-th family, $j \neq k$, in $]z_m, \widetilde y_m[$ vanishes in the  
  in the $m \to + \infty$ limit then there is $\unpo \nu >0$ and a group $\mathcal I_m$ of rarefaction wave-fronts of the $k$-th family  such that 
  $$
       \sum_{\substack{\alpha \in \mathcal I_m \\ i_\alpha =k, s_\alpha <0 } }|s_\alpha| \ge \unpo \nu, \quad \lambda_k (\mf u^+_{m \alpha}) \ge \unpo \nu >0 \; \text{for every $\alpha \in \mathcal I_m$ },
  $$
  where $\mf u^+_{m \alpha}$ represents the state at the right-hand side of the wave-front $\alpha$. This however contradicts the assumption that the total strength of all the waves different from the ones in~\eqref{e:case5} vanishes in the $m \to + \infty$ limit and hence concludes the proof of the lemma.
 $\qedhere$
\end{itemize}
\end{proof}
\begin{lemma}
\label{l:bc}
For a.e. $t \in G$ the trace $\mf{\bar v}(t) : = \lim_{x \to 0^+} \mf v(t, x)$ satisfies $\mf{\bar v} (t) \sim_{\mf D} \mf v_b(t)$, in the sense of Definition~\ref{d:equiv}. 
In particular, conditions i) and ii) are satisfied provided $\mf{\underline v} (t): = \mf v (\mf {\underline u}(t))$ and $ \mf {\underline u} (t)$ is the same as in~\eqref{e:tracciainternalimite} and satisfies~\eqref{e:limitedasx} and~\eqref{e:tracciainterna}.
\end{lemma}
\begin{proof}
We proceed according to the following steps.\\
{\sc Step 1}: we establish property ii) in Definition~\ref{d:equiv}, provided $\mf{\underline v}: = \mf v (\mf {\underline u}(t))$ and $ \mf {\underline u} (t)$  is the same as in~\eqref{e:tracciainternalimite}. By construction, we have that\footnote{We recall that in \S\ref{ss:tinE}, as in \S\ref{ss:notinE}, we assume~\eqref{e:a11nz} and hence focus on case A in Hypothesis~\ref{h:eulerlag}} for a.e. $t \in \R_+$
\be \label{e:xiemmet}
  \boldsymbol{\beta}( \boldsymbol{\phi} (\mf{\bar u}_m(t), \xi^m_{\ell +1}(t), \dots, \xi^m_k(t)),   \mf u^{\ee_m}_{b}(t) )=\mf 0_{N} ,
\eq
with $\mf{\bar u}_m(t)$ and $\xi^m_k(t)$ satisfying~\eqref{e:protra}. Owing to~\eqref{e:converge2} we have that $\lim_{m \to + \infty} \mf u_{bm}(t) = \mf u_b(t)$ for a.e. $t \in \R_+$ and by~\eqref{e:tracciainterna} we can pass to the limit in~\eqref{e:xiemmet} and conclude that for a.e. $t \in G$ we have that
$\xi^m_1(t), \dots, \xi^m_k(t)$ converge as $m \to + \infty$ to some $\xi_1(t), \dots, \xi_k(t)$ satisfying 
\be \label{e:xiemmet2}
   \boldsymbol{\beta}( \boldsymbol{\phi} (\mf{\underline u}(t), \xi_{\ell +1}(t), \dots, \xi_k(t)), \mf u_{b}(t)) = \mf 0_{N}.
\eq
Owing to~\eqref{e:protra} for a.e. $t \in G$ we can assume that up to subsequences one of the following conditions is satisfied for every $m$. 
\begin{itemize}
\item $\lambda_k (\mf{\bar u}_m(t)) > 0$ and $\xi^m_k(t)=0$, which by~\eqref{e:tracciainterna} yields $\lambda_k (\mf{\underline u}(t)) \ge 0$ and $\xi_k(t)=0$. Owing to~\eqref{e:p} and~\eqref{e:bl10}, this implies 
$$
    \boldsymbol{\phi} (\mf{\underline u}(t), \xi_{\ell +1}(t), \dots, \xi_k(t))= 
    \boldsymbol{\psi}_s (\mf{\underline u}(t), \xi_{\ell +1}(t), \dots, \xi_{k-1}(t))
$$
and by plugging this equality into~\eqref{e:xiemmet2}, recalling the definition~\eqref{e:phiesse} of $\boldsymbol{\psi}_s$ and arguing as in the proof of Lemma~\ref{l:bl} this yields the existence of a boundary layer $\mf w$ satisfying~\eqref{e:andiamoav2}. 
\item $\lambda_k (\mf{\bar u}_m(t)) \leq  0$ and $\xi^m_k(t) \leq \underline s(\mf{\bar u}_m(t))$, which  by~\eqref{e:tracciainterna} yields  $\lambda_k (\mf{\underline u}(t)) \leq  0$ and $\xi_k(t) \leq \underline s(\mf{\underline u}(t))$. Since $\mf{\underline u}(t)$ is defined as in~\eqref{e:tracciainternalimite}, with $\lambda_k(\mf{\bar u}(t))\leq 0$ by~\eqref{e:vallatraccia}, we have  $\lambda_k (\mf{\underline u}(t)) \ge  0$. 
We conclude that $\lambda_k (\mf{\underline u}(t)) =0$, which implies $\underline s (\mf{\underline u}(t)) =0$ and hence $\xi_k \leq 0$ and, by~\eqref{e:tracciainternalimite}, $\mf{\underline u}(t) = \mf{\bar u}(t)$.
We conclude that 
\begin{equation*}
\begin{split}
    \boldsymbol{\phi} (\mf{\underline u}(t), \xi_{\ell +1}(t), \dots, \xi_k(t))& \stackrel{\eqref{e:bl10}}{=}
    \boldsymbol{\psi}_s (\underbrace{\boldsymbol{\zeta}_k (\mf{\underline u}(t)}_{=\boldsymbol{\pi}_{\mf u} \circ \mf b_k (\mf{\underline u}(t)) \ \text{by~\eqref{e:uguale1}}}, \xi_k(t)), \xi_{\ell +1}(t), \dots, \xi_{k-1}(t)) \\&
    +   \boldsymbol{\psi}_p (\mf{\underline u}(t), \xi_{\ell +1}(t), \dots, \xi_{k-1}(t), \xi_k(t)).\\
\end{split}
\end{equation*}
We now recall the construction of the Slaving Manifold given in \S\ref{ss:case2}. In particular, we recall~\eqref{e:bl00} and that the solution of the Cauchy problem obtained by coupling~\eqref{e:bltw2} and  the initial datum $\boldsymbol{\omega}(0)$ in~\eqref{e:bl00} satisfies~\eqref{e:slavingman}, where $\boldsymbol{\omega}_0$
is by~\eqref{e:festanatale} the solution of~\eqref{e:bltw2} with initial datum $\boldsymbol{\gamma}_k (\mf{\underline u}(t), \xi_k(t))= \mf b_k (\mf{\underline u}(t), \xi_k(t))$. Owing to Lemma~\ref{l:min02}, the orbit $\boldsymbol{\omega}_0$ converges to $(\mf{\underline u}(t), \mf 0_{N-h})$ as $y \to + \infty$. Wrapping up, we conclude that the solution of the Cauchy problem obtained by coupling~\eqref{e:bltw2} and  $\boldsymbol{\omega}(0)$ in~\eqref{e:bl00} converges to $(\mf{\underline u}(t), \mf 0_{N-h})$ as $y \to + \infty$. Recalling that the function $\mf a$ in~\eqref{e:bltw2} is given by~\eqref{e:gamma}, this yields the existence of a boundary layer $\mf w$ satisfying~\eqref{e:andiamoav2} with asymptotic state 
$\mf{\underline u}(t)$. 
\end{itemize}
{\sc Step 2:} we establish property i) in Definition~\eqref{d:equiv}. We have 
$$
  \mf f(\mf{\underline v}(t))= \mf f \circ \mf v \circ \mf{\underline u}(t)  
  \stackrel{\eqref{e:tracciainterna}}{=} \lim_{m \to + \infty} \mf f \circ \mf v \circ \mf{\bar u}_m (t) \stackrel{\eqref{e:flupointwise},\eqref{e:baruemme}}{=} \mf f (\mf{\bar v}(t)),
$$
and since $\lambda_k (\mf{\bar u}(t)) \leq 0$ by~\eqref{e:vallatraccia} the $0$-speed shock between  $\mf{\underline u}$ (on the left) and $\mf{\bar u}(t)$ (on the right) is Lax admissible. This concludes the proof of Lemma~\ref{l:bc}.
\end{proof}
\subsection{Case of a linearly degenerate case characteristic field} \label{ss:bclindeg}
We now establish the proof of Theorem~\ref{t:maintr} in the case of a linearly degenerate boundary characteristic field~\eqref{e:lindeg}. The proof follows the same outline as in the case of a genuinely nonlinear characteristic field, so we only discuss the points where the argument is different. 
\subsubsection{Proof of Lemma~\ref{t:tvf} (linearly degenerate $k$-th characteristic field)}
The argument is very similar to the one in the case of a genuinely nonlinear characteristic field, so we only provide a sketch. First, we recall \S\ref{sss:532} and also we point out that instead of~\eqref{e:dflux} we have the inequality 
\begin{equation}
\label{e:dflux2}
      | \mf f \circ \mf v (\mf{\widetilde u}) -  \mf f \circ \mf v (\mf t_k (\mf{\widetilde u}, s))|
      \leq
          \unpo  |s| | \lambda_k (\mf{\widetilde u})|. 
\eq
Next, we consider the same functional $\Lambda$ as in~\eqref{e:Lambda}. We recall that now the functional $S$ in $\Upsilon$ is given by the second line of~\eqref{e:S}.  We now discuss {\sc Cases I,$\dots$,III} in the proof of Lemma~\ref{t:tvf}. \\
{\sc Case I:} it suffices to establish~\eqref{e:deltaF2}. To this end, we term $\mf u^-$ and $\mf u^+$ the left and right state of the hitting wave of the $j$-th family, $j<k$, and by $\mf{\hat u}$ the same state as in~\eqref{e:hatu}.  Next, we point out that there are two possibilities:
\begin{itemize}
\item if $\lambda_k (\mf{\hat u})>0$ by following the same computations done  in~\eqref{e:deltaF} for cases i),$\dots$, iii) in \S\ref{ss:abrie}, we arrive at 
$$
    \Delta F(\tau) \leq  \unpo|s_j| +  |\mf f (\mf v (\mf{\hat u})) - \mf f (\mf v (\mf t_k (\mf{\hat u}, s_k) ))| 
$$
 and the second term at the right-hand side of the above expression can be controlled by arguing as in~\eqref{e:cvelocita}.
 \item  if $\lambda_k (\mf{\hat u}) \leq 0$ there is no wave front of the $k$-th family entering the domain after the interaction and hence the last term at the right hand side of~\eqref{e:deltaF} is not there, and in this way we arrive at~\eqref{e:deltaF2}. 
\end{itemize}
{\sc Case II:} we term $\mf u^-$ and $\mf u^+$ the left and right state of the hitting wave front of the $k$-th family, and $\mf{\hat u}$ the same state as in~\eqref{e:hatu}. Assume that we have shown that
\be \label{e:deltaFmarzo}
    \Delta F(\tau) \leq \unpo |s_k| \big( |\xi_k|+ [\lambda_k (\mf u^+)]^- \big);
\eq
then we can choose the constant $K_5$ in such a way that the function $\Lambda$ is decreasing at $\tau$. To establish~\eqref{e:deltaFmarzo} we first assume that $\lambda_k (\mf{\hat u}) > 0$ and we point out that 
\begin{equation*}
\begin{split}
    \Delta F(\tau) & \leq | \mf f (\mf v (\mf u^-))-  \mf f (\mf v (\mf t_k (\mf{\hat u}, s'_k) ))   | \\ &  \leq 
    | \mf f (\mf v (\mf u^-))-  \mf f (\mf v (\mf u^+))   | + | \mf f (\mf v (\mf u^+))-  \mf f (\mf v (\mf{\hat u}))| +
    |\mf f (\mf v (\mf{\hat u})) - \mf f (\mf v (\mf t_k (\mf{\hat u}, s'_k) ))|  \\ &
    \stackrel{\eqref{e:melindeg},\eqref{e:dflux2}}{\leq} 
   [\lambda_k (\mf u^+)]^-|s_k| + \unpo  |s_k|  |\xi_k|+  |\lambda_k (\mf{\hat u})||s'_k|  \\ & \leq 
     [\lambda_k (\mf u^+)]^-|s_k| + \unpo  |s_k|  |\xi_k|+ \delta ( [\lambda_k (\mf{\hat u})]^+
     -  [\lambda_k (\mf u^+)]^+ ) \\ & \leq  [\lambda_k (\mf u^+)]^-|s_k| + \unpo  |s_k|  |\xi_k|+ \delta \unpo |\mf{\hat u} - \mf u^+| 
     \stackrel{\eqref{e:melindeg}}{\leq}   [\lambda_k (\mf u^+)]^-|s_k| + \unpo  |s_k|  |\xi_k|+\delta |s_k| |\xi_k|,
    \end{split}
\end{equation*}
that is~\eqref{e:deltaFmarzo}. If $\lambda_k (\mf{\hat u}) \leq 0$ the estimate is actually simpler because we do not have the last term at the right hand side of the above expression since there is no $k$-th wave front entering the domain after the interaction. \\
{\sc Case III:} as in the proof of Lemma~\ref{t:tvf}  we use the following notation: 
\begin{itemize}
\item $\mf u_{bm}^+$ and $\mf u_{bm}^-$ are the right and left limits $\mf u^{\ee_m}_{b}(\tau^+)$ and $\mf u^{\ee_m}_{b}(\tau^-)$, respectively;
\item $\mf{\bar u}^-$ and $\mf{\bar u}^+$ are the traces of $\mf u_m(t, \cdot)$ for $t= \tau^-$ and $t=\tau^+$, respectively; 
\item  $\mf{\hat u}$ is the same state as in~\eqref{e:hatu}.  
\end{itemize}
As in the genuinely nonlinear case, it suffices to establish~\eqref{e:deltaFd2}.  We recall that if $\lambda_k (\mf{\hat u})\leq 0$ then there is no wave front of the $k$-th family entering the domain immediately after time $\tau$. This yields $\mf{\bar u}^+ =\mf{\hat u}$ and hence 
$$
    \Delta F(\tau) \leq   | \mf f (\mf v (\mf{\hat u}))-  \mf f (\mf v (\mf{u}^-)) | \leq \unpo \sum_{j=k+1}^N |s'_j| \stackrel{\eqref{e:jd1}}{\leq} \unpo |\mf u_{bm}^+ - \mf{u}^-_{bm}|,
$$
that is~\eqref{e:deltaF2}. We now consider the case $\lambda_k (\mf{\hat u})> 0$ and assume that there is a wave-front of the $k$-th family with strength $s_k$ entering the domain after the interaction. We  point out that 
\begin{equation} \label{e:discindata}
\begin{split}
    \Delta F(\tau) & \leq   | \mf f (\mf v (\mf{\hat u}))-  \mf f (\mf v (\mf{ \bar u}^-)) | +  | \mf f (\mf v (\mf{\hat u}))-  \mf f (\mf v (\underbrace{\mf t_k (\mf{\hat u}, s_k)}_{= \mf {\bar u}^+}) | 
     \! \! \stackrel{\eqref{e:jd1}, \eqref{e:dflux2}}{\leq} \! \! \unpo |\mf u_{bm}^+ - \mf{u}^-_{bm}| + 
         \unpo  |s_k| | \lambda_k (\mf{\hat u})| 
\end{split}
\end{equation}
If $\lambda_k (\mf{\bar u}^-) \leq 0$ then we can control the last term at the right-hand side of~\eqref{e:discindata} as follows: 
\begin{equation*} \begin{split}
    |s_k| | \lambda_k (\mf{\hat u})| =  |s_k| [\lambda_k (\mf{\hat u})]^+ \leq 
    |s_k| | \big( [\lambda_k (\mf{\hat u})]^+ -  [\lambda_k (\mf{\bar u}^-)]^+| \big) \leq \unpo |s_k | |\mf{\hat u} - \mf{\bar u}^-|  \stackrel{\eqref{e:jd1}}{\leq} \unpo |s_k| |\mf u_{bm}^+ - \mf{u}^-_{bm}|
    \end{split}
\end{equation*}
 and by plugging the above inequality into~\eqref{e:discindata} we arrive at~\eqref{e:deltaF2}. If $\lambda_k (\mf{\bar u}^-) > 0$ then we can apply~\eqref{e:protra2} with $\xi_k$ and $\mf{\bar u}^-$ in place of $\xi'_k$ and $\mf{\bar u}$, respectively, and conclude that $\xi_k=0$, which in turn owing to~\eqref{e:jd1} (with $s_k$ in place of $s'_k$) yields $|s_k| \leq \unpo |\mf u_{bm}^+ - \mf{u}^-_{bm}|$ and by plugging this inequality into~\eqref{e:discindata} we eventually arrive at~\eqref{e:deltaF2}. This concludes the proof of Lemma~\ref{t:tvf} in the linearly degenerate case.

\subsubsection{Definition~\ref{d:equiv} in the linearly degenerate case}\label{sss:952}
We now show that the limit function $\mf v$ in~\eqref{e:16bis} satisfies the boundary condition in the sense of Definition~\ref{d:equiv}. We recall that if the $k$-th characteristic field is linearly degenerate  the Lax admissibility requirement in condition i) is redundant because contact discontinuities are always admissible. To ease the exposition, we focus on case B) in Hypothesis~\ref{h:eulerlag}, the analysis in the other cases being similar but easier.
We consider the sequence $\{ \mf{\bar u}_m(t) \}$ defined as in~\eqref{e:baruemme} and point out that
\be \label{e:948}
       \boldsymbol{\beta} (\boldsymbol{\phi} (\mf{\bar u}_m(t), \xi^m_{\ell (\mf u^{\ee_m}_{b}(t)) +1}(t), \dots, \xi^m_k(t)), \mf u^{\ee_m}_{b}(t) )= \mf 0_{N} 
   \quad \text{for a.e. $t \in \R_+$ and every $m \in \mathbb N$}.
\eq
for suitable functions $\xi^m_{\ell (\mf u^{\ee_m}_{b}(t)) +1}(t), \dots, \xi^m_k(t))$. We now fix $t \in \R_+$ such that both~\eqref{e:948} and the pointwise convergence in~\eqref{e:flupointwise} hold true, and such that $\lim_{m \to + \infty} \mf{u}^{\ee_m}_{b} (t) = \mf u_{b}(t)$.  Up to subsequences\footnote{We explicitly point out that, in general, the particular subsequence such that we have convergence \emph{depends} on the point $t$, so we are not stating that the whole function $\mf{\bar u}_m$ converges at almost every point.}, we have 
\be \label{e:grave}
      \lim_{m \to + \infty} \mf{\bar u}_m (t) = \grave{\mf u}(t)
\eq 
for some limit value $\grave{\mf u}(t) \in \R^N$.  Owing to~\eqref{e:flupointwise}, we have
\be \label{e:949}
      \mf f (\mf v(\mf {\bar u} (t)))= \mf f (\mf v(\grave{\mf u} (t))),
\eq
where $\mf{\bar u}$ is the same as in~\eqref{e:traccialimite}.  By recalling~\eqref{e:rh}, formula~\eqref{e:949} implies that either
\be \label{e:950}
      \mf {\bar u} (t) = \grave{\mf u} (t)
\eq
or 
\be \label{e:951}
      \lambda_k (\grave{\mf u} (t))= \lambda_k (\bar{\mf u} (t))=0 \; \text{and} \; \grave{\mf u} (t) = \mf t_k ( \mf {\bar u}  (t), s_k(t)) \; \text{for some $s_k(t) \in [- \delta, \delta]$}. 
\eq
We now turn our attention to~\eqref{e:948}, recall the analysis in~\cite[\S2.2]{BianchiniSpinolo}, and in particular~\cite[formula (2.17)]{BianchiniSpinolo}, and we conclude that $\ell (\mf u^{\ee_m}_{b}(t))$ depends on the sign of a certain continuous function $\alpha$, and that the sign of $\alpha$ only depends on the last $N-h$ components of $\mf u^{\ee_m}_{b}(t)$, where we recall that $h$ denotes the dimension of the kernel of the matrix $\mf B$, see~\eqref{e:B}.  Since $\lim_{t \to + \infty} \mf u^{\ee_m}_{b} (t) = \mf u_b(t)$, up to subsequences one of the following conditions is satisfied: \\
{\sc Case 1:} $\alpha (\mf u^{\ee_m}_{b}(t)) \leq 0$ for every $m$, which implies $\alpha (\mf u_{b}(t)) \leq 0$. Owing to~\cite[formula (2.17)]{BianchiniSpinolo} in this case we have $\ell(\mf u_{bm})= h = \ell(\mf u_{b})$ for every $m$. We can then pass to the limit in~\eqref{e:948} and arrive at 
\be
\label{e:limitdc}
       \boldsymbol{\beta} (\boldsymbol{\phi} (\mf{\grave u}(t), \xi_{\ell (\mf u_{b}(t)) +1}(t), \dots, \xi_k(t)), \mf u_{b}(t) )= \mf 0_{N}.
\eq
for suitable values $\xi_{\ell (\mf u_{b}(t)) +1}(t), \dots, \xi_k(t)$. \\
{\sc Case 1a:} we have~\eqref{e:950}. In this case formula~\eqref{e:limitdc} boils down to 
\be \label{e:953}
     \boldsymbol{\beta} (\boldsymbol{\phi} (\mf{\bar u}(t), \xi_{\ell (\mf u_{b}(t)) +1}(t), \dots, \xi_k(t)), \mf u_{b}(t) )= \mf 0_{N}. 
\eq
We separately consider the following cases.
\begin{itemize}
\item $\lambda_k (\mf{\bar u}(t)))>0$. First, we point out that 
\be \label{e:giusto}
      \xi_k(t)=0 \; \text{if $\lambda_k (\mf{\bar u}(t)))>0$}.
\eq 
To establish~\eqref{e:giusto}, it suffices to recall~\eqref{e:protra2}, infer that $\lambda_k (\mf{\bar u}_m(t))>0$ and thus $\xi^m_k(t)=0$, and then pass to the limit. We then have 
\begin{equation*}
\begin{split}
       \boldsymbol{\phi} &(\mf{\bar u}(t), \xi_{\ell (\mf u_{b}(t)) +1}(t), \dots, \xi_k(t)) 
\stackrel{\eqref{e:giusto}}{=}
       \boldsymbol{\phi} (\mf{\bar u}(t), \xi_{\ell (\mf u_{b}(t)) +1}(t), \dots, \xi_{k-1}(t), 0)
       \\ &
     \stackrel{\eqref{e:bl10}}{=}
         \boldsymbol{\psi}_s (\mf{\bar u}(t), \xi_{\ell (\mf u_{b}(t)) +1}(t), \dots, \xi_{k-1}(t)) +
      \boldsymbol{\psi}_p (\mf{\bar u}(t), \xi_{\ell (\mf u_{b}(t)) +1}(t), \dots, \xi_{k-1}(t), 0)\\ & 
      \stackrel{\eqref{e:p2}}{=}  
      \boldsymbol{\psi}_s (\mf{\bar u}(t), \xi_{\ell (\mf u_{b}(t)) +1}(t), \dots, \xi_{k-1}(t))
\end{split}
\end{equation*}
By the definition~\eqref{e:phiesse} there is a solution of~\eqref{e:bltw} with initial datum given by the value $\boldsymbol{\psi}_s (\mf{\bar u}(t), \xi_{\ell (\mf u_{b}(t)) +1}(t), \dots, \xi_{k-1}(t))$ which converges to $\mf{\bar u}(t)$ as $ y \to + \infty$ and owing to~\eqref{e:953} we conclude that there is a solution of~\eqref{e:andiamoav2} with $\mf{\underline u} = \mf{\bar u}(t)$, which in turn yields condition i) (without the redundant Lax admissibility requirement) and  ii) in Definition~\ref{d:equiv} with
$\mf{\underline u} = \mf{\bar u}(t)$.  
\item $\lambda_k (\mf{\bar u}(t)))=0$. We have 
\begin{equation*}
\begin{split}
       \boldsymbol{\phi} &(\mf{\bar u}(t), \xi_{\ell (\mf u_{b}(t)) +1}(t), \dots, \xi_k(t))=        
     \stackrel{\eqref{e:bl10}}{=}
         \boldsymbol{\psi}_s (\boldsymbol{\zeta}_k(\mf{\bar u}(t), \xi_k(t)), \xi_{\ell (\mf u_{b}(t)) +1}(t), \dots, \xi_{k-1}(t)) \\ & +
      \boldsymbol{\psi}_p (\mf{\bar u}(t), \xi_{\ell (\mf u_{b}(t)) +1}(t), \dots, \xi_{k-1}(t), \xi_k(t))\\ & 
      \stackrel{\eqref{e:p2}, \ \lambda_k (\mf{\bar u}(t)))=0}{=}  
      \boldsymbol{\psi}_s (\boldsymbol{\zeta}_k(\mf{\bar u}(t), \xi_k(t)), \xi_{\ell (\mf u_{b}(t)) +1}(t), \dots, \xi_{k-1}(t))
\end{split}
\end{equation*}
and by definition of $ \boldsymbol{\psi}_s $ the above equality combined with~\eqref{e:953} dictates that there is a solution of~\eqref{e:andiamoav2} which converges to $\boldsymbol{\zeta}_k(\mf{\bar u}(t), \xi_k(t))$ as $y\to + \infty$. This yields conditions i) (without the redundant Lax admissibility requirement) and ii) in Definition~\ref{d:equiv} with $\mf{\underline u}= \boldsymbol{\zeta}_k(\mf{\bar u}(t), \xi_k(t))$. 
\item  $\lambda_k (\mf{\bar u}(t)))< 0$. In this case we recall~\eqref{e:bl00} and that $\boldsymbol{\omega}(0)$ and $\boldsymbol{\omega}_0(0)$ are as in~\eqref{e:festanatale} and~\eqref{e:festanatale2}, respectively. We also use the equality $\boldsymbol{\gamma}_k (\mf{\bar u}(t), \xi_k(t)) = \mf b_k (\mf{\bar u}(t), \xi_k(t))$, which follows from~\eqref{e:servepure}. Note that, by construction, $\boldsymbol{\omega}(0)$ is the initial datum for an orbit of~\eqref{e:bltw} that has the same asymptotic state as the orbit with initial datum $\boldsymbol{\omega}_0(0)$. Owing to~\eqref{e:allafineserve}, the $\lim_{y \to + \infty } \boldsymbol{\pi}_{\mf u} (\boldsymbol{\omega}_0(y))= \mf{\bar u}(t)$. Wrapping up, we conclude that there is a solution of~\eqref{e:bltw} with initial datum given by the value $\boldsymbol{\phi} (\mf{\bar u}(t), \xi_{\ell (\mf u_{b}(t)) +1}(t), \dots, \xi_k(t))$ converging to $\mf{\bar u}(t)$ as $y \to + \infty$. This yields conditions i) (without the redundant Lax admissibility requirement) and ii) in Definition~\ref{d:equiv} with $\mf{\underline u}= \mf{\bar u}(t)$. 
\end{itemize}
{\sc Case 1b:} we have~\eqref{e:951}. By arguing as in the second item of {\sc Case 1a} this yields  
\begin{equation*}
\begin{split}
      \boldsymbol{\phi} & (\mf{\grave u}(t), \xi_{\ell (\mf u_{b}(t)) +1}(t), \dots, \xi_k(t))=        
      \stackrel{\eqref{e:bl10}\eqref{e:p2}, \ \lambda_k (\mf{\grave u}(t)))=0}{=}  
      \boldsymbol{\psi}_s (\mf t_k(\mf{\grave u}(t), \xi_k(t)), \xi_{\ell (\mf u_{b}(t)) +1}(t), \dots, \xi_{k-1}(t)) \\ & 
     \stackrel{\eqref{e:951}}{=}
       \boldsymbol{\psi}_s (\mf t_k(\mf{\bar u}(t), s_k(t)+ \xi_k(t)), \xi_{\ell (\mf u_{b}(t)) +1}(t), \dots, \xi_{k-1}(t))
\end{split}
\end{equation*}
and owing to~\eqref{e:limitdc} this yields conditions i) (without the redundant Lax admissibility requirement) and ii) in Definition~\ref{d:equiv} with $\mf{\underline u}= \mf t_k(\mf{\bar u}(t), s_k(t) + \xi_k(t))$. \\
{\sc Case 2:} $\alpha (\mf u_{bm}(t)) > 0$ for every $m$ and $\alpha (\mf u_{b}(t)) > 0$. In this case  $\ell(\mf u_{bm}(t))= 0= \ell(\mf u_{b}(t))$ for every $m$ and by arguing as in the previous case we get Definition~\ref{d:equiv} without the redundant Lax admissibility requirement in condition i). \\
{\sc Case 3:} $\alpha (\mf u_{bm}(t)) > 0$ for every $m$ and $\alpha (\mf u_{b}(t)) =0$. In this case  $\ell(\mf u_{bm}(t))= 0$ for every $m$, but $\ell(\mf u_{b}(t))=h$. This means that the number of boundary condition imposed on the boundary layers drops in the limit. To handle this case we have to recall some facts from the analysis in~\cite{BianchiniSpinolo}. 
\begin{itemize}
\item First, we recall the explicit expression of the function $\boldsymbol{\beta}$ (see~\cite[(2.17)]{BianchiniSpinolo}), and that $h$ denotes the dimension of the kernel of the matrix $\mf B$ in~\eqref{e:B}. We conclude that since $\alpha (\mf u^{\ee_m}_{b}(t)) > 0$ formula~\eqref{e:948} 
boils down to 
\be \label{e:9482} 
       \boldsymbol{\phi} (\mf{\bar u}_m(t), \xi^m_{1}(t), \dots, \xi^m_k(t))= \mf u^{\ee_m}_{b}(t)    \quad \text{for every $m \in \mathbb N$}.
\eq
We extract a converging subsequence $\{ \xi^m_{1}(t), \dots, \xi^m_k(t) \}_{m \in \mathbb N}$ and by passing to the limit in~\eqref{e:9482} and recalling~\eqref{e:grave} we arrive at 
\be \label{e:94823} 
       \boldsymbol{\phi} (\mf{\grave u}(t), \xi_{1}(t), \dots, \xi_k(t))= \mf u_{b}(t).
\eq 
\item Next, we briefly overwiew the construction of the map $ \boldsymbol{\phi}$ given in the proof of~\cite[Lemma 6.1]{BianchiniSpinolo}\footnote{Note that the map $\boldsymbol{\phi}$ in~\eqref{e:94823} is given by the projection onto the first $N$ components of the map $\boldsymbol{\widetilde{\psi}}_b$ defined in the statement of~\cite[Lemma 6.1]{BianchiniSpinolo}.}
To construct $ \boldsymbol{\phi}$ one considers system~\cite[(1.14)]{BianchiniSpinolo}, where $\mf{h}$ is given by~\cite[(8.1)]{BianchiniSpinolo}, and poses $\sigma=0$, which yields\footnote{Note that in~\cite[\S6]{BianchiniSpinolo} one assumes $h=1$, and that the case $h>1$ is handled in~\cite[\S8]{BianchiniSpinolo}.}
\begin{equation}\label{e:accaqui}
\left\{
\begin{array}{ll}
   \dot{\mf{u}}_1= -  \mf E_{11}^{-1} \mf A_{21}^t \mf z_2 \\
    \dot{\mf u}_2 = \alpha \mf{z}_2 \\
   \dot{\mf z}_2=  
   \mf B^{-1}_{22} \displaystyle{\big[ 
   ( \mf G_1 -\mf A_{21} ) \mf E_{11}^{-1} \mf A_{21}^t  + \alpha [\mf A_{22}  - \mf G_2]  \big] \mf z_2}. \\
\end{array}
\right. \quad \mf u_1 \in \R^h, \; \mf u_2 \in \R^{N-h}, \; \mf z_2 \in \R^{N-h},
\end{equation}
where the functions $\mf E_{11}$, $\mf A_{21}$, $\mf B_{22}$, $\mf G_1$, $\mf A_{22}$, $\mf G_2$ are the same as in Hypothesis~\ref{h:normal}. In what follows, we will systemtically use the notation
$$
    \mf u = \left(
    \begin{array}{cc}
    \mf u_1 \\
     \mf u_2 \\
    \end{array}
  \right) \quad \mf u_1 \in \R^h, \; \mf u_2 \in \R^{N-h}. 
$$
In the proof of~\cite[Lemma 6.1]{BianchiniSpinolo} one applies the Slaving Manifold Lemma (see~\cite[\S11]{BianchiniSpinolo}) and constructs the manifold $\boldsymbol{\phi}$ in such a way that, under~\eqref{e:94823}, there is a solution $(\mf u_1(y), \mf u_2(y), \mf z_2(y))$ of~\eqref{e:accaqui} such that $(\mf u_1(0), \mf u_2(0))= \mf u_b$, and 
\be \label{e:doveva}
   \lim_{y \to + \infty} | (\mf u_1(y), \mf u_2 (y)) - ( \mf w_1(y), \mf w_2 (y))| =0,
\eq
where $(\mf w_1(y), \mf w_2(y))$ denote the first $N$ components of a suitable orbit of~\eqref{e:accaqui}, lying on the manifold $\mathcal M^0$ defined in~\cite[\S3]{BianchiniSpinolo}
and satisfying 
$(\mf w_1 (0), \mf w_2 (0))= \boldsymbol{\psi}_{sl} (\mf{\grave u}(t), \xi_{h}(t), \dots, \xi_k(t))$, where the function $ \boldsymbol{\psi}_{sl} $ is defined in~\cite[Theorem 5.2]{BianchiniSpinolo}.
\item We now recall that, owing to~\cite[(2.12)]{BianchiniSpinolo}, if $\alpha (\mf u(0)) =0$ and $(\mf u(y), \mf z_2(y))$ is a solution of~\eqref{e:accaqui} then $\alpha (\mf u(y)) =0$ for every $y$. Since $\alpha(\mf u_b(t)) =0$ by assumption, then by combining the second line of~\eqref{e:accaqui} with~\eqref{e:doveva} and the identities $(\mf u_1(0), \mf u_2(0))= \mf u_b$ and $(\mf w_1 (0), \mf w_2 (0))= \boldsymbol{\psi}_{sl} (\mf{\grave u}(t), \xi_{h}(t), \dots, \xi_k(t))$
we eventually arrive at 
\be \label{e:pi2}
   \mf p_{\mf u_2} (\boldsymbol{\psi}_{sl} (\mf{\grave u}(t), \xi_{h}(t), \dots, \xi_k(t)))=
   \mf p_{\mf u_2} (\mf u_b(t)),
\eq
where $\mf p_{\mf u_2}(\mf u)$ represents the projection of $\mf u$ onto the last $N-h$ components.
\item We now use again~\cite[(2.17)]{BianchiniSpinolo} and conclude that, since $\alpha (\mf u_b(t)) =0$, then~\eqref{e:pi2} is equivalent to 
\be \label{e:fine?}
      \boldsymbol{\beta}(\boldsymbol{\psi}_{sl} (\mf{\grave u}(t), \xi_{h}(t), \dots, \xi_k(t)),  \mf u_{b}(t)) = \mf 0_N. 
\eq
By using the properties of the function $\boldsymbol{\psi}_{sl} $ given in the statement of~\cite[Theorem 5.2]{BianchiniSpinolo} we can then argue as in {\sc Case 1} above and conclude the proof of Theorem~\ref{t:maintr}.  
\end{itemize}
\section*{Acknowledgments}
All the authors are members of the GNAMPA group of INDAM and of the PRIN 
Project 20204NT8W4 (PI Stefano Bianchini). FA and LVS are also members of the PRIN 2022 PNRR Project P2022XJ9SX (PI Roberta Bianchini). LVS is also member of the PRIN Project 2022YXWSLR (PI Paolo Antonelli). Part of this work was done when LVS was visiting the Department of Mathematics of the University of Padua: its kind hospitality is gratefully acknowledged. 
\bibliographystyle{plain}
\bibliography{wft}
\end{document}